\documentclass[reqno,11pt]{amsart}

\usepackage[top=2.0cm,bottom=2.0cm,left=3cm,right=3cm]{geometry}
\usepackage{amsthm,amsmath,amssymb,dsfont}
\usepackage{mathrsfs,amsfonts,functan,extarrows,mathtools}
\usepackage[colorlinks]{hyperref}

\usepackage{marginnote}
\usepackage{xcolor}
\usepackage{stmaryrd}
\usepackage{esint}
\usepackage{graphicx}
\usepackage{bm}
\usepackage{lipsum}

\newtheorem{theorem}{Theorem}[section]
\newtheorem{proposition}{Proposition}[section]

\newtheorem{remark}{Remark}[section]
\newtheorem{lemma}{Lemma}[section]

\numberwithin{equation}{section}
\allowdisplaybreaks

\def\d{\mathrm{d}}
\def\no{\nonumber}
\def\R{\mathbb{R}}
\def\T{\mathbb{T}}

\def\eps{\varepsilon}
\def\div{\mathrm{div}}
\def\u{\mathrm{u}}
\def\dr{\mathrm{d}}
\def\A{\mathrm{A}}
\def\B{\mathrm{B}}
\def\l{\left\langle}
\def\r{\right\rangle}
\def\J{\mathcal{J}}
\def\D{\mathrm{D}}

\newcounter{wronumber}

\begin{document}
\title[zero inertia limit of Ericksen-Leslie's model]
			{The zero inertia limit of Ericksen-Leslie's model for liquid crystals}

\author[Ning Jiang]{Ning Jiang}
\address[Ning Jiang]{\newline School of Mathematics and Statistics, Wuhan University, Wuhan, 430072, P. R. China}
\email{njiang@whu.edu.cn}

\author[Yi-Long Luo]{Yi-Long Luo}
\address[Yi-Long Luo]
		{\newline School of Mathematics and Statistics, Wuhan University, Wuhan, 430072, P. R. China}
\email{yl-luo@whu.edu.cn}


\maketitle

\begin{abstract}
  In this paper we study the zero inertia limit that is from the hyperbolic to parabolic Ericksen-Leslie's liquid crystal flow. By introducing an initial layer and constructing an energy norm and energy dissipation functional depending on the solutions of the limiting system, we derive a global in time uniform energy bound to the remainder system under the small size of the initial data. \\

  \noindent\textsc{Keywords.} Zero inertia limit; Ericksen-Leslie model; Uniform energy bounds; Initial layer \\

  \noindent\textsc{AMS subject classifications.}  35L05; 35L30; 35L81; 37D50; 80A20
\end{abstract}



\tableofcontents


\section{Introduction and main theorem}

The hydrodynamic theory of incompressible liquid crystals was established by Ericksen \cite{Ericksen-1961-TSR, Ericksen-1987-RM, Ericksen-1990-ARMA} and Leslie \cite{Leslie-1968-ARMA, Leslie-1979} in the 1960's (see also Section 5.1 of \cite{Lin-Liu-2001} ). In this paper, we study the following hyperbolic Ericksen-Leslie's incompressible liquid crystal flow in $(t,x) \in \R^+ \times \T^3$
\begin{equation}\label{HLQ}
  \left\{
    \begin{array}{c}
      \partial_t \u^\eps + \u^\eps \cdot \nabla \u^\eps - \tfrac{1}{2} \mu_4 \Delta \u^\eps + \nabla p^\eps = - \div ( \nabla \dr^\eps \odot \nabla \dr^\eps ) + \div \bm{\sigma}^\eps \,, \\[3mm]
      \div \u^\eps = 0 \,, \\[3mm]
      \eps \D^2_{\u^\eps} \dr^\eps = \Delta \dr^\eps + \bm{\gamma}^\eps \dr^\eps + \lambda_1 ( \D_{\u^\eps} \dr^\eps + \B^\eps \dr^\eps ) + \lambda_2 \A^\eps \dr^\eps \,, \\[3mm]
      |\dr^\eps| = 1 \,,
    \end{array}
  \right.
\end{equation}
where the Lagrangian multiplier $\bm{\gamma}^\eps$ (for the geometric constraint $|\dr^\eps|=1$) is
\begin{equation}\label{gamma-eps}
  \bm{\gamma}^\eps = - \eps | \D_{\u^\eps} \dr^\eps |^2 + |\nabla \dr^\eps|^2 - \lambda_2 \A^\eps : \dr^\eps \otimes \dr^\eps \,,
\end{equation}
and the extra stress $\bm{\sigma}^\eps$ is
\begin{equation}\label{sigma}
  \begin{aligned}
    \bm{\sigma}^\eps \equiv \bm{\sigma} (\u^\eps , \dr^\eps) = & \mu_1 ( \A^\eps : \dr^\eps \otimes \dr^\eps ) \dr^\eps \otimes \dr^\eps + \mu_2 ( \D_{\u^\eps} \dr^\eps + \B^\eps \dr^\eps ) \otimes \dr^\eps \\
    + & \mu_3 \dr^\eps \otimes ( \D_{\u^\eps} \dr^\eps + \B^\eps \dr^\eps ) + \mu_5 (\A^\eps \dr^\eps) \otimes \dr^\eps + \mu_6 \dr^\eps \otimes (\A^\eps \dr^\eps) \,.
  \end{aligned}
\end{equation}
Here $\u^\eps (t,x) \in \R^3$ is the bulk velocity, $\dr^\eps (t,x) \in \mathbb{S}^2$ is the unit director fields of the liquid molecules, $p^\eps (t,x) \in \R$ is the pressure. The notation $\D_{\u^\eps} \dr^\eps$ denotes the first order material derivative of $\dr^\eps$ with respect to the velocity $\u^\eps$ by
\begin{equation}
  \D_{\u^\eps} \dr^\eps = \partial_t \dr^\eps + \u^\eps \cdot \nabla \dr^\eps \,,
\end{equation}
and  $\D^2_{\u^\eps} \dr^\eps$ represents the second order material derivative,
\begin{equation}
  \D^2_{\u^\eps} \dr^\eps = \D_{\u^\eps} ( \D_{\u^\eps} \dr^\eps ) \,.
\end{equation}
 In this paper we work on the periodic spatial domain $\T^3 = \R^3 / \mathbb{L}^3$, where $\mathbb{L}^3 \subset \R^3$ is any $3$-dimensional lattice. The notations $\A^\eps = \frac{1}{2} ( \nabla \u^\eps + \nabla \u^\eps{}^\top )$ and $\B^\eps = \frac{1}{2} ( \nabla \u^\eps - \nabla \u^\eps{}^\top )$ represent the rate of strain tensor, skew-symmetric part of the strain rate of by fluid velocity, respectively. More precisely, the entries of $\A^\eps$ and $\B^\eps$ are given as
\begin{equation}
  \A^\eps_{ij} = \tfrac{1}{2} ( \partial_j \u^\eps_i + \partial_i \u^\eps_j ) \,, \ \B^\eps_{ij} = \tfrac{1}{2} ( \partial_j \u^\eps_i - \partial_i \u^\eps_j )
\end{equation}
for $1 \leq i,j \leq 3$. One notices $\B^\eps_{ij} = - \B^\eps_{ji}$. The components of the vector $\B^\eps \dr^\eps$ and $\A^\eps \dr^\eps$ are $(\B^\eps \dr^\eps)_i = \B^\eps_{ki} \dr^\eps_k$ and $(\A^\eps \dr^\eps)_i = \A^\eps_{ki} \dr^\eps_k$, respectively. The entries of the matrix $\nabla \dr^\eps \odot \nabla \dr^\eps$ are $(\nabla \dr^\eps \odot \nabla \dr^\eps)_{ij} = \partial_i \dr^\eps_k \partial_j \dr^\eps_k$ and the symbol $a \otimes b$ means $(a \otimes b)_{ij} = a_i b_j$ for $1 \leq i, j \leq 3$. For any two matrix $M, N \in \R^{3 \times 3}$, we denote by $M:N = M_{ij} N_{ij}$. Furthermore, the symbol $\div M$ means a vector field in $\R^3$ with the components $(\div M)_i = \partial_j M_{ij}$ for $1 \leq  i \leq 3$. We emphasize that the Einstein summation convention is used throughout this paper.

The parameter $\eps > 0$ is the inertia constant, which is usually small in the physical experiments. $\mu_4 > 0$ is the viscosity of the flow. The material coefficients $\lambda_1 \leq 0$ and $\lambda_2 \in \R$ reflect the molecular shape and the slippery part between the fluid and the particles. The coefficients $\mu_i (i = 1,2,3,5,6 )$, which may depend on material and temperature, are usually called Leslie coefficients, and are related to certain local correlations in the fluid. Moreover, the previous coefficients have the relations $\mu_4 > 0$ and
\begin{equation}\label{Coefficients-Relations}
\lambda_1=\mu_2-\mu_3\,, \quad\lambda_2 = \mu_5-\mu_6\,,\quad \mu_2+\mu_3 = \mu_6-\mu_5\,.
\end{equation}
The first two relations are necessary conditions in order to satisfy the equation of motion identically, while the third relation is called {\em Parodi's relation}, which is derived from Onsager reciprocal relations expressing the equality of certain relations between flows and forces in thermodynamic systems out of equilibrium.

For the system \eqref{HLQ} we take the initial data independent of $\eps$, i.e.,
\begin{equation}\label{IC-HyperLQ}
  \begin{aligned}
    \u^\eps (0,x) = \u^{in} (x) \in \R^3 \,, \ \dr^\eps (0,x) = \dr^{in} (x) \in \mathbb{S}^2 \,, ( \D_{\u^\eps} \dr^\eps ) (0,x) = \tilde{\dr}^{in} (x) \in \R^3
  \end{aligned}
\end{equation}
with the compatibilities
\begin{equation}\label{IC-compatibility}
  \begin{aligned}
    \div \u^{in} (x) = 0 \,, \ \dr^{in} (x) \cdot \tilde{\dr}^{in} (x) = 0 \,.
  \end{aligned}
\end{equation}

\subsection{Initial layer vs well-prepared initial data}
Formally letting $\u^\eps \rightarrow \u_0$ and $\dr^\eps \rightarrow \dr_0$ as $\eps \rightarrow 0$ in the hyperbolic liquid crystal system \eqref{HLQ} deduces to the parabolic liquid crystal model
\begin{equation}\label{PLQ}
  \left\{
    \begin{array}{c}
      \partial_t \u_0 + \u_0 \cdot \nabla \u_0 - \tfrac{1}{2} \mu_4 \Delta \u_0 + \nabla p_0 = - \div ( \nabla \dr_0 \odot \nabla \dr_0 ) + \div \bm{\sigma}_0 \,, \\[2mm]
      \div \u_0 = 0 \,, \\[2mm]
      -\lambda_1 (\D_{\u_0} \dr_0 + \B_0 \dr_0 ) = \Delta \dr_0 + \bm{\gamma}_0 \dr_0 + \lambda_2 \A_0 \dr_0 \,, \\[2mm]
      |\dr_0| = 1 \,,
    \end{array}
  \right.
\end{equation}
where the Lagrangian multiplier $\bm{\gamma}_0$ is
\begin{equation}
  \begin{aligned}
    \bm{\gamma}_0 = |\nabla \dr_0|^2 - \lambda_2 \A_0 : \dr_0 \otimes \dr_0 \,,
  \end{aligned}
\end{equation}
and the extra stress $\bm{\sigma}_0$ is
\begin{equation}
  \begin{aligned}
    \bm{\sigma}_0 \equiv & \bm{\sigma} (\u_0 , \dr_0) \,,
  \end{aligned}
\end{equation}
i.e. replacing $(\u^\eps , \dr^\eps)$ by $(\u_0 , \dr_0)$ in \eqref{sigma}. Furthermore, the initial data of the limit system \eqref{PLQ} will naturally be
\begin{equation}\label{IC-ParaLQ}
  \begin{aligned}
    \u_0 (0,x) = \u^{in} (x) \in \R^3 \,, \ \dr_0 (0,x) = \dr^{in} (x) \in \mathbb{S}^2 \,.
  \end{aligned}
\end{equation}

The limit considered in this paper is a limit from a hyperbolic-type system for $\eps > 0$ to a parabolic system for $\eps = 0$. One notices that $\dr^\eps$-equation in \eqref{HLQ} is a system of wave equations with two initial conditions, while the $\dr_0$-equation in \eqref{PLQ} is a  parabolic system with only a single initial condition. In general, the solution $( \u_0, \dr_0)$ to the limit system \eqref{PLQ}-\eqref{IC-ParaLQ} does not satisfy the third initial condition in \eqref{IC-HyperLQ}. To overcome this disparity, we will take two ways: 1) introduce an {\em initial layer} in times; 2) give a well-prepared initial condition.

{\bf Initial layer.} This disparity between the initial conditions of the hyperbolic type system \eqref{HLQ} and of the parabolic type system \eqref{PLQ} indicates that one should expect an {\em initial layer} in time, appearing in the limit process $\eps \rightarrow 0$. Specifically, this disparity, denoted by $\D^{in} (x)$, between the initial conditions \eqref{IC-HyperLQ} and \eqref{IC-ParaLQ} is defined as
\begin{equation}\label{Differ-IC}
\begin{aligned}
\D^{in} : = \tilde{\dr}^{in} - \D_{\u_0} \dr_0 |_{t=0} = \tilde{\dr}^{in} + \B^{in} \dr^{in} + \tfrac{1}{\lambda_1} ( \Delta \dr^{in} + \bm{\gamma}_0^{in} \dr^{in} + \lambda_2 \A^{in} \dr^{in} ) \,.
\end{aligned}
\end{equation}
Here $\A^{in} = \tfrac{1}{2} ( \nabla \u^{in} + (\nabla \u^{in})^\top )$, $\B^{in} = \tfrac{1}{2} ( \nabla \u^{in} - (\nabla \u^{in})^\top )$ and $\bm{\gamma}_0^{in} = |\nabla \dr^{in}|^2 - \lambda_2 \A^{in} : \dr^{in} \otimes \dr^{in}$.

We will justify rigorously this limit by employing the Hilbert expansion method in which the leading term is given by solutions to the limit system \eqref{PLQ}. In this approach, the solution to the limit system \eqref{PLQ} is known beforehand. Then a special class of solutions to the original system \eqref{HLQ} can be constructed around the limit system. A key of this approach is to construct a correct ansatz of the solutions to the original system. \eqref{HLQ}. Besides the limit system \eqref{PLQ} and the remainder term $(\u_R^\eps, \dr_R^\eps)$, an initial layer $ \eps^\beta \dr_I ( \tfrac{t}{\eps^\beta}, x ) $ to adsorb the disparity $\D^{in} (x)$ should be included in the ansatz. More precisely, we take the following ansatz of the solution $( \u^\eps , \dr^\eps)$ to the system \eqref{HLQ}
\begin{equation}\label{Ansatz}
\begin{aligned}
\u^\eps (t,x) = \u_0 (t,x) + \sqrt{\eps} \u_R^\eps (t,x) \,, \ \dr^\eps (t,x) = \dr_0 (t,x) + \eps^\beta \dr_I ( \tfrac{t}{\eps^\beta}, x ) + \sqrt{\eps} \dr_R^\eps (t,x)
\end{aligned}
\end{equation}
for a fixed $\beta > 0$ to be determined, where $(\u_0 , \dr_0)$ is the solution to \eqref{PLQ} with initial data \eqref{IC-ParaLQ}.

It is easy to derive by plugging the expansions \eqref{Ansatz} into the system \eqref{HLQ} that the leading relation is
\begin{equation*}
   \begin{aligned}
     \eps^{1-\beta} \partial^2_{\tau \tau} \dr_I - \lambda_1 \partial_\tau \dr_I = \eps^\beta \Delta \dr_I \,,
   \end{aligned}
\end{equation*}
where $\tau = \frac{t}{\eps^\beta}$. Then we can design the initial layer satisfying the following linear damped wave system on $(\tau, x) \in \R^+ \times \T^3$ (called {\em initial layer system}):
\begin{equation}\label{Initial-layer-equ}
  \left\{
    \begin{array}{l}
      \partial^2_{\tau \tau} \dr_I + \tfrac{-\lambda_1}{\eps^{1-\beta}} \partial_\tau \dr_I = \eps^{2 \beta - 1} \Delta \dr_I \,, \\ [2mm]
      \dr_I (\infty , x) = \lim\limits_{\tau \rightarrow \infty} \dr_I (\tau, x) = 0 \,, \\ [2mm]
      \partial_\tau \dr_I (0,x) = \D^{in} (x) \,.
    \end{array}
  \right.
\end{equation}
We emphasize that if $\beta > \tfrac{1}{2}$, the Laplacian term $ \eps^{2 \beta - 1 } \Delta \dr_I $ is also a higher order term as $\eps \rightarrow 0$., which means that it can be ignored. For instance, in Jiang-Luo-Tang-Zarnescu's work \cite{Jiang-Luo-Tang-Zarnescu-2019-CMS} to justify this limit corresponding to the background velocity $\u^\eps \equiv 0$, the value of $\beta$ is taken as $\beta=1$, and ignored the term $\eps \Delta \dr_I$. For completeness and generality (for example, if we consider the case $\lambda_1=0$ in a forthcoming separate paper, the term  $\eps^{2 \beta - 1} \Delta \dr_I$ must kept, and $\beta$ should be taken as $1/2$), we keep this term in the initial layer structure here. Then, we can easily solve \eqref{Initial-layer-equ}:
\begin{equation}\label{Initial-layer-sol}
  \begin{aligned}
    \dr_I (\tau, x) = 2 \eps^{1-\beta} \Big( \lambda_1 - \sqrt{\lambda_1^2 + 4 \eps \Delta} \Big)^{-1} \exp \Big( \tfrac{\lambda_1 - \sqrt{\lambda_1^2 + 4 \eps \Delta}}{2 \eps^{1-\beta} } \tau \Big) \D^{in} (x)\,.
  \end{aligned}
\end{equation}
 Here the operator $ \Big( \lambda_1 - \sqrt{\lambda_1^2 + 4 \eps \Delta} \Big)^{-1} \exp \Big( \tfrac{\lambda_1 - \sqrt{\lambda_1^2 + 4 \eps \Delta}}{2 \eps^{1-\beta} } \tau \Big) $ is understood in the sense of Fourier multiplier. As a consequence, the initial layer $ \dr_I $ designed in \eqref{Initial-layer-sol} can be explicitly presented as
\begin{equation}\label{Initial-Layer}
  \begin{aligned}
    \eps^\beta \dr_I ( \tfrac{t}{\eps^\beta} , x ) = \eps \D_I^\eps (t,x) : = 2 \eps \Big( \lambda_1 - \sqrt{\lambda_1^2 + 4 \eps \Delta} \Big)^{-1} \exp \Big( \tfrac{\lambda_1 - \sqrt{\lambda_1^2 + 4 \eps \Delta}}{2 \eps } t \Big) \D^{in} (x) \,.
  \end{aligned}
\end{equation}
One observes that $\eps^\beta \dr_I ( \tfrac{t}{\eps^\beta} , x )$ is independent of $\beta > 0$. In other words, this initial layer is an intrinsic structure resulted from the disparity $\D^{in} (x)$ of the initial conditions between the original system \eqref{HLQ} and the limit equations \eqref{PLQ}. Since the real part $\Re e \big( \lambda_1 - \sqrt{\lambda_1^2 - 4 \eps |k|^2} \big) < 0$ for all $k \in \mathbb{Z}^3$, both the initial layer $ \D_I^\eps (t,x) $ and its time derivative $\partial_t \D_I^\eps (t,x)$ exponentially decay to zero as $\eps \rightarrow 0$ for every $(t,x) \in \R^+ \times \T^3$. This means that the disparity $\D^{in} (x)$ just affects the whole evolution process in a very short beginning time.

We now write the ansatz \eqref{Ansatz} as the form
\begin{equation}\label{Ansatz-1}
\begin{aligned}
\u^\eps (t,x) = \u_0 (t,x) + \sqrt{\eps} \u_R^\eps (t,x) \,, \ \dr^\eps (t,x) = \dr_0 (t,x) + \eps \D_I^\eps (t,x) + \sqrt{\eps} \dr_R^\eps (t,x) \,.
\end{aligned}
\end{equation}
Then, plugging the ansatz \eqref{Ansatz-1} into the original system \eqref{HLQ} implies that the remainder $(\u_R^\eps , \dr_R^\eps)$ satisfies the following system
\begin{equation}\label{Remainder-u-d}
  \left\{
    \begin{array}{c}
      \partial_t \u_R^\eps - \tfrac{1}{2} \mu_4 \Delta \u_R^\eps + \nabla p_R^\eps = \mu_1 \div \big[ (\A_R^\eps : \dr_0 \otimes \dr_0) \dr_0 \otimes \dr_0 \big] \\
      \qquad\qquad + \mathcal{K}_{\u} + \div ( \mathcal{C}_{\u} + \mathcal{T}_{\u} + \sqrt{\eps} \mathcal{R}_{\u} ) + \eps \div \mathcal{Q}_{\u}(\D_I) \,,\\[2mm]
      \div \u_R^\eps = 0 \,, \\[2mm]
      \D^2_{\u_0 + \sqrt{\eps} \u_R^\eps} \dr_R^\eps + \tfrac{- \lambda_1}{\eps} \D_{\u_0 + \sqrt{\eps} \u_R^\eps} \dr_R^\eps - \tfrac{1}{\eps} \Delta \dr_R^\eps + \partial_t ( \u_R^\eps \cdot \nabla \dr_0 + \sqrt{\eps} \u_R^\eps \cdot \nabla \D_I^\eps ) \\
      = \tfrac{1}{\eps} \mathcal{C}_{\dr} + \tfrac{1}{\eps} \mathcal{S}^1_{\dr} + \tfrac{1}{\sqrt{\eps}} \mathcal{S}^2_{\dr} + \mathcal{R}_{\dr} + \mathcal{Q}_{\dr}(\D_I)
    \end{array}
  \right.
\end{equation}
with the constraint
\begin{equation}\label{Constraint-1}
  \begin{aligned}
    2 \dr_0 \cdot ( \dr_R^\eps + \sqrt{\eps} \D_I^\eps ) + \sqrt{\eps} | \dr_R^\eps + \sqrt{\eps} \D_I^\eps |^2 = 0 \,,
  \end{aligned}
\end{equation}
where the tensor $\mathcal{C}_{\u}$ is
\begin{equation}\label{C-u}
  \begin{aligned}
    \mathcal{C}_{\u} = & \mu_2 ( \D_{\u_0 + \sqrt{\eps} \u_R^\eps} \dr_R^\eps + \B_R^\eps \dr_0) \otimes \dr_0 + \mu_3 \dr_0 \otimes ( \D_{\u_0 + \sqrt{\eps} \u_R^\eps} \dr_R^\eps + \B_R^\eps \dr_0 ) \\
    & + \mu_5 (\A_R^\eps \dr_0) \otimes \dr_0 + \mu_6 \dr_0 \otimes ( \A_R^\eps \dr_0  ) \,,
  \end{aligned}
\end{equation}
the vector field $\mathcal{C}_{\dr}$ is
\begin{equation}\label{C-d}
  \mathcal{C}_{\dr} = \lambda_1 \B_R^\eps \dr_0 + \lambda_2 \A_R^\eps \dr_0 \,,
\end{equation}
and the tedious terms $\mathcal{T}_{\u}$, $\mathcal{K}_{\u}$, $\mathcal{S}^1_{\dr}$, $\mathcal{S}^2_{\dr}$,  $\mathcal{R}_{\dr}$, $\mathcal{R}_{\u}$, $\mathcal{Q}_{\u} (\D_I)$ and $\mathcal{Q}_{\dr} (\D_I) $ are defined as the forms \eqref{Tu}, \eqref{Lu}, \eqref{Sd-1}, \eqref{Sd-2}, \eqref{Rd}, \eqref{Ru}, \eqref{Q-u} and \eqref{Q-d} in Appendix \ref{Appendix} respectively.

Next we consider the initial conditions of the remainder system \eqref{Remainder-u-d}. Since our goal is to seek a solution to \eqref{HLQ} with the form \eqref{Ansatz-1}, the initial data of $(\u_R^\eps, \dr_R^\eps)$ should subject to
\begin{equation}\label{IC-relations}
  \left\{
    \begin{array}{l}
      \sqrt{\eps} \u_R^\eps (0,x) = \u^\eps (0,x) - \u_0 (0,x) \,, \\
      \sqrt{\eps} \dr_R^\eps (0,x) = \dr^\eps (0,x) - \dr_0 (0,x) - \eps \D_I^\eps (0,x) \,, \\
      \sqrt{\eps} ( \D_{\u_0 + \sqrt{\eps} \u_R^\eps} \dr_R^\eps ) (0, x) = ( \D_{\u^\eps} \dr^\eps ) (0,x) - ( \D_{\u_0} \dr_0 ) (0,x) - \sqrt{\eps} ( \u_R^\eps \cdot \nabla \dr_0 ) (0,x) \\
      \qquad \qquad \qquad \qquad \qquad - \eps \partial_t \D_I^\eps ( 0, x ) - \eps ( \u_0 \cdot \nabla \D_I^\eps ) (0,x) - \sqrt{\eps}^3 ( \u_R^\eps \cdot \nabla \D_I^\eps ) (0,x) \,.
    \end{array}
  \right.
\end{equation}
Recalling that the initial conditions of the original system \eqref{HLQ} and the limit system \eqref{PLQ} satisfy
\begin{equation}
  \begin{aligned}
    \u^\eps (0,x) = \u_0 (0,x) = \u^{in} (x) \,, \ \dr^\eps (0,x) = \dr_0 (0,x) = \dr^{in} (x) \,, (\D_{\u^\eps} \dr^\eps) (0,x) = \tilde{\dr}^{in} (x)
  \end{aligned}
\end{equation}
and the initial data of the initial layer $\dr_I$ in \eqref{Initial-layer-equ} is imposed on
\begin{equation}
  \begin{aligned}
    \eps \partial_t \D_I^\eps (0,x) = \partial_\tau \dr_I (0,x) = \D^{in} (x) \,,
  \end{aligned}
\end{equation}
where $\D^{in}(x)$ is the disparity defined in \eqref{Differ-IC}, we derive from the initial relations \eqref{IC-relations} that the initial data of the remainder system \eqref{Remainder-u-d} should be
\begin{equation}\label{IC-Remainder}
  \left\{
    \begin{array}{l}
      \u_R^\eps (0, x) = 0 \,, \\ [2mm]
      \dr_R^\eps (0,x) = - \sqrt{\eps} \D_I^\eps (0,x) = - \sqrt{\eps} \widetilde{\D}^{in}_\eps (x) \,, \\ [2mm]
      ( \D_{\u_0 + \sqrt{\eps} \u_R^\eps } \dr_R^\eps ) (0,x) = - \sqrt{\eps} ( \u_0 \cdot \nabla \D_I^\eps ) (0,x) = - \sqrt{\eps} ( \u^{in} \cdot \nabla \widetilde{\D}^{in}_\eps ) (x) \,,
    \end{array}
  \right.
\end{equation}
where the vector field $\widetilde{\D}^{in}_\eps (x)$ is defined as
\begin{equation}
   \begin{aligned}
     \widetilde{\D}^{in}_\eps (x) = 2 \Big( \lambda_1 - \sqrt{\lambda_1^2 + 4 \eps \Delta} \Big)^{-1} \D^{in} (x) \,.
   \end{aligned}
\end{equation}
We emphasize that for any fixed $s \geq 0$, the norm bound
\begin{equation}
  \begin{aligned}
    \| \widetilde{\D}^{in}_\eps \|^2_{H^s} \leq \tfrac{4}{\lambda_1^2} \| \D^{in} \|^2_{H^s}
  \end{aligned}
\end{equation}
holds uniformly in $\eps > 0$, since the uniform lower bound $ \big| \lambda_1 - \sqrt{\lambda_1^2 - 4 \eps |k|^2} \big| \geq - \lambda_1 > 0 $ for all $k \in \mathbb{Z}^3$ reduces to
\begin{equation}
  \begin{aligned}
    \| \widetilde{\D}^{in}_\eps \|^2_{H^s} = & \sum_{k \in \mathbb{Z}^3} ( 1 + |k|^2 )^s \tfrac{4}{ \big| \lambda_1 - \sqrt{\lambda_1^2 - 4 \eps |k|^2} \big|^2 } \big| \int_{\T^3} \D^{in} (x) e^{-i x \cdot k} \d x \big|^2 \\
    \leq & \tfrac{4}{\lambda_1^2} \sum_{k \in \mathbb{Z}^3} ( 1 + |k|^2 )^s \big| \int_{\T^3} \D^{in} (x) e^{-i x \cdot k} \d x \big|^2 = \tfrac{4}{\lambda_1^2} \| \D^{in} \|^2_{H^s} \,.
  \end{aligned}
\end{equation}
Here the $H^s$-norm will be defined in the later.

{\bf Well-prepared initial data.} As mentioned before, one of our goal in this paper is to deal with the disparity $\D^{in} (x)$ given in \eqref{Differ-IC} resulted from the initial conditions \eqref{IC-HyperLQ}. Besides introducing a so-called initial layer in time to overcome this disparity, we can also impose the original system \eqref{HLQ} on the so-called {\em well-prepared initial data}. To be more precise, we can skilfully select the initial values \eqref{IC-HyperLQ} of the system \eqref{HLQ} such that the disparity $\D^{in} (x)$ vanishes, hence
\begin{equation}\label{IC-Well-prepared}
  \begin{aligned}
    \tilde{\dr}^{in} (x) = ( \D_{\u_0} \dr_0 ) (0,x) = - ( \B^{in} \dr^{in} ) (x) - \tfrac{1}{\lambda_1} ( \Delta \dr^{in} + \bm{\gamma}_0^{in} \dr^{in} + \lambda_2 \A^{in} \dr^{in} ) (x) \,.
  \end{aligned}
\end{equation}
Actually, if the disparity $\D^{in} (x) = 0$, the initial layer $\eps \D_I^\eps (t,x)$ defined in \eqref{Initial-Layer} will automatically be zero. Consequently, we shall take ansatz that the system \eqref{HLQ} imposed on the well-prepared initial data \eqref{IC-HyperLQ}-\eqref{IC-Well-prepared} has a solution $(\u^\eps, \dr^\eps)$ with the form
\begin{equation}\label{Ansatz-2}
  \left\{
    \begin{array}{l}
      \u^\eps (t,x) = \u_0 (t,x) + \sqrt{\eps} \u_R^\eps (t,x) \,, \\ [2mm]
      \dr^\eps (t,x) = \dr_0 (t,x) + \sqrt{\eps} \dr_R^\eps (t, x) \,.
    \end{array}
  \right.
\end{equation}
By substituting the ansatz \eqref{Ansatz-2} into the original hyperbolic type system \eqref{HLQ}, one easily derives that the remainder $(\u_R^\eps , \dr_R^\eps)$ subjects to the system with the similar structure of \eqref{Remainder-u-d}, just removing the terms $\sqrt{\eps} \partial_t (\u_R^\eps \cdot \nabla \D_I^\eps)$, $\eps \div \mathcal{Q}_{\u} (\D_I)$ and $\mathcal{Q}_{\dr} (\D_I)$ appearing in the system \eqref{Remainder-u-d}. More precisely, the remainder $(\u_R^\eps , \dr_R^\eps)$ in \eqref{Ansatz-2} satisfies
\begin{equation}\label{Remainder-u-d-2}
\left\{
\begin{array}{c}
\partial_t \u_R^\eps - \tfrac{1}{2} \mu_4 \Delta \u_R^\eps + \nabla p_R^\eps = \mu_1 \div \big[ (\A_R^\eps : \dr_0 \otimes \dr_0) \dr_0 \otimes \dr_0 \big] \\
\qquad\qquad + \mathcal{K}_{\u} + \div ( \mathcal{C}_{\u} + \mathcal{T}_{\u} + \sqrt{\eps} \mathcal{R}_{\u} ) \,,\\[1mm]
\div \u_R^\eps = 0 \,, \\[1mm]
\D^2_{\u_0 + \sqrt{\eps} \u_R^\eps} \dr_R^\eps + \tfrac{- \lambda_1}{\eps} \D_{\u_0 + \sqrt{\eps} \u_R^\eps} \dr_R^\eps - \tfrac{1}{\eps} \Delta \dr_R^\eps + \partial_t ( \u_R^\eps \cdot \nabla \dr_0 ) \\
= \tfrac{1}{\eps} \mathcal{C}_{\dr} + \tfrac{1}{\eps} \mathcal{S}^1_{\dr} + \tfrac{1}{\sqrt{\eps}} \mathcal{S}^2_{\dr} + \mathcal{R}_{\dr}
\end{array}
\right.
\end{equation}
with the constraint
\begin{equation}\label{Constraint-2}
\begin{aligned}
2 \dr_0 \cdot \dr_R^\eps + \sqrt{\eps} | \dr_R^\eps |^2 = 0 \,,
\end{aligned}
\end{equation}
where the tensor term $\mathcal{C}_{\u}$ and the vector term $ \mathcal{C}_{\dr} $ are defined in \eqref{C-u} and \eqref{C-d} respectively, and the accurate expressions of the tedious terms $\mathcal{T}_{\u}$, $\mathcal{K}_{\u}$, $\mathcal{S}^1_{\dr}$, $\mathcal{S}^2_{\dr}$,  $\mathcal{R}_{\dr}$ and $\mathcal{R}_{\u}$ are all given in Appendix \ref{Appendix}. Furthermore, based on the initial relations \eqref{IC-relations}, we know that the initial conditions of the remainder system \eqref{Remainder-u-d-2} should be imposed on
\begin{equation}\label{IC-Rmd-2}
  \begin{aligned}
    \u_R^\eps (0,x) = 0 \,, \ \dr_R^\eps (0,x) = 0 \,, \ ( \D_{\u_0 + \sqrt{\eps} \u_R^\eps } \dr_R^\eps ) (0,x) = 0 \,.
  \end{aligned}
\end{equation}

\subsection{Main results.} To state our main results, we collect here the notations we will use throughout this paper. We denote by $A \thicksim B$ if there are two constants $C_1 , C_2 > 0$, independent of $\eps > 0$, such that $C_1 A \leq B \leq C_2 A$. For convenience, we also denote by
$$ L^p = L^p (\T^3) $$
for all $p \in [ 1, \infty ]$, which endows with the norm $ \| f \|_{L^p} = \left( \int_{\T^3} |f(x)|^p \d x \right)^\frac{1}{p} $ for $p \in [ 1 , \infty )$ and $\| f \|_{L^\infty} = \underset{x \in \T^3}{\textrm{ess sup}} \, |f(x)|$. For $p = 2$, we use the notation $ \langle \cdot \,, \cdot \rangle $ to represent the inner product on the Hilbert space $L^2$.

For any multi-index $ m = ( m_1, m_2, m_3 )$ in $\mathbb{N}^3$, we denote the $m^{th}$ partial derivative by
\begin{equation*}
\partial^m = \partial^{m_1}_{x_1} \partial^{m_2}_{x_2} \partial^{m_3}_{x_3} \,.
\end{equation*}
If each component of $m \in \mathbb{N}^3$ is not greater than that of $\tilde{m}$'s, we denote by $m \leq \tilde{m}$. The symbol $m < \tilde{m}$ means $m \leq \tilde{m}$ and $|m| < | \tilde{m} |$, where $|m|= m_1 + m_2 + m_3$. We define the Sobolev space $H^N = H^N (\T^3)$ by the norm
\begin{align*}
\| f \|_{H^N} = \bigg(  \sum_{|m| \leq N} \| \partial^m f \|^2_{L^2} \bigg)^\frac{1}{2} < \infty \,,
\end{align*}
or the equivalent norm
\begin{equation}
  \begin{aligned}
    \| f \|_{H^N} = \Big( \sum_{k \in \mathbb{Z}^3} ( 1 + |k|^2 )^s | \widehat{f} (k) |^2 \Big)^\frac{1}{2} < + \infty \,,
  \end{aligned}
\end{equation}
where the symbol $\widehat{f} (k)$ is the Fourier transform of $f(x)$ on $x \in \T^3$, hence,
\begin{equation*}
  \begin{aligned}
    \widehat{f} (k) = \int_{\T^3} f(x) e^{i x \cdot k} \d x
  \end{aligned}
\end{equation*}
fro all $k \in \mathbb{Z}^3$. For any integer $N \geq 2$, we define a number $S_{\!_N} \in \mathbb{N}$ as
\begin{equation}\label{SN-integer}
\begin{aligned}
S_{\!_N} = \min \{ k \in \mathbb{N} ; 2 k \geq N + 2  \} \,.
\end{aligned}
\end{equation}
Actually, if $N$ is even, $S_{\!_N} = \tfrac{1}{2} N + 1$ and if $N$ is odd, $S_{\!_N} = \tfrac{1}{2} (N+3)$.

Now we state our main theorem:
\begin{theorem}\label{Thm-main}
	Let $N \geq 2$ be an integer and $ ( \u^{in} (x), \dr^{in} (x), \tilde{\dr}^{in} (x) ) \in \R^3 \times \mathbb{S}^2 \times \R^3 $ satisfy the compatibility conditions \eqref{IC-compatibility} and $  \u^{in} \,, \tilde{\dr}^{in} \,, \nabla \dr^{in} \in H^{2 S_{\!_N}} $. If the Leslie's coefficients satisfy
	\begin{equation}\label{Assumption-coefficients}
	\begin{aligned}
	\mu_4 > 0 \,, \ \lambda_1 < 0 \,, \ \mu_1 \geq 0 \,, \ \mu_5 + \mu_6 + \tfrac{\lambda_2^2}{\lambda_1} \geq 0
	\end{aligned}
	\end{equation}
	and there exist small $\eps_0, \xi_0 \in ( 0, 1 ]$, depending on the Leslie's coefficients and $N$, such that
	\begin{equation}\label{Initial-Energy}
	  \begin{aligned}
	    E^{in} \overset{\Delta}{=} \| \u^{in} \|^2_{H^{2 S_{\!_N}}} + \| \tilde{\dr}^{in} \|^2_{H^{2 S_{\!_N}}} + \| \nabla \dr^{in} \|^2_{H^{2 S_{\!_N}}} \leq \xi_0
	  \end{aligned}
	\end{equation}
	for all $ \eps \in ( 0, \eps_0 ]$, then the system \eqref{HLQ} with the initial conditions \eqref{IC-HyperLQ} admits a unique solution $(\u^\eps , \dr^\eps)$ satisfying
	\begin{equation}
	  \begin{aligned}
	    \u^\eps , \nabla \dr^\eps , \D_{\u^\eps} \dr^\eps \in L^\infty ( \R^+ ; H^N ) \,, \ \nabla \u^\eps \in L^2 ( \R^+ ; H^N ) \,.
	  \end{aligned}
	\end{equation}
	Moreover, the solution $(\u^\eps , \dr^\eps)$ is of the form
	\begin{equation}
	  \left\{
	    \begin{array}{l}
	      \u^\eps (t,x) = \u_0 (t,x) + \sqrt{\eps} \u_R^\eps (t,x) \,, \\[2mm]	
	      \dr^\eps (t,x) = \dr_0 (t,x) + \eps \D_I^\eps (t,x) + \sqrt{\eps} \dr_R^\eps (t,x) \,,
	    \end{array}
	  \right.
	\end{equation}
	where $(\u_0, \dr_0)$ is the solution to the incompressible parabolic Ericksen-Leslie's liquid crystal model \eqref{PLQ} with the initial data \eqref{IC-ParaLQ}, the initial layer $\eps \D_I^\eps (t,x)$ is defined in \eqref{Initial-Layer}, and $ ( \u_R^\eps , \dr_R^\eps ) $ obeys the remainder system \eqref{Remainder-u-d} with the initial condition \eqref{IC-Remainder}. Furthermore, $ ( \u_R^\eps , \dr_R^\eps ) $ satisfies the uniform energy bound
	\begin{equation}\label{Energ-Bnds-Rmd}
	  \begin{aligned}
	    \Big( \tfrac{1}{\eps} \| \u_R^\eps \|^2_{H^N} + \| \D_{\u_0 + \sqrt{\eps} \u_R^\eps } \dr_R^\eps \|^2_{H^N} + \tfrac{1}{\eps} \| \dr_R^\eps \|^2_{H^{N+1}} \Big) (t) + \tfrac{1}{\eps} \int_0^t \| \nabla \u_R^\eps \|^2_{H^N} (\tau) \d \tau \leq C \xi_0
	  \end{aligned}
	\end{equation}
	for all $t \geq 0$, $\eps \in ( 0, \eps_0 ]$ and for some constant $C > 0$, independent of $\eps$ and $t$.
\end{theorem}

\begin{remark}
	The small number $\xi_0$ is firstly smaller than $\beta_{S_{\!_N}, 0}$ mentioned in Proposition \ref{Prop-WP-PLQ} proved by Wang-Zhang-Zhang in \cite{Wang-Zhang-Zhang-2013-ARMA}, so that the limit system \eqref{PLQ}-\eqref{IC-ParaLQ} admits a unique global in time classical solution $(\u_0 , \dr_0)$ under the constraint of small size $\xi_0$ to the initial data. Thus, the vectors $\u_0$ and $\dr_0$ appeared in the remainder system \eqref{Remainder-u-d} can be regarded as the known coefficients.
\end{remark}

\begin{remark}
	Under the assumptions of Theorem \ref{Thm-main}, if the well-prepared initial data is further assumed, i.e., \eqref{IC-Well-prepared} holds, then the hyperbolic Ericksen-Leslie's liquid crystal system \eqref{HLQ}-\eqref{IC-HyperLQ} has a unique global classical solution $(\u^\eps , \dr^\eps)$ with the form
	\begin{equation}
	  \left\{
	    \begin{array}{l}
	      \u^\eps (t,x) = \u_0 (t,x) + \sqrt{\eps} \u_R^\eps (t,x) \,, \\[2mm]	
	      \dr^\eps (t,x) = \dr_0 (t,x) + \sqrt{\eps} \dr_R^\eps (t,x) \,,
	    \end{array}
	  \right.
	\end{equation}
	and the other conclusions are the same as stated in Theorem \ref{Thm-main}.
\end{remark}

\begin{remark}
	The uniform bound \eqref{Energ-Bnds-Rmd} implies that
	\begin{equation}
	   \| \D_{\u_0 + \sqrt{\eps} \u_R^\eps } \dr_R^\eps \|^2_{L^\infty(\R^+;H^N}) + \tfrac{1}{\eps} \| \dr_R^\eps \|^2_{L^\infty( \R^+; H^{N+1})} \leq C \xi_0 \,,
	\end{equation}
	which shows us the convergence rate $\sqrt{\eps}$ on the orientation $\dr^\eps$ we obtain is optimal. This optimality can be also seen in Remark 1.1 of Jiang-Luo-Tang-Zarnescu's work \cite{Jiang-Luo-Tang-Zarnescu-2019-CMS}. We also derive from the uniform bound \eqref{Energ-Bnds-Rmd} that
	\begin{equation}
	  \begin{aligned}
	    \tfrac{1}{\eps} \| \u_R^\eps \|^2_{L^\infty (\R^+; H^N)} + \tfrac{1}{\eps} \| \nabla \u_R^\eps \|^2_{L^2(\R^+;H^N)} \leq C \xi_0 \,.
	  \end{aligned}
	\end{equation}
	Let $\widetilde{\u}_R^\eps = \tfrac{1}{\sqrt{\eps}} \u_R^\eps$. Then the expansion of $\u^\eps$ can be rewritten as $ \u^\eps (t,x) = \u_0 (t,x) + \eps \widetilde{\u}_R^\eps (t,x) $ and the uniform bound $ \| \widetilde{\u}_R^\eps \|^2_{L^\infty (\R^+; H^N)} + \| \nabla \widetilde{\u}_R^\eps \|^2_{L^2(\R^+;H^N)} \leq C \xi_0 $ holds, which means that the convergence rate of the velocity $\u^\eps $ is $\eps$.
\end{remark}

\subsection{Ideals and novelties} Generally speaking, there are three aspects of the disparity between the original system and the limit equations on the limit problem. First, the form of the limit equations is obviously different from that of the original system. Second, the initial conditions of original system and limit equations are not the same type, which will result to the initial layers. Finally, the boundary conditions will also lead to the disparity, which can be covered by the boundary layers.

In the current paper, we study the limit problem on $\T^3$ in the regime of classical solutions, which does not involve the boundary conditions. It is a very difficult problem to derive the energy bounds of $(\u^\eps , \dr^\eps)$ uniform in small $\eps > 0$ to the original hyperbolic Ericksen-Leslie's liquid crystal model \eqref{HLQ}. So, we take the Hilbert expansion method to rigorously justify the limit from hyperbolic Ericksen-Leslie's liquid crystal model to the parabolic case, in which the remainder term $( \u_R^\eps , \dr_R^\eps )$ is utilized to deal with the difference between the forms of two systems. Because the $\dr^\eps$-equation in \eqref{HLQ} is a wave type equation and the corresponding limit $\dr_0$-equation in \eqref{PLQ} is a parabolic type equation, the $\dr^\eps$-equation of original system need impose on two initial conditions but the $\dr_0$-equation of limit system only need impose on one initial data. One of the methods to overcome this disparity is to introduce an initial layer $\eps \D_I^\eps (t,x)$ (defined in \eqref{Initial-Layer}), so that the disparity $\D^{in} (x)$ is absorbed, for details see the analysis before. Thus we give the formal expansion \eqref{Ansatz-1}, namely
$$ \u^\eps (t,x) = \u_0 (t,x) + \sqrt{\eps} \u_R^\eps (t,x) \,, \ \dr^\eps (t,x) = \dr_0 (t,x) + \eps \D_I^\eps (t,x) + \sqrt{\eps} \dr_R^\eps (t,x) \,. $$
Another method to overcome the disparity of the initial conditions is to give the well-prepared initial data such that $\D^{in} (x) = 0$, which immediately implies that the initial layer $\D_I^\eps (t,x)$ vanishes. In this case, the expansion form will be \eqref{Ansatz-2}, hence
$$ \u^\eps (t,x) = \u_0 (t,x) + \sqrt{\eps} \u_R^\eps (t,x) \,, \ \dr^\eps (t,x) = \dr_0 (t,x) + \sqrt{\eps} \dr_R^\eps (t,x) \,. $$

The main goal of this work is to derive the global energy bound of the remainder $( \u_R^\eps , \dr_R^\eps )$ uniform in small $\eps > 0$ under the small size of the initial data $(\u^{in}, \dr^{in}, \tilde{\dr}^{in})$. Since the initial layer $\D_I^\eps (t,x)$ and its time derivative $\partial_t \D_I^\eps (t,x)$ are infinitely small as $\eps \ll 1$, we can merely consider the remainder $(\u_R^\eps , \dr_R^\eps)$ in the expansion \eqref{Ansatz-2} with respect to the case of well-prepared initial data, which satisfying the system \eqref{Remainder-u-d-2}. The major structure of \eqref{Remainder-u-d-2} reads
\begin{equation}\label{Rmd-simple}
  \left\{
    \begin{array}{c}
      \partial_t \u_R^\eps - \tfrac{1}{2} \mu_4 \Delta \u_R^\eps + \nabla p_R^\eps - \div \mathcal{C}_{\u} = \textrm{some other terms} \,, \\ [2mm]
      \div \u_R^\eps = 0 \,, \\ [2mm]
      \D_{\u_0 + \sqrt{\eps} \u_R^\eps }^2 \dr_R^\eps - \tfrac{\lambda_1}{\eps} \D_{\u_0 + \sqrt{\eps} \u_R^\eps } \dr_R^\eps - \tfrac{1}{\eps} \Delta \dr_R^\eps + \partial_t ( \u_R^\eps \cdot \nabla \dr_0 ) - \tfrac{1}{\eps} \mathcal{C}_{\dr} = \textrm{some other terms} \,,
    \end{array}
  \right.
\end{equation}
where $\lambda_1 < 0$. The {\bf key point} in this work is how to control the term $\partial_t (\u_R^\eps \cdot \nabla \dr_0)$ in the energy estimates to the remainder system \eqref{Remainder-u-d-2}. We will design a energy functional, which sensitively depends on the limit vector field $(\u_0 , \dr_0)$, to deal with this term. More precisely, we multiply by $\u_R^\eps$ and $\D_{\u_0 + \sqrt{\eps} \u_R^\eps } \dr_R^\eps$ in the first and third equations of \eqref{Rmd-simple} respectively and integrate by parts over $x \in \T^3$. Then, combining the cancellation \eqref{Cancellations} of the case $m=0$ in Lemma \ref{Lmm-Cancellation}, hence
\begin{equation*}
   \begin{aligned}
     - \l \div \mathcal{C}_{\u} , \u_R^\eps \r - \l \mathcal{C}_{\dr} , \D_{\u_0 + \sqrt{\eps} \u_R^\eps } \dr_R^\eps \r = & - \lambda_1 \| \D_{\u_0 + \sqrt{\eps} \u_R^\eps } \dr_R^\eps + \B_R^\eps \dr_0 + \tfrac{\lambda_2}{\lambda_1} \A_R^\eps \dr_0 \|_{L^2} \\
     + & \lambda_1 \| \D_{\u_0 + \sqrt{\eps} \u_R^\eps } \dr_R^\eps \|^2_{L^2} + ( \mu_5 + \mu_6 + \tfrac{\lambda_2^2}{\lambda_1} ) \| \A_R^\eps \dr_0 \|^2_{L^2} \,,
   \end{aligned}
\end{equation*}
we obtain the main part of the $L^2$-energy equality
\begin{equation}
  \begin{aligned}
    & \tfrac{1}{2} \tfrac{\d}{\d t} \big( \underset{The \ energy}{ \underbrace{ \tfrac{1}{\eps} \| \u_R^\eps \|^2_{L^2} + \| \D_{\u_0 + \sqrt{\eps} \u_R^\eps } \dr_R^\eps \|^2_{L^2} + \tfrac{1}{\eps} \| \nabla \dr_R^\eps \|^2_{L^2} } } \big) + \l \partial_t ( \u_R^\eps \cdot \nabla \dr_0 ) , \D_{\u_0 + \sqrt{\eps} \u_R^\eps } \dr_R^\eps \r  \\
    & + \underset{The \ energy \ dissipative \ rate}{ \underbrace{ \tfrac{1}{2 \eps} \mu_4 \| \nabla \u_R^\eps \|^2_{L^2} - \tfrac{\lambda_1}{\eps} \| \D_{\u_0 + \sqrt{\eps} \u_R^\eps } \dr_R^\eps + \B_R^\eps \dr_0 + \tfrac{\lambda_2}{\lambda_1} \A_R^\eps \dr_0 \|^2_{L^2} + \tfrac{1}{\eps} (\mu_5 + \mu_6 + \tfrac{\lambda_2^2}{\lambda_1}) \| \A_R^\eps \dr_0 \|_{L^2} } } \\
    & = \cdots \cdots
  \end{aligned}
\end{equation}
Under the coefficient conditions \eqref{Assumption-coefficients}, the energy dissipative rate is positive. If we regard the term $\partial_t (\u_R^\eps \cdot \nabla \dr_0)$ as a source term, the quantity $ \l \partial_t ( \u_R^\eps \cdot \nabla \dr_0 ) , \D_{\u_0 + \sqrt{\eps} \u_R^\eps } \dr_R^\eps \r $ should be dominated by the energy and the energy dissipative rate defined in the above equality. However, it is impossible, since the regularity of $\partial_t ( \u_R^\eps \cdot \nabla \dr_0 )$ is equivalent to $\tfrac{1}{2} \mu_4 \Delta \u_R^\eps$ by using the first $\u_R^\eps$-equation of \eqref{Rmd-simple} and the highest order regularity of the energy dissipative rate is $\nabla \u_R^\eps$. In order to overcome this difficulty, we try to design this quantity as a part of the energy. More precisely, we take the following {\em important } deformation
\begin{equation}
  \begin{aligned}
    & \l \partial_t ( \u_R^\eps \cdot \nabla \dr_0 ) , \D_{\u_0 + \sqrt{\eps} \u_R^\eps } \dr_R^\eps \r \\
    = & \tfrac{\d}{\d t} \l \u_R^\eps \cdot \nabla \dr_0 , \D_{\u_0 + \sqrt{\eps} \u_R^\eps } \dr_R^\eps \r  - \l \u_R^\eps \cdot \nabla \dr_0 , \partial_t \D_{\u_0 + \sqrt{\eps} \u_R^\eps } \dr_R^\eps \r \\
    = & \tfrac{\d}{\d t} \l \u_R^\eps \cdot \nabla \dr_0 , \D_{\u_0 + \sqrt{\eps} \u_R^\eps } \dr_R^\eps \r  - \l \u_R^\eps \cdot \nabla \dr_0 , \D_{\u_0 + \sqrt{\eps} \u_R^\eps }^2 \dr_R^\eps \r \\
    & + \l \u_R^\eps \cdot \nabla \dr_0 , ( \u_0 + \sqrt{\eps} \u_R^\eps ) \cdot \nabla \D_{\u_0 + \sqrt{\eps} \u_R^\eps } \dr_R^\eps \r \\
    = & \tfrac{1}{2} \tfrac{\d}{\d t} \big( \| \u_R^\eps \cdot \nabla \dr_0 \|^2_{L^2} + 2 \l \u_R^\eps \cdot \nabla \dr_0 , \D_{\u_0 + \sqrt{\eps} \u_R^\eps } \dr_R^\eps \r \big) \\
    & \underset{P_1}{ \underbrace{ - \l ( \u_0 + \sqrt{\eps} \u_R^\eps ) \cdot \nabla ( \u_R^\eps \cdot \nabla \dr_0 ) , \D_{\u_0 + \sqrt{\eps} \u_R^\eps } \dr_R^\eps \r + \tfrac{1}{\eps} \l \u_R^\eps \cdot \nabla \dr_0 , \mathcal{C}_{\dr} \r }} \\
    & \underset{P_1}{ \underbrace{ - \tfrac{1}{\eps} \l \nabla (\u_R^\eps \cdot \nabla \dr_0) , \nabla \dr_R^\eps \r + \tfrac{\lambda_1}{\eps} \l \u_R^\eps \cdot \nabla \dr_0 , \D_{\u_0 + \sqrt{\eps} \u_R^\eps } \dr_R^\eps \r }} \\
    & + \cdots \cdots
  \end{aligned}
\end{equation}
where the last equality is derived from the $\dr_R^\eps$-equation of \eqref{Rmd-simple}. Then, we obtain the relation
\begin{equation}
\begin{aligned}
& \tfrac{1}{2} \tfrac{\d}{\d t} \Big( \underset{The \ new \ energy}{ \underbrace{ \tfrac{1}{\eps} \| \u_R^\eps \|^2_{L^2} + \| \D_{\u_0 + \sqrt{\eps} \u_R^\eps } \dr_R^\eps \|^2_{L^2} + \tfrac{1}{\eps} \| \nabla \dr_R^\eps \|^2_{L^2} + \| \u_R^\eps \cdot \nabla \dr_0 \|^2_{L^2} + 2 \l \u_R^\eps \cdot \nabla \dr_0 , \D_{\u_0 + \sqrt{\eps} \u_R^\eps } \dr_R^\eps \r }} \Big) \\
& + \underset{The \ energy \ dissipative \ rate}{ \underbrace{ \tfrac{1}{2 \eps} \mu_4 \| \nabla \u_R^\eps \|^2_{L^2} - \tfrac{\lambda_1}{\eps} \| \D_{\u_0 + \sqrt{\eps} \u_R^\eps } \dr_R^\eps + \B_R^\eps \dr_0 + \tfrac{\lambda_2}{\lambda_1} \A_R^\eps \dr_0 \|^2_{L^2} + \tfrac{1}{\eps} (\mu_5 + \mu_6 + \tfrac{\lambda_2^2}{\lambda_1}) \| \A_R^\eps \dr_0 \|_{L^2} } } \\
& = P_1 + P_2 + \cdots \cdots
\end{aligned}
\end{equation}
where the quantities $P_1$ and $P_2$ can be controlled by the energy and energy dissipative rate. Although the new energy is not positive for all $\eps > 0$, it will be a definitely positive sign with sufficiently small $\eps > 0$ under the fixed coefficient $\nabla \dr_0$. Consequently, we have designed a complicated energy functional, which sensitively depends on the solutions to the limit system, to deal with the trouble quantity $ \l \partial_t ( \u_R^\eps \cdot \nabla \dr_0 ) , \D_{\u_0 + \sqrt{\eps} \u_R^\eps } \dr_R^\eps \r $.

The advantage of the expansion \eqref{Ansatz-2} of the solutions $(\u^\eps, \dr^\eps)$ to the system \eqref{HLQ} is such that the remainder equations \eqref{Remainder-u-d-2} of $(\u_R^\eps , \dr_R^\eps)$ have weaker nonlinearities than the original system \eqref{HLQ}, despite the system \eqref{Remainder-u-d-2} is still nonlinear and singular (with singular terms of the type $\frac{1}{\eps}$). To be more precise, after utilizing the expansion \eqref{Ansatz-2}, the nonlinearity and singularity are separated. For instance, the term $ ( - \eps |\D_{\u^\eps} \dr^\eps|^2 + |\nabla \dr^\eps|^2 - \lambda_2 \A^\eps : \dr^\eps \otimes \dr^\eps ) \dr^\eps $ in the original system \eqref{HLQ} is replaced by either linear terms ( the unknown $\dr^\eps$ and $\u^\eps$ are superseded by the known $\u_0$ and $\dr_0$ ) or a nonlinear term with the same form but with some higher order power of $\eps$ in front. So, it will be easier to get the energy bound, uniform in small $\eps > 0$, of the remainder system \eqref{Remainder-u-d-2}-\eqref{IC-Rmd-2}.

\subsection{Historical remarks}

In this subsection, we review some history of the mathematical analytic works on the liquid crystals, in particular Ericksen-Leslie's system. The static analogue of the parabolic Ericksen-Leslie's system \eqref{PLQ} is the so-called Oseen-Frank model, whose mathematical study was initialed from Hardt-Kinderlehrer-Lin \cite{Hardt-Kinderlehrer-Lin-CMP1986}. Since then there have been many works in this direction. In particular, the existence and regularity or partial regularity of the approximation (usually Ginzburg-Landau approximation as in \cite{Lin-Liu-CPAM1995}) dynamical Ericksen-Leslie's system was started by the work of Lin and Liu in \cite{Lin-Liu-CPAM1995}, \cite{Lin-Liu-DCDS1996} and \cite{Lin-Liu-ARMA2000}. For the simplest system preserving the basic energy law which can be obtained by neglecting the Leslie stress and by specifying some elastic constants,  in 2-D case, the existence of global weak solutions with at most a finite number of singular times were proved by Lin-Lin-Wang \cite{Lin-Lin-Wang-ARMA2010}. Recently, Lin and Wang proved global existence of weak solution for 3-D case with the initial director field $\dr^{in} (x)$ lying in the hemisphere in \cite{Lin-Wang-CPAM2016}. For the more general parabolic Ericksen-Leslie's system,  local well-posedness is proved by Wang-Zhang-Zhang in \cite{Wang-Zhang-Zhang-2013-ARMA}, and in \cite{Huang-Lin-Wang-CMP2014} existence of global solutions and regularity in $\mathbb{R}^2$ was established by Huang-Lin-Wang.  For more complete review of the works for the parabolic Ericksen-Leslie's system, please see the reference listed above.

For the hyperbolic Ericksen-Leslie's system \eqref{HLQ}, much less is known. For the most simplified model, i.e. taking the bulk velocity field $\u=0$, neglecting the Leslie's coefficients, and the spatial dimension is $1$, the system \eqref{HLQ} can be reduced to a so-called nonlinear variational wave equation which is already highly nontrivial. Zhang and Zheng studied systematically the dissipative and energy conservative solutions in series work starting from late 90's (\cite{Zhang-Zheng-ARMA2010, Zhang-Zheng-CPAM2012, Zhang-Zheng-CPDE2001}).

Recently, there started some works on the original hyperbolic Ericksen-Leslie's system \eqref{HLQ} for multi-dimentional case. the authors of the current paper studied in \cite{Jiang-Luo-2019-SIAM} the well-posedness in the context of classical solutions of \eqref{HLQ}. More precisely, in \cite{Jiang-Luo-2019-SIAM} under some natural constraints on the Leslie coefficients which ensure the basic energy law is dissipative, it was proved the local-in-time existence and uniqueness of the classical solution to the system \eqref{HLQ} with finite initial energy. Furthermore, with an additional assumption on the coefficients which provides a damping effect, i.e. $\lambda_1 < 0$, and the smallness of the initial energy, the unique global classical solution was established. Here we remark that the assumption $\lambda_1 < 0$ plays a crucial role in the global-in-time well-posedness. Later, Cai-Wang \cite{CW} made progress for the simplied Ericksen-Leslie system, namely, the case with  $\mu_i = 0, i=1\,,\cdots\,,6$, $i\neq 4$ in \eqref{sigma}. They proved the global regularity of \eqref{HLQ} near the constant equilibrium by employing the vector field method. More recently, in \cite{Huang-Jiang-Luo-Zhao}, the authors of the current papers with Huang and Zhao, considered the more general case: still $\mu_2=\mu_3=0$, but $0 \neq \mu_5=\mu_6 >-\mu_4$, and $0 \neq \mu_1 > -2(\mu_4+\mu_5)$, and proved results similar to \cite{CW}.

Regarding to the inertia limit, i.e. $\eps\rightarrow 0$, for a given bulk velocity and well-prepared initial data, together with Tang, we justified this limit in \cite{Jiang-Luo-Tang-2019}. For the case without the bulk velocity and general initial data, by constructing an initial layer, we with Tang and Zarnescu, justified this limit in \cite{Jiang-Luo-Tang-Zarnescu-2019-CMS}. In this sense, the current paper, proved this inertia limit for the much more general case with bulk velocity field, under the assumption $\lambda_1 <0$. The case $\lambda_1=\lambda_2=0$ will be analytically more subtle, for which the limiting system is the harmonic map to $\mathbf{S}^2$, and furthermore, the initial layer will be a wave equation which preserve the energy. This work is under preparation, together with Huang and Zhao.

The organization of this paper is as follows: in the next section, we first derive the canceled relations between $\mathcal{C}_{\u}$ and $\mathcal{C}_{\dr}$ contained in the remainder system of $(\u_R^\eps, \dr_R^\eps)$, which will play an essential role in estimating the energy of the remainder $(\u_R^\eps, \dr_R^\eps)$. Then we shown that the constraints \eqref{Constraint-1} and \eqref{Constraint-2} will hold at any time provided they initially hold, see Lemma \ref{Lmm-Constraints}. In Section \ref{Sec:Uniform-Ener}, we estimate the uniform energy bound on small $\eps > 0$ of the remainder system of $(\u_R^\eps, \dr_R^\eps)$. Then, based on the uniform energy estimates in the previous section, Theorem \ref{Thm-main} of the current paper is proved in Section \ref{Sec:Proof-Main-Thm}. Finally, we accurately present the all tedious term of the remainder system \eqref{Remainder-u-d} (also \eqref{Remainder-u-d-2}) in Appendix \ref{Appendix}.

\section{Some basic cancellations and constraints}

In this section, we will first derive some basic cancellations on the remainder equations \eqref{Remainder-u-d} (or \eqref{Remainder-u-d-2}), which will play an essential role in deriving the global in time energy estimates uniformly in $\eps > 0$ to the remainder system \eqref{Remainder-u-d} (or \eqref{Remainder-u-d-2}) with small initial data. We then prove that the constraints \eqref{Constraint-1} or \eqref{Constraint-2}, which come from the geometric constraint $|\dr^\eps| = 1$ in the original system \eqref{HLQ}, will hold for all time $t \geq 0$ provided they initially hold.

First, we will give the following lemma, which displays the cancellations between the terms $ \mathcal{C}_{\u} $ and $\mathcal{C}_{\dr}$.

\begin{lemma}\label{Lmm-Cancellation}
	Under the relations \eqref{Coefficients-Relations}, for all multi-indexes $m \in \mathbb{N}^3$, one has
	\begin{equation}\label{Cancellations}
	  \begin{aligned}
	    & \l \div \partial^m \mathcal{C}_{\u} , \partial^m \u_R^\eps \r + \l \partial^m \mathcal{C}_{\dr} , \partial^m \D_{\u_0 + \sqrt{\eps} \u_R^\eps} \dr_R^\eps \r \\
	    = & \lambda_1 \left\| \partial^m \D_{\u_0 + \sqrt{\eps} \u_R^\eps} \dr_R^\eps + ( \partial^m \B_R^\eps ) \dr_0 + \tfrac{\lambda_2}{\lambda_1} ( \partial^m \A_R^\eps ) \dr_0 \right\|^2_{L^2} - \lambda_1 \left\| \partial^m \D_{\u_0 + \sqrt{\eps} \u_R^\eps} \dr_R^\eps \right\|^2_{L^2} \\
	    & - ( \mu_5 + \mu_6 + \tfrac{\lambda_2^2}{\lambda_1} ) \left\| (\partial^m \A_R^\eps) \dr_0 \right\|^2_{L^2}  + \mathcal{G}_m \,,
	  \end{aligned}
	\end{equation}
	where the quantity $\mathcal{G}_m$ is defined as follows:
	if $m = 0$,
	\begin{equation}
	  \mathcal{G}_m = 0 \,,
	\end{equation}
	and if $m \neq 0$,
	\begin{equation}\label{Gm}
	  \begin{aligned}
	    \mathcal{G}_m = & \sum_{m' < m} C_m^{m'} \Big[ \l \lambda_1 ( \partial^{m'} \B_R^\eps ) \partial^{m-m'} \dr_0 + \lambda_2 ( \partial^{m'} \A_R^\eps ) \partial^{m-m'} \dr_0 , \partial^m \D_{\u_0 + \sqrt{\eps} \u_R^\eps} \dr_R^\eps \r \\
	    - & \l \mu_2 \partial^{m-m'} ( \dr_{0,j} \dr_{0,k} ) \partial^{m'} \B_{R, ki}^\eps + \mu_3 \partial^{m-m'} ( \dr_{0,i} \dr_{0,k} ) \partial^{m'} \B_{R,kj}^\eps , \partial^m \partial_j \u_{R,i}^\eps \r \\
	    - & \l \mu_5 \partial^{m-m'} ( \dr_{0,j} \dr_{0,k} ) \partial^{m'} \A_{R,ki}^\eps + \mu_6 \partial^{m-m'} ( \dr_{0,i} \dr_{0,k} )  \partial^{m'} \A_{R,kj}^\eps , \partial^m \partial_j \u_{R,i}^\eps \r \\
	    - & \l \mu_2 \partial^{m-m'} \dr_{0,j} \partial^{m'} ( \D_{\u_0 + \sqrt{\eps} \u_R^\eps} \dr_R^\eps )_i + \mu_3 \partial^{m-m'} \dr_{0,i} \partial^{m'} ( \D_{\u_0 + \sqrt{\eps} \u_R^\eps} \dr_R^\eps )_j , \partial^m \partial_j \u_{R, i}^\eps \r \Big] \,.
	  \end{aligned}
	\end{equation}
\end{lemma}

\begin{proof}[Proof of Lemma \ref{Lmm-Cancellation}]
	Recalling the definition of $\mathcal{C}_{\u}$ in \eqref{C-u}, one deduces
	\begin{equation}
	  \begin{aligned}
	    \l \div \partial^m \mathcal{C}_{\u} , \partial^m \u_R^\eps \r = & \ \underset{U_1}{\underbrace{ \l \partial_j \partial^m ( \mu_2 \dr_{0,j} \B_{R, ki}^\eps \dr_{0,k} + \mu_3 \dr_{0,i} \B_{R,kj}^\eps \dr_{0,k} ) , \partial^m \u_{R,i}^\eps \r }} \\
	    & \ + \underset{U_2}{\underbrace{ \l \partial_j \partial^m ( \mu_5 \dr_{0,j} \dr_{0,k} \A_{R,ki}^\eps + \mu_6 \dr_{0,i} \dr_{0,k} \A_{R,kj}^\eps ) , \partial^m \u_{R,i}^\eps \r }} \\
	    & \ + \underset{U_3}{\underbrace{\l \partial_j \partial^m \big[ \mu_2 \dr_{0,j} ( \D_{\u_0 + \sqrt{\eps} \u_R^\eps} \dr_R^\eps )_i + \mu_3 \dr_{0,i} ( \D_{\u_0 + \sqrt{\eps} \u_R^\eps} \dr_R^\eps )_j \big] , \partial^m \u_{R,i}^\eps \r}} \,.
	  \end{aligned}
	\end{equation}
	Then we will directly calculate the terms $U_1$, $U_2$ and $U_3$ for the case $m \neq 0$. The case $m = 0$ can be similarly justified. For the term $U_1$, we derive from the integration by parts over $x \in \T^3$ that
	\begin{equation}
	  \begin{aligned}
	    & U_1 = \ \underset{U_{11}}{\underbrace{- \l \mu_2 \dr_{0,j} \partial^m \B_{R,ki}^\eps \dr_{0,k} + \mu_3 \dr_{0,i} \partial^m \B_{R,kj}^\eps \dr_{0,k} , \partial^m \partial_j \u_{R, i}^\eps \r }} \\
	    & \ \underset{U_{12}}{\underbrace{ - \sum_{m' < m} C_m^{m'} \l \mu_2 \partial^{m-m'} ( \dr_{0,j} \dr_{0,k} ) \partial^{m'} \B_{R, ki}^\eps + \mu_3 \partial^{m-m'} ( \dr_{0,i} \dr_{0,k} ) \partial^{m'} \B_{R,kj}^\eps , \partial^m \partial_j \u_{R,i}^\eps \r }} \,.
	  \end{aligned}
	\end{equation}
	Since $\lambda_1 = \mu_2 - \mu_3$, $\lambda_2 = \mu_5 - \mu_6$ and $\mu_6 - \mu_5 = \mu_2 + \mu_3$, we have
	\begin{equation}
	  \begin{aligned}
	    U_{11} = & - \big\langle (\mu_2 - \mu_3) \dr_{0,j} (\partial^m \B_{R,ki}^\eps ) \dr_{0,k} \\
	    & + \mu_3 ( \dr_{0,j} ( \partial^m \B_{R,ki}^\eps ) \dr_{0,k} + \dr_{0,i} ( \partial^m \B_{R,kj}^\eps ) \dr_{0,k} ) , \partial^m \A_{R,ij}^\eps + \partial^m \B_{R,ij}^\eps \big\rangle \\
	    = & - \lambda_1 \l \dr_{0,j} ( \partial^m \B_{R,ki}^\eps ) \dr_{0,k} , \partial^m \A_{R,ij}^\eps + \partial^m \B_{R,ij}^\eps \r \\
	    & - \mu_3 \l \dr_{0,j} ( \partial^m \B_{R,ki}^\eps ) \dr_{0,k} + \dr_{0,i} ( \partial^m \B_{R,kj}^\eps ) \dr_{0,k} , \partial^m \A_{R,ij}^\eps \r \\
	    = & \lambda_1 \| ( \partial^m \B_R^\eps ) \dr_0 \|^2_{L^2} - ( \lambda_1 + 2 \mu_3 ) \l (\partial^m \B_{R,ki}^\eps ) \dr_{0,k} , ( \partial^m \A_{R,ij}^\eps ) \dr_{0,j} \r \\
	    = & \lambda_1 \| ( \partial^m \B_R^\eps ) \dr_0 \|^2_{L^2} + \lambda_2 \l (\partial^m \B_R^\eps) \dr_0 , (\partial^m \A_R^\eps) \dr_0 \r \,,
	  \end{aligned}
	\end{equation}
	where we make use of the relation $\B_{R,ij}^\eps = - \B_{R,ji}^\eps$.
	
	For the term $U_2$, integrating by parts over $x \in \T^3$ implies
	\begin{equation}
	  \begin{aligned}
	    & U_2 = \underset{U_{21}}{\underbrace{ - \l \mu_5 \dr_{0,j} \dr_{0,k} (  \partial^m \A_{R,ki}^\eps ) + \mu_6 \dr_{0,i} \dr_{0,k} ( \partial^m \A_{R,kj}^\eps ) , \partial^m \partial_j \u_{R, i}^\eps \r }} \\
	    & \underset{U_{22}}{\underbrace{ - \sum_{m' < m} C_m^{m'} \l \mu_5 \partial^{m-m'} ( \dr_{0,j} \dr_{0,k} ) \partial^{m'} \A_{R,ki}^\eps + \mu_6 \partial^{m-m'} ( \dr_{0,i} \dr_{0,k} )  \partial^{m'} \A_{R,kj}^\eps , \partial^m \partial_j \u_{R,i}^\eps \r }} \,.
	  \end{aligned}
	\end{equation}
	By the similar calculations on $U_{11}$, we compute the term $U_{21}$
	\begin{equation}
	  \begin{aligned}
	    U_{21} = & - \big\langle (\mu_5 - \mu_6) \dr_{0,j} \dr_{0,k} ( \partial^m \A_{R,ki}^\eps ) \\
	    & + \mu_6 ( \dr_{0,j} \dr_{0,k} ( \partial^m \A_{R,ki}^\eps ) + \dr_{0,i} \dr_{0,k} ( \partial^m \A_{R,kj}^\eps ) ) , \partial^m \A_{ij}^\eps + \partial^m \B_{ij}^\eps \big\rangle \\
	    = & - \lambda_2 \| ( \partial^m \A_R^\eps ) \dr_0 \|^2_{L^2} + \lambda_2 \l ( \partial^m \A_R^\eps ) \dr_0 , ( \partial^m \B_R^\eps ) \dr_0 \r - 2 \mu_6 \| ( \partial^m \A_R^\eps ) \dr_0 \|^2_{L^2} \\
	    = & - (\mu_5 + \mu_6) \| ( \partial^m \A_R^\eps ) \dr_0 \|^2_{L^2} + \lambda_2 \l ( \partial^m \A_R^\eps ) \dr_0 , ( \partial^m \B_R^\eps ) \dr_0 \r \,.
	  \end{aligned}
	\end{equation}
	
	For the term $U_3$, we deduce from integrating by parts over $x \in \T^3$ that
	\begin{equation}
	  \begin{aligned}
	    & U_3 = \underset{U_{31}}{\underbrace{ - \l \mu_2 \dr_{0,j} \partial^m ( \D_{\u_0 + \sqrt{\eps} \u_R^\eps} \dr_R^\eps )_i + \mu_3 \dr_{0,i} \partial^m ( \D_{\u_0 + \sqrt{\eps} \u_R^\eps} \dr_R^\eps )_j , \partial^m \partial_j \u_{R, i}^\eps \r }}  \\
	    & \underset{U_{32}}{\underbrace{ - \sum_{m' < m} C_m^{m'} \l \mu_2 \partial^{m-m'} \dr_{0,j} \partial^{m'} ( \D_{\u_0 + \sqrt{\eps} \u_R^\eps} \dr_R^\eps )_i + \mu_3 \partial^{m-m'} \dr_{0,i} \partial^{m'} ( \D_{\u_0 + \sqrt{\eps} \u_R^\eps} \dr_R^\eps )_j , \partial^m \partial_j \u_{R, i}^\eps \r }} \,.
	  \end{aligned}
	\end{equation}
	It is easily yielded that by the analogous arguments in computations of $U_{11}$ or $U_{21}$
	\begin{equation}
	  \begin{aligned}
	    U_{31} = & - \big\langle (\mu_2 - \mu_3) \dr_{0,j} \partial^m ( \D_{\u_0 + \sqrt{\eps} \u_R^\eps} \dr_R^\eps )_i \\
	    & + \mu_3 ( \dr_{0,j} \partial^m ( \D_{\u_0 + \sqrt{\eps} \u_R^\eps} \dr_R^\eps )_i + \dr_{0,i} \partial^m ( \D_{\u_0 + \sqrt{\eps} \u_R^\eps} \dr_R^\eps )_j ) , \partial^m \A_{R,ij}^\eps + \partial^m \B_{R, ij}^\eps \big\rangle \\
	    = & - \lambda_1 \l \partial^m D_{\u_0 + \sqrt{\eps} \u_R^\eps} \dr_R^\eps , (\partial^m \A_R^\eps ) \dr_0 \r  - 2 \mu_3 \l \partial^m D_{\u_0 + \sqrt{\eps} \u_R^\eps} \dr_R^\eps , (\partial^m \A_R^\eps ) \dr_0 \r \\
	    & + \lambda_1 \l \partial^m D_{\u_0 + \sqrt{\eps} \u_R^\eps} \dr_R^\eps , (\partial^m \B_R^\eps ) \dr_0 \r \\
	    = & \lambda_1 \l \partial^m D_{\u_0 + \sqrt{\eps} \u_R^\eps} \dr_R^\eps , (\partial^m \B_R^\eps ) \dr_0 \r + \lambda_2 \l \partial^m D_{\u_0 + \sqrt{\eps} \u_R^\eps} \dr_R^\eps , (\partial^m \A_R^\eps ) \dr_0 \r \,.
	  \end{aligned}
	\end{equation}
	We thereby obtain
	\begin{equation}\label{C-u-inner}
	  \begin{aligned}
	    \l \div \partial^m \mathcal{C}_{\u} , \partial^m \u_R^\eps \r = & \lambda_1 \| ( \partial^m \B_R^\eps ) \dr_0 \|^2_{L^2} + \lambda_2 \l \partial^m \D_{\u_0 + \sqrt{\eps} \u_R^\eps} \dr_R^\eps + 2 ( \partial^m \B_R^\eps ) \dr_0 , ( \partial^m \A_R^\eps ) \dr_0 \r \\
	    & - (\mu_5 + \mu_6) \| ( \partial^m \A_R^\eps ) \dr_0 \|^2_{L^2} + \lambda_1 \l \partial^m \D_{\u_0 + \sqrt{\eps} \u_R^\eps } \dr_R^\eps , ( \partial^m \B_R^\eps ) \dr_0 \r \\
	    & + U_{12} + U_{22} + U_{32} \,.
	  \end{aligned}
	\end{equation}
	
	We finally compute the quantity $ \l \partial^m \mathcal{C}_{\dr} , \partial^m \D_{\u_0 + \sqrt{\eps} \u_R^\eps} \dr_R^\eps \r $. Recalling the definition of $\mathcal{C}_{\dr}$ in \eqref{C-d}, we deduce
	\begin{equation}\label{C-d-inner}
	  \begin{aligned}
	    & \l \partial^m \mathcal{C}_{\dr} , \partial^m \D_{\u_0 + \sqrt{\eps} \u_R^\eps} \dr_R^\eps \r \\
	    = &  \lambda_1 \l ( \partial^m \B_R^\eps ) \dr_0 , \partial^m \D_{\u_0 + \sqrt{\eps} \u_R^\eps} \dr_R^\eps \r + \lambda_2 \l ( \partial^m \A_R^\eps ) \dr_0 , \partial^m \D_{\u_0 + \sqrt{\eps} \u_R^\eps} \dr_R^\eps \r \\
	    + & \underset{U_4}{\underbrace{ \sum_{m' < m} C_m^{m'} \l \lambda_1 ( \partial^{m'} \B_R^\eps ) \partial^{m-m'} \dr_0 + \lambda_2 ( \partial^{m'} \A_R^\eps ) \partial^{m-m'} \dr_0 , \partial^m \D_{\u_0 + \sqrt{\eps} \u_R^\eps} \dr_R^\eps \r }} \,.
	  \end{aligned}
	\end{equation}
	Adding the equality \eqref{C-u-inner} to \eqref{C-d-inner} tells us
	\begin{equation}
	  \begin{aligned}
	    & \l \div \partial^m \mathcal{C}_{\u} , \partial^m \u_R^\eps \r + \l \partial^m \mathcal{C}_{\dr} , \partial^m \D_{\u_0 + \sqrt{\eps} \u_R^\eps} \dr_R^\eps \r \\
	    = & \lambda_1 \| ( \partial^m \B_R^\eps ) \dr_0 \|^2_{L^2} + 2 \lambda_2 \l \partial^m \D_{\u_0 + \sqrt{\eps} \u_R^\eps} \dr_R^\eps + ( \partial^m \B_R^\eps ) \dr_0 , ( \partial^m \A_R^\eps ) \dr_0 \r \\
	    & - (\mu_5 + \mu_6) \| ( \partial^m \A_R^\eps ) \dr_0 \|^2_{L^2} + 2 \lambda_1 \l \partial^m \D_{\u_0 + \sqrt{\eps} \u_R^\eps} \dr_R^\eps , ( \partial^m \B_R^\eps ) \dr_0 \r \\
	    & + U_{12} + U_{22} + U_{32} + U_4 \\
	    = & \lambda_1 \left\| \partial^m \D_{\u_0 + \sqrt{\eps} \u_R^\eps} \dr_R^\eps + ( \partial^m \B_R^\eps ) \dr_0 + \tfrac{\lambda_2}{\lambda_1} ( \partial^m \A_R^\eps ) \dr_0 \right\|^2_{L^2} - \lambda_1 \left\| \partial^m \D_{\u_0 + \sqrt{\eps} \u_R^\eps} \dr_R^\eps \right\|^2_{L^2} \\
	    & - ( \mu_5 + \mu_6 + \tfrac{\lambda_2^2}{\lambda_1} ) \left\| (\partial^m \A_R^\eps) \dr_0 \right\|^2_{L^2}  + + U_{12} + U_{22} + U_{32} + U_4 \,,
	  \end{aligned}
	\end{equation}
	and then the proof of Lemma \ref{Lmm-Cancellation} is completed.
\end{proof}

Then, inspired by Lemma 4.1 of \cite{Jiang-Luo-2019-SIAM}, the work of the same two authors of this paper, we derive the following lemma to show how the constraints \eqref{Constraint-1} or \eqref{Constraint-2} hold under the corresponding initial constraints.

\begin{lemma}\label{Lmm-Constraints}
	Assume $(\u_0, \dr_0)$ and $(\u_R^\eps, \dr_R^\eps)$ are two sufficiently smooth solutions to the limit system \eqref{PLQ}-\eqref{IC-ParaLQ} and the remainder system \eqref{Remainder-u-d}-\eqref{IC-Remainder}, respectively. Let $\D_I^\eps (t,x)$ is the initial layer defined in \eqref{Initial-Layer}. If the initial constraint
	\begin{equation}\label{IC-constraint-1}
	  \begin{aligned}
	    2 ( \dr_0 \cdot ( \dr_R^\eps + \sqrt{\eps} \D_I^\eps ) ) (0,x) + \sqrt{\eps} |\dr_R^\eps + \sqrt{\eps} \D_I^\eps|^2 (0,x) = 0
	  \end{aligned}
	\end{equation}
	holds, then the constraint \eqref{Constraint-1}, i.e.,
	\begin{equation}
	  \begin{aligned}
	    2 \dr_0 \cdot ( \dr_R^\eps + \sqrt{\eps} \D_I^\eps ) + \sqrt{\eps} |\dr_R^\eps + \sqrt{\eps} \D_I^\eps|^2 = 0
	  \end{aligned}
	\end{equation}
	holds for all $t \geq 0$.
\end{lemma}

\begin{remark}\label{Rmk-constraint-2}
	For the remainder system \eqref{Remainder-u-d-2}-\eqref{IC-Rmd-2} corresponding to the well-prepared initial data, the initial layer $\D_I^\eps (t,x) \equiv 0$. If the initial constraint $ 2 (\dr_0 \cdot \dr_R^\eps) (0,x) + \sqrt{\eps} |\dr_R^\eps|^2 (0,x) = 0 $ holds, then the constraint \eqref{Constraint-2}, hence $2 \dr_0 \cdot \dr_R^\eps + \sqrt{\eps} |\dr_R^\eps|^2 = 0$ still holds for all $t \geq 0$.
\end{remark}

Before proving Lemma \ref{Lmm-Constraints}, for convenience to readers, we list Lemma 4.1 in \cite{Jiang-Luo-2019-SIAM} as follows:
\begin{lemma}[Lemma 4.1 in \cite{Jiang-Luo-2019-SIAM}]\label{Lmm-Lagrangian-gamma}
	Assume $(\u^\eps , \dr^\eps)$ is a sufficiently smooth solution to the system \eqref{HLQ}-\eqref{IC-HyperLQ}. If the constraint $|\dr^\eps| = 1$ is further assumed, then the Lagrangian multiplier $\bm{\gamma}^\eps$ is \eqref{gamma-eps}, i.e.,
	\begin{equation*}
	\begin{aligned}
	\bm{\gamma}^\eps = - \eps | \D_{\u^\eps} \dr^\eps |^2 + |\nabla \dr^\eps|^2 - \lambda_2 \A^\eps : \dr^\eps \otimes \dr^\eps \,.
	\end{aligned}
	\end{equation*}
	Conversely, if we give the form of $\bm{\gamma}^\eps$ as \eqref{gamma-eps} and $\dr^\eps$ satisfies the initial compatibility $|\dr^\eps| \big|_{t=0} = 1$ and $ ( \dr^\eps \cdot \D_{\u^\eps} \dr^\eps ) |_{t=0} = 0 $, then $|\dr^\eps| = 1$ holds at any time.
\end{lemma}

\begin{proof}[Proof of Lemma \ref{Lmm-Constraints}]
	According to the formal analysis given in Section Introduction, we know that
	\begin{equation}
	  \left\{
	    \begin{array}{l}
	      \u^\eps (t,x) = \u_0 (t,x) + \sqrt{\eps} \u_R^\eps (t,x) \,, \\ [2mm]
	      \dr^\eps (t,x) = \dr_0 (t,x) + \eps \D_I^\eps (t,x) + \sqrt{\eps} \dr_R^\eps (t,x)
	    \end{array}
	  \right.
	\end{equation}
	subject to the equations of \eqref{HLQ} but without the geometric constraint $|\dr^\eps|=1$. Via the initial relations \eqref{IC-relations}, one easily derives from \eqref{IC-constraint-1} and the compatibilities \eqref{IC-compatibility} that
	\begin{equation}
	  \begin{aligned}
	    |\dr^\eps| (0,x) = 1 \,, \ ( \dr^\eps \cdot \D_{\u^\eps} \dr^\eps ) (0,x) = 0 \,.
	  \end{aligned}
	\end{equation}
	Therefore, Lemma \ref{Lmm-Lagrangian-gamma} implies that $| \dr^\eps | = 1$ holds for all $t \geq 0$. Noticing that $|\dr_0 (t,x)| \equiv 1$, we easily derive from the expression of $\dr^\eps (t,x)$ and the geometric constraint $\dr^\eps (t,x) \in \mathbb{S}^2$ that the constraint \eqref{Constraint-1} holds for all $t \geq 0$. Then the proof of Lemma \ref{Lmm-Constraints} is finished.
\end{proof}

The same arguments in justifying Lemma \ref{Lmm-Constraints} will also imply the conclusions of Remark \ref{Rmk-constraint-2}, just letting $\D_I^\eps (t,x) \equiv 0$.

\section{A priori uniform energy estimates for the remainder system}\label{Sec:Uniform-Ener}

In this section, we will derive the a priori uniform (in small $\eps > 0$) energy estimates for the remainder systems \eqref{Remainder-u-d} or \eqref{Remainder-u-d-2} globally in times. Notice that there two differences between the equations \eqref{Remainder-u-d} and \eqref{Remainder-u-d-2}:

1) the system \eqref{Remainder-u-d} involves the terms $\sqrt{\eps} \partial_t ( \u_R^\eps \cdot \nabla \D_I^\eps )$, $\eps \div \mathcal{Q}_{\u} (\D_I)$ and $\mathcal{Q}_{\dr} (\D_I)$;

2) The initial conditions of \eqref{Remainder-u-d} are nontrivial (see \eqref{IC-Remainder}) but that of \eqref{Remainder-u-d-2} are all imposed on zero (see \eqref{IC-Rmd-2}).

However, the initial data \eqref{IC-Remainder} of \eqref{Remainder-u-d} are infinitely small quantities and the initial layer $\D_I^\eps (t,x)$ defined in \eqref{Initial-Layer} is also infinitely small as $\eps \ll 1$. The terms $\eps \div \mathcal{Q}_{\u} (\D_I)$ and $\mathcal{Q}_{\dr} (\D_I)$ involved in \eqref{Remainder-u-d}-\eqref{IC-Remainder} will not result to any more difficulty in the energy estimates comparing to the energy estimates of the system \eqref{Remainder-u-d-2}-\eqref{IC-Rmd-2}. For term $\sqrt{\eps} \partial_t ( \u_R^\eps \cdot \nabla \D_I^\eps )$, it will be dealt with the same way as the term $\partial_t ( \u_R^\eps \cdot \nabla \dr_0 )$, the details of which will be given later. To avoid the tedious calculations in controlling the tedious terms $\eps \div \mathcal{Q}_{\u}$ and $\mathcal{Q}_{\dr}$, we only derive the a priori estimates of the remainder system \eqref{Remainder-u-d-2} with the initial conditions \eqref{IC-Rmd-2}, for simplicity. Actually, its calculations remain very annoying and complicated.

Next, we aim at deriving the a priori estimates of the remainder equations \eqref{Remainder-u-d-2} with the initial data \eqref{IC-Rmd-2}. The {\em key points} are:
\begin{itemize}
	\item The $H^N$-norm of term $\partial_t (\u_R^\eps \cdot \nabla \dr_0)$ in the $\dr_R^\eps$-equation of \eqref{Remainder-u-d-2} has the same regularity as the $H^N$-norm of $\Delta \u_R^\eps$ and the energy dissipative rate will only supply a regularity of $\| \nabla \u_R^\eps \|^2_{H^N}$, so that we can not crudely view it as a source term to be controlled in the right-hand side of the energy inequality. To overcome this, we deal with this term as an energy term, just like $\partial_t \u_R^\eps$ or $\partial_t \D_{\u_0 + \sqrt{\eps} \u_R^\eps } \dr_R^\eps$. Consequently, the energy functional of the remainder system \eqref{Remainder-u-d-2} is depended on the vector field $\dr_0$, which is a solution to the limit equations \eqref{PLQ}.
	
	\item The relations \eqref{Cancellations} is essential to deal the the terms $\mathcal{C}_{\u}$ and $\mathcal{C}_{\dr}$, which are all linearly dependent on the $\u_R^\eps$ and $\dr_R^\eps$ and with the coefficient $\dr_0$. We derive some useful dissipative structures from these two terms. So, we will derive an energy dissipative rate of the remainder system depending also on $\dr_0$.
\end{itemize}

Before deriving the a priori uniform energy estimates on the remainder system \eqref{Remainder-u-d-2}, we state the global existence, which has been proved by Wang-Zhang-Zhang in \cite{Wang-Zhang-Zhang-2013-ARMA}, to the incompressible parabolic Ericksen-Leslie's liquid crystal model \eqref{PLQ} with small initial data \eqref{IC-ParaLQ}. For convenience to readers, we restate this result here. We first define the following energy functionals $\mathscr{E}_{s,0}$ and $\mathscr{D}_{s,0}$ for any integer $s \geq 2$ (see Section 5 of \cite{Wang-Zhang-Zhang-2013-ARMA}):
\begin{equation}
  \begin{aligned}
    \mathscr{E}_{s,0} (t) = & \ \| \nabla \dr_0 \|^2_{L^2} + \| \nabla \Delta^s \dr_0 \|^2_{L^2} + \| \u_0 \|^2_{L^2} + \| \Delta^s \u_0 \|^2_{L^2} \,, \\
    \mathscr{D}_{s,0} (t) = & \ \tfrac{1}{- \lambda_1} \| \Delta \dr_0 \|^2_{L^2} + \tfrac{1}{- \lambda_1} \| \Delta^{s+1} \dr_0 \|^2_{L^2} + \tfrac{1}{2} \mu_4 \| \nabla \u_0 \|^2_{L^2} + \tfrac{1}{2} \mu_4 \| \nabla \Delta^s \u_0 \|^2_{L^2} \,.
  \end{aligned}
\end{equation}
By interpolation, one easily know that
\begin{equation}\label{Bnd-Es0-Ds0}
  \begin{aligned}
    \mathscr{E}_{s,0} (t) & \thicksim \| \u_0 \|^2_{H^{2s}} + \| \nabla \dr_0 \|^2_{H^{2s}} \,, \\
    \mathscr{D}_{s,0} (t) & \thicksim \| \nabla \u_0 \|^2_{H^{2s}} + \| \Delta \dr_0 \|^2_{H^{2s}} \,.
  \end{aligned}
\end{equation}
Then, the following result holds:
\begin{proposition}[Wang-Zhang-Zhang in \cite{Wang-Zhang-Zhang-2013-ARMA}]\label{Prop-WP-PLQ}
	Let $s \geq 2$ be an integer. Assume that the Leslie coefficients satisfy \eqref{Coefficients-Relations} and \eqref{Assumption-coefficients}, and the initial data $ \nabla \dr^{in} \in H^{2s} $, $ \u^{in} \in H^{2s} $. Then if there exists a $\beta_{s,0} > 0$ such that
	\begin{equation*}
	  \| \nabla \dr^{in} \|^2_{H^{2s}} + \| \u^{in} \|^2_{H^{2s}} \leq \beta_{s,0} \,,
	\end{equation*}
	the incompressible parabolic Ericksen-Leslie's liquid crystal model \eqref{PLQ} with initial conditions \eqref{IC-ParaLQ} admits a unique global classical solution
	\begin{equation*}
	  \u_0, \nabla \dr_0 \in C( \R^+; H^{2s} )\,, \ \nabla \u_0 \in L^2( \R^+; H^{2s} )
	\end{equation*}
	satisfying the uniform bound
	\begin{equation}\label{u0-d0-Bnd}
	  \begin{aligned}
	    \| \u_0 \|^2_{H^{2s}} + \| \nabla \dr_0 \|^2_{H^{2s}} \leq c_0^{-1} \mathscr{E}_{s,0} \leq c_0^{-1} C_0 \beta_{s,0}
	  \end{aligned}
	\end{equation}
	holds for all $t \geq 0$. Moreover, the following energy inequality holds:
	\begin{equation}\label{u0-d0-differential-ineq}
	  \begin{aligned}
	    \tfrac{\d}{\d t} \mathscr{E}_{s,0} + \mathscr{D}_{s,0} \leq 0 \,, \quad \forall t \geq 0 \,.
	  \end{aligned}
	\end{equation}
\end{proposition}

We now introduce the following energy functional $\mathscr{E}_{N, \eps} (t)$
\begin{equation}
  \begin{aligned}
    \mathscr{E}_{N, \eps} (t) = & \tfrac{1}{\eps} \| \u_R^\eps \|^2_{H^N} + ( 1 - \delta ) \| \D_{\u_0 + \sqrt{\eps} \u_R^\eps } \dr_R^\eps \|^2_{H^N} + \tfrac{1}{\eps} \| \nabla \dr_R^\eps \|^2_{H^N} \\
    & + ( \tfrac{- \delta \lambda_1}{\eps} - \tfrac{5}{4} \delta ) \| \dr_R^\eps \|^2_{H^N} + \| \u_R^\eps \cdot \nabla \dr_0 + \tfrac{\delta}{2} \dr_R^\eps \|^2_{H^N} \\
    & + \delta \| \D_{\u_0 + \sqrt{\eps} \u_R^\eps } \dr_R^\eps + \dr_R^\eps \|^2_{H^N}  + 2 \sum_{|m| \leq N} \l \partial^m ( \u_R^\eps \cdot \nabla \dr_0 ), \partial^m \D_{\u_0 + \sqrt{\eps} \u_R^\eps } \dr_R^\eps \r
  \end{aligned}
\end{equation}
and the energy dissipative rate $\mathscr{D}_{N, \eps} (t)$
\begin{equation}
  \begin{aligned}
    \mathscr{D}_{N, \eps} (t) = & \tfrac{3 \mu_4}{8 \eps} \| \nabla \u_R^\eps \|^2_{H^N} + \tfrac{\delta}{2 \eps} \| \nabla \dr_R^\eps \|^2_{H^N} - \delta \| \D_{\u_0 + \sqrt{\eps} \u_R^\eps } \dr_R^\eps \|^2_{H^N}  \\
    & + \tfrac{\mu_1}{\eps} \sum_{|m| \leq N} \| (\partial^m \A_R^\eps) : \dr_0 \otimes \dr_0 \|^2_{L^2} + \tfrac{1}{\eps} ( \mu_5 + \mu_6 + \tfrac{\lambda_2^2}{\lambda_1} ) \sum_{|m| \leq N} \| ( \partial^m \A_R^\eps ) \dr_0 \|^2_{L^2} \\
    &  + \tfrac{- \lambda_1}{\eps} \sum_{|m| \leq N} \| \partial^m \D_{\u_0 + \sqrt{\eps} \u_R^\eps } \dr_R^\eps + ( \partial^m \B_R^\eps ) \dr_0 + \tfrac{\lambda_2}{\lambda_1} ( \partial^m \A_R^\eps ) \dr_0 \|^2_{L^2} \,,
  \end{aligned}
\end{equation}
where $\delta \in (0, \tfrac{1}{2}]$ is a fixed constant, depending only on $\lambda_1$, $\lambda_2$ and $N$.

One notices that the energy $\mathscr{E}_{N, \eps} (t)$ and the energy dissipative rate $\mathscr{D}_{N, \eps } (t)$ may not be nonnegative, since there is an indefinitely signed term $ \sum_{|m| \leq N} \l \partial^m \u_R^\eps \cdot \nabla \dr_0 , \partial^m \D_{\u_0 + \sqrt{\eps} \u_R^\eps } \dr_R^\eps \r $ appearing in $\mathscr{E}_{N, \eps} (t)$ and the functional $ \mathscr{D}_{N, \eps} (t) $ includes a negative term $ - \delta \| \D_{\u_0 + \sqrt{\eps} \u_R^\eps } \dr_R^\eps \|^2_{H^N} $. However, if the inertia constant $\eps > 0$ is sufficiently small, the functionals $\mathscr{E}_{N, \eps} (t)$ and $ \mathscr{D}_{N, \eps} (t) $ will be both nonnegative. More precisely, we derive the following lemma.

\begin{lemma}\label{Lmm-Energies}
	If the integer $N \geq 2$ and the Leslie's coefficients satisfy relations \eqref{Coefficients-Relations} and \eqref{Assumption-coefficients}, then there is a small $\eps_0 > 0$, depending only on $N$, $\beta_{S_{\!_N}, 0}$, and the all Leslie's coefficients, such that the energy $\mathscr{E}_{N, \eps} (t)$ and the energy dissipative rate $\mathscr{D}_{N, \eps} (t)$ are both nonnegative for any $\eps \in ( 0, \eps_0 )$. Moreover, for all $\eps \in (0, \eps_0)$, we have
	\begin{equation}\label{Energy-E-bnd}
	  \begin{aligned}
	     \mathscr{E}_{N, \eps} (t) \thicksim  \tfrac{1}{\eps} \| \u_R^\eps \|^2_{H^N} + \| \D_{\u_0 + \sqrt{\eps} \u_R^\eps } \dr_R^\eps \|^2_{H^N} + \tfrac{1}{\eps} \| \nabla \dr_R^\eps \|^2_{H^N} + \tfrac{1}{\eps} \| \dr_R^\eps \|^2_{H^N}  \,,
	  \end{aligned}
	\end{equation}
	and
	\begin{equation}\label{Energy-D-bnd}
	  \begin{aligned}
	     \mathscr{D}_{N, \eps} (t) \thicksim &  \tfrac{1}{\eps} \| \nabla \u_R^\eps \|^2_{H^N} + \tfrac{1}{\eps} \| \nabla \dr_R^\eps \|^2_{H^N} + \| \D_{\u_0 + \sqrt{\eps} \u_R^\eps } \dr_R^\eps \|^2_{H^N} \\
	    &  + \tfrac{1}{\eps} \sum_{|m| \leq N} \| \partial^m \D_{\u_0 + \sqrt{\eps} \u_R^\eps } \dr_R^\eps + ( \partial^m \B_R^\eps ) \dr_0 + ( \partial^m \A_R^\eps ) \dr_0 \|^2_{L^2} \,.
	  \end{aligned}
	\end{equation}
	Here the small positive constant $\beta_{S_{\!_N}, 0}$ is given in Proposition \ref{Prop-WP-PLQ}.
\end{lemma}

\begin{proof}[Proof of Lemma \ref{Lmm-Energies}]
	We now find a constant $\eps_0 > 0$ such that
	\begin{equation}
	  \begin{aligned}
	    \mathscr{E}_{N, \eps} \geq 0 \,, \ \textrm{and } \mathscr{D}_{N, \eps } (t) \geq 0
	  \end{aligned}
	\end{equation}
	for all $\eps \in (0, \eps_0)$. First, we require the coefficient $ \tfrac{- \delta \lambda_1}{ \eps} - \tfrac{5}{4} \delta $ in $\mathscr{E}_{N, \eps} (t)$ satisfies
	\begin{equation*}
	  \tfrac{- \delta \lambda_1}{\eps} - \tfrac{5}{4} \delta \geq \tfrac{- \delta \lambda_1}{2 \eps} > 0 \,,
	\end{equation*}
	which means
	\begin{equation}
	 0 < \eps \leq \tfrac{- 2 \lambda_1}{5} \,.
	\end{equation}
	Noticing that $0 < \delta \leq \tfrac{1}{2}$, we have
	\begin{equation}
	  \begin{aligned}
	    & \| \D_{\u_0 + \sqrt{\eps} \u_R^\eps } \dr_R^\eps \|^2_{H^N} - \delta \| \D_{\u_0 + \sqrt{\eps} \u_R^\eps } \dr_R^\eps \|^2_{\dot{H}^N} \\
	    \geq & \tfrac{1}{2} \| \D_{\u_0 + \sqrt{\eps} \u_R^\eps } \dr_R^\eps \|^2_{H^N} + ( \tfrac{1}{2} - \delta ) \| \D_{\u_0 + \sqrt{\eps} \u_R^\eps } \dr_R^\eps \|^2_{\dot{H}^N} + \tfrac{1}{2} \| \D_{\u_0 + \sqrt{\eps} \u_R^\eps } \dr_R^\eps \|^2_{L^2} \\
	    \geq & \tfrac{1}{2} \| \D_{\u_0 + \sqrt{\eps} \u_R^\eps } \dr_R^\eps \|^2_{H^N} \geq  0 \,.
	  \end{aligned}
	\end{equation}
	Then by utilizing the H\"older inequality, the Sobolev embedding theory, the Young's inequality and the bound \eqref{u0-d0-Bnd} in Proposition \ref{Prop-WP-PLQ}, we estimate
	\begin{equation}
	  \begin{aligned}
	    & \Big| \sum_{|m| \leq N} \l \partial^m ( \u_R^\eps \cdot \nabla \dr_0 ) , \partial^m \D_{\u_0 + \sqrt{\eps} \u_R^\eps } \dr_R^\eps \r \Big| \\
	    \leq & \| \u_R^\eps \cdot \nabla \dr_0 \|_{H^N} \| \D_{\u_0 + \sqrt{\eps} \u_R^\eps } \dr_R^\eps \|_{H^N} \\
	    \leq & C \| \nabla \dr_0 \|_{L^\infty( \R^+ ; H^N)} \| \u_R^\eps \|_{H^N} \| \D_{\u_0 + \sqrt{\eps} \u_R^\eps } \dr_R^\eps \|_{H^N} \\
	    \leq & C \| \nabla \dr_0 \|^2_{L^\infty(\R^+ ; H^N)} \| \u_R^\eps \|^2_{H^N} + \tfrac{1}{4} \| \D_{\u_0 + \sqrt{\eps} \u_R^\eps } \dr_R^\eps \|^2_{H^N}  \\
	    \leq & C \beta_{S_{\!_N},0} \| \u_R^\eps \|^2_{H^N} + \tfrac{1}{4} \| \D_{\u_0 + \sqrt{\eps} \u_R^\eps } \dr_R^\eps \|^2_{H^N}
	  \end{aligned}
	\end{equation}
	for some constant $C = C(N) > 0$. Consequently, we deduce
	\begin{equation}
	  \begin{aligned}
	    & \tfrac{1}{\eps} \| \u_R^\eps \|^2_{H^N} + \tfrac{1}{2} \| \D_{\u_0 + \sqrt{\eps} \u_R^\eps } \dr_R^\eps \|^2_{H^N} + \sum_{|m| \leq N} \l \partial^m ( \u_R^\eps \cdot \nabla \dr_0 ) , \partial^m \D_{\u_0 + \sqrt{\eps} \u_R^\eps } \dr_R^\eps \r \\
	    \geq & \big( \tfrac{1}{\eps} - C \beta_{S_{\!_N},0} \big) \| \u_R^\eps \|^2_{H^N} + \tfrac{1}{4} \| \D_{\u_0 + \sqrt{\eps} \u_R^\eps } \dr_R^\eps \|^2_{H^N} \,.
	  \end{aligned}
	\end{equation}
	If we require $\tfrac{1}{\eps} - C \beta_{S_{\!_N},0} \geq \tfrac{1}{2 \eps} $, hence $ 0 < \eps \leq \tfrac{1}{2 C \beta_{S_{\!_N},0}} $, then
	\begin{equation}
	  \begin{aligned}
	    & \tfrac{1}{\eps} \| \u_R^\eps \|^2_{H^N} + \tfrac{1}{2} \| \D_{\u_0 + \sqrt{\eps} \u_R^\eps } \dr_R^\eps \|^2_{H^N} + \sum_{|m| \leq N} \l \partial^m ( \u_R^\eps \cdot \nabla \dr_0 ) , \partial^m \D_{\u_0 + \sqrt{\eps} \u_R^\eps } \dr_R^\eps \r \\
	    & \geq \tfrac{1}{2 \eps} \| \u_R^\eps \|^2_{H^N} + \tfrac{1}{4} \| \D_{\u_0 + \sqrt{\eps} \u_R^\eps } \dr_R^\eps \|^2_{H^N} \geq 0 \,.
	  \end{aligned}
	\end{equation}
	As a result, if $0 < \eps \leq \min \big\{ \tfrac{- 2 \lambda_1 }{5} , \tfrac{1}{2 C \beta_{S_{\!_N},0} } \big\}$, we have $\mathscr{E}_{N, \eps} (t) \geq 0$.
	
	Next we consider the energy dissipative rate $ \mathscr{D}_{N , \eps} (t) $. One observes that there is only a negative term $ - \delta \| \D_{\u_0 + \sqrt{\eps} \u_R^\eps } \dr_R^\eps \|^2_{H^N}$ in $\mathscr{D}_{N, \eps} (t)$ under assumption \eqref{Assumption-coefficients}. Via the following elementary estimates
	\begin{equation}
	  \begin{aligned}
	    \| \D_{\u_0 + \sqrt{\eps} \u_R^\eps } \dr_R^\eps \|^2_{H^N} \leq & 2 \sum_{|m| \leq N } \Big( \| \partial^m \D_{\u_0 + \sqrt{\eps} \u_R^\eps } \dr_R^\eps + ( \partial^m \B_R^\eps ) \dr_0 + \tfrac{\lambda_2}{\lambda_1} ( \partial^m \A_R^\eps ) \dr_0 \|^2_{L^2} \\
	    & + 2 \| ( \partial^m \B_R^\eps ) \dr_0 + \tfrac{\lambda_2}{\lambda_1} ( \partial^m \A_R^\eps ) \dr_0 \|^2_{L^2}  \Big) \\
	    \leq & 2 \sum_{|m| \leq N} \| \partial^m \D_{\u_0 + \sqrt{\eps} \u_R^\eps } \dr_R^\eps + ( \partial^m \B_R^\eps ) \dr_0 + \tfrac{\lambda_2}{\lambda_1} ( \partial^m \A_R^\eps ) \dr_0 \|^2_{L^2} \\
	    & + 4 ( 1 + \tfrac{\lambda_2^2}{\lambda_1^2} ) \| \nabla \u_R^\eps \|^2_{H^N} \,,
	  \end{aligned}
	\end{equation}
	we know
	\begin{equation}\label{Dissip-D-part}
	  \begin{aligned}
	    & \tfrac{3 \mu_4}{8 \eps} \| \nabla \u_R^\eps \|^2_{H^N} - 2 \delta \| \D_{\u_0 + \sqrt{\eps} \u_R^\eps } \dr_R^\eps \|^2_{H^N} \\
	    & + \tfrac{- \lambda_1}{\eps} \sum_{|m| \leq N} \| \partial^m \D_{\u_0 + \sqrt{\eps} \u_R^\eps } \dr_R^\eps + ( \partial^m \A_R^\eps ) \dr_0 + \tfrac{\lambda_2}{\lambda_1} ( \partial^m \A_R^\eps ) \dr_0 \|^2_{L^2} \\
	    \geq & \big( \tfrac{3 \mu_4}{8 \eps} - 8 \delta - \tfrac{ 8 \delta \lambda_2^2}{\lambda_1^2} \big) \| \nabla \u_R^\eps \|^2_{H^N} \\
	    & + \big( \tfrac{- \lambda_1}{\eps} - 4 \delta \big) \sum_{|m| \leq N} \| \partial^m \D_{\u_0 + \sqrt{\eps} \u_R^\eps } \dr_R^\eps + ( \partial^m \A_R^\eps ) \dr_0 + \tfrac{\lambda_2}{\lambda_1} ( \partial^m \A_R^\eps ) \dr_0 \|^2_{L^2} \,.
	  \end{aligned}
	\end{equation}
	If we require $ \tfrac{3 \mu_4}{8 \eps} - 8 \delta - \tfrac{ 8 \delta \lambda_2^2}{\lambda_1^2} \geq \tfrac{\mu_4}{4 \eps} $ and $ \tfrac{- \lambda_1}{\eps} - 4 \delta \geq \tfrac{- \lambda_1}{2 \eps} $, namely, $ 0 < \eps \leq \min \big\{ \tfrac{- \lambda_1}{8 \delta} , \tfrac{\mu_4 \lambda_1^2}{64 \delta ( \lambda_1^2 + \lambda_2^2 )} \big\} $, then the quantity in the right-hand side of the inequality \eqref{Dissip-D-part} has a lower bound
	\begin{equation*}
	  \begin{aligned}
	    \tfrac{\mu_4}{4 \eps} \| \nabla \u_R^\eps \|^2_{H^N} + \tfrac{- \lambda_1}{2 \eps} \sum_{|m| \leq N} \| \partial^m \D_{\u_0 + \sqrt{\eps} \u_R^\eps } \dr_R^\eps + ( \partial^m \A_R^\eps ) \dr_0 + \tfrac{\lambda_2}{\lambda_1} ( \partial^m \A_R^\eps ) \dr_0 \|^2_{L^2} \geq 0 \,.
	  \end{aligned}
	\end{equation*}
	In summary, we can take $ \eps_0 = \min \big\{ 1 , \tfrac{- 2 \lambda_1}{5}, \tfrac{- \lambda_1}{8 \delta}, \tfrac{\mu_4 \lambda_1^2}{64 \delta ( \lambda_1^2 + \lambda_2^2 )} , \tfrac{1}{2 C \beta_{S_{\!_N},0} } \big\} > 0$, so that $\mathscr{E}_{N, \eps} ( t) \geq 0$ and $\mathscr{D}_{N, \eps} (t) \geq 0$ hold for all $\eps \in ( 0 , \eps_0 ]$. Moreover, following the previous estimates, one can easily derive the lower bounds of the inequalities \eqref{Energy-E-bnd} and \eqref{Energy-D-bnd}, and the upper bounds are obviously holds. Then the proof of Lemma \ref{Lmm-Energies} is finished.	
\end{proof}

Next we derive the following key energy inequality, which will reduce to the uniform energy bound under the assumption of small size of the initial data. More precisely, we will give the following proposition.

\begin{proposition}\label{Prop-Unif-Bnd}
	Let $N \geq 2$ be an integer and assume that $ ( \u_R^\eps , \dr_R^\eps ) $ is a sufficiently smooth solution to the remainder system \eqref{Remainder-u-d} on $[0,T]$. Then there are constants $C > 0$ and $\theta_0 \gg 1$, depending only on the Leslie coefficients and $\beta_{S_{\!_N},0}$ given in Proposition \ref{Prop-WP-PLQ}, such that
	\begin{equation}\label{Apriori-Energy-Inq}
	\begin{aligned}
	& \tfrac{\d}{\d t} \Big[ \mathscr{E}_{N,\eps} (t) + \theta_0 \mathscr{E}_{S_{\!_N}, 0} (t) \Big] + \mathscr{D}_{N,\eps} (t) + \tfrac{\theta_0}{2} \mathscr{D}_{S_{\!_N}, 0} (t) \\
	\leq & C \big[ \mathscr{E}^2_{N,\eps} (t) + \mathscr{E}^\frac{1}{2}_{N,\eps} (t) + \mathscr{E}^\frac{1}{2}_{S_{\!_N},\eps} (t) \big] \big[ \mathscr{D}_{N,\eps} (t) + \tfrac{\theta_0}{2} \mathscr{D}_{S_{\!_N}, 0} (t) \big]
	\end{aligned}
	\end{equation}
	holds for all $t \in [0,T]$ and $\eps \in ( 0, \eps_0 ]$, where the small positive constant $\eps_0$ is mentioned in Lemma \ref{Lmm-Energies}, and the integer $S_{_{\!N}}$ is defined in \eqref{SN-integer}.
\end{proposition}

\begin{proof}[ Proof of Proposition \ref{Prop-Unif-Bnd} ]
	For all multi-indexes $m \in \mathbb{N}^3$ with $|m| \leq N \ ( N \geq 2 )$, we take the derivative operator $\partial^m $ on the first equation of the remainder system \eqref{Remainder-u-d} and take $L^2$-inner product by multiplying $\partial^m \u_R^\eps$ and integrating by parts over $x \in \T^3$. We hence obtain
	\begin{equation}\label{u-equ}
	  \begin{aligned}
	    & \tfrac{1}{2} \tfrac{\d}{\d t} \| \partial^m \u_R^\eps \|^2_{L^2} + \tfrac{1}{2} \mu_4 \| \nabla \partial^m \u_R^\eps \|^2_{L^2} + \mu_1 \big\| \partial^m  \A_R^\eps : \dr_0 \otimes \dr_0 \big\|^2_{L^2} \\
	    = & - \mu_1 \sum_{0 \neq m' \leq m} C_m^{m'} \sum_{m'' \leq m'} C_{m'}^{m''} \l ( \partial^{m - m'} \A_R^\eps : \partial^{m' - m''} ( \dr_0 \otimes \dr_0 ) ) \partial^{m''} ( \dr_0 \otimes \dr_0 ) , \nabla \partial^m \u_R^\eps \r \\
	    & + \l \div \partial^m \mathcal{C}_{\u} , \partial^m \u_R^\eps \r + \l \partial^m \mathcal{K}_{\u} , \partial^m \u_R^\eps \r + \l \partial^m \div ( \mathcal{T}_{\u} + \sqrt{\eps} \mathcal{R}_{\u}  ) , \partial^m \u_R^\eps \r \,,
	  \end{aligned}
	\end{equation}
	where we make use of the divergence-free property of $\u_R^\eps$, the relation $\nabla \partial^m \u_R^\eps = \partial^m \A_R^\eps + \partial^m \B_R^\eps$ and the skew-symmetry of $ \partial^m \B_R^\eps $.
	
	Acting the derivative operator $\partial^m$ on the third equation of \eqref{Remainder-u-d}, taking $L^2$-inner product by dot with $\partial^m \D_{\u_0 + \sqrt{\eps} \u_R^\eps } \dr_R^\eps$ and integrating by parts over $x \in \T^3$, we know
	\begin{equation}\label{d-equ}
	  \begin{aligned}
	    & \tfrac{1}{2} \tfrac{\d}{\d t} \big( \| \partial^m \D_{\u_0 + \sqrt{\eps} \u_R^\eps } \dr_R^\eps \|^2_{L^2} + \tfrac{1}{\eps} \| \nabla \partial^m \dr_R^\eps \|^2_{L^2} \big) \\
	    & + \l \partial_t \partial^m ( \u_R^\eps \cdot \nabla \dr_0 ) , \partial^m \D_{\u_0 + \sqrt{\eps} \u_R^\eps } \dr_R^\eps \r + \tfrac{- \lambda_1}{\eps} \big\| \partial^m \D_{\u_0 + \sqrt{\eps} \u_R^\eps } \dr_R^\eps \big\|^2_{L^2} \\
	    = & - \l \partial^m ( \sqrt{\eps} \u_R^\eps \cdot \nabla \D_{\u_0 + \sqrt{\eps} \u_R^\eps } \dr_R^\eps ) , \partial^m \D_{\u_0 + \sqrt{\eps} \u_R^\eps } \dr_R^\eps \r - \tfrac{1}{\eps} \l \nabla \partial^m \dr_R^\eps , \nabla \partial^m ( \sqrt{\eps} \u_R^\eps \cdot \nabla \dr_R^\eps ) \r \\
	    & + \l \tfrac{1}{\eps} \partial^m \mathcal{C}_{\dr} + \tfrac{1}{\eps} \partial^m \mathcal{S}_{\dr}^1 + \tfrac{1}{\sqrt{\eps}} \partial^m \mathcal{S}_{\dr}^2 + \partial^m \mathcal{R}_{\dr} , \partial^m \D_{\u_0 + \sqrt{\eps} \u_R^\eps } \dr_R^\eps \r \,.
	  \end{aligned}
	\end{equation}
	Accounting for the cancellation \eqref{Cancellations} in Lemma \ref{Lmm-Cancellation}, we add the $\frac{1}{\eps}$ times of \eqref{u-equ} to \eqref{d-equ}, and then deduce that for $|m| \leq N$
	\begin{equation}\label{u-d-equ-1}
	  \begin{aligned}
	    & \tfrac{1}{2} \tfrac{\d}{\d t} \big( \tfrac{1}{\eps} \| \partial^m \u_R^\eps \|^2_{L^2} + \big\| \partial^m \D_{\u_0 + \sqrt{\eps} \u_R^\eps } \dr_R^\eps \big\|^2_{L^2} + \tfrac{1}{\eps} \| \nabla \partial^m \dr_R^\eps \|^2_{L^2} \big) \\
	    & + \tfrac{\mu_4}{2 \eps} \| \nabla \partial^m \u_R^\eps \|^2_{L^2} + \tfrac{\mu_1}{\eps} \| \partial^m \A_R^\eps : \dr_0 \otimes \dr_0 \|^2_{L^2}  + \l \partial_t \partial^m ( \u_R^\eps \cdot \nabla \dr_0 ) , \partial^m \D_{\u_0 + \sqrt{\eps} \u_R^\eps } \dr_R^\eps \r \\
	    & + \tfrac{- \lambda_1}{\eps} \| \partial^m \D_{\u_0 + \sqrt{\eps} \u_R^\eps } \dr_R^\eps + (\partial^m \B_R^\eps) \dr_0 + \tfrac{\lambda_2}{\lambda_1} ( \partial^m \A_R^\eps ) \dr_0 \|^2_{L^2} + \tfrac{1}{\eps} ( \mu_5 + \mu_6 + \tfrac{\lambda_2^2}{\lambda_1} ) \| (\partial^m \A_R^\eps) \dr_0 \|^2_{L^2} \\
	    = & - \tfrac{ \mu_1}{\eps} \sum_{\substack{0 \neq m' \leq m \\ m'' \leq m'} } C_m^{m'} C_{m'}^{m''} \l ( \partial^{m - m'} \A_R^\eps : \partial^{m' - m''} ( \dr_0 \otimes \dr_0 ) ) \partial^{m''} ( \dr_0 \otimes \dr_0 ) , \nabla \partial^m \u_R^\eps \r \\
	    &  + \tfrac{1}{\eps} \l \partial^m \div ( \mathcal{T}_{\u} + \sqrt{\eps} \mathcal{R}_{\u}  ) , \partial^m \u_R^\eps \r + \tfrac{1}{\eps} \l \partial^m \mathcal{K}_{\u} , \partial^m \u_R^\eps \r \\
	    & - \l \partial^m ( \sqrt{\eps} \u_R^\eps \cdot \nabla \D_{\u_0 + \sqrt{\eps} \u_R^\eps } \dr_R^\eps ) , \partial^m \D_{\u_0 + \sqrt{\eps} \u_R^\eps } \dr_R^\eps \r - \tfrac{1}{\eps} \l \nabla \partial^m \dr_R^\eps , \nabla \partial^m ( \sqrt{\eps} \u_R^\eps \cdot \nabla \dr_R^\eps ) \r \\
	    & + \l \tfrac{1}{\eps} \partial^m \mathcal{S}_{\dr}^1 + \tfrac{1}{\sqrt{\eps}} \partial^m \mathcal{S}_{\dr}^2 + \partial^m \mathcal{R}_{\dr} , \partial^m \D_{\u_0 + \sqrt{\eps} \u_R^\eps } \dr_R^\eps \r + \tfrac{1}{\eps} \mathcal{G}_m \,.
	  \end{aligned}
	\end{equation}
	We next deal with the term $ \l \partial_t \partial^m ( \u_R^\eps \cdot \nabla \dr_0 ) , \partial^m \D_{\u_0 + \sqrt{\eps} \u_R^\eps } \dr_R^\eps \r $. Straightforward calculation reduces to
	\begin{equation}
	  \begin{aligned}
	    & \l \partial_t \partial^m ( \u_R^\eps \cdot \nabla \dr_0 ) , \partial^m \D_{\u_0 + \sqrt{\eps} \u_R^\eps } \dr_R^\eps \r \\
	    = & \tfrac{\d}{\d t} \l \partial^m ( \u_R^\eps \cdot \nabla \dr_0 ) , \partial^m \D_{\u_0 + \sqrt{\eps} \u_R^\eps } \dr_R^\eps \r - \l \partial^m ( \u_R^\eps \cdot \nabla \dr_0 ) , \partial^m \partial_t \D_{\u_0 + \sqrt{\eps} \u_R^\eps } \dr_R^\eps \r \\
	    = & \tfrac{\d}{\d t} \l \partial^m ( \u_R^\eps \cdot \nabla \dr_0 ) , \partial^m \D_{\u_0 + \sqrt{\eps} \u_R^\eps } \dr_R^\eps \r + \l \sqrt{\eps} \u_R^\eps \cdot \nabla \partial^m \D_{\u_0 + \sqrt{\eps} \u_R^\eps } \dr_R^\eps , \partial^m ( \u_R^\eps \cdot \nabla \dr_0 ) \r \\
	    & - \l \partial^m ( \u_R^\eps \cdot \nabla \dr_0 ) , \partial^m \D^2_{\u_0 + \sqrt{\eps} \u_R^\eps } \dr_R^\eps \r \,.
	  \end{aligned}
	\end{equation}
	Then, from utilizing the third $\dr_R^\eps$-equation of \eqref{Remainder-u-d}, we derive that
	\begin{equation}
	  \begin{aligned}
	    & - \l \partial^m ( \u_R^\eps \cdot \nabla \dr_0 ) , \partial^m \D^2_{\u_0 + \sqrt{\eps} \u_R^\eps } \dr_R^\eps \r \\
	    = & \tfrac{1}{2} \tfrac{\d}{\d t} \| \partial^m ( \u_R^\eps \cdot \nabla \dr_0 ) \|^2_{L^2} + \tfrac{- \lambda_1}{\eps} \l \partial^m (\u_R^\eps \cdot \nabla \dr_0), \partial^m \D_{\u_0 + \sqrt{\eps} \u_R^\eps } \dr_R^\eps \r \\
	    & + \tfrac{1}{\eps } \l \nabla \partial^m \dr_R^\eps , \nabla \partial^m ( \u_R^\eps \cdot \nabla \dr_0 ) \r \\
	    & - \l \tfrac{1}{\eps} \partial^m \mathcal{C}_{\dr} + \tfrac{1}{\eps} \partial^m \mathcal{S}_{\dr}^1 + \tfrac{1}{\sqrt{\eps}} \partial^m \mathcal{S}_{\dr}^2 + \partial^m \mathcal{R}_{\dr} , \partial^m ( \u_R^\eps \cdot \nabla\dr_0 ) \r \,.
	  \end{aligned}
	\end{equation}
	In summary, we obtain the {\em key} relation to deal with the term $ \partial_t ( \u_R^\eps \cdot \nabla \dr_0 ) $ that
	\begin{equation}\label{u-d0}
	  \begin{aligned}
	    & \l \partial_t \partial^m ( \u_R^\eps \cdot \nabla \dr_0 ) , \partial^m \D_{\u_0 + \sqrt{\eps} \u_R^\eps } \dr_R^\eps \r \\
	    = & \tfrac{1}{2} \tfrac{\d}{\d t} \big( \| \partial^m ( \u_R^\eps \cdot \nabla \dr_0 ) \|^2_{L^2} + 2 \big\langle \partial^m ( \u_R^\eps \cdot \nabla \dr_0   ) , \partial^m \D_{\u_0 + \sqrt{\eps} \u_R^\eps } \dr_R^\eps \big\rangle \big) \\
	    & - \l \sqrt{\eps} \partial^m \D_{\u_0 + \sqrt{\eps} \u_R^\eps } \dr_R^\eps ,  \u_R^\eps \cdot \nabla \partial^m ( \u_R^\eps \cdot \nabla \dr_0 ) \r \\
	    & + \tfrac{- \lambda_1}{\eps} \l \partial^m (\u_R^\eps \cdot \nabla \dr_0), \partial^m \D_{\u_0 + \sqrt{\eps} \u_R^\eps } \dr_R^\eps \r \\
	    &  + \tfrac{1}{\eps } \l \nabla \partial^m \dr_R^\eps , \nabla \partial^m ( \u_R^\eps \cdot \nabla \dr_0 ) \r \\
	    & - \l \tfrac{1}{\eps} \partial^m \mathcal{C}_{\dr} + \tfrac{1}{\eps} \partial^m \mathcal{S}_{\dr}^1 + \tfrac{1}{\sqrt{\eps}} \partial^m \mathcal{S}_{\dr}^2 + \partial^m \mathcal{R}_{\dr} , \partial^m ( \u_R^\eps \cdot \nabla\dr_0 ) \r \,.
	  \end{aligned}
	\end{equation}
	Consequently, it is derived from the equalities \eqref{u-d-equ-1}, \eqref{u-d0} and summing up for $|m| \leq N$ that
	  \begin{align}\label{u-d-2}
	    \no & \tfrac{1}{2} \tfrac{\d}{\d t} \big( \tfrac{1}{\eps} \| \u_R^\eps \|^2_{H^N} + \big\| \D_{\u_0 + \sqrt{\eps} \u_R^\eps} \dr_R^\eps \big\|^2_{H^N} + \tfrac{1}{\eps} \| \nabla \dr_R^\eps \|^2_{H^N} \\
	    \no & \qquad + \| \u_R^\eps \cdot \nabla \dr_0 \|^2_{H^N} + 2 \sum_{|m| \leq N} \big\langle \partial^m ( \u_R^\eps \cdot \nabla \dr_0   ) , \partial^m \D_{\u_0 + \sqrt{\eps} \u_R^\eps} \dr_R^\eps \big\rangle \big) \\
	    \no & + \tfrac{\mu_4}{2 \eps} \| \nabla \u_R^\eps \|^2_{H^N}  + \tfrac{- \lambda_1}{\eps} \sum_{|m| \leq N} \| \partial^m \D_{\u_0 + \sqrt{\eps} \u_R^\eps} \dr_R^\eps + (\partial^m \B_R^\eps) \dr_0 + \tfrac{\lambda_2}{\lambda_1} ( \partial^m \A_R^\eps ) \dr_0 \|^2_{L^2} \\
	    \no &  + \tfrac{\mu_1}{\eps} \sum_{|m| \leq N} \| \partial^m \A_R^\eps : \dr_0 \otimes \dr_0 \|^2_{L^2} + \tfrac{1}{\eps} ( \mu_5 + \mu_6 + \tfrac{\lambda_2^2}{\lambda_1} ) \sum_{|m| \leq N} \| (\partial^m \A_R^\eps) \dr_0 \|^2_{L^2} \\
	    \no = & - \tfrac{ \mu_1}{\eps} \sum_{|m| \leq N} \sum_{\substack{0 \neq m' \leq m \\ m'' \leq m'} } C_m^{m'} C_{m'}^{m''} \l ( \partial^{m - m'} \A_R^\eps : \partial^{m' - m''} ( \dr_0 \otimes \dr_0 ) ) \partial^{m''} ( \dr_0 \otimes \dr_0 ) , \nabla \partial^m \u_R^\eps \r \\
	    \no & + \tfrac{1}{\eps} \sum_{|m| \leq N} \l \partial^m \mathcal{K}_{\u} , \partial^m \u_R^\eps \r + \tfrac{1}{\eps} \sum_{|m| \leq N} \l \partial^m \div \mathcal{T}_{\u} , \partial^m \u_R^\eps \r  + \tfrac{1}{\sqrt{\eps}} \l \partial^m \div \mathcal{R}_{\u}  , \partial^m \u_R^\eps \r\\
	    \no & - \sum_{|m| \leq N} \l \partial^m ( \sqrt{\eps} \u_R^\eps \cdot \nabla \D_{\u_0 + \sqrt{\eps} \u_R^\eps} \dr_R^\eps ) , \partial^m \D_{\u_0 + \sqrt{\eps} \u_R^\eps} \dr_R^\eps \r \\
	    \no & - \tfrac{1}{\sqrt{\eps}} \sum_{|m| \leq N} \l \nabla \partial^m \dr_R^\eps , \nabla \partial^m ( \u_R^\eps \cdot \nabla \dr_R^\eps ) \r \\
	     \no & - \sum_{|m| \leq N} \l \sqrt{\eps} \partial^m \D_{\u_0 + \sqrt{\eps} \u_R^\eps} \dr_R^\eps ,  \u_R^\eps \cdot \nabla \partial^m ( \u_R^\eps \cdot \nabla \dr_0 ) \r \\
	     \no & + \tfrac{- \lambda_1}{\eps} \sum_{|m| \leq N} \l \partial^m (\u_R^\eps \cdot \nabla \dr_0), \partial^m \D_{\u_0 + \sqrt{\eps} \u_R^\eps} \dr_R^\eps \r \\
	     \no & + \tfrac{1}{\eps } \sum_{|m| \leq N} \l \nabla \partial^m \dr_R^\eps , \nabla \partial^m ( \u_R^\eps \cdot \nabla \dr_0 ) \r + \tfrac{1}{\eps} \sum_{|m| \leq N} \mathcal{G}_m - \tfrac{1}{\eps} \sum_{|m| \leq N} \l \partial^m ( \u_R^\eps \cdot \nabla \dr_0 ) , \partial^m \mathcal{C}_{\dr} \r \\
	     & + \sum_{|m| \leq N} \l \tfrac{1}{\eps} \partial^m \mathcal{S}_{\dr}^1 + \tfrac{1}{\sqrt{\eps}} \partial^m \mathcal{S}_{\dr}^2 + \partial^m \mathcal{R}_{\dr} , \partial^m \D_{\u_0 + \sqrt{\eps} \u_R^\eps} \dr_R^\eps - \partial^m ( \u_R^\eps \cdot \nabla\dr_0 ) \r \,.
	  \end{align}
	
	Next we will seek an additional dissipative structure $ \tfrac{1}{\eps} \| \nabla \dr_R^\eps \|^2_{H^N} = \tfrac{1}{\eps} \sum_{|m| \leq N} \| \nabla \partial^m \dr_R^\eps \|^2_{L^2}$ from the term $\tfrac{1}{\eps} \Delta \dr_R^\eps$ in the third $\dr_R^\eps$-equation of the remainder equations \eqref{Remainder-u-d}. More precisely, we act the derivative operator $\partial^m$ on $\eqref{Remainder-u-d}_3$, take $L^2$-inner product by dot with $\partial^m \dr_R^\eps$ and integrate by parts over $x \in \T^3$. Then we have for $ |m| \leq N$
	\begin{equation}\label{d-nabla-d-decay-1}
	  \begin{aligned}
	    & \tfrac{- \lambda_1}{2 \eps} \tfrac{\d}{\d t} \| \partial^m \dr_R^\eps \|^2_{L^2} + \tfrac{1}{\eps} \| \nabla \partial^m \dr_R^\eps \|^2_{L^2} \\
	    & + \l \partial^m \D^2_{\u_0 + \sqrt{\eps} \u_R^\eps } \dr_R^\eps , \partial^m \dr_R^\eps \r + \l \partial_t \partial^m ( \u_R^\eps \cdot \nabla \dr_0 ) , \partial^m \dr_R^\eps \r \\
	    = & \tfrac{\lambda_1}{\eps} \l \partial^m ( (\u_0 + \sqrt{\eps} \u_R^\eps ) \cdot \nabla \dr_R^\eps ) , \partial^m \dr_R^\eps \r \\
	    & + \l \tfrac{1}{\eps} \partial^m \mathcal{C}_{\dr} + \tfrac{1}{\eps} \partial^m \mathcal{S}_{\dr}^1 + \tfrac{1}{\sqrt{\eps}} \partial^m \mathcal{S}_{\dr}^2 + \partial^m \mathcal{R}_{\dr} , \partial^m \dr_R^\eps \r \,.
	  \end{aligned}
	\end{equation}
	It is easily calculated that
	\begin{equation}\label{d-d0-1}
	  \begin{aligned}
	    & \l \partial_t \partial^m ( \u_R^\eps \cdot \nabla \dr_0 ) , \partial^m \dr_R^\eps \r \\
	    = & \tfrac{\d}{\d t} \l \partial^m ( \u_R^\eps \cdot \nabla \dr_0 ) , \partial^m \dr_R^\eps \r -\l \partial^m ( \u_R^\eps \cdot \nabla \dr_0 ) , \partial^m \partial_t \dr_R^\eps \r \\
	    = & \tfrac{\d}{\d t} \l \partial^m ( \u_R^\eps \cdot \nabla \dr_0 ) , \partial^m \dr_R^\eps \r - \l \partial^m ( \u_R^\eps \cdot \nabla \dr_0 ), \partial^m \D_{\u_0 + \sqrt{\eps} \u_R^\eps} \dr_R^\eps \r \\
	    & + \l \partial^m ( \u_R^\eps \cdot \nabla \dr_0 ) , \partial^m ( (\u_0 + \sqrt{\eps} \u_R^\eps ) \cdot \nabla \dr_R^\eps ) \r \,.
	  \end{aligned}
	\end{equation}
	Furthermore, from the integration by parts over $t \in \R^+$, we deduce that
	  \begin{align}\label{d-d0-2}
	    \no & \l \partial^m \D^2_{\u_0 + \sqrt{\eps} \u_R^\eps } \dr_R^\eps , \partial^m \dr_R^\eps \r \\
	    \no = & \l \partial_t \partial^m \D_{\u_0 + \sqrt{\eps} \u_R^\eps} \dr_R^\eps , \partial^m \dr_R^\eps \r + \l \partial^m ( \sqrt{\eps} \u_R^\eps \cdot  \nabla \D_{\u_0 + \sqrt{\eps} \u_R^\eps} \dr_R^\eps ) , \partial^m \dr_R^\eps \r \\
	    \no = & \tfrac{\d}{\d t} \l \partial^m \D_{\u_0 + \sqrt{\eps} \u_R^\eps} \dr_R^\eps , \partial^m \dr_R^\eps \r - \| \partial^m \D_{\u_0 + \sqrt{\eps} \u_R^\eps} \dr_R^\eps \|^2_{L^2} \\
	    \no & + \l \partial^m ( \sqrt{\eps} \u_R^\eps \cdot \nabla \D_{\u_0 + \sqrt{\eps} \u_R^\eps} \dr_R^\eps ) , \partial^m \dr_R^\eps \r - \l \partial^m ( \sqrt{\eps} \u_R^\eps \otimes \D_{\u_0 + \sqrt{\eps} \u_R^\eps} \dr_R^\eps ) , \nabla \partial^m \dr_R^\eps \r \\
	    \no = & \tfrac{1}{2} \tfrac{\d}{\d t} \big( \| \partial^m \D_{\u_0 + \sqrt{\eps} \u_R^\eps} \dr_R^\eps + \partial^m \dr_R^\eps \|^2_{L^2} - \| \partial^m \D_{\u_0 + \sqrt{\eps} \u_R^\eps} \dr_R^\eps \|^2_{L^2} - \| \partial^m \dr_R^\eps \|^2_{L^2} \big)  \\
	    \no & - \| \partial^m \D_{\u_0 + \sqrt{\eps} \u_R^\eps} \dr_R^\eps \|^2_{L^2} + \sum_{0 \neq m' \leq m} C_m^{m'} \l \partial^m \D_{\u_0 + \sqrt{\eps} \u_R^\eps} \dr_R^\eps , \sqrt{\eps} \partial^{m'} \u_R^\eps \cdot \nabla \partial^{m - m'} \dr_R^\eps \r \\
	    & - \sum_{0 \neq m' \leq m} C_m^{m'} \l \sqrt{\eps} \partial^{m'} \u_R^\eps \otimes \partial^{m - m'} \D_{\u_0 + \sqrt{\eps} \u_R^\eps} \dr_R^\eps , \nabla \partial^m \dr_R^\eps \r
	  \end{align}
	Substituting the relations \eqref{d-d0-1} and \eqref{d-d0-2} into \eqref{d-nabla-d-decay-1} and summing up for $|m| \leq N$ reduce to
	\begin{equation}\label{d-nabla-d}
	  \begin{aligned}
	    & \tfrac{1}{2} \tfrac{\d}{\d t} \Big( \| \D_{\u_0 + \sqrt{\eps} \u_R^\eps} \dr_R^\eps + \dr_R^\eps \|^2_{H^N} - \| \D_{\u_0 + \sqrt{\eps} \u_R^\eps} \dr_R^\eps \|^2_{H^N} + ( \tfrac{- \lambda_1}{\eps} - 1 ) \| \dr_R^\eps \|^2_{H^N} \\
	    & \qquad + \sum_{|m| \leq N} \l \partial^m ( \u_R^\eps \cdot \nabla \dr_0 ) , \partial^m \dr_R^\eps \r \Big) + \tfrac{1}{\eps} \| \nabla \dr_R^\eps \|^2_{H^N} - \| \D_{\u_0 + \sqrt{\eps} \u_R^\eps} \dr_R^\eps \|^2_{H^N} \\
	    = & - \sum_{|m| \leq N} \l \partial^m ( \u_R^\eps \cdot \nabla \dr_0 ) , \partial^m ( (\u_0 + \sqrt{\eps} \u_R^\eps ) \cdot \nabla \dr_R^\eps ) \r \\
	    & + \sum_{|m| \leq N} \l \partial^m ( \u_R^\eps \cdot \nabla \dr_0 ) , \partial^m \D_{\u_0 + \sqrt{\eps} \u_R^\eps} \dr_R^\eps \r \\
	    & + \tfrac{\lambda_1}{\eps} \sum_{ |m| \leq N} \l \partial^m ( ( \u_0 + \sqrt{\eps} \u_R^\eps ) \cdot \nabla \dr_R^\eps ) , \partial^m \dr_R^\eps \r +  \tfrac{1}{\eps} \sum_{|m| \leq N} \l \partial^m \mathcal{C}_{\dr} , \partial^m \dr_R^\eps \r \\
	    & + \sum_{|m| \leq N} \l \tfrac{1}{\eps} \partial^m \mathcal{S}_{\dr}^1 + \tfrac{1}{\sqrt{\eps}} \partial^m \mathcal{S}_{\dr}^2 + \partial^m \mathcal{R}_{\dr} , \partial^m \dr_R^\eps \r \\
	    & + \sum_{|m| \leq N} \sum_{0 \neq m' \leq m} C_m^{m'} \l \sqrt{\eps} \partial^{m'} \u_R^\eps \otimes \partial^{m - m'} \D_{\u_0 + \sqrt{\eps} \u_R^\eps} \dr_R^\eps , \nabla \partial^m \dr_R^\eps \r \\
	    &  - \sum_{|m| \leq N} \sum_{0 \neq m' \leq m} C_m^{m'} \l \partial^m \D_{\u_0 + \sqrt{\eps} \u_R^\eps} \dr_R^\eps , \sqrt{\eps} \partial^{m'} \u_R^\eps \cdot \nabla \partial^{m - m'} \dr_R^\eps \r
	  \end{aligned}
	\end{equation}
	for all $\eps > 0$. Recalling the definition of $\mathcal{C}_{\dr}$ in \eqref{C-d}, we calculate that
	\begin{equation}
	  \begin{aligned}
	    & \tfrac{1}{\eps} \sum_{1 \leq |m| \leq N} \l \partial^m \mathcal{C}_{\dr} , \partial^m \dr_R^\eps \r \\
	    = & \tfrac{1}{\eps} \sum_{ 1 \leq |m| \leq N} \l \lambda_1 ( \partial^m \B_R^\eps ) \dr_0 + \lambda_2 ( \partial^m \A_R^\eps ) \dr_0 , \partial^m \dr_R^\eps \r \\
	    & + \tfrac{1}{\eps} \sum_{1 \leq |m| \leq N} \sum_{0 \neq m' \leq m} C_m^{m'} \l \lambda_1 ( \partial^{m - m'} \B_R^\eps ) \partial^{m'} \dr_0 + \lambda_2 ( \partial^{m - m'} \A_R^\eps ) \partial^{m'} \dr_0 , \partial^m \dr_R^\eps \r \,,
	  \end{aligned}
	\end{equation}
	where the first term in the right-hand side of the previous equality can be bounded by
	\begin{equation}
	  \begin{aligned}
	    & \tfrac{1}{\eps} \sum_{1 \leq |m| \leq N} \big( |\lambda_1| \| \partial^m \B_R^\eps \|_{L^2} + |\lambda_2| \| \partial^m \A_R^\eps \|_{L^2} \big) \| \partial^m \dr_R^\eps \|_{L^2} \\
	    \leq & \tfrac{1}{\eps} \sqrt{c'_0 ( \lambda_1 , \lambda_2 , N )} \| \nabla \u_R^\eps \|_{H^N} \| \nabla \dr_R^\eps \|_{H^N} \\
	    \leq & \tfrac{1}{\eps} c'_0 ( \lambda_1 , \lambda_2 , N ) \| \nabla \u_R^\eps \|^2_{H^N} + \tfrac{1}{4 \eps } \| \nabla \dr_R^\eps \|^2_{H^N}
	  \end{aligned}
	\end{equation}
	for some constant $c'_0 ( \lambda_1 , \lambda_2 , N ) > 0$. Furthermore, thanks to $ \int_{\T^3} \d x = |\T^3|^3 < \infty $, we derive from the Gagliardo-Nirenberg interpolation inequality $ \| f \|_{L^6} \leq C \| \nabla f \|_{L^2} $ and the Young's inequality that
	\begin{equation}\label{Cd-d-linear}
	  \begin{aligned}
	    \tfrac{1}{\eps} \l \lambda_1 \B_R^\eps \dr_0 + \lambda_2 \A_R^\eps \dr_0 , \dr_R^\eps \r \leq & \tfrac{1}{\eps} ( |\lambda_1| + |\lambda_2| ) \big( \int_{\T^3} |\dr_0|^3 \d x  \big)^\frac{1}{3} \| \nabla \u_R^\eps \|_{L^2} \| \dr_R^\eps \|_{L^6} \\
	    \leq & \tfrac{1}{\eps} ( |\lambda_1| + |\lambda_2| ) | \T^3 |^\frac{1}{3} \| \nabla \u_R^\eps \|_{H^N} \| \nabla \dr_R^\eps \|_{H^N} \\
	    \leq & \tfrac{1}{\eps} c_0^* (\lambda_1, \lambda_2) \| \nabla \u_R^\eps \|^2_{H^N} + \tfrac{1}{4 \eps} \| \nabla \dr_R^\eps \|^2_{H^N}
	  \end{aligned}
	\end{equation}
	for some constant $c_0^* (\lambda_1, \lambda_2) > 0$. Consequently, we know that
	\begin{equation}
	  \begin{aligned}
	    & \tfrac{1}{\eps} \sum_{|m| \leq N} \l \partial^m \mathcal{C}_{\dr} , \partial^m \dr_R^\eps \r \leq \tfrac{1}{\eps} c_0 ( \lambda_1, \lambda_2, N ) \| \nabla \u_R^\eps \|^2_{H^N} + \tfrac{1}{2 \eps} \| \nabla \dr_R^\eps \|^2_{H^N} \\
	    & + \tfrac{1}{\eps} \sum_{1 \leq |m| \leq N} \sum_{0 \neq m' \leq m} C_m^{m'} \l \lambda_1 ( \partial^{m - m'} \B_R^\eps ) \partial^{m'} \dr_0 + \lambda_2 ( \partial^{m - m'} \A_R^\eps ) \partial^{m'} \dr_0 , \partial^m \dr_R^\eps \r
	  \end{aligned}
	\end{equation}
	holds for $N \geq 2$, where $c_0 ( \lambda_1, \lambda_2, N ) = c_0' (\lambda_1, \lambda_2, N) + c_0^* (\lambda_1, \lambda_2) > 0$.
	
	We now take a small constant
	\begin{equation}
	  \begin{aligned}
	    \delta = \min \big\{ \tfrac{1}{2} , \tfrac{ \mu_4 }{8 c_0 (\lambda_1, \lambda_2, N)} \big\} \in  ( 0, \tfrac{1}{2} ]
	  \end{aligned}
	\end{equation}
	such that $\delta c_0 (\lambda_1, \lambda_2, N) \leq \tfrac{\mu_4}{8}$. Combining the relation \eqref{Cd-d-linear}, we multiply \eqref{d-nabla-d} by $\delta$ and then add it to the \eqref{u-d-2}, which gives us
	\begin{equation}\label{u-d-decay-2}
	  \begin{aligned}
	    \tfrac{1}{2} \tfrac{\d}{\d t} \mathscr{E}_{N, \eps} (t) + \mathscr{D}_{N, \eps} (t) \leq \mathcal{I}_N^{(1)} + \mathcal{I}_N^{(2)} + \mathcal{I}_N^{(3)} + \mathcal{I}_N^{(4)}
	  \end{aligned}
	\end{equation}
	for $N \geq 2$, where the symbols $\mathcal{I}_N^{(i)}$ $(1 \leq i \leq 4)$ are defined as
	\begin{equation}
	   \begin{aligned}
	     & \mathcal{I}_N^{(1)} = \tfrac{1}{\eps} \sum_{|m| \leq N} \l \partial^m \mathcal{K}_{\u} , \partial^m \u_R^\eps \r + \tfrac{1}{\eps } \sum_{|m| \leq N} \l \nabla \partial^m \dr_R^\eps , \nabla \partial^m ( \u_R^\eps \cdot \nabla \dr_0 ) \r \\
	     & +  \tfrac{\delta \lambda_1}{ \eps } \sum_{ |m| \leq N} \l \partial^m ( ( \u_0 + \sqrt{\eps} \u_R^\eps ) \cdot \nabla \dr_R^\eps ) , \partial^m \dr_R^\eps \r \\
	     & - \sum_{|m| \leq N} \l \partial^m ( \sqrt{\eps} \u_R^\eps \cdot \nabla \D_{\u_0 + \sqrt{\eps} \u_R^\eps} \dr_R^\eps ) , \partial^m \D_{\u_0 + \sqrt{\eps} \u_R^\eps} \dr_R^\eps \r \\
	     & - \tfrac{1}{\sqrt{\eps}} \sum_{|m| \leq N} \l \nabla \partial^m \dr_R^\eps , \nabla \partial^m ( \u_R^\eps \cdot \nabla \dr_R^\eps ) \r - \sum_{|m| \leq N} \l \sqrt{\eps} \partial^m \D_{\u_0 + \sqrt{\eps} \u_R^\eps} \dr_R^\eps ,  \u_R^\eps \cdot \nabla \partial^m ( \u_R^\eps \cdot \nabla \dr_0 ) \r \\
	     & + \tfrac{- \lambda_1}{\eps} \sum_{|m| \leq N} \l \partial^m (\u_R^\eps \cdot \nabla \dr_0), \partial^m \D_{\u_0 + \sqrt{\eps} \u_R^\eps} \dr_R^\eps \r \\
	     & - \tfrac{ \mu_1}{\eps} \sum_{|m| \leq N} \sum_{\substack{0 \neq m' \leq m \\ m'' \leq m'} } C_m^{m'} C_{m'}^{m''} \l ( \partial^{m - m'} \A_R^\eps : \partial^{m' - m''} ( \dr_0 \otimes \dr_0 ) ) \partial^{m''} ( \dr_0 \otimes \dr_0 ) , \nabla \partial^m \u_R^\eps \r \\
	     & - \delta \sum_{ |m| \leq N} \l \partial^m (  \u_R^\eps \cdot \nabla \dr_0 ) , \partial^m ( ( \u_0 + \sqrt{\eps} \u_R^\eps ) \cdot \nabla \dr_R^\eps ) \r + \delta \sum_{ |m| \leq N} \l \partial^m ( \u_R^\eps \cdot \nabla \dr_0 ) , \partial^m \D_{\u_0 + \sqrt{\eps} \u_R^\eps} \dr_R^\eps \r \\
	     & + \delta \sum_{1 \leq |m| \leq N} \sum_{0 \neq m' \leq m} C_m^{m'} \l \sqrt{\eps} \partial^{m'} \u_R^\eps \otimes \partial^{m - m'} \D_{\u_0 + \sqrt{\eps} \u_R^\eps} , \nabla \partial^m \dr_R^\eps \r \\
	     &  - \delta \sum_{1 \leq |m| \leq N} \sum_{0 \neq m' \leq m} C_m^{m'} \l \partial^m \D_{\u_0 + \sqrt{\eps} \u_R^\eps} \dr_R^\eps , \sqrt{\eps} \partial^{m'} \u_R^\eps \cdot \nabla \partial^{m - m'} \dr_R^\eps \r
	   \end{aligned}
	\end{equation}
	for the functional $\mathcal{K}_{\u}$ defined in \eqref{Lu}, and
	\begin{equation}
	  \begin{aligned}
	    \mathcal{I}_N^{(2)} = - \tfrac{1}{\eps } \sum_{|m| \leq N} \l \partial^m \mathcal{T}_{\u} , \nabla \partial^m \u_R^\eps \r - \tfrac{1}{\sqrt{\eps}} \sum_{|m| \leq N} \l \partial^m \mathcal{R}_{\u} , \nabla \partial^m \u_R^\eps \r
	  \end{aligned}
	\end{equation}
	for the expressions $\mathcal{T}_{\u}$, $\mathcal{R}_{\u}$ given in \eqref{Tu}, \eqref{Ru}, respectively, and
	\begin{equation}
	  \begin{aligned}
	    \mathcal{I}_N^{(3)} = \tfrac{1}{\eps} \sum_{|m| \leq N} \mathcal{G}_m
	  \end{aligned}
	\end{equation}
	with the quantity $\mathcal{G}_m$ mentioned in Lemma \ref{Lmm-Cancellation}, and
	\begin{equation}\label{IN4}
	  \begin{aligned}
	    \mathcal{I}_N^{(4)} = & - \tfrac{1}{\eps} \sum_{|m| \leq N} \l \partial^m ( \u_R^\eps \cdot \nabla \dr_0 ) , \partial^m \mathcal{C}_{\dr} \r \\
	    & + \sum_{|m| \leq N} \l \tfrac{1}{\eps} \partial^m \mathcal{S}_{\dr}^1 + \tfrac{1}{\sqrt{\eps}} \partial^m \mathcal{S}_{\dr}^2 + \partial^m \mathcal{R}_{\dr} , \partial^m \D_{\u_0 + \sqrt{\eps} \u_R^\eps} \dr_R^\eps - \partial^m ( \u_R^\eps \cdot \nabla\dr_0 ) \r \\
	    & + \delta \sum_{ |m| \leq N} \l \tfrac{1}{\eps} \partial^m \mathcal{S}_{\dr}^1 + \tfrac{1}{\sqrt{\eps}} \partial^m \mathcal{S}_{\dr}^2 + \partial^m \mathcal{R}_{\dr} , \partial^m \dr_R^\eps \r \\
	    & + \tfrac{\delta}{\eps} \sum_{1 \leq |m| \leq N} \sum_{0 \neq m' \leq m} C_m^{m'} \l \lambda_1 ( \partial^{m-m'} \B_R^\eps ) \partial^{m'} \dr_0 + \lambda_2 ( \partial^{m-m'} \A_R^\eps ) \partial^{m'} \dr_0 , \partial^m \dr_R^\eps \r \,.
	  \end{aligned}
	\end{equation}
	Here the vectors $\mathcal{C}_{\dr}$, $\mathcal{S}_{\dr}^1$, $\mathcal{C}_{\dr}^2$ and $\mathcal{R}_{\dr}$ are determined in \eqref{C-d}, \eqref{Sd-1}, \eqref{Sd-2} and \eqref{Rd}, respectively.
	
	It remains to control the terms $\mathcal{I}_N^{(i)}$ $(1 \leq i \leq 4)$ for $N \geq 2$. We emphasize that, in the following estimates, we will frequently use the Sobolev interpolation inequality $\| f \|_{L^6(\T^3)} \leq C \| \nabla f \|_{L^2(\T^3)}$, Sobolev embeddings $H^1 (\T^3) \hookrightarrow L^4 (\T^3)$ (or $L^3 (\T^3)$), $ H^2 (\T^3) \hookrightarrow L^\infty (\T^3) $ and the inequalities \eqref{Energy-E-bnd}, \eqref{Energy-D-bnd} with the constraint $\eps \in (0, \eps_0)$ in Lemma \ref{Lmm-Energies}. Furthermore, for $|m| \leq N$ $(N \geq 2)$, the calculus inequality (see \cite{Majda-2002-BOOK}, for instance)
	\begin{equation}
	  \begin{aligned}
	    \| \partial^m ( f g ) \|_{L^2} \leq C \| f \|_{H^N} \| g \|_{H^N}
	  \end{aligned}
	\end{equation}
	will also be frequently utilized. The geometric constraint $|\dr_0| = 1$ is also considered in the following energy estimates.\\
	
	{\em Step 1. Control the term $\mathcal{I}_N^{(1)}$.}
	
	Via the divergence-free property of $\u_0$, we have
	\begin{equation}\label{IN1-1}
	  \begin{aligned}
	    - \tfrac{1}{\eps} & \l \partial^m ( \u_0 \cdot \nabla \u_R^\eps ) , \partial^m \u_R^\eps \r = - \tfrac{1}{\eps} \sum_{0 \neq m' \leq m} C_m^{m'} \l \partial^{m'} \u_0 \cdot \nabla \partial^{m-m'} \u_R^\eps , \partial^m \u_R^\eps \r \\
	    \leq & \tfrac{C}{\eps} \sum_{0 \neq m' \leq m} \| \partial^{m'} \u_0 \|_{L^4} \| \nabla \partial^{m-m'} \u_R^\eps \|_{L^4} \| \partial^m \u_R^\eps \|_{L^2} \\
	    \leq & \tfrac{C}{\eps} \sum_{0 \neq m' \leq m} \| \partial^{m'} \u_0 \|_{H^1} \| \nabla \partial^{m-m'} \u_R^\eps \|_{H^1} \| \partial^m \u_R^\eps \|_{L^2} \\
	    \leq & \tfrac{C}{\eps} \| \nabla \u_0 \|_{H^N} \| \nabla \u_R^\eps \|_{H^N} \| u_R^\eps \|_{H^N} \leq C \mathscr{E}^\frac{1}{2}_{N, \eps} (t) \mathscr{D}^\frac{1}{2}_{N, \eps} (t) \| \nabla \u_0 \|_{H^N}
	  \end{aligned}
	\end{equation}
	holds for all $|m| \leq N$. If $1 \leq |m| \leq N$, we estimate
	\begin{equation}
	  \begin{aligned}
	    - \tfrac{1}{\eps} & \l \partial^m ( \u_R^\eps \cdot \nabla \u_0 ) , \partial^m \u_R^\eps \r = - \tfrac{1}{\eps} \sum_{m' \leq m} C_m^{m'} \l \partial^{m'} \u_R^\eps \cdot \nabla \partial^{m-m'} \u_0 , \partial^m \u_R^\eps \r \\
	    \leq & \tfrac{C}{\eps} \| \u_R^\eps \|_{L^\infty} \| \nabla \partial^m \u_0 \|_{L^2} \| \partial^m \u_R^\eps \|_{L^2} \\
	    & + \tfrac{C}{\eps} \sum_{0 \neq m' \leq m} \| \partial^{m'} \u_R^\eps \| \nabla \partial^{m-m'} \u_0 \|_{L^4} \| \partial^m \u_R^\eps \|_{L^4} \\
	    \leq & \tfrac{C}{\eps} \| \u_R^\eps \|_{H^N} \| \nabla \u_R^\eps \|_{H^N} \| \nabla \u_0 \|_{H^N} \leq C \mathscr{E}^\frac{1}{2}_{N, \eps} (t) \mathscr{D}^\frac{1}{2}_{N, \eps} (t) \| \nabla \u_0 \|_{H^N} \,.
	  \end{aligned}
	\end{equation}
	Moreover, for $m = 0$, we have
	\begin{equation}
	  \begin{aligned}
	    - \tfrac{1}{\eps} & \l \u_R^\eps \cdot \nabla \u_0 , \u_R^\eps \r \leq \tfrac{1}{\eps} \| \u_R^\eps \|^2_{L^4} \| \nabla \u_0 \|_{L^2} \leq \tfrac{C}{\eps} \| \nabla \u_R^\eps \|^\frac{3}{2}_{L^2} \| \u_R^\eps \|^\frac{1}{2}_{L^2} \\
	    \leq & \tfrac{C}{\eps} \| \u_R^\eps \|_{H^1} \| \nabla \u_R^\eps \|_{L^2} \| \nabla \u_0 \|_{L^2} \leq C \mathscr{E}^\frac{1}{2}_{N, \eps} (t) \mathscr{D}^\frac{1}{2}_{N, \eps} (t) \| \nabla \u_0 \|_{H^N} \,,
	  \end{aligned}
	\end{equation}
	where we utilize the Sobolev interpolation inequality $ \| f \|_{L^4} \leq C \| f \|^\frac{1}{4}_{L^2} \| \nabla f \|^\frac{3}{4}_{L^2} $. Then, we obtain
	\begin{equation}\label{IN1-2}
	  \begin{aligned}
	    - \tfrac{1}{\eps} \l \partial^m ( \u_R^\eps \cdot \nabla \u_0 ) , \partial^m \u_R^\eps \r \leq - \tfrac{1}{\eps} \l \partial^m ( \u_R^\eps \cdot \nabla \u_0 ) , \partial^m \u_R^\eps \r
	  \end{aligned}
	\end{equation}
	for all $|m| \leq N$ $(N \geq 2)$. It is easy to derive from the divergence-free property of $\u_R^\eps$ that
	\begin{equation}\label{IN1-3}
	  \begin{aligned}
	    - \tfrac{1}{\eps} & \l \partial^m ( \sqrt{\eps} \u_R^\eps \cdot \nabla \u_R^\eps ) , \partial^m \u_R^\eps \r = - \tfrac{1}{\sqrt{\eps}} \sum_{0 \neq m' \leq m} C_m^{m'} \l \partial^{m'} \u_R^\eps \cdot \nabla \partial^{m-m'} \u_R^\eps , \partial^m \u_R^\eps \r \\
	    \leq & \tfrac{C}{\sqrt{\eps}} \sum_{0 \neq m' \leq m} \| \partial^{m'} \u_R^\eps \|_{L^4} \| \nabla \partial^{m-m'} \u_R^\eps \|_{L^4} \| \partial^m \u_R^\eps \|_{L^2} \leq \tfrac{C}{\sqrt{\eps}} \| \u_R^\eps \|_{H^N} \| \nabla \u_R^\eps \|^2_{H^N} \\
	    \leq & C \eps \mathscr{E}^\frac{1}{2}_{N, \eps} (t) \mathscr{D}_{N, \eps} (t) \,.
	  \end{aligned}
	\end{equation}
	Moreover, we can deduce
	\begin{equation}\label{IN1-4}
	  \begin{aligned}
	    & - \tfrac{1}{\eps} \l \partial^m \div ( \nabla \dr_0 \odot \nabla \dr_R^\eps + \nabla \dr_R^\eps \odot \nabla \dr_0 + \sqrt{\eps} \nabla \dr_R^\eps \odot \nabla \dr_R^\eps ) , \partial^m \u_R^\eps \r \\
	    = & \tfrac{1}{\eps} \l \partial^m ( \nabla \dr_0 \odot \nabla \dr_R^\eps + \nabla \dr_R^\eps \odot \nabla \dr_0 + \sqrt{\eps} \nabla \dr_R^\eps \odot \nabla \dr_R^\eps ) , \nabla \partial^m \u_R^\eps \r \\
	    \leq & \tfrac{C}{\eps} \big( \| \partial^m ( \nabla \dr_0 \odot \nabla \dr_R^\eps ) \|_{L^2} + \sqrt{\eps} \| \partial^m ( \nabla \dr_R^\eps \odot \nabla \dr_R^\eps ) \|_{L^2} \big) \| \nabla \partial^m \u_R^\eps \|_{L^2} \\
	    \leq & \tfrac{C}{\eps} \| \nabla \dr_R^\eps \|_{H^N} ( \| \nabla \dr_0 \|_{H^N} + \sqrt{\eps} \| \nabla \dr_R^\eps \|_{H^N} ) \| \nabla \u_R^\eps \|_{H^N} \\
	    \leq & C \big( \| \nabla \dr_0 \|_{H^N} + \eps \mathscr{E}^\frac{1}{\eps}_{N, \eps} (t) \big) \mathscr{D}_{N, \eps} (t)
	  \end{aligned}
	\end{equation}
	for all $|m| \leq N$. Recalling the definition of $\mathcal{K}_{\u}$ in \eqref{Lu}, we derive from the bounds \eqref{IN1-1}, \eqref{IN1-2}, \eqref{IN1-3} and \eqref{IN1-4} that for $|m| \leq N$
	\begin{equation}\label{IN1-5}
	  \begin{aligned}
	    \tfrac{1}{\eps} \l \partial^m \mathcal{K}_{\u} , \partial^m \u_R^\eps \r \leq C \mathscr{E}^\frac{1}{2}_{N, \eps} (t) \mathscr{D}^\frac{1}{2}_{N, \eps} (t) \| \nabla \u_0 \|_{H^N} + C \big( \| \nabla \dr_0 \|_{H^N} + \eps \mathscr{E}^\frac{1}{\eps}_{N, \eps} (t) \big) \mathscr{D}_{N, \eps} (t) \,.
	  \end{aligned}
	\end{equation}
	
	For $|m| \leq N$, we calculate
	\begin{equation}\label{IN1-6}
	  \begin{aligned}
	    & \tfrac{1}{\eps} \l \nabla \partial^m \dr_R^\eps , \nabla \partial^m ( \u_R^\eps \cdot \nabla \dr_0 ) \r = \tfrac{1}{\eps} \l \nabla \partial^m \dr_R^\eps , \partial^m ( \u_R^\eps \cdot \nabla^2 \dr_0 + \nabla \u_R^\eps \cdot  \nabla \dr_0 ) \r \\
	    = & \tfrac{1}{\eps} \sum_{m' \leq m} C_m^{m'} \l \nabla \partial^m \dr_R^\eps , \partial^{m'} \u_R^\eps \cdot \nabla^2 \partial^{m-m'} \dr_0 + \nabla \partial^{m'} \u_R^\eps \cdot \nabla \partial^{m-m'} \dr_0 \r \\
	    \leq & \tfrac{C}{\eps} \sum_{m' \leq m} \| \nabla \partial^m \dr_R^\eps \|_{L^2} \Big( \| \partial^{m'} \u_R^\eps \|_{L^6} \| \nabla^2 \partial^{m-m'} \dr_0 \|_{L^3} + \| \nabla \partial^{m'} \u_R^\eps \|_{L^2} \| \nabla \partial^{m-m'} \dr_0 \|_{L^\infty} \Big) \\
	    \leq & \tfrac{C}{\eps} \| \nabla \dr_R^\eps \|_{H^N} \| \nabla \u_R^\eps \|_{H^N} \| \nabla \dr_0 \|_{H^{N+2}} \leq C \| \nabla \dr_0 \|_{H^{N+2}} \mathscr{D}_{N, \eps} (t) \,.
	  \end{aligned}
	\end{equation}
	Via the divergence-free property of $\u_R^\eps$, we yield that
	\begin{equation}\label{IN1-7}
	  \begin{aligned}
	    & \tfrac{\delta \lambda_1}{\sqrt{\eps}} \l \partial^m ( \u_R^\eps \cdot \nabla \dr_R^\eps ), \partial^m \dr_R^\eps \r = \tfrac{\delta \lambda_1}{\sqrt{\eps}} \sum_{0 \neq m' \leq m} C_m^{m'} \l \partial^{m'} \u_R^\eps \cdot \nabla \partial^{m-m'} \dr_R^\eps , \partial^m \dr_R^\eps \r \\
	    \leq & \tfrac{C}{\sqrt{\eps}} \sum_{0 \neq m' \leq m} \| \partial^{m'} \u_R^\eps \|_{L^4} \| \nabla \partial^{m-m'} \dr_R^\eps \|_{L^4} \| \partial^m \dr_R^\eps \|_{L^2} \\
	    \leq & \tfrac{C}{\sqrt{\eps}} \| \nabla \dr_R^\eps \|_{H^N} \| \dr_R^\eps \|_{H^N} \| \nabla \u_R^\eps \|_{H^N} \leq C \eps \mathscr{E}^\frac{1}{2}_{N, \eps} (t) \mathscr{D}_{N, \eps} (t)
	  \end{aligned}
	\end{equation}
	holds for $|m| \leq N$. Similarly, we deduce from the fact $\div \u_R^\eps = 0$ that
	\begin{equation}\label{IN1-8}
	  \begin{aligned}
	     & - \l \partial^m ( \sqrt{\eps} \u_R^\eps \cdot \nabla \D_{\u_0 + \sqrt{\eps} \u_R^\eps} \dr_R^\eps ) , \partial^m \D_{\u_0 + \sqrt{\eps} \u_R^\eps} \dr_R^\eps \r \\
	     = & - \sqrt{\eps} \sum_{0 \neq m' \leq m} C_m^{m'} \l \partial^{m'} \u_R^\eps \cdot \nabla \partial^{m-m'} \D_{\u_0 + \sqrt{\eps} \u_R^\eps} \dr_R^\eps , \partial^m \D_{\u_0 + \sqrt{\eps} \u_R^\eps} \dr_R^\eps \r \\
	     \leq & C \sqrt{\eps} \sum_{|m'| = 1} \| \partial^{m'} \u_R^\eps \|_{L^\infty} \| \nabla \partial^{m-m'} \D_{\u_0 + \sqrt{\eps} \u_R^\eps} \dr_R^\eps \|_{L^2} \| \partial^m \D_{\u_0 + \sqrt{\eps} \u_R^\eps} \dr_R^\eps \|_{L^2} \\
	     & + C \sqrt{\eps} \sum_{2 \leq |m'| \leq |m|} \| \partial^{m'} \u_R^\eps \|_{L^4} \| \nabla \partial^{m-m'} \D_{\u_0 + \sqrt{\eps} \u_R^\eps} \dr_R^\eps \|_{L^4} \| \partial^m \D_{\u_0 + \sqrt{\eps} \u_R^\eps} \dr_R^\eps \|_{L^2} \\
	     \leq & C \sqrt{\eps} \| \nabla \u_R^\eps \|_{H^N} \| \D_{\u_0 + \sqrt{\eps} \u_R^\eps} \dr_R^\eps \|^2_{H^N} \leq C \eps \mathscr{E}^\frac{1}{2}_{N, \eps} (t) \mathscr{D}_{N, \eps} (t)
	  \end{aligned}
	\end{equation}
	for all multi-indexes $m \in \mathbb{N}^3$ with $|m| \leq N$. For the term $ - \tfrac{1}{\sqrt{\eps}} \l \nabla \partial^m \dr_R^\eps , \nabla \partial^m ( \u_R^\eps \cdot \nabla \dr_R^\eps ) \r $, one estimates that for $|m| \leq N$
	\begin{equation}\label{IN1-9}
	  \begin{aligned}
	    & - \tfrac{1}{\sqrt{\eps}} \l \nabla \partial^m \dr_R^\eps , \nabla \partial^m ( \u_R^\eps \cdot \nabla \dr_R^\eps ) \r \\
	    = & - \tfrac{1}{\sqrt{\eps}} \sum_{0 \neq m' \leq m} C_m^{m'} \l \nabla \partial^m \dr_R^\eps , \partial^{m'} \u_R^\eps \cdot \nabla \partial^{m-m'} \nabla \dr_R^\eps \r - \tfrac{1}{\sqrt{\eps}} \l \nabla \partial^m \dr_R^\eps , \partial^m ( \nabla \u_R^\eps \cdot \nabla \dr_R^\eps ) \r \\
	    \leq & \tfrac{C}{\sqrt{\eps}} \sum_{|m'|=1} \| \partial^{m'} \u_R^\eps \|_{L^\infty} \| \nabla \partial^{m-m'} \nabla \dr_R^\eps \|_{L^2} \| \nabla \partial^m \dr_R^\eps \|_{L^2} \\
	    & + \tfrac{C}{\sqrt{\eps}} \sum_{2 \leq |m'| \leq |m|} \| \partial^{m'} \u_R^\eps \|_{L^4} \| \nabla \partial^{m-m'} \nabla \dr_R^\eps \|_{L^4} \| \nabla \partial^m \dr_R^\eps \|_{L^2} \\
	    & + \tfrac{C}{\sqrt{\eps}} \| \nabla \partial^m \dr_R^\eps \|{L^2} \| \partial^m ( \nabla \u_R^\eps \cdot \nabla \dr_R^\eps ) \|_{L^2} \\
	    \leq & \tfrac{C}{\sqrt{\eps}} \| \nabla \u_R^\eps \|_{H^N} \| \nabla \dr_R^\eps \|^2_{H^N} \leq C \eps \mathscr{E}^\frac{1}{2}_{N, \eps} (t) \mathscr{D}_{N, \eps} (t) \,,
	  \end{aligned}
	\end{equation}
	where the fact $\div \u_R^\eps = 0$ is also utilized.
	
	We derive the bound of the term $ - \l \sqrt{\eps} \partial^m \D_{\u_0 + \sqrt{\eps} \u_R^\eps} \dr_R^\eps, \u_R^\eps \cdot \nabla \partial^m ( \u_R^\eps \cdot \nabla \dr_0 ) \r $ as follows:
	\begin{equation}\label{IN1-10}
	  \begin{aligned}
	    & - \l \sqrt{\eps} \partial^m \D_{\u_0 + \sqrt{\eps} \u_R^\eps} \dr_R^\eps, \u_R^\eps \cdot \nabla \partial^m ( \u_R^\eps \cdot \nabla \dr_0 ) \r \\
	    = & - \l \sqrt{\eps} \partial^m \D_{\u_0 + \sqrt{\eps} \u_R^\eps} \dr_R^\eps, \u_R^\eps \cdot \partial^m ( \nabla \u_R^\eps \cdot \nabla \dr_0 + \u_R^\eps \cdot \nabla \nabla \dr_0 ) \r \\
	    \leq & \sqrt{\eps} \| \partial^m \D_{\u_0 + \sqrt{\eps} \u_R^\eps} \dr_R^\eps \|_{L^2} \| \u_R^\eps \|_{L^\infty} \| \partial^m ( \nabla \u_R^\eps \cdot \nabla \dr_0 ) \|_{L^2} \\
	    & + \sqrt{\eps} \| \partial^m \D_{\u_0 + \sqrt{\eps} \u_R^\eps} \dr_R^\eps \|_{L^2} \| \u_R^\eps \|^2_{L^6} \| \nabla \partial^m \dr_0 \|_{L^6} \\
	    & + C \sqrt{\eps} \sum_{0 \neq m' \leq m} \| \partial^m \D_{\u_0 + \sqrt{\eps} \u_R^\eps} \dr_R^\eps \|_{L^2} \| \u_R^\eps \|_{L^\infty} \| \partial^{m'} \u_R^\eps \|_{L^4} \| \nabla \partial^{m-m'} \dr_0 \|_{L^4} \\
	    \leq & C \sqrt{\eps} \| \D_{\u_0 + \sqrt{\eps} \u_R^\eps} \dr_R^\eps \|_{H^N} \| \u_R^\eps \|_{H^N} \| \nabla \u_R^\eps \|_{H^N} \| \nabla \dr_0 \|_{H^{N+2}} \\
	    \leq & C \sqrt{\eps}^3 \| \nabla \dr_0 \|_{H^{N+2}} \mathscr{E}^\frac{1}{2}_{N, \eps} (t) \mathscr{D}_{N, \eps} (t) \,.
	  \end{aligned}
	\end{equation}
	For all $|m| \leq N$, it is derived that
	  \begin{align}\label{IN1-11}
	    \no & \tfrac{- \lambda_1}{\eps} \l \partial^m ( \u_R^\eps \cdot \nabla \partial^m \dr_0 ) , \partial^m \D_{\u_0 + \sqrt{\eps} \u_R^\eps} \dr_R^\eps \r \\
	    \no = & \tfrac{- \lambda_1}{\eps} \l \u_R^\eps \cdot \nabla \partial^m \dr_0 , \partial^m \D_{\u_0 + \sqrt{\eps} \u_R^\eps} \dr_R^\eps \r + \tfrac{- \lambda_1}{\eps} \sum_{0 \neq m' \leq m } \l \partial^{m'} \u_R^\eps \cdot \nabla \partial^{m-m'} \dr_0 , \partial^m \D_{\u_0 + \sqrt{\eps} \u_R^\eps} \dr_R^\eps \r \\
	    \no \leq & \tfrac{- \lambda_1}{\eps} \| \u_R^\eps \|_{L^6} \| \nabla \partial^m \dr_0 \|_{L^3} \| \partial^m \D_{\u_0 + \sqrt{\eps} \u_R^\eps} \dr_R^\eps \|_{L^2} \\
	    \no & + \tfrac{- C \lambda_1}{\eps} \sum_{0 \neq m' leq m} \| \partial^{m'} \u_R^\eps \|_{L^4} \| \nabla \partial^{m-m'} \dr_0 \|_{L^4} \| \partial^m \D_{\u_0 + \sqrt{\eps} \u_R^\eps} \dr_R^\eps \|_{L^2} \\
	    \no \leq & \tfrac{C}{\eps} \| \nabla \u_R^\eps \|_{H^N} \| \nabla \dr_0 \|_{H^{N+1}} \| \partial^m \D_{\u_0 + \sqrt{\eps} \u_R^\eps} \dr_R^\eps \|_{L^2} \\
	    \no \leq & \tfrac{C}{\eps} \| \nabla \u_R^\eps \|_{H^N} \| \nabla \dr_0 \|_{H^{N+1}} \Big( \| \partial^m \D_{\u_0 + \sqrt{\eps} \u_R^\eps} \dr_R^\eps + ( \partial^m \B_R^\eps ) \dr_0 + \tfrac{\lambda_2}{\lambda_1} (  \partial^m \A_R^\eps ) \dr_0 \|_{L^2} \\
	    \no & \qquad \qquad + ( 1 + \tfrac{|\lambda_2|}{-\lambda_1} ) \| \nabla \partial^m \u_R^\eps \|_{L^2} \Big) \\
	    \no \leq & \tfrac{C}{\eps} \| \nabla \u_R^\eps \|_{H^N} \| \nabla \dr_0 \|_{H^{N+1}} \Big( \| \partial^m \D_{\u_0 + \sqrt{\eps} \u_R^\eps} \dr_R^\eps + ( \partial^m \B_R^\eps ) \dr_0 + \tfrac{\lambda_2}{\lambda_1} (  \partial^m \A_R^\eps ) \dr_0 \|_{L^2} + \| \nabla \partial^m \u_R^\eps \|_{L^2} \Big) \\
	    \leq & C \| \nabla \dr_0 \|_{H^{N+1}} \mathscr{D}_{N, \eps} (t) \,.
	  \end{align}
	We also can yield that
	\begin{equation}\label{IN1-12}
	  \begin{aligned}
	    & - \tfrac{\mu_1}{\eps} \sum_{\substack{ 0 \neq m' \leq m \\ m'' \leq m' }} C_m^{m'} C_{m'}^{m''} \l ( \partial^{m-m'} \A_R^\eps : \partial^{m'-m''} ( \dr_0 \otimes \dr_0 ) ) \partial^{m''} ( \dr_0 \otimes \dr_0 ) , \nabla \partial^m \u_R^\eps \r \\
	    \leq & \tfrac{C}{\eps} \| \nabla \partial^m \u_R^\eps \|_{L^2} \sum_{\substack{ 0 \neq m' \leq m \\ m'' \leq m' }} \| \partial^{m-m'} \A_R^\eps \|_{L^2} \| \partial^{m'-m''} ( \dr_0 \otimes \dr_0 ) \|_{L^\infty} \| \partial^{m''} ( \dr_0 \otimes \dr_0 ) \|_{L^\infty} \\
	    \leq & \tfrac{C}{\eps} \| \nabla \u_R^\eps \|^2_{H^N} \big( \| \nabla \dr_0 \|_{H^{N+1}} + \| \nabla \dr_0 \|^2_{H^{N+1}} + \| \nabla \dr_0 \|^3_{H^{N+1}} + \| \nabla \dr_0 \|^4_{H^{N+1}} \big) \\
	    \leq & C \big( \| \nabla \dr_0 \|_{H^{N+1}} + \| \nabla \dr_0 \|^4_{H^{N+1}} \big) \mathscr{D}_{N, \eps} (t)
	  \end{aligned}
	\end{equation}
	holds for all $|m| \leq N$. For the term $ - \delta \l \partial^m ( \u_R^\eps \cdot \nabla \dr_0 ) , \partial^m ( \sqrt{\eps} \u_R^\eps \cdot \nabla \dr_R^\eps ) \r $, we have
	\begin{equation}\label{IN1-13}
	  \begin{aligned}
	    & - \delta \l \partial^m ( \u_R^\eps \cdot \nabla \dr_0 ) , \partial^m ( \sqrt{\eps} \u_R^\eps \cdot \nabla \dr_R^\eps ) \r \\
	    \leq & C \sqrt{\eps} \| \u_R^\eps \|_{L^6} \| \nabla \partial^m \dr_0 \|_{L^3} \| \partial^m ( \u_R^\eps \cdot \nabla \dr_R^\eps ) \|_{L^2} \\
	    & + C \sqrt{\eps} \sum_{0 \neq m' \leq m} \| \partial^{m'} \u_R^\eps \|_{L^4} \| \nabla \partial^{m - m'} \dr_0 \|_{L^4} \| \partial^m (  \u_R^\eps \cdot \nabla \dr_R^\eps ) \|_{L^2} \\
	    \leq & C \sqrt{\eps} \| \nabla \u_R^\eps \|_{H^N} \| \nabla \dr_0 \|_{H^{N+1}} \| \u_R^\eps \|_{H^N} \| \nabla \dr_R^\eps \|_{H^N} \\
	    \leq & C \eps^2 \| \nabla \dr_0 \|_{H^{N+1}} \mathscr{E}^\frac{1}{2}_{N, \eps} (t) \mathscr{D}_{N, \eps} (t)
	  \end{aligned}
	\end{equation}
	for $|m| \leq N$. It is easy to estimate
	\begin{equation}\label{IN1-14}
	  \begin{aligned}
	    & \delta \l \partial^m ( \u_R^\eps \cdot \nabla \dr_0 ) , \partial^m \D_{\u_0 + \sqrt{\eps} \u_R^\eps} \dr_R^\eps \r \\
	    \leq & C \sum_{m' \leq m} \| \partial^{m'} \u_R^\eps \|_{L^6} \| \nabla \partial^{m-m'} \dr_0 \|_{L^3} \| \partial^m \D_{\u_0 + \sqrt{\eps} \u_R^\eps} \dr_R^\eps \|_{L^2} \\
	    \leq & C \| \nabla \u_R^\eps \|_{H^N} \| \nabla \dr_0 \|_{H^{N+1}} \| \D_{\u_0 + \sqrt{\eps} \u_R^\eps} \dr_R^\eps \|_{L^2} \\
	    \leq & C \sqrt{\eps} \| \nabla \dr_0 \|_{H^{N+1}} \mathscr{D}_{N, \eps} (t)
	  \end{aligned}
	\end{equation}
	for $|m| \leq N$. The term $ \sum_{0 \neq m' \leq m} C_m^{m'} \l \sqrt{\eps} \partial^{m'} \u_R^\eps \otimes \partial^{m-m'} \D_{\u_0 + \sqrt{\eps} \u_R^\eps} \dr_R^\eps , \nabla \partial^m \dr_R^\eps \r $ can be bounded by
	\begin{equation}\label{IN1-15}
	  \begin{aligned}
	    &  \sum_{0 \neq m' \leq m} C_m^{m'} \l \sqrt{\eps} \partial^{m'} \u_R^\eps \otimes \partial^{m-m'} \D_{\u_0 + \sqrt{\eps} \u_R^\eps} \dr_R^\eps , \nabla \partial^m \dr_R^\eps \r \\
	    \leq & C \sqrt{\eps} \sum_{0 \neq m' \leq m} \| \partial^{m'} \u_R^\eps \|_{L^4} \| \partial^{m-m'} \D_{\u_0 + \sqrt{\eps} \u_R^\eps} \dr_R^\eps \|_{L^4} \| \nabla \partial^m \dr_R^\eps \|_{L^2} \\
	    \leq & C \sqrt{\eps} \| \nabla \u_R^\eps \|_{H^N} \| \D_{\u_0 + \sqrt{\eps} \u_R^\eps} \dr_R^\eps \|_{H^N} \\
	    \leq & C \eps \mathscr{E}^\frac{1}{2}_{N, \eps} (t) \mathscr{D}_{N, \eps} (t) \,.
	  \end{aligned}
	\end{equation}
	Via the similar arguments in \eqref{IN1-15}, we yield
	\begin{equation}\label{IN1-16}
	  \begin{aligned}
	    & - \sum_{0 \neq m' \leq m} C_m^{m'} \l \partial^m \D_{\u_0 + \sqrt{\eps} \u_R^\eps} \dr_R^\eps , \sqrt{eps} \partial^{m'} \u_R^\eps \cdot \nabla \partial^{m-m'} \dr_R^\eps \r \\
	    \leq & C \sqrt{\eps} \| \nabla \u_R^\eps \|_{H^N} \| \nabla \dr_R^\eps \|_{H^N} \| \D_{\u_0 + \sqrt{\eps} \u_R^\eps} \dr_R^\eps \|_{H^N} \\
	    \leq & C \eps \mathscr{E}^\frac{1}{2}_{N, \eps} (t) \mathscr{D}_{N, \eps} (t) \,.
	  \end{aligned}
	\end{equation}
	
	Collecting the bounds \eqref{IN1-5}-\eqref{IN1-16} reduces to
	\begin{equation}\label{IN1-Bnd}
	  \begin{aligned}
	    \mathcal{I}_N^{(1)} \leq &  C \mathscr{E}^\frac{1}{2}_{N,\eps} (t) \mathscr{D}^\frac{1}{2}_{N, \eps} (t) \| \nabla \u_0 \|_{H^N} + C \big( \| \nabla \dr_0 \|_{H^{N+1}} + \| \nabla \dr_0 \|^4_{H^{N+1}} \big) \mathscr{D}_{N, \eps} (t) \\
	    & + C \eps ( 1 + \| \nabla \dr_0 \|_{H^{N+2}} ) \mathscr{E}^\frac{1}{2}_{N,\eps} (t) \mathscr{D}_{N, \eps} (t)
	  \end{aligned}
	\end{equation}
	for all $\eps \in (0, \eps_0)$. \\
	
	{\em Step 2. Control the term $\mathcal{I}_N^{(2)}$.}
	
	We need to estimate the terms $ - \tfrac{1}{\eps} \l \partial^m \mathcal{T}_{\u} , \nabla \partial^m \u_R^\eps \r $ and $ - \tfrac{1}{\sqrt{\eps}} \l \partial^m \mathcal{R}_{\u}, \nabla \partial^m \u_R^\eps \r $, where the expressions $ \mathcal{T}_{\u} $ and $\mathcal{R}_{\u}$ are defined in \eqref{Tu} and \eqref{Ru}, respectively. First, we have
	\begin{equation}\label{IN2-1}
	  \begin{aligned}
	    & - \tfrac{\mu_1}{\eps} \l \partial^m \big[ ( \A_0 : ( \dr_R^\eps \otimes \dr_0 + \dr_0 \otimes \dr_R^\eps ) \dr_0 \otimes \dr_0 ) \big] , \nabla \partial^m \u_R^\eps \r \\
	    \leq & \tfrac{C}{\eps} \sum_{m' \leq m} \| \partial^{m'} \u_R^\eps \|_{L^6} \| \partial^{m-m'} ( \A_0 \dr_0 \otimes \dr_0 \otimes \dr_0 ) \|_{L^3} \| \nabla \partial^m \u_R^\eps \|_{L^2} \\
	    \leq & \tfrac{C}{\eps} \| \nabla \dr_R^\eps \|_{H^N} \| \u_0 \|_{H^{N+2}} \| \nabla \u_R^\eps \|_{H^N} \big( 1 + \| \nabla \dr_0 \|_{H^{N+1}} + \| \nabla \dr_0 \|^2_{H^{N+1}} + \| \nabla \dr_0 \|^3_{H^{N+1}} \big) \\
	    \leq & \tfrac{C}{\eps} \| \nabla \u_R^\eps \|_{H^N} \| \nabla \dr_R^\eps \|_{H^N} \| \u_0 \|_{H^{N+2}} ( 1 + \| \nabla \dr_0 \|^3_{H^{N+1}} ) \\
	    \leq & C  \| \u_0 \|_{H^{N+2}} ( 1 + \| \nabla \dr_0 \|^3_{H^{N+1}} ) \mathscr{D}_{N, \eps} (t)
	  \end{aligned}
	\end{equation}
	holds for all $|m| \leq N$. We derive from the similar arguments in \eqref{IN2-1} that for $|m| \leq N$
	\begin{equation}\label{IN2-2}
	  \begin{aligned}
	    & - \tfrac{\mu_1}{\eps} \l \partial^m \big[ ( \A_0 : \dr_0 \otimes \dr_0 ) ( \dr_R^\eps \otimes \dr_0 + \dr_0 \otimes \dr_R^\eps ) \big] , \nabla \partial^m \u_R^\eps \r \\
	    & \leq C  \| \u_0 \|_{H^{N+2}} ( 1 + \| \nabla \dr_0 \|^3_{H^{N+1}} ) \mathscr{D}_{N, \eps} (t) \,.
	  \end{aligned}
	\end{equation}
	For the term $ - \tfrac{\mu_2}{\eps} \l \partial^m \big[ ( \D_{\u_0} \dr_0 + \B_0 \dr_0 ) \otimes \dr_R^\eps \big] , \nabla \partial^m \u_R^\eps \r $, we estimate
	\begin{equation}
	  \begin{aligned}
	    & - \tfrac{\mu_2}{\eps} \l \partial^m \big[ ( \D_{\u_0} \dr_0 + \B_0 \dr_0 ) \otimes \dr_R^\eps \big] , \nabla \partial^m \u_R^\eps \r \\
	    \leq & \tfrac{C}{\eps} \sum_{m' \leq m} \| \partial^{m'} \dr_R^\eps \|_{L^6}  \| \partial^m ( \D_{\u_0} \dr_0 + \B_0 \dr_0  ) \|_{L^3} \| \nabla \partial^m \u_R^\eps \|_{L^2} \\
	    \leq & \tfrac{C}{\eps} \| \nabla \dr_R^\eps \|_{H^N} \| \nabla \u_R^\eps \|_{H^N} \big( \| \D_{\u_0} \dr_0 \|_{H^{N+1}} + \| \u_0 \|_{H^{N+2}} ( 1 + \| \nabla \dr_0 \|_{H^{N+1}} ) \big) \\
	    \leq & C \big( \| \D_{\u_0} \dr_0 \|_{H^{N+1}} + \| \u_0 \|_{H^{N+2}} ( 1 + \| \nabla \dr_0 \|_{H^{N+1}} ) \big) \mathscr{D}_{N, \eps} (t) \,.
	  \end{aligned}
	\end{equation}
	Recall that $\dr_0$ obeys the third equation of \eqref{PLQ}, hence
	\begin{equation}
	  \begin{aligned}
	    -\lambda_1 ( \D_{\u_0} \dr_0 + \B_0 \dr_0 ) = \Delta \dr_0 + \bm{\gamma}_0 \dr_0 + \lambda_2 \A_0 \dr_0 \,,
	  \end{aligned}
	\end{equation}
	where the Lagrangian multiplier $\bm{\gamma}_0$ is
	\begin{equation}
	  \begin{aligned}
	    \bm{\gamma}_0 = |\nabla \dr_0|^2 - \lambda_2 \A_0 : \dr_0 \otimes \dr_0 \,.
	  \end{aligned}
	\end{equation}
	Then one can easily yield that if the integer $N \geq 2$,
	\begin{equation}\label{Du0-d0-HN}
	  \begin{aligned}
	    \| \D_{\u_0} \dr_0 \|_{H^N} \leq C ( \| \nabla \u_0 \|_{H^N} + \| \Delta \dr_0 \|_{H^N} ) ( 1 + \| \nabla \dr_0 \|^3_{H^N} ) \,.
	  \end{aligned}
	\end{equation}
	and
	\begin{equation}\label{D2u0-d0-HN}
	  \begin{aligned}
	    \| \D_{\u_0}^2 \dr_0 \|_{H^N} \leq C ( \| \nabla \u_0 \|_{H^{N+2}} + \| \Delta \dr_0 \|_{H^{N+2}} ) ( 1 + \| \u_0 \|^6_{H^{N+2}} + \| \nabla \dr_0 \|^6_{H^{N+2}} ) \,.
	  \end{aligned}
	\end{equation}
	Consequently, we obtain the bound
	\begin{equation}\label{IN2-3}
	  \begin{aligned}
	    & - \tfrac{\mu_2}{\eps} \l \partial^m \big[ ( \D_{\u_0} \dr_0 + \B_0 \dr_0 ) \otimes \dr_R^\eps \big] , \nabla \partial^m \u_R^\eps \r \\
	    \leq & C ( \| \u_0 \|_{H^{N+2}} + \| \nabla \dr_0 \|_{H^{N+2}} ) ( 1 + \| \nabla \dr_0 \|^3_{H^{N+2}} ) \mathscr{D}_{N, \eps} (t) \,.
	  \end{aligned}
	\end{equation}
	Next, we deduce
	\begin{equation}\label{IN2-4}
	  \begin{aligned}
	    & - \tfrac{\mu_2}{\eps} \l \partial^m \big[ ( \B_0 \dr_R^\eps + \u_0 \cdot \nabla \dr_R^\eps + \u_R^\eps \cdot \nabla \dr_0 ) \otimes \dr_0 \big] , \nabla \partial^m \u_R^\eps \r \\
	    \leq & \tfrac{C}{\eps} \| \nabla \partial^m \u_R^\eps \|_{L^2} \sum_{m' \leq m} \Big( \| \partial^{m'} \dr_R^\eps \|_{L^6} \| \partial^{m-m'} ( \B_0 \otimes \dr_0 ) \|_{L^3} \\
	    & \qquad \qquad + \| \partial^{m'} \u_R^\eps \|_{L^6} \| \partial^{m-m'} (\nabla \dr_0 \otimes \dr_0) \|_{L^3} \Big) \\
	    & + \tfrac{C}{\eps} \| \u_0 \otimes \dr_0 \|_{L^\infty} \| \nabla \partial^m \dr_R^\eps \|_{L^2} \| \nabla \partial^m \u_R^\eps \|_{L^2} \\
	    & + \tfrac{C}{\eps} \sum_{0 \neq m' \leq m} \| \partial^{m'} ( \u_0 \otimes \dr_0 ) \|_{L^4} \| \nabla \partial^{m-m'} \dr_R^\eps \|_{L^4} \| \nabla \partial^m \u_R^\eps \|_{L^2} \\
	    \leq & \tfrac{C}{\eps} \| \nabla \dr_R^\eps \|_{H^N} \| \nabla \u_R^\eps \|_{H^N} \| \u_0 \|_{H^{N+2}} ( 1 + \| \nabla \dr_0 \|_{H^{N+2}} ) \\
	    & + \tfrac{C}{\eps} \| \nabla \u_R^\eps \|^2_{H^N} ( \| \nabla \dr_0 \|_{H^{N+1}} + \| \nabla \dr_0 \|^2_{H^{N+1}} ) \\
	    \leq & C ( \| \u_0 \|_{H^{N+2}} + \| \nabla \dr_0 \|_{H^{N+2}} ) ( 1 + \| \nabla \dr_0 \|_{H^{H^{N+1}}} ) \mathscr{D}_{N, \eps} (t) \,.
	  \end{aligned}
	\end{equation}
	Analogous estimates in \eqref{IN2-3} tell us
	\begin{equation}\label{IN2-5}
	  \begin{aligned}
	    & - \tfrac{\mu_3}{\eps} \l \partial^m \big[ \dr_R^\eps \otimes ( \D_{\u_0} \dr_0 + \B_0 \dr_0 ) \big], \nabla \partial^m \u_R^\eps \r \\
	    \leq & C ( \| \u_0 \|_{H^{N+2}} + \| \nabla \dr_0 \|_{H^{N+2}} ) ( 1 + \| \nabla \dr_0 \|^3_{H^{N+2}} ) \mathscr{D}_{N, \eps} (t) \,,
	  \end{aligned}
	\end{equation}
	and the similar calculations in \eqref{IN2-4} reduce to
	\begin{equation}\label{IN2-6}
	  \begin{aligned}
	    & - \tfrac{\mu_3}{\eps} \l \partial^m \big[ \dr_0 \otimes ( \B_0 \dr_R^\eps + \u_0 \cdot \nabla \dr_R^\eps + \u_R^\eps \cdot \nabla \dr_0 ) \big] , \nabla \partial^m \u_R^\eps \r \\
	    \leq & C ( \| \u_0 \|_{H^{N+2}} + \| \nabla \dr_0 \|_{H^{N+2}} ) ( 1 + \| \nabla \dr_0 \|_{H^{H^{N+1}}} ) \mathscr{D}_{N, \eps} (t) \,.
	  \end{aligned}
	\end{equation}
	It is easy to derive
	\begin{equation}\label{IN2-7}
	  \begin{aligned}
	    & - \tfrac{\mu_5}{\eps} \l \partial^m \big[ ( \A_0 \dr_0 ) \otimes \dr_R^\eps + ( \A_0 \dr_R^\eps ) \otimes \dr_0 \big] , \nabla \partial^m \u_R^\eps \r \\
	    \leq & \tfrac{C}{\eps} \sum_{m' \leq m} \| \partial^{m'} \dr_R^\eps \|_{L^6}  \| \partial^{m-m'} ( \A_0 \dr_0 ) \|_{L^3} \| \nabla \partial^m \dr_R^\eps \|_{L^2} \\
	    \leq & \tfrac{C}{\eps} \| \nabla \dr_R^\eps \|_{H^N} \| \nabla \u_R^\eps \|_{H^N} \| \u_0 \|_{H^{N+2}} ( 1 + \| \nabla \dr_0 \|_{H^{N}} ) \\
	    \leq & C \| \u_0 \|_{H^{N+2}} ( 1 + \| \nabla \dr_0 \|_{H^N} ) \mathscr{D}_{N, \eps} (t)
	  \end{aligned}
	\end{equation}
	for all $|m| \leq N$. Furthermore, via the analogous estimates in \eqref{IN2-7}, we imply that
	\begin{equation}\label{IN2-8}
	  \begin{aligned}
	    & - \tfrac{\mu_6}{\eps} \l \partial^m \big[ \dr_R^\eps \otimes (\A_0 \dr_0) + \dr_0 \otimes ( \A_0 \dr_R^\eps ) \big] , \nabla \partial^m \u_R^\eps \r \\
	    \leq & C \| \u_0 \|_{H^{N+2}} ( 1 + \| \nabla \dr_0 \|_{H^N} ) \mathscr{D}_{N, \eps} (t) \,.
	  \end{aligned}
	\end{equation}
	As a consequence, collecting the estimates \eqref{IN2-1}, \eqref{IN2-2}, \eqref{IN2-3}, \eqref{IN2-4}, \eqref{IN2-5}, \eqref{IN2-6}, \eqref{IN2-7} and \eqref{IN2-8}, we deduce from the definition of $\mathcal{T}_{\u}$ in \eqref{Tu} that
	\begin{equation}\label{IN2-9}
	  \begin{aligned}
	    - \tfrac{1}{\eps} \l \partial^m \mathcal{T}_{\u} , \nabla \partial^m \u_R^\eps \r \leq C ( \| \u_0 \|_{H^{N+2}} + \| \nabla \dr_0 \|_{H^{N+2}} ) ( 1 + \| \nabla \dr_0 \|^3_{H^{N+2}} ) \mathscr{D}_{N, \eps} (t) \,.
	  \end{aligned}
	\end{equation}
	
	We next estimate the quantity $ - \tfrac{1}{\sqrt{\eps}} \l \partial^m \mathcal{R}_{\u} , \nabla \partial^m \u_R^\eps \r $. More precisely, we will consider the terms $ - \tfrac{1}{\sqrt{\eps}} \l \sqrt{\eps}^{i-1} \partial^m \mathcal{M}_i , \nabla \partial^m \u_R^\eps \r $ for $1 \leq i \leq 4$. First, we estimate
	  \begin{align}\label{IN2-10}
	    \no & - \tfrac{\mu_1}{\sqrt{\eps}} \l \partial^m \big[ ( \A_0 : \dr_R^\eps \otimes \dr_R^\eps ) \dr_0 \otimes \dr_0 \big] , \nabla \partial^m \u_R^\eps \r \\
	    \no \leq & \tfrac{C}{\sqrt{\eps}} \| \partial^m ( \A_0 : \dr_0 \otimes \dr_0  ) \|_{L^6} \| \dr_R^\eps \|^2_{L^6} \| \nabla \partial^m \u_R^\eps \|_{L^2} \\
	    \no & + \tfrac{C}{\sqrt{\eps}} \sum_{0 \neq m' \leq m} \| \partial^{m-m'} ( \A_0 : \dr_0 \otimes \dr_0 ) \|_{L^4} \| \partial^{m'} ( \dr_R^\eps \otimes \dr_R^\eps ) \|_{L^4} \| \nabla \partial^m \u_R^\eps \|_{L^2} \\
	    \no \leq & \tfrac{C}{\sqrt{\eps}} \| \u_0 \|_{H^{N+2}} ( 1 + \| \nabla \dr_0 \|^2_{H^{N+2}} ) \| \nabla \dr_R^\eps \|^2_{H^N} \| \nabla \u_R^\eps \|_{H^N} \\
	    \leq & C \eps \| \u_0 \|_{H^{N+2}} ( 1 + \| \nabla \dr_0 \|^2_{H^{N+2}} ) \mathscr{E}^\frac{1}{2}_{N, \eps} (t) \mathscr{D}_{N, \eps} (t) \,.
	  \end{align}
	The term $ - \tfrac{\mu_1}{\sqrt{\eps}} \l \partial^m \big[ 2 ( \A_R^\eps : \dr_0 \otimes \dr_R^\eps ) \dr_0 \otimes \dr_0 \big] , \nabla \partial^m \u_R^\eps \r $ can be controlled by
	\begin{equation}\label{IN2-11}
	  \begin{aligned}
	    & - \tfrac{\mu_1}{\sqrt{\eps}} \l \partial^m \big[ 2 ( \A_R^\eps : \dr_0 \otimes \dr_R^\eps ) \dr_0 \otimes \dr_0 \big] , \nabla \partial^m \u_R^\eps \r \\
	    \leq & \tfrac{C}{\sqrt{\eps}} \| \partial^m \A_R^\eps \|_{L^2} \| ( \dr_R^\eps \otimes \dr_0 ) ( \dr_0 \otimes \dr_0) \|_{L^\infty} \| \nabla \partial^m \u_R^\eps \|_{L^2} \\
	    & + \tfrac{C}{\sqrt{\eps}} \sum_{\substack{0 \neq m' \leq m \\ m'' \leq m' }} \| \partial^{m-m'} \A_R^\eps \|_{L^4} \| \partial^{m''} \dr_R^\eps \|_{L^4} \| \partial^{m'-m''} ( \dr_0 ( \dr_0 \otimes \dr_0 ) ) \|_{L^\infty} \| \nabla \partial^m \u_R^\eps \|_{L^2} \\
	    \leq & \tfrac{C}{\sqrt{\eps}} \| \dr_R^\eps \|_{H^N} \| \nabla \u_R^\eps \|^2_{H^N} + \tfrac{C}{\sqrt{\eps}} \| \nabla \u_R^\eps \|^2_{H^N} ( \| \dr_R^\eps \|_{H^N} + \| \nabla \dr_R^\eps \|_{H^N} ) ( 1 + \| \nabla \dr_0 \|^3_{H^{N+2}} ) \\
	    \leq & C \eps ( 1 + \| \nabla \dr_0 \|^3_{H^{N+2}}  ) \mathscr{E}^\frac{1}{2}_{N, \eps} (t) \mathscr{D}_{N, \eps} (t) \,.
	  \end{aligned}
	\end{equation}
	Via the similar arguments in \eqref{IN2-10} and \eqref{IN2-11}, we imply that for $|m| \leq N$
	\begin{equation}\label{IN2-12}
	  \begin{aligned}
	    & - \tfrac{\mu_1}{\sqrt{\eps}} \l \partial^m \big[ ( \A_0 : \dr_0 \otimes \dr_0 ) \dr_R^\eps \otimes \dr_R^\eps \big] , \nabla \partial^m \u_R^\eps \r \\
	    & - \tfrac{\mu_1}{\sqrt{\eps}} \l \partial^m \big[ ( 2 \A_0 : \dr_0 \otimes \dr_R^\eps + \A_R^\eps : \dr_0 \otimes \dr_0  ) ( \dr_R^\eps \otimes \dr_0 + \dr_0 \otimes \dr_R^\eps ) \big] , \nabla \partial^m \u_R^\eps \r \\
	    \leq & C \eps \| \u_0 \|_{H^{N+2}} ( 1 + \| \nabla \dr_0 \|^2_{H^{N+2}} ) \mathscr{E}^\frac{1}{2}_{N, \eps} (t) \mathscr{D}_{N, \eps} (t) \\
	    & + C \eps ( 1 + \| \nabla \dr_0 \|^3_{H^{N+2}}  ) \mathscr{E}^\frac{1}{2}_{N, \eps} (t) \mathscr{D}_{N, \eps} (t) \\
	    \leq & C ( 1 + \| \u_0 \|^3_{H^{N+2}} + \| \nabla \dr_0 \|^3_{H^{N+2}} ) \mathscr{E}^\frac{1}{2}_{N, \eps} (t) \mathscr{D}_{N, \eps} (t) \,.
	  \end{aligned}
	\end{equation}
	We calculate that
	\begin{equation}\label{IN2-13}
	  \begin{aligned}
	    & - \tfrac{\mu_2}{\sqrt{\eps}} \l \partial^m [ (\B_R^\eps \dr_R^\eps) \otimes \dr_0 ] , \nabla \partial^m \u_R^\eps \r \\
	    \leq & \tfrac{C}{\sqrt{\eps}} \| \nabla \partial^{m'} \u_R^\eps \|_{L^2} \| \partial^m (\B_R^\eps \dr_R^\eps) \|_{L^2} \\
	    & + \tfrac{C}{\sqrt{\eps}} \sum_{0 \neq m' \leq m} \| \partial^{m-m'} ( \B_R^\eps \dr_R^\eps ) \|_{L^4} \| \partial^{m'} \dr_0 \|_{L^4} \| \nabla \partial^m \u_R^\eps \|_{L^2} \\
	    \leq & \tfrac{C}{\sqrt{\eps}} \| \nabla \u_R^\eps \|^2_{H^N} \| \dr_R^\eps \|_{H^N} + \tfrac{C}{\sqrt{\eps}} \| \nabla \u_R^\eps \|^2_{H^N} \| \dr_R^\eps \|_{H^N} \| \nabla \dr_0 \|_{H^N} \\
	    \leq & C \eps ( 1 + \| \nabla \dr_0 \|_{H^N} ) \mathscr{E}^\frac{1}{2}_{N, \eps} (t) \mathscr{D}_{N, \eps} (t) \,.
	  \end{aligned}
	\end{equation}
	Moreover, it is easily calculated that
	  \begin{align}\label{IN2-14}
	    \no & - \tfrac{\mu_2}{\sqrt{\eps}} \l \partial^m \big[ ( \D_{\u_0 + \sqrt{\eps} \u_R^\eps} \dr_R^\eps + \u_0 \cdot \nabla \dr_R^\eps + \u_R^\eps \cdot \nabla \dr_0 + \B_0 \dr_R^\eps + \B_R^\eps \dr_0  ) \otimes \dr_R^\eps \big] , \nabla \partial^m \u_R^\eps \r \\
	    \no \leq & \tfrac{\mu_2}{\sqrt{\eps}} \| \partial^m ( \D_{\u_0 + \sqrt{\eps} \u_R^\eps} \dr_R^\eps \otimes \dr_R^\eps ) \|_{L^2} \| \nabla \partial^m \u_R^\eps \|_{L^2}  + \tfrac{\mu_2}{\sqrt{\eps}} \| \partial^m ( (\u_0 \cdot \nabla \dr_R^\eps ) \otimes \dr_R^\eps ) \|_{L^2} \| \nabla \partial^m \u_R^\eps \|_{L^2} \\
	    \no & + \tfrac{\mu_2}{\sqrt{\eps}} \| \partial^m ( ( \B_R^\eps \dr_0 ) \otimes \dr_R^\eps ) \|_{L^2} \| \nabla \partial^m \u_R^\eps \|_{L^2} \\
	    \no & + \tfrac{C}{\sqrt{\eps}} \sum_{m' \leq m} \| \partial^{m'} \u_R^\eps \|_{L^6} \| \partial^{m-m'} ( \nabla \dr_0 \otimes \dr_R^\eps ) \|_{L^3} \| \nabla \partial^m \u_R^\eps \|_{L^2} \\
	    \no & + \tfrac{C}{\sqrt{\eps}} \sum_{m' \leq m} \| \partial^{m'} \dr_R^\eps \|_{L^6} \| \partial^{m-m'} ( \B_0 \dr_R^\eps ) \|_{L^3} \| \nabla \partial^m \u_R^\eps \|_{L^2} \\
	    \no \leq & \tfrac{C}{\sqrt{\eps}} \| \D_{\u_0 + \sqrt{\eps} \u_R^\eps} \dr_R^\eps \|_{H^N} \| \dr_R^\eps \|_{H^N} \| \nabla \u_R^\eps \|_{H^N} + \tfrac{C}{\sqrt{\eps}} \| \nabla \u_R^\eps \|_{H^N} \| \u_0 \|_{H^N} \| \nabla \dr_R^\eps \|_{H^N} \| \dr_R^\eps \|_{H^N} \\
	    \no & + \tfrac{C}{\sqrt{\eps}} \| \nabla \u_R^\eps \|^2_{H^N} \| \dr_R^\eps \|_{H^N} ( 1 + \| \nabla \dr_0 \|_{H^{N+1}} ) \\
	    \no & + \tfrac{C}{\sqrt{\eps}} \| \nabla \dr_R^\eps \|_{H^N} \| \nabla \u_R^\eps \|_{H^N} ( \| \dr_R^\eps \|_{H^N} + \| \nabla \dr_R^\eps \|_{H^N} ) \| \u_0 \|_{H^{N+2}} \\
	    \no \leq & \tfrac{C}{\sqrt{\eps}} ( \| \nabla \u_R^\eps \|^2_{H^N} + \| \nabla \dr_R^\eps \|^2_{H^N} ) ( \| \dr_R^\eps \|_{H^N} + \| \nabla \dr_R^\eps \|_{H^N} ) ( 1 + \| \u_0 \|_{H^{N+2}} + \| \nabla \dr_0 \|_{H^{N+2}} ) \\
	    \leq & C \eps ( 1 + \| \u_0 \|_{H^{N+2}} + \| \nabla \dr_0 \|_{H^{N+2}} ) \mathscr{E}^\frac{1}{2}_{N, \eps} (t) \mathscr{D}_{N,\eps} (t) \,.
	  \end{align}
	We employ the similar arguments of \eqref{IN2-13} to obtain
	\begin{equation}\label{IN2-15}
	  \begin{aligned}
	    - \tfrac{\mu_3}{\sqrt{\eps}} \l \partial^m [ \dr_0 \otimes ( \B_R^\eps \dr_R^\eps ) ] , \nabla \partial^m \u_R^\eps \r \leq C \eps ( 1 + \| \nabla \dr_0 \|_{H^N} ) \mathscr{E}^\frac{1}{2}_{N,\eps} (t) \mathscr{D}_{N,\eps} (t) \,,
	  \end{aligned}
	\end{equation}
	and employ the analogous estimates of \eqref{IN2-14} to imply
	\begin{equation}\label{IN2-16}
	  \begin{aligned}
	    & - \tfrac{\mu_3}{\sqrt{\eps}} \l \partial^m \big[ \dr_R^\eps \otimes ( \D_{\u_0 + \sqrt{\eps} \u_R^\eps} \dr_R^\eps + \u_0 \cdot \nabla \dr_R^\eps + \u_R^\eps \cdot \nabla \dr_0 + \B_0 \dr_R^\eps + \B_R^\eps \dr_0  ) \big] , \nabla \partial^m \u_R^\eps \r \\
	    & \leq C \eps ( 1 + \| \u_0 \|_{H^{N+2}} + \| \nabla \dr_0 \|_{H^{N+2}} ) \mathscr{E}^\frac{1}{2}_{N, \eps} (t) \mathscr{D}_{N,\eps} (t) \,.
	  \end{aligned}
	\end{equation}
	We now estimate that for all $|m| \leq N$
	\begin{equation}\label{IN2-17}
	  \begin{aligned}
	    & - \tfrac{\mu_5}{\sqrt{\eps}} \l \partial^m \big[ ( \A_R^\eps \dr_R^\eps ) \otimes \dr_0 + ( \A_R^\eps \dr_0 + \A_0 \dr_R^\eps ) \otimes \dr_R^\eps \big] , \nabla \partial^m \u_R^\eps \r \\
	    \leq & \tfrac{C}{\sqrt{\eps}} \| \partial^m \A_R^\eps \|_{L^2} \| \dr_R^\eps \otimes \dr_0 \|_{L^\infty} \| \nabla \partial^m \u_R^\eps \|_{L^2} \\
	    & + \tfrac{C}{\sqrt{\eps}} \sum_{0 \neq m' \leq m} \| \partial^{m-m'} \A_R^\eps \|_{L^4} \| \partial^{m'} ( \dr_R^\eps \otimes \dr_0 ) \|_{L^4} \| \nabla \partial^m \u_R^\eps \|_{L^2} \\
	    & + \tfrac{C}{\sqrt{\eps}} \sum_{m' \leq m} \| \partial^{m'} \dr_R^\eps \|_{L^6} \| \partial^{m-m'} ( \A_0 \dr_R^\eps ) \|_{L^3} \| \nabla \partial^m \u_R^\eps \|_{L^2} \\
	    \leq & \tfrac{C}{\sqrt{\eps}} \| \nabla \u_R^\eps \|^2_{H^N} \| \dr_R^\eps \|_{H^N} + \tfrac{C}{\sqrt{\eps}} \| \nabla \u_R^\eps \|^2_{H^N} ( \| \dr_R^\eps \|_{H^N} + \| \nabla \dr_R^\eps \|_{H^N} ) ( 1 + \| \nabla \dr_0 \|_{H^N} ) \\
	    & + \tfrac{C}{\sqrt{\eps}} \| \nabla \u_R^\eps \|_{H^N} \| \nabla \dr_R^\eps \|_{H^N} ( \| \dr_R^\eps \|_{H^N} + \| \nabla \dr_R^\eps \|_{H^N} ) \| \u_0 \|_{H^{N+2}} \\
	    \leq &  \tfrac{C}{\sqrt{\eps}} ( \| \nabla \u_R^\eps \|^2_{H^N} + \| \nabla \dr_R^\eps \|^2_{H^N} ) ( \| \dr_R^\eps \|_{H^N} + \| \nabla \dr_R^\eps \|_{H^N} ) ( 1 + \| \u_0 \|_{H^{N+2}} + \| \nabla \dr_0 \|_{H^{N+2}} ) \\
	    \leq & C \eps  ( 1 + \| \u_0 \|_{H^{N+2}} + \| \nabla \dr_0 \|_{H^{N+2}} ) \mathscr{E}^\frac{1}{2}_{N, \eps} (t) \mathscr{D}_{N, \eps} (t) \,,
	  \end{aligned}
	\end{equation}
	and similarly we have
	\begin{equation}\label{IN2-18}
	  \begin{aligned}
	    & - \tfrac{\mu_6}{\sqrt{\eps}} \l \partial^m \big[ \dr_0 \otimes ( \A_R^\eps \dr_R^\eps ) + \dr_R^\eps \otimes ( \A_R^\eps \dr_0 + \A_0 \dr_R^\eps ) \big] , \nabla \partial^m \u_R^\eps \r \\
	    & \leq C \eps  ( 1 + \| \u_0 \|_{H^{N+2}} + \| \nabla \dr_0 \|_{H^{N+2}} ) \mathscr{E}^\frac{1}{2}_{N, \eps} (t) \mathscr{D}_{N, \eps} (t) \,.
	  \end{aligned}
	\end{equation}
	Recalling the definition of $\mathcal{M}_1$ in \eqref{M1}, we derive from collecting the bounds \eqref{IN2-10}, \eqref{IN2-11}, \eqref{IN2-12}, \eqref{IN2-13}, \eqref{IN2-14}, \eqref{IN2-15}, \eqref{IN2-16}, \eqref{IN2-17} and \eqref{IN2-18} that for all $|m| \leq N$
	\begin{equation}\label{IN2-19}
	  \begin{aligned}
	    - \tfrac{1}{\sqrt{\eps}} \l \partial^m \mathcal{M}_1 , \nabla \partial^m \u_R^\eps \r \leq C \eps ( 1 + \| \u_0 \|^3_{H^{N+2}} + \| \nabla \dr_0 \|^3_{H^{N+2}} )  \mathscr{E}^\frac{1}{2}_{N, \eps} (t) \mathscr{D}_{N, \eps} (t) \,.
	  \end{aligned}
	\end{equation}
	
	Next we estimate the term $ - \tfrac{1}{\sqrt{\eps}} \l \sqrt{\eps} \partial^m \mathcal{M}_2 , \nabla \partial^m \u_R^\eps \r $ for $|m| \leq N$. Firstly, we have
	\begin{equation}\label{IN2-20}
	  \begin{aligned}
	    & - \tfrac{\mu_1}{\sqrt{\eps}} \l \sqrt{\eps} \partial^m \big[ ( \A_R^\eps : \dr_0 \otimes \dr_0 + 2 \A_0 : \dr_0 \otimes \dr_R^\eps ) \dr_R^\eps \otimes \dr_R^\eps \big], \nabla \partial^m \u_R^\eps \r \\
	    \leq & \mu_1 \| \partial^m \A_R^\eps \|_{L^2} \| ( \dr_0 \otimes \dr_0  ) ( \dr_R^\eps \otimes \dr_R^\eps ) \|_{L^\infty} \| \nabla \partial^m \u_R^\eps \|_{L^2} \\
	    & + C \sum_{0 \neq m' \leq m} \| \partial^{m-m'} \A_R^\eps \|_{L^4} \| \partial^{m'} ( ( \dr_0 \otimes \dr_0  ) ( \dr_R^\eps \otimes \dr_R^\eps ) ) \|_{L^4} \| \nabla \partial^m \u_R^\eps \|_{L^2} \\
	    & + C \sum_{m' \leq m} \| \partial^{m'} \dr_R^\eps \|_{L^6} \| \partial^{m-m'}( ( \A_0 : \dr_0 \otimes \dr_R^\eps ) \dr_R^\eps ) \|_{L^3} \| \nabla \partial^m \u_R^\eps \|_{L^2} \\
	    \leq & C \| \nabla \u_R^\eps \|^2_{H^N} \| \dr_R^\eps \|^2_{H^N} + C \| \nabla \u_R^\eps \|^2_{H^N} ( \| \dr_R^\eps \|^2_{H^N} + \| \nabla \dr_R^\eps \|^2_{H^N} ) ( 1 + \| \nabla \dr_0 \|^2_{H^N} ) \\
	    & + C \| \nabla \u_R^\eps \|_{H^N} \| \nabla \dr_R^\eps \|_{H^N} ( \| \dr_R^\eps \|^2_{H^N} + \| \nabla \dr_R^\eps \|^2_{H^N} ) \| \u_0 \|_{H^{N+2}} \\
	    \leq & C ( \| \nabla \u_R^\eps \|^2_{H^N} + \| \nabla \dr_R^\eps \|^2_{H^N} ) ( \| \dr_R^\eps \|^2_{H^N} + \| \nabla \dr_R^\eps \|^2_{H^N} ) ( 1 + \| \u_0 \|^2_{H^{N+2}} + \| \nabla \dr_0 \|^2_{H^{N+2}} ) \\
	    \leq & C \eps^2 ( 1 + \| \u_0 \|^2_{H^{N+2}} + \| \nabla \dr_0 \|^2_{H^{N+2}} ) \mathscr{E}_{N,\eps} (t) \mathscr{D}_{N,\eps} (t) \,.
	  \end{aligned}
	\end{equation}
	Via employing the analogous estimates in \eqref{IN2-20}, we have
	\begin{equation}\label{IN2-21}
	  \begin{aligned}
	    & - \tfrac{\mu_1}{\sqrt{\eps}} \l \sqrt{\eps} \partial^m [ ( \A_R^\eps : \dr_R^\eps \otimes \dr_R^\eps ) \dr_0 \otimes \dr_0 ] , \nabla \partial^m \u_R^\eps \r \\
	    & \leq C \eps^2 ( 1 + \| \nabla \dr_0 \|^2_{H^{N+2}} ) \mathscr{E}_{N,\eps} (t) \mathscr{D}_{N,\eps} (t)
	  \end{aligned}
	\end{equation}
	and
	\begin{equation}\label{IN2-22}
	  \begin{aligned}
	    & - \tfrac{\mu_1}{\sqrt{\eps}} \l \sqrt{\eps} \partial^m [ ( 2 \A_R^\eps : \dr_0 \otimes \dr_R^\eps + \A_0 : \dr_R^\eps \otimes \dr_R^\eps ) ( \dr_0 \otimes \dr_R^\eps + \dr_R^\eps \otimes \dr_0 ) ] , \nabla \partial^m \u_R^\eps \r \\
	    & \leq C \eps^2 ( 1 + \| \u_0 \|^2_{H^{N+2}} + \| \nabla \dr_0 \|^2_{H^{N+2}} ) \mathscr{E}_{N,\eps} (t) \mathscr{D}_{N,\eps} (t) \,.
	  \end{aligned}
	\end{equation}
	For the term $ - \tfrac{\mu_2}{\sqrt{\eps}} \l \sqrt{\eps} \partial^m [ ( \B_R^\eps \dr_R^\eps ) \otimes \dr_R^\eps ] , \nabla \partial^m \u_R^\eps \r  $, we deduce that
	\begin{equation}\label{IN2-23}
	  \begin{aligned}
	    & - \tfrac{\mu_2}{\sqrt{\eps}} \l \sqrt{\eps} \partial^m [ ( \B_R^\eps \dr_R^\eps ) \otimes \dr_R^\eps ] , \nabla \partial^m \u_R^\eps \r \\
	    \leq & C \| \partial^m \B_R^\eps \|_{L^2} \| \dr_R^\eps \otimes \dr_R^\eps \|_{L^\infty} \| \nabla \partial^m \u_R^\eps \|_{L^2} \\
	    & + C \sum_{0 \neq m' \leq m} \| \partial^{m-m'} \B_R^\eps \|_{L^4} \| \partial^{m'} (\dr_R^\eps \otimes \dr_R^\eps) \|_{L^4} \| \nabla \partial^m \u_R^\eps \|_{L^2} \\
	    \leq & C \| \nabla \u_R^\eps \|^2_{H^N} \| \dr_R^\eps \|^2_{H^N} + C \| \nabla \u_R^\eps \|^2_{H^N} ( \| \dr_R^\eps \|^2_{H^N} + \| \nabla \dr_R^\eps \|^2_{H^N} ) \\
	    \leq & C \eps^2 \mathscr{E}_{N,\eps} (t) \mathscr{D}_{N,\eps} (t) \,.
	  \end{aligned}
	\end{equation}
	By utilizing the similar calculations in \eqref{IN2-23}, we yield that
	\begin{equation}\label{IN2-24}
	  \begin{aligned}
	    - \tfrac{1}{\sqrt{\eps}} \l \sqrt{\eps} \partial^m [ \mu_3 \dr_R^\eps \otimes ( \B_R^\eps \dr_R^\eps ) + \mu_5 ( \A_R^\eps \dr_R^\eps ) \otimes \dr_R^\eps + \mu_6 \dr_R^\eps \otimes ( \A_R^\eps \dr_R^\eps ) ] , \nabla \partial^m \u_R^\eps \r \\
	    \leq C \eps^2 \mathscr{E}_{N,\eps} (t) \mathscr{D}_{N,\eps} (t) \,.
	  \end{aligned}
	\end{equation}
	Recalling the definition of $\mathcal{M}_2$ in \eqref{M2}, one deduces from collecting the estimates \eqref{IN2-20}, \eqref{IN2-21}, \eqref{IN2-22}, \eqref{IN2-23} and \eqref{IN2-24} that
	\begin{equation}\label{IN2-25}
	  \begin{aligned}
	    - \tfrac{1}{\sqrt{\eps}} \l \sqrt{\eps} \partial^m \mathcal{M}_2 , \nabla \partial^m \u_R^\eps \r \leq C \eps^2 ( 1 + \| \u_0 \|^2_{H^{N+2}} + \| \nabla \dr_0 \|^2_{H^{N+2}} ) \mathscr{E}_{N,\eps} (t) \mathscr{D}_{N,\eps} (t)
	  \end{aligned}
	\end{equation}
	for all multi-indexes $m \in \mathbb{N}^3$ with $|m| \leq N$.
	
	We next estimate the term $ - \tfrac{1}{\sqrt{\eps}} \l \sqrt{\eps}^2 \partial^m \mathcal{M}_3 , \nabla \partial^m \u_R^\eps \r $ for all $|m| \leq N$. First, we compute that
	\begin{equation}\label{IN2-26}
	  \begin{aligned}
	    & - \tfrac{\mu_1}{\sqrt{\eps}} \l \sqrt{\eps}^2 \partial^m \big[ ( 2 \A_R^\eps : \dr_0 \otimes \dr_R^\eps ) \dr_R^\eps \otimes \dr_R^\eps \big] , \nabla \partial^m \u_R^\eps \r \\
	    \leq & C \sqrt{\eps} \| \partial^m \A_R^\eps \|_{L^2} \| ( \dr_0 \otimes \dr_R^\eps ) ( \dr_R^\eps \otimes \dr_R^\eps ) \|_{L^\infty} \| \nabla \partial^m \u_R^\eps \|_{L^2} \\
	    & + C \sqrt{\eps} \sum_{0 \neq m' \leq m} \| \partial^{m-m'} \A_R^\eps \|_{L^4} \| \partial^{m'} ( ( \dr_0 \otimes \dr_R^\eps ) ( \dr_R^\eps \otimes \dr_R^\eps ) ) \|_{L^4} \| \nabla \partial^m \u_R^\eps \|_{L^2} \\
	    \leq & C \sqrt{\eps} \| \nabla \u_R^\eps \|^2_{H^N} \| \dr_R^\eps \|^3_{H^N} + C \sqrt{\eps} \| \nabla \u_R^\eps \|^2_{H^N} ( \| \dr_R^\eps \|^3_{H^N} + \| \nabla \dr_R^\eps \|^3_{H^N} ) ( 1 + \| \nabla \dr_0 \|_{H^{N+1}} )  \\
	    \leq & C \eps^3 ( 1 + \| \nabla \dr_0 \|_{H^{N+1}} ) \mathscr{E}^\frac{3}{2}_{N, \eps} (t) \mathscr{D}_{N, \eps} (t) \,,
	  \end{aligned}
	\end{equation}
	Similarly as in \eqref{IN2-26}, we have
	\begin{equation}\label{IN2-27}
	  \begin{aligned}
	    - \tfrac{\mu_1}{\sqrt{\eps}} \l \sqrt{\eps}^2 \partial^m \big[ ( \A_R^\eps : \dr_R^\eps \otimes \dr_R^\eps ) ( \dr_R^\eps \otimes \dr_0 + \dr_0 \otimes \dr_R^\eps ) \big] , \nabla \partial^m \u_R^\eps \r \\
	    \leq C \eps^3 ( 1 + \| \nabla \dr_0 \|_{H^{N+1}} ) \mathscr{E}^\frac{3}{2}_{N, \eps} (t) \mathscr{D}_{N, \eps} (t) \,.
	  \end{aligned}
	\end{equation}
	For all $|m| \leq N$, we deduce that
	  \begin{align}\label{IN2-28}
	    \no & - \tfrac{\mu_1}{\sqrt{\eps}} \l \sqrt{\eps}^2 \partial^m \big[ ( \A_0 : \dr_R^\eps \otimes \dr_R^\eps ) \dr_R^\eps \otimes \dr_R^\eps \big] , \nabla \partial^m \u_R^\eps \r \\
	    \no \leq & C \sqrt{\eps} \| \partial^m \A_0 \|_{L^3} \| \dr_R^\eps \|_{L^6} \| \dr_R^\eps \otimes \dr_R^\eps \|_{L^\infty} \| \nabla \partial^m \u_R^\eps \|_{L^2} \\
	    \no & + C \sqrt{\eps} \sum_{0 \neq m' \leq m} \| \partial^{m-m'} \A_0 \|_{L^4} \| \partial^{m'} ( (\dr_R^\eps \otimes \dr_R^\eps) ( \dr_R^\eps \otimes \dr_R^\eps ) ) \|_{L^4} \| \nabla \partial^m \u_R^\eps \|_{L^2} \\
	    \no \leq & C \sqrt{\eps} \| \u_0 \|_{H^{N+2}} \| \nabla \u_R^\eps \|_{H^N} \| \nabla \dr_R^\eps \|_{H^N} \| \dr_R^\eps \|^3_{H^N} \\
	    \leq & C \eps^3 \| \u_0 \|_{H^{N+2}} \mathscr{E}^\frac{3}{2}_{N, \eps} (t) \mathscr{D}_{N, \eps} (t) \,.
	  \end{align}
	Noticing the definition of $\mathcal{M}_3$ in \eqref{M3}, we collect the estimates \eqref{IN2-26}, \eqref{IN2-27} and \eqref{IN2-28}, and then obtain
	\begin{equation}\label{IN2-29}
	  \begin{aligned}
	    - \tfrac{1}{\sqrt{\eps}} \l \sqrt{\eps}^2 \partial^m \mathcal{M}_3 , \nabla \partial^m \u_R^\eps \r \leq C \eps^3 ( 1 + \| \u_0 \|_{H^{N+2}} + \| \nabla \dr_0 \|_{H^{N+2}} ) \mathscr{E}^\frac{3}{2}_{N, \eps} (t) \mathscr{D}_{N, \eps} (t) \,.
	  \end{aligned}
	\end{equation}
	
	Recalling the definition of $\mathcal{M}_4$ in \eqref{M4}, we calculate that
	\begin{equation}\label{IN2-30}
	  \begin{aligned}
	    & - \tfrac{1}{\sqrt{\eps}} \l \sqrt{\eps}^3 \partial^m \mathcal{M}_4 , \nabla \partial^m \u_R^\eps \r = - \mu_1 \eps \| (\partial^m \A_R^\eps ) : \dr_R^\eps \otimes \dr_R^\eps \|_{L^2} \\
	    & - \mu_1 \eps \sum_{\substack{ 0 \neq m' \leq m \\ m'' \leq m' }} \l \partial^{m - m'} \A_R^\eps : \partial^{m''} ( \dr_R^\eps \otimes \dr_R^\eps ) , \nabla \partial^m \u_R^\eps : \partial^{m' - m''} ( \dr_R^\eps \otimes \dr_R^\eps ) \r \\
	    \leq & C \eps \sum_{\substack{ 0 \neq m' \leq m \\ m'' \leq m' }} \| \partial^{m - m'} \A_R^\eps \|_{L^4} \| \partial^{m''} ( \dr_R^\eps \otimes \dr_R^\eps ) \partial^{m' - m''} ( \dr_R^\eps \otimes \dr_R^\eps ) \|_{L^4} \| \nabla \partial^m \u_R^\eps \|_{L^2} \\
	    \leq & C \eps \| \nabla \u_R^\eps \|^2_{H^N} ( \| \dr_R^\eps \|^4_{H^N} + \| \nabla \dr_R^\eps \|^4_{H^N} ) \\
	    \leq & C \eps^4 \mathscr{E}^2_{N, \eps} (t) \mathscr{D}_{N, \eps} (t) \,.
	  \end{aligned}
	\end{equation}
	Therefore, from the definition of $\mathcal{R}_{\u}$ in \eqref{Ru} and the bounds \eqref{IN2-19}, \eqref{IN2-25}, \eqref{IN2-29} and \eqref{IN2-30}, we derive that
	\begin{equation}\label{IN2-31}
	  \begin{aligned}
	    - \tfrac{1}{\sqrt{\eps}} \l \partial^m \mathcal{R}_{\u} , \nabla \partial^m \u_R^\eps \r = & - \tfrac{1}{\sqrt{\eps}} \l \partial^m ( \mathcal{M}_1 + \sqrt{\eps} \mathcal{M}_2 + \sqrt{\eps}^2 \mathcal{M}_3 + \sqrt{\eps}^3 \mathcal{M}_4 ) , \nabla \partial^m \u_R^\eps \r \\
	    \leq & C \eps ( 1 + \| \u_0 \|^3_{H^{N+2}} + \| \nabla \dr_0 \|^3_{H^{N+2}} ) \mathscr{E}^\frac{1}{2}_{N, \eps} (t) \mathscr{D}_{N, \eps} (t) \\
	    & + C \eps^2 ( 1 + \| \u_0 \|^2_{H^{N+2}} + \| \nabla \dr_0 \|^2_{H^{N+2}} ) \mathscr{E}_{N, \eps} (t) \mathscr{D}_{N, \eps} (t) \\
	    & + C \eps^3 ( 1 + \| \u_0 \|_{H^{N+2}} + \| \nabla \dr_0 \|_{H^{N+2}} ) \mathscr{E}^\frac{3}{2}_{N, \eps} (t) \mathscr{D}_{N, \eps} (t) \\
	    & + C \eps^4 \mathscr{E}^2_{N, \eps} (t) \mathscr{D}_{N, \eps} (t) \\
	    \leq & C \eps ( 1 + \| \u_0 \|^3_{H^{N+2}} + \| \nabla \dr_0 \|^3_{H^{N+2}} ) \sum_{p=1}^4 \mathscr{E}^\frac{p}{2}_{N, \eps} (t) \mathscr{D}_{N, \eps} (t) \\
	    \leq & C \eps ( 1 + \| \u_0 \|^3_{H^{N+2}} + \| \nabla \dr_0 \|^3_{H^{N+2}} ) \big( \mathscr{E}^\frac{1}{2}_{N, \eps} (t) + \mathscr{E}^2_{N, \eps} (t) \big) \mathscr{D}_{N, \eps} (t)
	  \end{aligned}
	\end{equation}
	for $0 < \eps \leq 1$. Consequently, via the estimates \eqref{IN2-9} and \eqref{IN2-31}, we obtain
	\begin{equation}\label{IN2-Bnd}
	  \begin{aligned}
	    \mathcal{I}_N^{(2)} = & - \tfrac{1}{\eps} \l \partial^m \mathcal{T}_{\u} , \nabla \partial^m \u_R^\eps \r - \tfrac{1}{\sqrt{\eps}} \l \partial^m \mathcal{R}_{\u} , \nabla \partial^m \u_R^\eps \r \\
	    \leq & C ( \| \u_0 \|_{H^{N+2}} + \| \nabla \dr_0 \|_{H^{N+2}} ) ( 1 + \| \nabla \dr_0 \|^3_{H^{N+2}} ) \mathscr{D}_{N, \eps} (t) \\
	    & + C \eps ( 1 + \| \u_0 \|^3_{H^{N+2}} + \| \nabla \dr_0 \|^3_{H^{N+2}} ) \big( \mathscr{E}^\frac{1}{2}_{N, \eps} (t) + \mathscr{E}^2_{N, \eps} (t) \big) \mathscr{D}_{N, \eps} (t)
	  \end{aligned}
	\end{equation}
	for all $|m| \leq N$ $(N \geq 2)$ and for all $0 < \eps \leq \eps_0 \leq 1$. \\
	
	{\em Step 3. Control the quantity $\mathcal{I}_N^{(3)}$.}
	
	Now we carefully estimate the term $\mathcal{I}_N^{(3)}$ for $N \geq 2$. First, we have
	  \begin{align}\label{IN3-1}
	    \no & \tfrac{1}{\eps} \sum_{|m| \leq N} \sum_{m' < m} C_m^{m'} \l \lambda_1 ( \partial^{m'} \B_R^\eps ) \partial^{m-m'} \dr_0 , \partial^m \D_{\u_0 + \sqrt{\eps} \u_R^\eps} \dr_R^\eps \r \\
	    \no \leq & \tfrac{C}{\eps} \sum_{|m| \leq N} \sum_{m' < m} \| \partial^{m'} \B_R^\eps \|_{L^4} \| \partial^{m-m'} \dr_0 \|_{L^4} \| \partial^m \D_{\u_0 + \sqrt{\eps} \u_R^\eps} \dr_R^\eps \|_{L^2} \\
	    \no \leq & \tfrac{C}{\eps} \| \nabla \u_R^\eps \|_{H^N} \| \nabla \dr_0 \|_{H^{N+1}} \| \partial^m \D_{\u_0 + \sqrt{\eps} \u_R^\eps} \dr_R^\eps \|_{L^2} \\
	    \no \leq &  \tfrac{C}{\eps} \| \nabla \dr_0 \|_{H^{N+1}} \| \nabla \u_R^\eps \|_{H^N} \| \partial^m \D_{\u_0 + \sqrt{\eps} \u_R^\eps} \dr_R^\eps + ( \partial^m \B_R^\eps ) \dr_0 + \tfrac{\lambda_2}{\lambda_1} ( \partial^m \A_R^\eps ) \dr_0 \|_{L^2} \\
	    \no & + \tfrac{C}{\eps} \| \nabla \dr_0 \|_{H^{N+1}} \| \nabla \u_R^\eps \|_{H^N} \big( \| ( \partial^m \B_R^\eps ) \dr_0 \|_{L^2} + \tfrac{|\lambda_2|}{|\lambda_1|} ( \partial^m \A_R^\eps ) \dr_0 \|_{L^2}  \big) \\
	    \no \leq & \tfrac{C}{\eps} \| \nabla \dr_0 \|_{H^{N+1}} \| \nabla \u_R^\eps \|_{H^N} \| \partial^m \D_{\u_0 + \sqrt{\eps} \u_R^\eps} \dr_R^\eps + ( \partial^m \B_R^\eps ) \dr_0 + \tfrac{\lambda_2}{\lambda_1} ( \partial^m \A_R^\eps ) \dr_0 \|_{L^2} \\
	    \no & + \tfrac{C}{\eps} \| \nabla \dr_0 \|_{H^{N+1}} \| \nabla \u_R^\eps \|^2_{H^N} \\
	    \leq & C \| \nabla \dr_0 \|_{H^{N+1}} \mathscr{D}_{N, \eps} (t) \,.
	  \end{align}
	Similarly as in \eqref{IN3-1}, we imply
	\begin{equation}\label{IN3-2}
	  \begin{aligned}
	    \tfrac{1}{\eps} \sum_{|m| \leq N} \sum_{m' < m} C_m^{m'} \l \lambda_2 ( \partial^{m'} \A_R^\eps ) \partial^{m-m'} \dr_0 , \partial^m \D_{\u_0 + \sqrt{\eps} \u_R^\eps} \dr_R^\eps \r \leq C \| \nabla \dr_0 \|_{H^{N+1}} \mathscr{D}_{N, \eps} (t) \,,
	  \end{aligned}
	\end{equation}
	and
	\begin{equation}\label{IN3-3}
	  \begin{aligned}
	    - \tfrac{1}{\eps} \sum_{|m| \leq N} \sum_{m' < m} \l \mu_2 \partial^{m-m'} \dr_{0,j} \partial^{m'} ( \D_{\u_0 + \sqrt{\eps} \u_R^\eps} \dr_R^\eps )_i + \mu_3 \partial^{m-m'} \dr_{0,i} \partial^{m'} ( \D_{\u_0 + \sqrt{\eps} \u_R^\eps} \dr_R^\eps )_j , \partial^m \partial_j \u_{R, i}^\eps \r \\
	    \leq C \| \nabla \dr_0 \|_{H^{N+1}} \mathscr{D}_{N, \eps} (t) \,.
	  \end{aligned}
	\end{equation}
	
	It is easy to estimate that
	\begin{equation}\label{IN3-4}
	  \begin{aligned}
	    & - \tfrac{1}{\eps} \sum_{|m| \leq N} \sum_{m' < m} C_m^{m'} \l \mu_2 \partial^{m-m'} ( \dr_{0,j} \dr_{0,k} ) \partial^{m'} \B_{R, ki}^\eps , \partial^m \partial_j \u_{R, i}^\eps \r \\
	    \leq & \tfrac{C}{\eps} \sum_{|m| \leq N} \sum_{m' < m} \| \partial^{m-m'} ( \dr_{0,j} \dr_{0,k} ) \|_{L^4} \| \partial^{m'} \B_{R, ki}^\eps \|_{L^4} \| \partial^m \partial_j \u_{R, i}^\eps \|_{L^2} \\
	    \leq & \tfrac{C}{\eps} \| \nabla \dr_0 \|_{H^{N+1}} ( 1 + \| \nabla \dr_0 \|_{H^{N+1}} ) \| \nabla \u_R^\eps \|^2_{H^N} \\
	    \leq & C \| \nabla \dr_0 \|_{H^{N+1}} ( 1 + \| \nabla \dr_0 \|_{H^{N+1}} ) \mathscr{D}_{N, \eps} (t) \,.
	  \end{aligned}
	\end{equation}
	Furthermore, via the analogous calculations in \eqref{IN3-4}, we imply
	\begin{equation}\label{IN3-5}
	  \begin{aligned}
	    & - \tfrac{1}{\eps} \sum_{|m| \leq N} \sum_{m' < m} C_m^{m'} \l \mu_3 \partial^{m-m'} ( \dr_{0,i} \dr_{0,k} ) \partial^{m'} \B_{R, kj}^\eps , \partial^m \partial_j \u_{R,i}^\eps \r \\
	    & - \tfrac{1}{\eps} \sum_{|m| \leq N} \sum_{m' < m} C_m^{m'} \l \mu_5 \partial^{m-m'} ( \dr_{0,j} \dr_{0,k} ) \partial^{m'} \A_{R, ki}^\eps , \partial^m \partial_j \u_{R,i}^\eps \r \\
	    & - \tfrac{1}{\eps} \sum_{|m| \leq N} \sum_{m' < m} C_m^{m'} \l \mu_6 \partial^{m-m'} ( \dr_{0,i} \dr_{0,k} ) \partial^{m'} \A_{R, kj}^\eps , \partial^m \partial_j \u_{R,i}^\eps \r \\
	    & \leq C \| \nabla \dr_0 \|_{H^{N+1}} ( 1 + \| \nabla \dr_0 \|_{H^{N+1}} ) \mathscr{D}_{N, \eps} (t) \,.
	  \end{aligned}
	\end{equation}
	Recalling the definition of $\mathcal{G}_m$ in \eqref{Gm} and noticing that $\mathcal{I}_N^{(3)} = \tfrac{1}{\eps} \sum_{|m| \leq N} \mathcal{G}_m$, we derive from the inequalities \eqref{IN3-1}, \eqref{IN3-2}, \eqref{IN3-3}, \eqref{IN3-4} and \eqref{IN3-5} that
	\begin{equation}\label{IN3-Bnd}
	  \begin{aligned}
	    \mathcal{I}_N^{(3)} \leq C \| \nabla \dr_0 \|_{H^{N+1}} ( 1 + \| \nabla \dr_0 \|_{H^{N+1}} ) \mathscr{D}_{N, \eps} (t)
	  \end{aligned}
	\end{equation}
	holds for all $N \geq 2$ and $0 < \eps \leq \eps_0$. \\
	
	{\em Step 4. Control the term $\mathcal{I}_N^{(4)}$.}
	
	We first rewrite the expression of $\mathcal{I}_N^{(4)}$ in \eqref{IN4} as
	  \begin{align}\label{IN4-1}
	    \no \mathcal{I}_N^{(4)} = & \underset{\mathcal{I}_N^{(4)} (\mathcal{C}_{\dr}) }{ \underbrace{  - \tfrac{1}{\eps} \sum_{|m| \leq N} \l \partial^m ( \u_R^\eps \cdot \nabla \dr_0 ) , \partial^m \mathcal{C}_{\dr} \r }} \\
	    \no + & \underset{\mathcal{I}_N^{(4)} (\mathcal{S}_{\dr}^1) }{ \underbrace{  \tfrac{1}{\eps} \sum_{|m| \leq N} \l \partial^m \mathcal{S}_{\dr}^1 , \partial^m \D_{\u_0 + \sqrt{\eps} \u_R^\eps} \dr_R^\eps - \partial^m ( \u_R^\eps \cdot \nabla \dr_0 ) + \delta \partial^m \dr_R^\eps \r }} \\
	    \no + & \underset{\mathcal{I}_N^{(4)} (\mathcal{S}_{\dr}^2) }{ \underbrace{  \tfrac{1}{\sqrt{\eps}} \sum_{|m| \leq N} \l \partial^m \mathcal{S}_{\dr}^2 , \partial^m \D_{\u_0 + \sqrt{\eps} \u_R^\eps} \dr_R^\eps - \partial^m ( \u_R^\eps \cdot \nabla \dr_0 ) + \delta \partial^m \dr_R^\eps \r }} \\
	    \no + & \underset{\mathcal{I}_N^{(4)} (\mathcal{S}_{\dr}^2) }{ \underbrace{  \sum_{|m| \leq N} \l \partial^m \mathcal{R}_{\dr} , \partial^m \D_{\u_0 + \sqrt{\eps} \u_R^\eps} \dr_R^\eps - \partial^m ( \u_R^\eps \cdot \nabla \dr_0 ) + \delta \partial^m \dr_R^\eps \r }} \\
	    + & \underset{\mathcal{I}_N^{(4)} (\A \B) }{ \underbrace{  \tfrac{\delta}{\eps} \sum_{1 \leq |m| \leq N} \sum_{0 \neq m' \leq m} C_m^{m'} \l \lambda_1 ( \partial^{m-m'} \B_R^\eps ) \partial^{m'} \dr_0 + \lambda_2 ( \partial^{m-m'} \A_R^\eps ) \partial^{m'} \dr_0 , \partial^m \dr_R^\eps \r }} \,.
	  \end{align}
	We next estimate the previous terms one by one. Before doing this, we derive the following inequality, which will be frequently used. More precisely, for all $|m| \leq N$,
	\begin{equation}\label{uR-nabla-d0}
	\begin{aligned}
	\| \partial^m ( \u_R^\eps \cdot \nabla \dr_0  ) \|_{L^2} \leq & C \| \u_R^\eps \cdot \nabla \partial^m \dr_0 \|_{L^2} + C \sum_{0 \neq m' \leq m} \| \partial^{m'} \u_R^\eps \cdot \nabla \partial^{m-m'} \dr_0 \|_{L^4} \\
	\leq & C \| \u_R^\eps \|_{L^6} \| \nabla \partial^m \dr_0 \|_{L^3} + C \sum_{0 \neq m' \leq m } \| \partial^{m'} \u_R^\eps \|_{L^4} \| \nabla \partial^{m-m'} \dr_0 \|_{L^4} \\
	\leq & C \| \nabla \u_R^\eps \|_{H^N} \| \nabla \dr_0 \|_{H^{N+1}} \,.
	\end{aligned}
	\end{equation}
	We initially estimate the quantity $ \mathcal{I}_N^{(4)} (\mathcal{C}_{\dr}) $. It is easy to deduce that
	\begin{equation}\label{IN4-2}
	  \begin{aligned}
	    & - \tfrac{1}{\eps} \l \lambda_1 \partial^m ( \B_R^\eps \dr_0 ) , \partial^m ( \u_R^\eps \cdot \nabla \dr_0 ) \r \\
	    \leq & \tfrac{C}{\eps} \Big( \| \partial^m \B_R^\eps \|_{L^2} + \sum_{0 \neq m' \leq m} \| \partial^{m-m'} \B_R^\eps \|_{L^4} \| \partial^{m'} \dr_0 \|_{L^4} \Big) \| \partial^m ( \u_R^\eps \cdot \nabla \dr_0 ) \|_{L^2} \\
	    \leq & \tfrac{C}{\eps} \| \nabla \u_R^\eps \|^2_{H^N} \| \nabla \dr_0 \|_{H^{N+1}} ( 1 + \| \nabla \dr_0 \|_{H^{N+1}} ) \\
	    \leq & C \| \nabla \dr_0 \|_{H^{N+1}} ( 1 + \| \nabla \dr_0 \|_{H^{N+1}} ) \mathscr{D}_{N, \eps} (t) \,,
	  \end{aligned}
	\end{equation}
	where we also make use of the inequality \eqref{uR-nabla-d0}. Similarly as in \eqref{IN4-2}, we yield that
	\begin{equation}\label{IN4-3}
	  \begin{aligned}
	    - \tfrac{1}{\eps} \l \lambda_2 \partial^m ( \A_R^\eps \dr_0 ) , \partial^m ( \u_R^\eps \cdot \nabla \dr_0 ) \r \leq C \| \nabla \dr_0 \|_{H^{N+1}} ( 1 + \| \nabla \dr_0 \|_{H^{N+1}} ) \mathscr{D}_{N, \eps} (t) \,.
	  \end{aligned}
	\end{equation}
	Based on the definition of $\mathcal{C}_{\dr}$ in \eqref{C-d}, the inequalities \eqref{IN4-2} and \eqref{IN4-3} reduces to
	\begin{equation}\label{IN4-Cd}
	  \begin{aligned}
	    \mathcal{I}_N^{(4)} (\mathcal{C}_{\dr}) \leq C \| \nabla \dr_0 \|_{H^{N+1}} ( 1 + \| \nabla \dr_0 \|_{H^{N+1}} ) \mathscr{D}_{N, \eps} (t) \,.
	  \end{aligned}
	\end{equation}
	For the term $ \mathcal{I}_N^{(4)} ( \A \B ) $, we estimate that
	\begin{equation}\label{IN4-AB}
	  \begin{aligned}
	    \mathcal{I}_N^{(4)} ( \A \B ) \leq & \tfrac{C}{\eps} \sum_{1 \leq |m| \leq N} \sum_{0 \neq m' \leq m} ( \| \partial^{m-m'} \B_R^\eps \|_{L^4} + \| \partial^{m-m'} \A_R^\eps \|_{L^4} ) \| \partial^{m'} \dr_0 \|_{L^4} \| \partial^m \dr_R^\eps \|_{L^2} \\
	    \leq & \tfrac{C}{\eps} \| \nabla \u_R^\eps \|_{H^N} \| \nabla \dr_R^\eps \|_{H^N} \| \nabla \dr_0 \|_{H^{N+1}} \leq C \| \nabla \dr_0 \|_{H^{N+1}} \mathscr{D}_{N, \eps} (t) \,.
	  \end{aligned}
	\end{equation}
	
	We now estimate the quantity $ \mathcal{I}_N^{(4)} ( \mathcal{S}_{\dr}^1 ) $. First, we have
	  \begin{align}\label{IN4-Sd1-1}
	    \no & \tfrac{1}{\eps} \l 2 \partial^m [ ( \nabla \dr_0 \cdot \nabla \dr_R^\eps ) \dr_0 ] , \partial^m \D_{\u_0 + \sqrt{\eps} \u_R^\eps} \dr_R^\eps \r \\
	    \no \leq & \tfrac{2}{\eps} \| \partial^m \nabla \dr_R^\eps \|_{L^2} \| \dr_0 \otimes \nabla \dr_0 \|_{L^\infty} \| \partial^m \D_{\u_0 + \sqrt{\eps} \u_R^\eps} \dr_R^\eps \|_{L^2} \\
	    \no & + \tfrac{C}{\eps} \sum_{0 \neq m' \leq m} \| \nabla \partial^{m-m'} \dr_R^\eps \|_{L^4} \| \partial^{m'} ( \dr_0 \otimes \dr_0 ) \|_{L^4} \| \partial^m \D_{\u_0 + \sqrt{\eps} \u_R^\eps} \dr_R^\eps \|_{L^2} \\
	    \no \leq & \tfrac{C}{\eps} \| \nabla \dr_0 \|_{H^N} \| \nabla \dr_R^\eps \|_{H^N} \| \partial^m \D_{\u_0 + \sqrt{\eps} \u_R^\eps} \dr_R^\eps \|_{L^2} \\
	    \no & + \tfrac{C}{\eps} \| \nabla \dr_0 \|_{H^N} ( 1 + \| \nabla \dr_0 \|_{H^N} ) \| \nabla \dr_R^\eps \|_{H^N} \| \partial^m \D_{\u_0 + \sqrt{\eps} \u_R^\eps} \dr_R^\eps \|_{L^2} \\
	    \no \leq & \tfrac{C}{\eps} \| \nabla \dr_0 \|_{H^N} ( 1 + \| \nabla \dr_0 \|_{H^N} ) \| \nabla \dr_R^\eps \|_{H^N} \| \partial^m \D_{\u_0 + \sqrt{\eps} \u_R^\eps} \dr_R^\eps + ( \partial^m \B_R^\eps ) \dr_0 + \tfrac{\lambda_2}{\lambda_1} (\partial^m \A_R^\eps ) \dr_0 \|_{L^2} \\
	    \no & + \tfrac{C}{\eps} \| \nabla \dr_0 \|_{H^N} ( 1 + \| \nabla \dr_0 \|_{H^N} ) \| \nabla \dr_R^\eps \|_{H^N} \| \nabla \u_R^\eps \|_{H^N} \\
	    \leq & C \| \nabla \dr_0 \|_{H^N} ( 1 + \| \nabla \dr_0 \|_{H^N} ) \mathscr{D}_{N,\eps} (t) \,.
	  \end{align}
	For the term $ \tfrac{1}{\eps} \l \partial^m ( |\nabla \dr_0|^2 \dr_R^\eps ) , \partial^m \D_{\u_0 + \sqrt{\eps} \u_R^\eps} \dr_R^\eps \r $, one can easily estimate that
	\begin{equation}\label{IN4-Sd1-2}
	  \begin{aligned}
	    & \tfrac{1}{\eps} \l \partial^m ( |\nabla \dr_0|^2 \dr_R^\eps ) , \partial^m \D_{\u_0 + \sqrt{\eps} \u_R^\eps} \dr_R^\eps \r \\
	    \leq & \tfrac{1}{\eps} \| \partial^m |\nabla \dr_0|^2 \|_{L^3} \| \dr_R^\eps \|_{L^6} \| \partial^m \D_{\u_0 + \sqrt{\eps} \u_R^\eps} \dr_R^\eps \|_{L^2} \\
	    & + \tfrac{C}{\eps} \sum_{0 \neq m' \leq m} \| \partial^{m-m'} |\nabla \dr_0|^2 \|_{L^4} \| \partial^{m'} \dr_R^\eps \|_{L^4} \| \partial^m \D_{\u_0 + \sqrt{\eps} \u_R^\eps} \dr_R^\eps \|_{L^2} \\
	    \leq & \tfrac{C}{\eps} \| \nabla \dr_0 \|^2_{H^{N+1}} \| \nabla \dr_R^\eps \|_{H^N} \| \partial^m \D_{\u_0 + \sqrt{\eps} \u_R^\eps} \dr_R^\eps \|_{L^2} \\
	    \leq & \tfrac{C}{\eps} \| \nabla \dr_0 \|^2_{H^{N+1}} \| \nabla \dr_R^\eps \|_{H^N} \| \partial^m \D_{\u_0 + \sqrt{\eps} \u_R^\eps} \dr_R^\eps + ( \partial^m \B_R^\eps ) \dr_0 + \tfrac{\lambda_2}{\lambda_1} (\partial^m \A_R^\eps ) \dr_0 \|_{L^2} \\
	    & + \tfrac{C}{\eps} \| \nabla \dr_0 \|^2_{H^{N+1}} \| \nabla \dr_R^\eps \|_{H^N} \| \nabla \u_R^\eps \|_{H^N} \\
	    \leq & C \| \nabla \dr_0 \|^2_{H^{N+1}} \mathscr{D}_{N, \eps} (t) \,.
	  \end{aligned}
	\end{equation}
	We then derive the following bound
	\begin{equation}\label{IN4-Sd1-3}
	  \begin{aligned}
	    & \tfrac{\lambda_1}{\eps} \l \partial^m ( \u_0 \cdot \nabla \dr_R^\eps + \u_R^\eps \cdot \nabla \dr_0 + \B_0 \dr_R^\eps ) , \partial^m \D_{\u_0 + \sqrt{\eps} \u_R^\eps} \dr_R^\eps \r \\
	    \leq & \tfrac{C}{\eps} \sum_{m' \leq m} \| \partial^{m'} \u_0 \|_{L^\infty} \| \nabla \partial^{m-m'} \dr_R^\eps \|_{L^2} \| \partial^m \D_{\u_0 + \sqrt{\eps} \u_R^\eps} \dr_R^\eps \|_{L^2} \\
	    & + \tfrac{C}{\eps} \Big( \| \u_R^\eps \|_{L^6} \| \nabla \dr_0 \|_{L^3} + \sum_{0 \neq m' \leq m} \| \partial^{m'} \u_R^\eps \|_{L^4} \| \nabla \partial^{m-m'} \dr_0 \|_{L^4} \Big) \| \partial^m \D_{\u_0 + \sqrt{\eps} \u_R^\eps} \dr_R^\eps \|_{L^2} \\
	    & + \tfrac{C}{\eps} \Big( \| \dr_R^\eps \|_{L^6} \| \B_0 \|_{L^3} + \sum_{0 \neq m' \leq m} \| \partial^{m'} \dr_R^\eps \|_{L^4} \| \partial^{m-m'} \B_0 \|_{L^4} \Big) \| \partial^m \D_{\u_0 + \sqrt{\eps} \u_R^\eps} \dr_R^\eps \|_{L^2} \\
	    \leq & \tfrac{C}{\eps} ( \| \u_0 \|_{H^{N+1}} + \| \nabla \dr_0 \|_{H^{N+1}} ) ( \| \nabla \u_R^\eps \|_{H^N} + \| \nabla \dr_R^\eps \|_{H^N} ) \| \partial^m \D_{\u_0 + \sqrt{\eps} \u_R^\eps} \dr_R^\eps \|_{L^2} \\
	    \leq & \tfrac{C}{\eps} ( \| \u_0 \|_{H^{N+1}} + \| \nabla \dr_0 \|_{H^{N+1}} ) ( \| \nabla \u_R^\eps \|_{H^N} + \| \nabla \dr_R^\eps \|_{H^N} ) \\
	    & \qquad \times \big( \| \partial^m \D_{\u_0 + \sqrt{\eps} \u_R^\eps} \dr_R^\eps + ( \partial^m \B_R^\eps ) \dr_0 + \tfrac{\lambda_2}{\lambda_1} ( \partial^m \A_R^\eps ) \dr_0 \|_{L^2} + \| \nabla \u_R^\eps \|_{H^N} \big) \\
	    \leq & C ( \| \u_0 \|_{H^{N+1}} + \| \nabla \dr_0 \|_{H^{N+1}} ) \mathscr{D}_{N, \eps} (t) \,.
	  \end{aligned}
	\end{equation}
	Via the analogous arguments in \eqref{IN4-Sd1-3}, we imply that
	\begin{equation}\label{IN4-Sd1-4}
	  \begin{aligned}
	    \tfrac{\lambda_2}{\eps} \l \partial^m ( \A_0 \dr_R^\eps ) , \partial^m \D_{\u_0 + \sqrt{\eps} \u_R^\eps} \dr_R^\eps \r \leq C \| \u_0 \|_{H^{N+1}} \mathscr{D}_{N, \eps} (t) \,.
	  \end{aligned}
	\end{equation}
	It is easily estimated that
	\begin{equation}\label{IN4-Sd1-5}
	  \begin{aligned}
	    & - \tfrac{\lambda_2}{\eps} \l \partial^m [ ( \A_0 : \dr_0 \otimes \dr_0 ) \dr_R^\eps ], \partial^m \D_{\u_0 + \sqrt{\eps} \u_R^\eps} \dr_R^\eps \r \\
	    \leq & \tfrac{C}{\eps} \| \partial^m ( \A_0 : \dr_0 \otimes \dr_0 ) \|_{L^3} \| \dr_R^\eps \|_{L^6} \| \partial^m \D_{\u_0 + \sqrt{\eps} \u_R^\eps} \dr_R^\eps \|_{L^2} \\
	    & + \tfrac{C}{\eps} \sum_{0 \neq m' \leq m} \| \partial^{m-m'} ( \A_0 : \dr_0 \otimes \dr_0 ) \|_{L^4} \| \partial^{m'} \dr_R^\eps \|_{L^4} \| \partial^m \D_{\u_0 + \sqrt{\eps} \u_R^\eps} \dr_R^\eps \|_{L^2} \\
	    \leq & \tfrac{C}{\eps} \| \u_0 \|_{H^{N+2}} ( 1 + \| \nabla \dr_0 \|^2_{H^{N+2}} ) \| \nabla \dr_R^\eps \|_{H^N} \| \partial^m \D_{\u_0 + \sqrt{\eps} \u_R^\eps} \dr_R^\eps \|_{L^2} \\
	    \leq & \tfrac{C}{\eps} \| \u_0 \|_{H^{N+2}} ( 1 + \| \nabla \dr_0 \|^2_{H^{N+2}} ) \| \nabla \dr_R^\eps \|_{H^N} \\
	    & \qquad \times \big( \| \partial^m \D_{\u_0 + \sqrt{\eps} \u_R^\eps} \dr_R^\eps + ( \partial^m \B_R^\eps ) \dr_0 + \tfrac{\lambda_2}{\lambda_1} ( \partial^m \A_R^\eps ) \dr_0 \|_{L^2} + \| \nabla \u_R^\eps \|_{H^N} \big) \\
	    \leq & C \| \u_0 \|_{H^{N+2}} ( 1 + \| \nabla \dr_0 \|^2_{H^{N+2}} ) \mathscr{D}_{N, \eps} (t) \,.
 	  \end{aligned}
	\end{equation}
	Similarly as in \eqref{IN4-Sd1-5}, we also have
	\begin{equation}\label{IN4-Sd1-6}
	  \begin{aligned}
	    - \tfrac{2 \lambda_2}{\eps} \l \partial^m [ ( \A_0 : \dr_0 \otimes \dr_R^\eps ) \dr_0 ] , \partial^m \D_{\u_0 + \sqrt{\eps} \u_R^\eps} \dr_R^\eps \r \leq C \| \u_0 \|_{H^{N+2}} ( 1 + \| \nabla \dr_0 \|^2_{H^{N+2}} ) \mathscr{D}_{N,\eps} (t) \,.
	  \end{aligned}
	\end{equation}
	
	Next, we estimate the quantity $ - \tfrac{\lambda_2}{\eps} \l \partial^m [ ( \A_R^\eps : \dr_0 \otimes \dr_0 ) \dr_0 ] , \partial^m \D_{\u_0 + \sqrt{\eps} \u_R^\eps } \dr_R^\eps \r $ for all $|m| \leq N$. Straightforward calculations imply that
	\begin{equation}\label{Const-1}
	  \begin{aligned}
	    & - \tfrac{\lambda_2}{\eps} \l \partial^m [ ( \A_R^\eps : \dr_0 \otimes \dr_0 ) \dr_0 ] , \partial^m \D_{\u_0 + \sqrt{\eps} \u_R^\eps } \dr_R^\eps \r \\
	    = & \tfrac{- \lambda_2}{\eps} \l \partial^m \A_R^\eps : \dr_0 \otimes \dr_0 , \partial_t ( \dr_0 \cdot \partial^m \dr_R^\eps ) - \partial_t \dr_0 \cdot \partial^m \dr_R^\eps \r \\
	    & + \tfrac{- \lambda_2}{\eps} \l ( \partial^m \A_R^\eps : \dr_0 \otimes \dr_0 ) \dr_0 , \partial^m [ ( \u_0 + \sqrt{\eps} \u_R^\eps ) \cdot \nabla \dr_R^\eps ] \r \\
	    & + \tfrac{- \lambda_2}{\eps} \sum_{0 \neq m' \leq m} C_m^{m'} \l \partial^{m-m'} \A_R^\eps : \partial^{m'} (\dr_0 \otimes \dr_0 \otimes \dr_0) , \partial^m \D_{\u_0 + \sqrt{\eps} \u_R^\eps } \dr_R^\eps \r \,.
	  \end{aligned}
	\end{equation}
	Recalling the constraint \eqref{Constraint-2}, hence $ \dr_0 \cdot \dr_R^\eps = - \tfrac{\sqrt{\eps}}{2} |\dr_R^\eps|^2 $, we easily derive that
	\begin{equation}\label{Const-3}
	  \begin{aligned}
	    \dr_0 \cdot \partial^m \dr_R^\eps = & \partial^m ( \dr_0 \cdot \dr_R^\eps ) - \sum_{0 \neq m' \leq m} C_m^{m'} \partial^{m'} \dr_0 \cdot \partial^{m-m'} \dr_R^\eps \\
	    = & - \tfrac{\sqrt{\eps}}{2} \partial^m |\dr_R^\eps|^2 - \sum_{0 \neq m' \leq m} C_m^{m'} \partial^{m'} \dr_0 \cdot \partial^{m-m'} \dr_R^\eps \,,
	  \end{aligned}
	\end{equation}
	which immediately reduces to
	\begin{equation}\label{Const-2}
	  \begin{aligned}
	    & \partial_t ( \dr_0 \cdot \partial^m \dr_R^\eps ) = - \sqrt{\eps} \partial^m ( \dr_R^\eps \cdot \partial_t \dr_R^\eps ) \\
	    & - \sum_{0 \neq m' \leq m} C_m^{m'} ( \partial^{m'} \partial_t \dr_0 \cdot \partial^{m-m'} \dr_R^\eps + \partial^{m'} \dr_0 \cdot \partial^{m-m'} \partial_t \dr_R^\eps ) \\
	    = & - \sqrt{\eps} \partial^m \big[ \dr_R^\eps \cdot \D_{\u_0 + \sqrt{\eps} \u_R^\eps} \dr_R^\eps - ( \u_0 + \sqrt{\eps} \u_R^\eps ) \cdot \nabla \dr_R^\eps \big] \\
	    & - \sum_{0 \neq m' \leq m} C_m^{m'} \big[ \partial^{m'} \partial_t \dr_0 \cdot \partial^{m-m'} \dr_R^\eps + \partial^{m'} \dr_0 \cdot \partial^{m-m'} \D_{\u_0 + \sqrt{\eps} \u_R^\eps } \dr_R^\eps \big] \\
	    & - \sum_{0 \neq m' \leq m} C_m^{m'} \big[ \partial^{m'} \dr_0 \cdot \partial^{m-m'} ( (\u_0 + \sqrt{\eps} \u_R^\eps) \cdot \nabla \dr_R^\eps ) \big] \,.
	  \end{aligned}
	\end{equation}
	Then, plugging the relation \eqref{Const-2} into the equality \eqref{Const-1} implies that
	  \begin{align}
	    \no & - \tfrac{\lambda_2}{\eps} \l \partial^m [ ( \A_R^\eps : \dr_0 \otimes \dr_0 ) \dr_0 ] , \partial^m \D_{\u_0 + \sqrt{\eps} \u_R^\eps } \dr_R^\eps \r \\
	    \no = & \underset{I_1}{ \underbrace{ \tfrac{- \lambda_2}{\eps} \l ( \partial^m \A_R^\eps : \dr_0 \otimes \dr_0  ) \dr_0 , \partial^m [ ( \u_0 + \sqrt{\eps} \u_R^\eps ) \cdot \nabla \dr_R^\eps ] \r } } \\
	    \no & + \underset{I_2}{ \underbrace{  \tfrac{- \lambda_2}{\eps} \sum_{0 \neq m' \leq m} C_m^{m'} \l \partial^{m-m'} \A_R^\eps : \partial^{m'} ( \dr_0 \otimes \dr_0 \otimes \dr_0 ) , \partial^m \D_{\u_0 + \sqrt{\eps} \u_R^\eps } \dr_R^\eps \r }} \\
	    \no & + \underset{I_3}{ \underbrace{ \tfrac{\lambda_2}{\eps} \l \partial^m \A_R^\eps : \dr_0 \otimes \dr_0 , \partial_t \dr_0 \cdot \partial^m \dr_R^\eps \r }}  \\
	    \no & + \underset{I_4}{ \underbrace{ \tfrac{\lambda_2}{\eps} \l \partial^m \A_R^\eps : \dr_0 \otimes \dr_0 , \sqrt{\eps} \partial^m ( \dr_R^\eps \cdot \D_{\u_0 + \sqrt{\eps} \u_R^\eps } \dr_R^\eps ) \r }}\\
	    \no & + \underset{I_5}{ \underbrace{ \tfrac{\lambda_2}{\eps} \sum_{0 \neq m' \leq m} C_m^{m'} \l \partial^m \A_R^\eps : \dr_0 \otimes \dr_0 , \partial^{m'} \partial_t \dr_0 \cdot \partial^{m-m'} \dr_R^\eps \r  }} \\
	    \no & + \underset{I_6}{ \underbrace{ \tfrac{\lambda_2}{\eps}  \sum_{0 \neq m' \leq m} C_m^{m'} \l \partial^m \A_R^\eps : \dr_0 \otimes \dr_0 , \partial^{m'} \dr_0 \cdot \partial^{m-m'} \D_{\u_0 + \sqrt{\eps} \u_R^\eps } \dr_R^\eps \r }} \\
	    & + \underset{I_7}{ \underbrace{ \tfrac{- \lambda_2}{\eps} \sum_{0 \neq m' \leq m} C_m^{m'} \l \partial^m \A_R^\eps : \dr_0 \otimes \dr_0 , \partial^{m'} \dr_0 \cdot \partial^{m-m'} [ (\u_0 + \sqrt{\eps} \u_R^\eps) \cdot \nabla \dr_R^\eps ] \r }} \,.
	  \end{align}
	For the term $I_1$, we estimate
	\begin{equation}
	  \begin{aligned}
	    I_1 \leq & \tfrac{C}{\eps} \| \partial^m \A_R^\eps \|_{L^2} \| \partial^m [ (\u_0 + \sqrt{\eps} \u_R^\eps) \cdot \nabla \dr_R^\eps ] \|_{L^2} \\
	    \leq &  \tfrac{C}{\eps} \| \nabla \u_R^\eps \|_{H^N} \| \nabla \dr_R^\eps \|_{H^N} ( \| \u_0 \|_{H^N} + \sqrt{\eps} \| \u_R^\eps \|_{H^N} ) \\
	    \leq & C ( \| \u_0 \|_{H^N} + \eps \mathscr{E}_{N,\eps}^\frac{1}{2} (t) ) \mathscr{D}_{N,\eps} (t) \,.
	  \end{aligned}
	\end{equation}
	For the term $I_2$, it is easily controlled that
	\begin{equation}
	  \begin{aligned}
	    I_2 \leq & \tfrac{C}{\eps} \| \A_R^\eps \|_{L^4} \| \partial^m ( \dr_0 \otimes \dr_0 \otimes \dr_0 ) \|_{L^4} \| \partial^m \D_{\u_0 + \sqrt{\eps} \u_R^\eps } \dr_R^\eps \|_{L^2} \\
	    & + \tfrac{C}{\eps} \sum_{0 \neq m' < m} \| \partial^{m-m'} \A_R^\eps \|_{L^2} \| \partial^{m'} ( \dr_0 \otimes \dr_0 \otimes \dr_0 ) \|_{L^\infty} \| \partial^m \D_{\u_0 + \sqrt{\eps} \u_R^\eps } \dr_R^\eps \|_{L^2}  \\
	    \leq & \tfrac{C}{\eps} \| \nabla \u_R^\eps \|_{H^N} \| \nabla \dr_0 \|_{H^N} ( 1 + \| \nabla \dr_0 \|^2_{H^N} ) \| \partial^m \D_{\u_0 + \sqrt{\eps} \u_R^\eps } \dr_R^\eps \|_{L^2} \\
	    \leq & \tfrac{C}{\eps} \| \nabla \u_R^\eps \|^2_{H^N} \| \nabla \dr_0 \|_{H^N} ( 1 + \| \nabla \dr_0 \|^2_{H^N} ) \\
	    & + \tfrac{C}{\eps} \| \nabla \u_R^\eps \|_{H^N} \| \nabla \dr_0 \|_{H^N} ( 1 + \| \nabla \dr_0 \|^2_{H^N} ) \| \partial^m \D_{\u_0 + \sqrt{\eps} \u_R^\eps } \dr_R^\eps + ( \partial^m \B_R^\eps ) \dr_0 + \tfrac{\lambda_2}{\lambda_1} ( \partial^m \A_R^\eps ) \dr_0 \|_{L^2} \\
	    \leq & C ( 1 + \| \nabla \dr_0 \|^2_{H^N} ) \| \nabla \dr_0 \|_{H^N} \mathscr{D}_{N,\eps} (t) \,.
	  \end{aligned}
	\end{equation}
	The quantity $I_3$ can be bounded by
	\begin{equation}
	  \begin{aligned}
	    I_3 \leq & \tfrac{|\lambda_2|}{\eps} \| \partial^m \A_R^\eps \|_{L^2} \| \partial_t \dr_0 \|_{L^3} \| \partial^m \dr_R^\eps \|_{L^6} \\
	    \leq & \tfrac{C}{\eps} \| \nabla \u_R^\eps \|_{H^N} \| \nabla \dr_R^\eps \|_{H^N} \| \partial_t \dr_0 \|_{H^1} \\
	    \leq & \tfrac{C}{\eps} ( 1 + \| \nabla \dr_0 \|^3_{H^N} ) ( \| \u_0 \|_{H^N} + \| \nabla \dr_0 \|_{H^N} ) \| \nabla \dr_R^\eps \|_{H^N} \| \nabla \u_R^\eps \|_{H^N} \\
	    \leq & C ( 1 + \| \nabla \dr_0 \|^3_{H^N} ) ( \| \u_0 \|_{H^N} + \| \nabla \dr_0 \|_{H^N} ) \mathscr{D}_{N,\eps} (t) \,.	
	  \end{aligned}
	\end{equation}
	Here we make use of the bound
	\begin{equation}
	  \begin{aligned}
	    \| \partial_t \dr_0 \|_{H^1} \leq & C ( 1 + \| \nabla \dr_0 \|^3_{H^N} ) ( \| \u_0 \|_{H^N} + \| \nabla \dr_0 \|_{H^N} )
	  \end{aligned}
	\end{equation}
	for $N \geq 2$, which is derived from the $\dr_0$-equation of \eqref{PLQ}. The term $I_4$ can be estimated as
	\begin{equation}
	  \begin{aligned}
	     I_4 \leq & \tfrac{|\lambda_2|}{\sqrt{\eps}} \| \partial^m \A_R^\eps \|_{L^2} \| \partial^m ( \dr_R^\eps \cdot \D_{\u_0 + \sqrt{\eps} \u_R^\eps } \dr_R^\eps ) \|_{L^2} \\
	     \leq & \tfrac{C}{\sqrt{\eps}} \| \nabla \u_R^\eps \|_{H^N} \| \dr_R^\eps \|_{H^N} \| \D_{\u_0 + \sqrt{\eps} \u_R^\eps } \dr_R^\eps \|_{H^N} \\
	     \leq & \tfrac{C}{\sqrt{\eps}} \| \nabla \u_R^\eps \|^2_{H^N} \| \dr_R^\eps \|_{H^N} \\
	     & + \tfrac{C}{\sqrt{\eps}} \| \nabla \u_R^\eps \|_{H^N} \| \dr_R^\eps \|_{H^N} \sum_{|m| \leq N} \| \partial^m \D_{\u_0 + \sqrt{\eps} \u_R^\eps} \dr_R^\eps + (\partial^m \B_R^\eps) \dr_0 + \tfrac{\lambda_2}{\lambda_1} ( \partial^m \A_R^\eps ) \dr_0 \|_{L^2} \\
	     \leq & C \eps \mathscr{E}_{N,\eps}^\frac{1}{2} (t) \mathscr{D}_{N,\eps} (t) \,.
	  \end{aligned}
	\end{equation}
	For the term $I_5$, we calculate that
	  \begin{align}
	    \no I_5 \leq & \tfrac{C}{\eps} \| \partial^m \A_R^\eps \|_{L^2} \sum_{0 \neq m' \leq m} \| \partial^{m'} \partial_t \dr_0 \|_{L^3} \| \partial^{m-m'} \dr_R^\eps \|_{L^6} \\
	    \no \leq & \tfrac{C}{\eps} \| \nabla \u_R^\eps \|_{H^N} \| \partial_t \dr_0 \|_{H^{N+1}} \| \nabla \dr_R^\eps \|_{H^N} \\
	    \no \leq & \tfrac{C}{\eps} ( 1 + \| \nabla \dr_0 \|^3_{H^{N+2}} ) ( \| \u_0 \|_{H^{N+2}} + \| \nabla \dr_0 \|_{H^{N+2}} ) \| \nabla \u_R^\eps \|_{H^N} \| \nabla \dr_R^\eps \|_{H^N} \\
	    \leq & C ( 1 + \| \nabla \dr_0 \|^3_{H^{N+2}} ) ( \| \u_0 \|_{H^{N+2}} + \| \nabla \dr_0 \|_{H^{N+2}} ) \mathscr{D}_{N,\eps} (t) \,,
	  \end{align}
	where the bound $ \| \partial_t \dr_0 \|_{H^{N+1}} \leq C ( 1 + \| \nabla \dr_0 \|^3_{H^{N+2}} ) ( \| \u_0 \|_{H^{N+2}} + \| \nabla \dr_0 \|_{H^{N+2}} ) $ is utilized. We compute the quantity $I_6$ that
	\begin{equation}
	  \begin{aligned}
	    I_6 \leq & \tfrac{C}{\eps} \sum_{0 \neq m' \leq m} \| \partial^m \A_R^\eps \|_{L^2} \| \partial^{m'} \dr_0 \|_{L^4} \| \partial^{m-m'} \D_{\u_0 + \sqrt{\eps} \u_R^\eps } \dr_R^\eps \|_{L^4} \\
	    \leq & \tfrac{C}{\eps} \| \nabla \u_R^\eps \|_{H^N} \| \nabla \dr_0 \|_{H^N} \| \D_{\u_0 + \sqrt{\eps} \u_R^\eps} \dr_R^\eps \|_{H^N} \\
	    \leq & \tfrac{C}{\eps} \| \nabla \dr_0 \|_{H^N} \| \nabla \dr_R^\eps \|_{H^N} \\
	    & \quad \times \Big( \| \nabla \u_R^\eps \|_{H^N} + \sum_{|m| \leq N} \| \partial^m \D_{\u_0 + \sqrt{\eps} \u_R^\eps } \dr_R^\eps + ( \partial^m \B_R^\eps ) \dr_0 + \tfrac{\lambda_2}{\lambda_1} (\partial^m \A_R^\eps) \dr_0 \|_{L^2} \Big) \\
	    \leq & C \| \nabla \dr_0 \|_{H^N} \mathscr{D}_{N,\eps} (t) \,.
	  \end{aligned}
	\end{equation}
	It is easy to be estimated that
	\begin{equation}
	  \begin{aligned}
	     I_7 \leq & \tfrac{C}{\eps} \sum_{0 \neq m' \leq m} \| \partial^m \A_R^\eps \|_{L^2} \| \partial^{m'} \dr_0 \|_{L^4} \| \partial^{m-m'} [ ( \u_0 + \sqrt{\eps} \u_R^\eps ) \cdot \nabla \dr_R^\eps ] \|_{L^4} \\
	     \leq & \tfrac{C}{\eps} \| \nabla \u_R^\eps \|_{H^N} \| \nabla \dr_0 \|_{H^N} \| (\u_0 + \sqrt{\eps} \u_R^\eps) \cdot \nabla \dr_R^\eps \|_{H^N} \\
	     \leq & \tfrac{C}{\eps} \| \nabla \dr_0 \|_{H^N} \| \nabla \u_R^\eps \|_{H^N} \| \nabla \dr_R^\eps \|_{H^N} ( \| \u_0 \|_{H^N} + \sqrt{\eps} \| \u_R^\eps \|_{H^N} ) \\
	     \leq & C \| \nabla \dr_0 \|_{H^N}( \| \u_0 \|_{H^N} + \eps \mathscr{E}_{N,\eps}^\frac{1}{2} (t) ) \mathscr{D}_{N,\eps} (t) \,.
	  \end{aligned}
	\end{equation}
	Collecting the bounds of $I_i$ $(1 \leq i \leq 7)$ above, we obtain
	\begin{equation}\label{IN4-Sd1-6-1}
	  \begin{aligned}
	    - \tfrac{\lambda_2}{\eps} \l \partial^m [ ( \A_R^\eps : \dr_0 \otimes \dr_0 ) \dr_0 ] , \partial^m \D_{\u_0 + \sqrt{\eps} \u_R^\eps } \dr_R^\eps \r \leq   C \eps ( 1 + \| \nabla \dr_0 \|_{H^{N+2}} ) \mathscr{E}_{N,\eps}^\frac{1}{2} (t) \mathscr{D}_{N,\eps} (t) \\
	    + C ( 1 + \| \nabla \dr_0 \|^3_{H^{N+2}} ) ( \| \u_0 \|_{H^{N+2}} + \| \nabla \dr_0 \|_{H^{N+2}} ) \mathscr{D}_{N,\eps} (t) \,.
	  \end{aligned}
	\end{equation}
	Recalling that the expression of $\mathcal{S}_{\dr}^1$ in \eqref{Sd-1}, we derive from the bounds \eqref{IN4-Sd1-1}, \eqref{IN4-Sd1-2}, \eqref{IN4-Sd1-3}, \eqref{IN4-Sd1-4}, \eqref{IN4-Sd1-5}, \eqref{IN4-Sd1-6} and \eqref{IN4-Sd1-6-1} that
	\begin{equation}\label{IN4-Sd1-7}
	  \begin{aligned}
	    \tfrac{1}{\eps} \sum_{|m| \leq N} \l \partial^m \mathcal{S}_{\dr}^1 , \partial^m \D_{\u_0 + \sqrt{\eps} \u_R^\eps} \dr_R^\eps \r \leq C \eps ( 1 + \| \nabla \dr_0 \|_{H^{N+2}} ) \mathscr{E}_{N,\eps}^\frac{1}{2} (t) \mathscr{D}_{N,\eps} (t) \\
	    + C ( \| \u_0 \|_{H^{N+2}} + \| \nabla \dr_0 \|_{H^{N+2}} ) ( 1 + \| \nabla \dr_0 \|^3_{H^{N+2}} ) \mathscr{D}_{N, \eps} (t) \,.
	  \end{aligned}
	\end{equation}
	
	Combining the similar arguments of the inequality \eqref{IN4-Sd1-7} and the bound \eqref{uR-nabla-d0}, we know that
	\begin{equation}\label{IN4-Sd1-8}
	  \begin{aligned}
	    & \tfrac{1}{\eps} \sum_{|m| \leq N} \l \partial^m \mathcal{S}_{\dr}^1 , \partial^m ( \u_R^\eps \cdot \nabla \dr_0 ) \r \\
	    \leq & \tfrac{C}{\sqrt{\eps}} ( \| \u_0 \|_{H^{N+2}} + \| \nabla \dr_0 \|_{H^{N+2}} ) ( 1 + \| \nabla \dr_0 \|^3_{H^{N+2}} ) \mathscr{D}^\frac{1}{2}_{N, \eps} (t) \| \partial^m ( \u_R^\eps \cdot \nabla \dr_0 ) \|_{L^2} \\
	    \leq & \tfrac{C}{\sqrt{\eps}} ( \| \u_0 \|_{H^{N+2}} + \| \nabla \dr_0 \|_{H^{N+2}} ) ( 1 + \| \nabla \dr_0 \|^3_{H^{N+2}} ) \mathscr{D}^\frac{1}{2}_{N, \eps} (t) \| \nabla \u_R^\eps \|_{H^N} \| \nabla \dr_0 \|_{H^{N+1}} \\
	    \leq & C ( \| \u_0 \|_{H^{N+2}} + \| \nabla \dr_0 \|_{H^{N+2}} ) ( 1 + \| \nabla \dr_0 \|^4_{H^{N+2}} ) \mathscr{D}_{N, \eps} (t) \,.
	  \end{aligned}
	\end{equation}
	
	We next estimate the quantity  $ \tfrac{\delta}{\eps } \sum_{|m| \leq N} \l \partial^m \mathcal{S}_{\dr}^1 , \partial^m \dr_R^\eps \r $ for $N \geq 2$. First, we have
	\begin{equation}\label{IN4-Sd1-9}
	  \begin{aligned}
	    & \tfrac{2 \delta}{\eps} \l \partial^m [ ( \nabla \dr_0 \cdot \nabla \dr_R^\eps ) \dr_0 ] , \partial^m \dr_R^\eps \r \\
	    \leq & \tfrac{C}{\eps} \sum_{m' \leq m} \| \partial^{m-m'} ( \dr_0 \otimes \nabla \dr_0 ) \|_{L^3} \| \nabla \partial^{m'} \dr_R^\eps \|_{L^2} \| \partial^m \dr_R^\eps \|_{L^6} \\
	    \leq & \tfrac{C}{\eps} \| \nabla \dr_0 \|_{H^{N+1}} (1 + \| \nabla \dr_0 \|_{H^{N+1}}  ) \| \nabla \dr_R^\eps \|^2_{H^N} \\
	    \leq & C \| \nabla \dr_0 \|_{H^{N+1}} (1 + \| \nabla \dr_0 \|_{H^{N+1}}  ) \mathscr{D}_{N,\eps} (t) \,.
	  \end{aligned}
	\end{equation}
	We then estimate
	\begin{equation}\label{IN4-Sd1-10}
	  \begin{aligned}
	    & \tfrac{\delta}{\eps} \l \partial^m ( |\nabla \dr_0|^2 \dr_R^\eps ) , \partial^m \dr_R^\eps \r \\
	    \leq & \tfrac{C}{\eps} \sum_{m' \leq m} \sum_{m'' \leq m'} \| \nabla \partial^{m''} \dr_0 \|_{L^3} \| \nabla \partial^{m'-m''} \dr_0 \|_{L^3} \| \partial^{m-m'} \dr_R^\eps \|_{L^6} \| \partial^m \dr_R^\eps \|_{L^6} \\
	    \leq & \tfrac{C}{\eps} \| \nabla \dr_0 \|^2_{H^{N+1}} \| \nabla \dr_R^\eps \|^2_{H^N} \leq C  \| \nabla \dr_0 \|^2_{H^{N+1}} \mathscr{D}_{N,\eps} (t) \,.
	  \end{aligned}
	\end{equation}
	It is easy to calculate that
	\begin{equation}\label{IN4-Sd1-11}
	  \begin{aligned}
	    & \tfrac{\delta \lambda_1}{\eps} \l \partial^m ( \u_0 \cdot \nabla \dr_R^\eps + \u_R^\eps \cdot \nabla \dr_0 + \B_0 \dr_R^\eps ) , \partial^m \dr_R^\eps \r \\
	    \leq & \tfrac{C}{\eps} \sum_{m' \leq m} \| \partial^{m'} \u_0 \|_{L^3} \| \nabla \partial^{m-m'} \dr_R^\eps \|_{L^2} \| \partial^m \dr_R^\eps \|_{L^6} \\
	    + & \tfrac{C}{\eps} \sum_{m' \leq m} \Big( \| \partial^{m'} \u_R^\eps \|_{L^6} \| \nabla \partial^{m-m'} \dr_0 \|_{L^2} + \| \partial^{m'} \dr_R^\eps \|_{L^6} \| \partial^{m-m'} \B_0 \|_{L^2} \Big) \| \partial^m \dr_R^\eps \|_{L^6} |\T^3|^\frac{1}{6} \\
	    \leq & \tfrac{C}{\eps} ( \| \u_0 \|_{H^{N+1}} + \| \nabla \dr_0 \|_{H^{N+1}} ) ( \| \nabla \dr_R^\eps \|_{H^N} + \| \nabla \u_R^\eps \|_{H^N} ) \| \nabla \dr_R^\eps \|_{H^N} \\
	    \leq & C ( \| \u_0 \|_{H^{N+1}} + \| \nabla \dr_0 \|_{H^{N+1}} ) \mathscr{D}_{N,\eps} (t) \,,
	  \end{aligned}
	\end{equation}
	where we intrinsically utilize the fact that the volume of $\T^3$ is finite. Similarly as in \eqref{IN4-Sd1-11}, we obtain
	\begin{equation}\label{IN4-Sd1-12}
	  \begin{aligned}
	    \tfrac{\delta \lambda_2}{\eps} \l \partial^m ( \A_0 \dr_R^\eps ) , \partial^m \dr_R^\eps \r \leq C \| \u_0 \|_{H^{N+1}} \mathscr{D}_{N,\eps} (t) \,,
	  \end{aligned}
	\end{equation}
	and
	\begin{equation}\label{IN4-Sd1-13}
	  \begin{aligned}
	    & - \tfrac{\delta \lambda_2}{\eps} \l \partial^m [ ( \A_0 : \dr_0 \otimes \dr_0 ) \dr_R^\eps ] , \partial^m \dr_R^\eps \r  - \tfrac{2 \delta \lambda_2}{\eps} \l \partial^m [ ( \A_0 : \dr_0 \otimes \dr_R^\eps ) \dr_0 ] , \partial^m \dr_R^\eps \r \\
	    & \leq C \| \u_0 \|_{H^{N+1}} ( 1 + \| \nabla \dr_0 \|^2_{H^{N+1}} ) \mathscr{D}_{N,\eps} (t) \,.
	  \end{aligned}
	\end{equation}
	It remains to estimate the quantity $ - \tfrac{\delta \lambda_2}{\eps} \l \partial^m ( (\A_R^\eps : \dr_0 \otimes \dr_0) \dr_0 ), \partial^m \dr_R^\eps \r $ for all $|m| \leq N$. From the relation \eqref{Const-3}, we easily derive that
	\begin{equation}
	  \begin{aligned}
	     & - \tfrac{\delta \lambda_2}{\eps} \l \partial^m ( (\A_R^\eps : \dr_0 \otimes \dr_0) \dr_0 ), \partial^m \dr_R^\eps \r \\
	     = & \underset{I\!I_1}{ \underbrace{ \tfrac{\delta \lambda_2}{2 \sqrt{\eps}} \sum_{m' \leq m} C_m^{m'} \l \partial^m \A_R^\eps : \dr_0 \otimes \dr_0 , \partial^{m'} \dr_R^\eps \cdot \partial^{m-m'} \dr_R^\eps \r }} \\
	     & + \underset{I\!I_2}{ \underbrace{ \tfrac{\delta \lambda_2}{\eps} \sum_{0 \neq m' \leq m} C_m^{m'} \l \partial^m \A_R^\eps : \dr_0 \otimes \dr_0 , \partial^{m'} \dr_0 \cdot \partial^{m-m'} \dr_R^\eps \r }} \\
	     & + \underset{I\!I_1}{ \underbrace{ \tfrac{- \delta \lambda_2}{\eps} \sum_{0 \neq m' \leq m} C_m^{m'} \l \partial^{m-m'} \A_R^\eps : \partial^{m'} (\dr_0 \otimes \dr_0 \otimes \dr_0) , \partial^m \dr_R^\eps \r }} \,.
	  \end{aligned}
	\end{equation}
	For the term $I\!I_1$, we estimate
	\begin{equation}
	  \begin{aligned}
	    I\!I_1 \leq & \tfrac{C}{\sqrt{\eps}} \sum_{m' \leq m} \| \partial^m \A_R^\eps \|_{L^2} \| \partial^{m'} \dr_R^\eps \|_{L^3} \| \partial^{m-m'} \dr_R^\eps \|_{L^6} \\
	    \leq & \tfrac{C}{\sqrt{\eps}} \| \nabla \u_R^\eps \|_{H^N} \| \nabla \dr_R^\eps \|_{H^N} ( \| \dr_R^\eps \|_{H^N} + \| \nabla \dr_R^\eps \|_{H^N} ) \\
	    \leq & C \eps \mathscr{E}_{N,\eps}^\frac{1}{2} (t) \mathscr{D}_{N,\eps} (t) \,.
	  \end{aligned}
	\end{equation}
	For the term $I\!I_2$, we have
	\begin{equation}
	  \begin{aligned}
	    I\!I_2 \leq & \tfrac{C}{\eps} \sum_{0 \neq m' \leq m} \| \partial^m \A_R^\eps \|_{L^2} \| \partial^{m'} \dr_0 \|_{L^3} \| \partial^{m-m'} \dr_R^\eps \|_{L^6} \\
	    \leq & \tfrac{C}{\eps} \| \nabla \u_R^\eps \|_{H^N} \| \nabla \dr_0 \|_{H^N} \| \nabla \dr_R^\eps \|_{H^N} \leq C \| \nabla \dr_0 \|_{H^N} \mathscr{D}_{N,\eps} (t) \,.
	  \end{aligned}
	\end{equation}
	For the term $I\!I_3$, we calculate that
	  \begin{align}
	    \no I\!I_3 \leq & \tfrac{C}{\eps} \sum_{0 \neq m' \leq m } \| \partial^{m-m'} \A_R^\eps \|_{L^2} \| \partial^{m'} (\dr_0 \otimes \dr_0 \otimes \dr_0) \|_{L^3} \| \partial^m \dr_R^\eps \|_{L^6} \\
	    \no \leq & \tfrac{C}{\eps} \| \nabla \u_R^\eps \|_{H^N} \| \nabla \dr_0 \|_{H^N} ( 1 + \| \nabla \dr_0 \|^2_{H^N} ) \| \nabla \dr_R^\eps \|_{H^N} \\
	    \leq & C ( 1 + \| \nabla \dr_0 \|^2_{H^N} ) \| \nabla \dr_0 \|_{H^N} \mathscr{D}_{N,\eps} (t) \,.
	  \end{align}
	Summarizing the estimates $I\!I_1$, $I\!I_2$ and $I\!I_3$, we obtain
	\begin{equation}\label{IN4-Sd1-13-1}
	  \begin{aligned}
	    & - \tfrac{\delta \lambda_2}{\eps} \l \partial^m ( (\A_R^\eps : \dr_0 \otimes \dr_0) \dr_0 ), \partial^m \dr_R^\eps \r \\
	    \leq & C ( 1 + \| \nabla \dr_0 \|^2_{H^N} ) \| \nabla \dr_0 \|_{H^N} \mathscr{D}_{N,\eps} (t) + C \eps \mathscr{E}_{N,\eps}^\frac{1}{2} (t) \mathscr{D}_{N,\eps} (t) \,.
	  \end{aligned}
	\end{equation}
	Recalling the expression of $\mathcal{S}_{\dr}^1$ in \eqref{Sd-1}, we derive from the inequalities \eqref{IN4-Sd1-9}, \eqref{IN4-Sd1-10}, \eqref{IN4-Sd1-11}, \eqref{IN4-Sd1-12}, \eqref{IN4-Sd1-13} and \eqref{IN4-Sd1-13-1} that
	\begin{equation}\label{IN4-Sd1-14}
	  \begin{aligned}
	    \tfrac{\delta}{\eps } \sum_{|m| \leq N} \l \partial^m \mathcal{S}_{\dr}^1 , \partial^m \dr_R^\eps \r \leq C \eps \mathscr{E}_{N,\eps}^\frac{1}{2} (t) \mathscr{D}_{N,\eps} (t) \\
	    + C ( \| \u_0 \|_{H^{N+1}} + \| \nabla \dr_0 \|_{H^{N+1}} ) ( 1 + \| \nabla \dr_0 \|^2_{H^{N+1}} ) \mathscr{D}_{N,\eps} (t) \,.
	  \end{aligned}
	\end{equation}
	Consequently, plugging the inequalities \eqref{IN4-Sd1-7}, \eqref{IN4-Sd1-8} and \eqref{IN4-Sd1-14} into the expression of $\mathcal{I}_N^{(4)} ( \mathcal{S}_{\dr}^1 )$ in \eqref{IN4-1} reduces to
	\begin{equation}\label{IN4-Sd1}
	  \begin{aligned}
	    \mathcal{I}_N^{(4)} ( \mathcal{S}_{\dr}^1 ) \leq  C ( \| \u_0 \|_{H^{N+2}} + \| \nabla \dr_0 \|_{H^{N+2}} ) ( 1 + \| \nabla \dr_0 \|^3_{H^{N+2}} ) \mathscr{D}_{N, \eps} (t) + C \eps \mathscr{E}_{N,\eps}^\frac{1}{2} (t) \mathscr{D}_{N,\eps} (t) \,.
	  \end{aligned}
	\end{equation}
	
	Next we estimate the quantity $\mathcal{I}_N^{(4)} (\mathcal{S}_{\dr}^2)$ for any integer $N \geq 2$. First, we have
	\begin{equation}\label{IN4-Sd2-1}
	  \begin{aligned}
	    & \tfrac{1}{\sqrt{\eps}} \l \partial^m ( |\nabla \dr_R^\eps|^2 \dr_0 ) , \partial^m \D_{\u_0 + \sqrt{\eps} \u_R^\eps } \dr_R^\eps \r \\
	    \leq & \tfrac{C}{\sqrt{\eps}} \sum_{m' \leq m} \| \partial^{m-m'} |\nabla \dr_R^\eps|^2 \|_{L^2} \| \partial^{m'} \dr_0 \|_{L^\infty} \| \partial^m \D_{\u_0 + \sqrt{\eps} \u_R^\eps } \dr_R^\eps \|_{L^2} \\
	    \leq & \tfrac{C}{\sqrt{\eps}} \| \nabla \dr_R^\eps \|^2_{H^N} \| \D_{\u_0 + \sqrt{\eps} \u_R^\eps } \dr_R^\eps \|_{H^N} ( 1 + \| \nabla \dr_0 \|_{H^{N+2}} ) \\
	    \leq & C \sqrt{\eps} ( 1 + \| \nabla \dr_0 \|_{H^{N+2}} ) \mathscr{E}^\frac{1}{2}_{N,\eps} (t) \mathscr{D}_{N,\eps} (t)
	  \end{aligned}
	\end{equation}
	and
	\begin{equation}\label{IN4-Sd2-2}
	  \begin{aligned}
	    & \tfrac{2}{\sqrt{\eps}} \l \partial^m ( ( \nabla \dr_0 \cdot \nabla \dr_R^\eps ) \dr_R^\eps ) , \partial^m \D_{\u_0 + \sqrt{\eps} \u_R^\eps } \dr_R^\eps \r \\
	    \leq & \tfrac{C}{\sqrt{\eps}} \sum_{m' \leq m} \| \nabla \partial^{m'} \dr_0 \|_{L^\infty} \| \partial^{m-m'} ( \dr_R^\eps \otimes \nabla \dr_R^\eps ) \|_{L^2} \| \partial^m \D_{\u_0 + \sqrt{\eps} \u_R^\eps } \dr_R^\eps \|_{L^2} \\
	    \leq & \tfrac{C}{\sqrt{\eps}} \| \nabla \dr_R^\eps \|_{H^N} \| \dr_R^\eps \|_{H^N} \| \D_{\u_0 + \sqrt{\eps} \u_R^\eps } \dr_R^\eps \|_{H^N} \| \nabla \dr_0 \|_{H^{N+2}} \\
	    \leq & C \sqrt{\eps} \| \nabla \dr_0 \|_{H^{N+2}} \mathscr{E}^\frac{1}{2}_{N,\eps} (t) \mathscr{D}_{N,\eps} (t)
	  \end{aligned}
	\end{equation}
	for all $|m| \leq N$. It is easy to deduce that
	\begin{equation}\label{IN4-Sd2-3}
	  \begin{aligned}
	    & - \tfrac{1}{\sqrt{\eps}} \l \partial^m \D_{\u_0}^2 \dr_0 , \partial^m \D_{\u_0 + \sqrt{\eps} \u_R^\eps } \dr_R^\eps \r \leq \tfrac{1}{\sqrt{\eps}} \| \partial^m \D_{\u_0}^2 \dr_0 \|_{L^2} \| \partial^m \D_{\u_0 + \sqrt{\eps} \u_R^\eps } \dr_R^\eps \|_{L^2} \\
	    \leq & \tfrac{C}{\sqrt{\eps}} \| \D_{\u_0}^2 \dr_0 \|_{H^N} \big( \| \partial^m \D_{\u_0 + \sqrt{\eps} \u_R^\eps } \dr_R^\eps + ( \partial^m \B_R^\eps ) \dr_0 + \tfrac{\lambda_2}{\lambda_1} ( \partial^m \A_R^\eps ) \dr_0 \|_{L^2} + \| \nabla \u_R^\eps \|_{H^N} \big) \\
	    \leq & C (1 + \| \u_0 \|^6_{N^{N+2}} + \| \nabla \dr_0 \|^6_{H^{N+2}}) ( \| \nabla \u_0 \|_{H^{N+2}} + \| \Delta \dr_0 \|_{H^{N+2}} ) \mathscr{D}^\frac{1}{2}_{N,\eps} (t) \,,
	  \end{aligned}
	\end{equation}
	where we make use of the inequality \eqref{D2u0-d0-HN}.
	
	For the term $ - \tfrac{1}{\sqrt{\eps}} \l \partial^m ( | \D_{\u_0} \dr_0 |^2 \dr_0 ) , \partial^m \D_{\u_0 + \sqrt{\eps} \u_R^\eps } \dr_R^\eps \r $, we deduce that
	\begin{equation}\label{IN4-Sd2-4}
	  \begin{aligned}
	    & - \tfrac{1}{\sqrt{\eps}} \l \partial^m ( | \D_{\u_0} \dr_0 |^2 \dr_0 ) , \partial^m \D_{\u_0 + \sqrt{\eps} \u_R^\eps } \dr_R^\eps \r \\
	    \leq & \tfrac{1}{\sqrt{\eps}} \| \partial^m |\D_{\u_0} \dr_0|^2 \|_{L^2} \| \partial^m \D_{\u_0 + \sqrt{\eps} \u_R^\eps } \dr_R^\eps \|_{L^2} \\
	    & + \tfrac{C}{\sqrt{\eps}} \sum_{0 \neq m' \leq m} \| \partial^{m-m'} |\D_{\u_0} \dr_0|^2 \|_{L^4} \| \partial^{m'} \dr_0 \|_{L^4} \| \partial^m \D_{\u_0 + \sqrt{\eps} \u_R^\eps } \dr_R^\eps \|_{L^2} \\
	    \leq & \tfrac{C}{\sqrt{\eps}} ( 1 + \| \nabla \dr_0 \|_{H^{N+1}} ) \| \D_{\u_0} \dr_0 \|^2_{H^N} \\
	    & \qquad \times \big( \| \partial^m \D_{\u_0 + \sqrt{\eps} \u_R^\eps } \dr_R^\eps + ( \partial^m \B_R^\eps ) \dr_0 + \tfrac{\lambda_2}{\lambda_1} ( \partial^m \A_R^\eps ) \dr_0 \|_{L^2} + \| \nabla \u_R^\eps \|_{H^N} \big) \\
	    \leq & C ( 1 + \| \nabla \dr_0 \|^7_{H^{N+1}} ) ( \| \nabla \u_0 \|^2_{H^N} + \| \Delta \dr_0 \|^2_{H^N} ) \mathscr{D}^\frac{1}{2}_{N,\eps} (t) \\
	    \leq &  C ( 1 + \| \u_0 \|^8_{H^{N+1}} + \| \nabla \dr_0 \|^8_{H^{N+1}} ) ( \| \nabla \u_0 \|_{H^N} + \| \Delta \dr_0 \|_{H^N} ) \mathscr{D}^\frac{1}{2}_{N,\eps} (t) \,,
	  \end{aligned}
	\end{equation}
	where the inequality \eqref{Du0-d0-HN} is utilized. We next estimate that
	\begin{equation}\label{IN4-Sd2-5}
	  \begin{aligned}
	    & \tfrac{\lambda_1}{\sqrt{\eps}} \l \partial^m ( \B_R^\eps \dr_R^\eps ) , \partial^m \D_{\u_0 + \sqrt{\eps} \u_R^\eps } \dr_R^\eps \r \\
	    \leq & \tfrac{C}{\sqrt{\eps}} \big( \| \partial^m \B_R^\eps \|_{L^2} \| \dr_R^\eps \|_{L^\infty} + \sum_{0 \neq m' \leq m} \| \partial^{m-m'} \B_R^\eps \|_{L^4} \| \partial^{m'} \dr_R^\eps \|_{L^4} \big) \| \partial^m \D_{\u_0 + \sqrt{\eps} \u_R^\eps } \dr_R^\eps \|_{L^2} \\
	    \leq & \tfrac{C}{\sqrt{\eps}} \| \nabla \u_R^\eps \|_{H^N} ( \| \dr_R^\eps \|_{H^N} + \| \nabla \dr_R^\eps \|_{H^N} ) \| \D_{\u_0 + \sqrt{\eps} \u_R^\eps } \dr_R^\eps \|_{H^N} \\
	    \leq & C \sqrt{\eps} \mathscr{E}^\frac{1}{2}_{N,\eps} (t) \mathscr{D}_{N,\eps} (t) \,.
	  \end{aligned}
	\end{equation}
	Similarly as in \eqref{IN4-Sd2-5}, we have
	\begin{equation}\label{IN4-Sd2-6}
	  \begin{aligned}
	    \tfrac{\lambda_2}{\sqrt{\eps}} \l \partial^m ( \A_R^\eps \dr_R^\eps ) , \partial^m \D_{\u_0 + \sqrt{\eps} \u_R^\eps } \dr_R^\eps \r \leq C \sqrt{\eps} \mathscr{E}^\frac{1}{2}_{N,\eps} (t) \mathscr{D}_{N,\eps} (t) \,.
	  \end{aligned}
	\end{equation}
	We deduce the following bound
	\begin{equation}\label{IN4-Sd2-7}
	  \begin{aligned}
	    & - \tfrac{\lambda_2}{\sqrt{\eps}} \l \partial^m [ \A_0 : ( \dr_0 \otimes \dr_R^\eps + \dr_R^\eps \otimes \dr_0 ) \dr_R^\eps ] , \partial^m \D_{\u_0 + \sqrt{\eps} \u_R^\eps } \dr_R^\eps \r \\
	    \leq & \tfrac{C}{\sqrt{\eps}} \sum_{\substack{ m' \leq m \\ m'' \leq m' }} \| \partial^{m-m'} ( \A_0 \dr_0 ) \|_{L^6} \| \partial^{m''} \dr_R^\eps \|_{L^6} \| \partial^{m' - m''} \dr_R^\eps \|_{L^6} \| \partial^m \D_{\u_0 + \sqrt{\eps} \u_R^\eps } \dr_R^\eps \|_{L^2} \\
	    \leq & \tfrac{C}{\sqrt{\eps}} \| \u_0 \|_{H^{N+2}} ( 1 + \| \nabla \dr_0 \|_{H^{N+2}} ) \| \nabla \dr_R^\eps \|^2_{H^N} \| \D_{\u_0 + \sqrt{\eps} \u_R^\eps } \dr_R^\eps \|_{H^N} \\
	    \leq & C \sqrt{\eps} \| \u_0 \|_{H^{N+2}} ( 1 + \| \nabla \dr_0 \|_{H^{N+2}} ) \mathscr{E}^\frac{1}{\eps}_{N,\eps} (t) \mathscr{D}_{N,\eps} (t) \,.
	  \end{aligned}
	\end{equation}
	We calculate that
	\begin{equation}\label{IN4-Sd2-8}
	  \begin{aligned}
	    & - \tfrac{\lambda_2}{\sqrt{\eps}} \l \partial^m [ ( \A_R^\eps : \dr_0 \otimes \dr_0 ) \dr_R^\eps ] , \partial^m \D_{\u_0 + \sqrt{\eps} \u_R^\eps } \dr_R^\eps \r \\
	    \leq & \tfrac{C}{\sqrt{\eps}} \| \partial^m \A_R^\eps \|_{L^2} \| ( \dr_0 \otimes \dr_0 ) \dr_R^\eps \|_{L^\infty} \| \partial^m \D_{\u_0 + \sqrt{\eps} \u_R^\eps } \dr_R^\eps \|_{L^2} \\
	    & + \tfrac{C}{\sqrt{\eps}} \sum_{0 \neq m' \leq m} \| \partial^{m-m'} \A_R^\eps \|_{L^4} \| \partial^{m'} ( (\dr_0 \otimes \dr_0) \dr_R^\eps ) \|_{L^4} \| \partial^m \D_{\u_0 + \sqrt{\eps} \u_R^\eps } \dr_R^\eps \|_{L^2} \\
	    \leq & \tfrac{C}{\sqrt{\eps}} \| \nabla \u_R^\eps \|_{H^N} \| \D_{\u_0 + \sqrt{\eps} \u_R^\eps } \dr_R^\eps \|_{H^N} ( \| \dr_R^\eps \|_{H^N} + \| \nabla \dr_R^\eps \|_{H^N} ) ( 1 + \| \nabla \dr_0 \|^2_{H^{N+1}} ) \\
	    \leq & C \sqrt{\eps} ( 1 + \| \nabla \dr_0 \|^2_{H^{N+1}} ) \mathscr{E}^\frac{1}{2}_{N,\eps} (t) \mathscr{D}_{N,\eps} (t) \,.
	  \end{aligned}
	\end{equation}
	It is derived from the analogous arguments in \eqref{IN4-Sd2-8} that
	\begin{equation}\label{IN4-Sd2-9}
	  \begin{aligned}
	    & - \tfrac{\lambda_2}{\sqrt{\eps}} \l \partial^m [ \A_R^\eps : ( \dr_0 \otimes \dr_R^\eps + \dr_R^\eps \otimes \dr_0 ) \dr_0 ] , \partial^m \D_{\u_0 + \sqrt{\eps} \u_R^\eps } \dr_R^\eps \r \\
	    & \leq C \sqrt{\eps} ( 1 + \| \nabla \dr_0 \|^2_{H^{N+1}} ) \mathscr{E}^\frac{1}{2}_{N,\eps} (t) \mathscr{D}_{N,\eps} (t) \,.
	  \end{aligned}
	\end{equation}
	Moreover, from the similar calculations in \eqref{IN4-Sd2-7}, we derive that
	\begin{equation}\label{IN4-Sd2-10}
	  \begin{aligned}
	    & - \tfrac{\lambda_2}{\sqrt{\eps}} \l \partial^m [ ( \A_0 : \dr_R^\eps \otimes \dr_R^\eps ) \dr_0 ] , \partial^m \D_{\u_0 + \sqrt{\eps} \u_R^\eps } \dr_R^\eps \r \\
	    & \leq C \sqrt{\eps} \| \u_0 \|_{H^{N+2}} ( 1 + \| \nabla \dr_0 \|_{H^{N+2}} ) \mathscr{E}^\frac{1}{2}_{N,\eps} (t) \mathscr{D}_{N,\eps} (t) \,.
	  \end{aligned}
	\end{equation}
	Recalling the definition of $\mathcal{S}_{\dr}^2$ in \eqref{Sd-2}, we collect the inequalities \eqref{IN4-Sd2-1}, \eqref{IN4-Sd2-2}, \eqref{IN4-Sd2-3}, \eqref{IN4-Sd2-4}, \eqref{IN4-Sd2-5}, \eqref{IN4-Sd2-6}, \eqref{IN4-Sd2-7}, \eqref{IN4-Sd2-8}, \eqref{IN4-Sd2-9} and \eqref{IN4-Sd2-10}, and then know that
	\begin{equation}\label{IN4-Sd2-11}
	  \begin{aligned}
	    & \tfrac{1}{\sqrt{\eps}} \l \partial^m \mathcal{S}_{\dr}^2 , \partial^m \D_{\u_0 + \sqrt{\eps} \u_R^\eps } \dr_R^\eps \r \\
	    & \leq C \sqrt{\eps} ( 1 + \| \u_0 \|^2_{H^{N+2} } + \| \nabla \dr_0 \|^2_{H^{N+2}} ) \mathscr{E}^\frac{1}{2}_{N,\eps} (t) \mathscr{D}_{N,\eps} (t) \\
	    & + C ( 1 + \| \u_0 \|^8_{H^{N+1}} + \| \nabla \dr_0 \|^8_{H^{N+1}} ) ( \| \nabla \u_0 \|_{H^N} + \| \Delta \dr_0 \|_{H^N} ) \mathscr{D}^\frac{1}{2}_{N,\eps} (t) \,.
	  \end{aligned}
	\end{equation}
	
	Furthermore, the inequality \eqref{uR-nabla-d0} and similar estimates on \eqref{IN4-Sd2-11} reduce to
	\begin{equation}\label{IN4-Sd2-12}
	  \begin{aligned}
	    & \tfrac{1}{\sqrt{\eps}} \l \partial^m \mathcal{S}_{\dr}^2 , \partial^m ( \u_R^\eps \cdot \nabla \dr_0 ) \r \\
	    \leq & C  ( 1 + \| \u_0 \|^2_{H^{N+2} } + \| \nabla \dr_0 \|^2_{H^{N+2}} ) \mathscr{E}^\frac{1}{2}_{N,\eps} (t) \mathscr{D}^\frac{1}{2}_{N,\eps} (t) \| \u_R^\eps \cdot \nabla \dr_0 \|_{H^N} \\
	    & + \tfrac{C}{\sqrt{\eps}} ( 1 + \| \u_0 \|^8_{H^{N+1}} + \| \nabla \dr_0 \|^8_{H^{N+1}} ) ( \| \nabla \u_0 \|_{H^N} + \| \Delta \dr_0 \|_{H^N} ) \| \u_R^\eps \cdot \nabla \dr_0 \|_{H^N} \\
	    \leq & C \sqrt{\eps} ( 1 + \| \u_0 \|^3_{H^{N+2} } + \| \nabla \dr_0 \|^3_{H^{N+2}} ) \mathscr{E}^\frac{1}{2}_{N,\eps} (t) \mathscr{D}_{N,\eps} (t) \\
	    & + C ( 1 + \| \u_0 \|^9_{H^{N+1}} + \| \nabla \dr_0 \|^9_{H^{N+1}} ) ( \| \nabla \u_0 \|_{H^N} + \| \Delta \dr_0 \|_{H^N} ) \mathscr{D}^\frac{1}{2}_{N,\eps} (t) \,.
	  \end{aligned}
	\end{equation}
	
	We next estimate the quantity $ \tfrac{\delta}{\sqrt{\eps}} \l \partial^m \mathcal{S}_{\dr}^2 , \partial^m \dr_R^\eps \r $ for all $|m| \leq N$. First, we have
	\begin{equation}\label{IN4-Sd2-13}
	  \begin{aligned}
	    & \tfrac{\delta}{\sqrt{\eps}} \l \partial^m ( |\nabla \dr_R^\eps|^2 \dr_0 ) , \partial^m \dr_R^\eps \r \\
	    \leq & \tfrac{C}{\sqrt{\eps}} \| \partial^m |\nabla \dr_R^\eps|^2 \|_{L^2} \| \partial^m \dr_R^\eps \|_{L^2} \\
	    & + \tfrac{C}{\sqrt{\eps}} \sum_{0 \neq m' \leq m} \| \partial^{m-m'} |\nabla \dr_R^\eps|^2 \|_{L^4} \| \partial^{m'} \dr_0 \|_{L^4} \| \partial^m \dr_R^\eps \|_{L^2} \\
	    \leq & \tfrac{C}{\sqrt{\eps}} \| \nabla \dr_R^\eps \|^2_{H^N} \| \dr_R^\eps \|_{H^N} ( 1 + \| \nabla \dr_0 \|_{H^{N+1}} ) \\
	    \leq & C \eps ( 1 + \| \nabla \dr_0 \|_{H^{N+1}} ) \mathscr{E}^\frac{1}{2}_{N,\eps} (t) \mathscr{D}_{N,\eps} (t) \,.
	  \end{aligned}
	\end{equation}
	It is estimated that
	\begin{equation}\label{IN4-Sd2-14}
	  \begin{aligned}
	    & \tfrac{2 \delta}{\sqrt{\eps}} \l \partial^m [ ( \nabla \dr_0 \cdot \nabla \dr_R^\eps ) \dr_R^\eps ] , \partial^m \dr_R^\eps \r \\
	    \leq & \tfrac{C}{\sqrt{\eps}} \sum_{m' \leq m} \sum_{m'' \leq m'} \| \nabla \partial^{m-m'} \dr_R^\eps \|_{L^2} \| \partial^{m''} \dr_R^\eps \|_{L^6} \| \nabla \partial^{m'-m''} \dr_0 \|_{L^6} \| \partial^m \dr_R^\eps \|_{L^6} \\
	    \leq & \tfrac{C}{\sqrt{\eps}} \| \nabla \dr_R^\eps \|^3_{H^N} \| \nabla \dr_0 \|_{H^{N+1}} \leq C \| \nabla \dr_0 \|_{H^{N+1}} \mathscr{E}^\frac{1}{2}_{N,\eps} (t) \mathscr{D}_{N,\eps} (t) \,.
	  \end{aligned}
	\end{equation}
	The following bound holds：
	\begin{equation}\label{IN4-Sd2-15}
	  \begin{aligned}
	    & - \tfrac{\delta}{\sqrt{\eps}} \l \partial^m \D_{\u_0}^2 \dr_0 , \partial^m \dr_R^\eps \r \leq \tfrac{\delta}{\sqrt{\eps}} \| \partial^m \D_{\u_0}^2 \dr_0 \|_{L^2} \| \partial^m \dr_R^\eps \|_{L^6} | \T^3 |^\frac{1}{3} \\
	    \leq & \tfrac{C}{\sqrt{\eps}} \| \D_{\u_0}^2 \dr_0 \|_{H^N} \| \nabla \dr_R^\eps \|_{H^N} \\
	    \leq & \tfrac{C}{\sqrt{\eps}} ( 1 + \| \u_0 \|^6_{H^{N+2}} + \| \nabla \dr_0 \|^6_{H^{N+2}} ) ( \| \nabla \u_0 \|_{H^{N+2}} + \| \Delta \dr_0 \|_{H^{N+2}} ) \| \nabla \dr_R^\eps \|_{H^N} \\
	    \leq & C ( 1 + \| \u_0 \|^6_{H^{N+2}} + \| \nabla \dr_0 \|^6_{H^{N+2}} ) ( \| \nabla \u_0 \|_{H^{N+2}} + \| \Delta \dr_0 \|_{H^{N+2}} ) \mathscr{D}^\frac{1}{2}_{N,\eps} (t) \,,
	  \end{aligned}
	\end{equation}
	where the inequality \eqref{D2u0-d0-HN} and the fact that the volume of $\T^3$ is finite are utilized. Moreover, from the bound \eqref{Du0-d0-HN} and the finiteness of the volume $\T^3$, we deduce that
	\begin{equation}\label{IN4-Sd2-16}
	  \begin{aligned}
	    & - \tfrac{\delta}{\sqrt{\eps}} \l \partial^m ( |\D_{\u_0} \dr_0|^2 \dr_0 ) , \partial^m \dr_R^\eps \r \\
	    \leq & \tfrac{\delta}{\sqrt{\eps}} \| \partial^m |\D_{\u_0} \dr_0|^2 \|_{L^2} \| \partial^m \dr_R^\eps \|_{L^6} |\T^3|^\frac{1}{3} \\
	    & + \tfrac{C}{\sqrt{\eps}} \sum_{0 \neq m' \leq m} \| \partial^{m-m'} |\D_{\u_0} \dr_0|^2 \|_{L^4} \| \partial^{m'} \dr_0 \|_{L^4} \| \partial^m \dr_R^\eps \|_{L^6} |\T^3|^\frac{1}{3} \\
	    \leq & \tfrac{C}{\sqrt{\eps}} \| \D_{\u_0} \dr_0 \|^2_{H^N} \| \nabla \dr_R^\eps \|_{H^N} ( 1 + \| \nabla \dr_0 \|_{H^N} ) \\
	    \leq & \tfrac{C}{\sqrt{\eps}} ( 1 + \| \nabla \dr_0 \|^7_{H^N} ) ( \| \nabla \u_0 \|^2_{H^N} + \| \Delta \dr_0 \|^2_{H^N} ) \| \nabla \dr_R^\eps \|_{H^N} \\
	    \leq & C ( 1 + \| \u_0 \|^8_{H^{N+2}} + \| \nabla \dr_0 \|^8_{H^{N+2}} ) ( \| \nabla \u_0 \|_{H^{N+2}} + \| \Delta \dr_0 \|_{H^{N+2}} ) \mathscr{D}^\frac{1}{2}_{N,\eps} (t) \,.
	  \end{aligned}
	\end{equation}
	For the term $ \tfrac{\delta \lambda_1}{\sqrt{\eps}} \l \partial^m ( \B_R^\eps \dr_R^\eps ) , \partial^m \dr_R^\eps \r $, we estimate that
	\begin{equation}\label{IN4-Sd2-17}
	  \begin{aligned}
	    & \tfrac{\delta \lambda_1}{\sqrt{\eps}} \l \partial^m ( \B_R^\eps \dr_R^\eps ) , \partial^m \dr_R^\eps \r \\
	    \leq & \tfrac{C}{\sqrt{\eps}} \sum_{m' \leq m} \| \partial^{m-m'} \B_R^\eps \|_{L^2} \| \partial^{m'} \dr_R^\eps \|_{L^3} \| \partial^m \dr_R^\eps \|_{L^6} \\
	    \leq & \tfrac{C}{\sqrt{\eps}} \| \nabla \u_R^\eps \|_{H^N} \| \nabla \dr_R^\eps \|_{H^N} ( \| \dr_R^\eps \|_{H^N} + \| \nabla \dr_R^\eps \|_{H^N}  ) \\
	    \leq & C \eps \mathscr{E}^\frac{1}{2}_{N,\eps} (t) \mathscr{D}_{N,\eps} (t) \,.
	  \end{aligned}
	\end{equation}
	Similarly as in \eqref{IN4-Sd2-17}, we have
	\begin{equation}\label{IN4-Sd2-18}
	  \begin{aligned}
	    \tfrac{\delta \lambda_2}{\sqrt{\eps}} \l \partial^m ( \A_R^\eps \dr_R^\eps ) , \partial^m \dr_R^\eps \r \leq C \eps \mathscr{E}^\frac{1}{2}_{N,\eps} (t) \mathscr{D}_{N,\eps} (t) \,.
	  \end{aligned}
	\end{equation}
	We next calculate that
	\begin{equation}\label{IN4-Sd2-19}
	  \begin{aligned}
	    & - \tfrac{\delta \lambda_2}{\sqrt{\eps}} \l \partial^m [ \A_0 : ( \dr_0 \otimes \dr_R^\eps + \dr_R^\eps \otimes \dr_0 ) \dr_R^\eps ] , \partial^m \dr_R^\eps \r \\
	    \leq & \tfrac{C}{\sqrt{\eps}} \sum_{\substack{ m' \leq m \\ m'' \leq m' }} \| \partial^{m-m'} ( \A_0 \dr_0 ) \|_{L^2} \| \partial^{m''} \dr_R^\eps \|_{L^6} \| \partial^{m'-m''} \dr_R^\eps \|_{L^6} \| \partial^m \dr_R^\eps \|_{L^6} \\
	    \leq & \tfrac{C}{\sqrt{\eps}} \| \u_0 \|_{H^{N+1}} ( 1 + \| \nabla \dr_0 \|_{H^{N+1}} ) \| \nabla \dr_R^\eps \|^3_{H^N} \\
	    \leq & C \eps \| \u_0 \|_{H^{N+1}} ( 1 + \| \nabla \dr_0 \|_{H^{N+1}} ) \mathscr{E}^\frac{1}{2}_{N,\eps} (t) \mathscr{D}_{N,\eps} (t) \,.
	  \end{aligned}
	\end{equation}
	One easily derives the following bound
	\begin{equation}\label{IN4-Sd2-20}
	  \begin{aligned}
	    & - \tfrac{\delta \lambda_2}{\sqrt{\eps}} \l \partial^m [ ( \A_R^\eps : \dr_0 \otimes \dr_0 ) \dr_R^\eps ] , \partial^m \dr_R^\eps \r \\
	    \leq & \tfrac{C}{\sqrt{\eps}} \sum_{m' \leq m} \| \partial^{m-m'} \A_R^\eps \|_{L^2} \| \partial^{m'} ( (\dr_0 \otimes \dr_0) \dr_R^\eps ) \|_{L^3} \| \partial^m \dr_R^\eps \|_{L^6} \\
	    \leq & \tfrac{C}{\sqrt{\eps}} \| \nabla \u_R^\eps \|_{H^N} \| \nabla \dr_R^\eps \|_{H^N} ( \| \dr_R^\eps \|_{H^N} + \| \nabla \dr_R^\eps \|_{H^N} ) ( 1 + \| \nabla \dr_0 \|^2_{H^N} ) \\
	    \leq & C ( 1 + \| \nabla \dr_0 \|^2_{H^N} ) \mathscr{E}^\frac{1}{2}_{N,\eps} (t) \mathscr{D}_{N,\eps} (t) \,.
	  \end{aligned}
	\end{equation}
	Via the analogous calculations in \eqref{IN4-Sd2-20}, we imply that
	\begin{equation}\label{IN4-Sd2-21}
	  \begin{aligned}
	    - \tfrac{\delta \lambda_2}{\sqrt{\eps}} \l \partial^m [ \A_R^\eps : ( \dr_0 \otimes \dr_R^\eps + \dr_R^\eps \otimes \dr_0 ) \dr_0 ] , \partial^m \dr_R^\eps \r \leq C ( 1 + \| \nabla \dr_0 \|^2_{H^N} ) \mathscr{E}^\frac{1}{2}_{N,\eps} (t) \mathscr{D}_{N,\eps} (t) \,.
	  \end{aligned}
	\end{equation}
	Furthermore, the similar calculations in \eqref{IN4-Sd2-19} tell us
	\begin{equation}\label{IN4-Sd2-22}
	  \begin{aligned}
	    - \tfrac{\delta \lambda_2}{\sqrt{\eps}} \l \partial^m [ ( \A_0 : \dr_R^\eps \otimes \dr_R^\eps ) \dr_0 ] , \partial^m \dr_R^\eps \r \leq C \eps \| \u_0 \|_{H^{N+1}} ( 1 + \| \nabla \dr_0 \|_{H^{N+1}} ) \mathscr{E}^\frac{1}{2}_{N,\eps} (t) \mathscr{D}_{N,\eps} (t) \,. \,.
	  \end{aligned}
	\end{equation}
	Recalling that the definition of $\mathcal{S}_{\dr}^2$ in \eqref{Sd-2}, we imply by collecting the bounds \eqref{IN4-Sd2-13}, \eqref{IN4-Sd2-14}, \eqref{IN4-Sd2-15}, \eqref{IN4-Sd2-16}, \eqref{IN4-Sd2-17}, \eqref{IN4-Sd2-18}, \eqref{IN4-Sd2-19}, \eqref{IN4-Sd2-20}, \eqref{IN4-Sd2-21} and \eqref{IN4-Sd2-22} that
	\begin{equation}\label{IN4-Sd2-23}
	  \begin{aligned}
	    & \tfrac{\delta}{\sqrt{\eps}} \l \partial^m \mathcal{S}_{\dr}^2 , \partial^m \dr_R^\eps \r \\
	    \leq & C \eps ( 1 + \| \u_0 \|^2_{H^{N+2}} + \| \nabla \dr_0 \|^2_{H^{N+2}} ) \mathscr{E}^\frac{1}{\eps}_{N,\eps} (t) \mathscr{D}_{N,\eps} (t) \\
	    & + C ( 1 + \| \u_0 \|^8_{H^{N+2}} + \| \nabla \dr_0 \|^8_{H^{N+2}} ) ( \| \nabla \u_0 \|_{H^{N+2}} + \| \Delta \dr_0 \|_{H^{N+2}} ) \mathscr{D}^\frac{1}{\eps}_{N,\eps} (t) \,.
 	  \end{aligned}
	\end{equation}
	From the inequalities \eqref{IN4-Sd2-11}, \eqref{IN4-Sd2-12} and \eqref{IN4-Sd2-23}, we deduce that
	\begin{equation}\label{IN4-Sd2}
	  \begin{aligned}
	    \mathcal{I}_N^{(4)} ( \mathcal{S}_{\dr}^2 ) = & \tfrac{1}{\sqrt{\eps}} \l \partial^m \mathcal{S}_{\dr}^2 , \partial^m \D_{\u_0 + \sqrt{\eps} \u_R^\eps } \dr_R^\eps - \partial^m (\u_R^\eps \cdot \nabla \dr_0) + \delta \partial^m \dr_R^\eps \r \\
	    \leq & C \sqrt{\eps} ( 1 + \| \u_0 \|^3_{H^{N+2} } + \| \nabla \dr_0 \|^3_{H^{N+2}} ) \mathscr{E}^\frac{1}{2}_{N,\eps} (t) \mathscr{D}_{N,\eps} (t) \\
	    & + C ( 1 + \| \u_0 \|^9_{H^{N+2}} + \| \nabla \dr_0 \|^9_{H^{N+2}} ) ( \| \nabla \u_0 \|_{H^{N+2}} + \| \Delta \dr_0 \|_{H^{N+2}} ) \mathscr{D}^\frac{1}{2}_{N,\eps} (t) \,.
	  \end{aligned}
	\end{equation}
	
	We next estimate the quantity $\mathcal{I}_N^{(4)} (\mathcal{R}_{\dr})$ for the integer $N \geq 2$. We start with the term $ \l \partial^m \mathcal{R}_{\dr} , \partial^m \D_{\u_0 + \sqrt{\eps} \u_R^\eps } \dr_R^\eps \r $ for $|m| \leq N$. First, we have
	\begin{equation}\label{IN4-Rd-1}
	  \begin{aligned}
	    & \l \partial^m ( |\nabla \dr_R^\eps|^2 \dr_R^\eps ) , \partial^m \D_{\u_0 + \sqrt{\eps} \u_R^\eps } \dr_R^\eps \r \\
	    \leq & \| \partial^m |\nabla \dr_R^\eps|^2 \|_{L^2} \| \dr_R^\eps \|_{L^\infty} \| \partial^m \D_{\u_0 + \sqrt{\eps} \u_R^\eps } \dr_R^\eps \|_{L^2} \\
	    & + C \sum_{0 \neq m' \leq m} \| \partial^{m-m'} |\nabla \dr_R^\eps|^2 \|_{L^4} \| \partial^{m'} \dr_R^\eps \|_{L^4} \| \partial^m \D_{\u_0 + \sqrt{\eps} \u_R^\eps } \dr_R^\eps \|_{L^2} \\
	    \leq & C \| \nabla \dr_R^\eps \|^2_{H^N} \| \D_{\u_0 + \sqrt{\eps} \u_R^\eps } \dr_R^\eps \|_{H^N} ( \| \dr_R^\eps \|_{H^N} + \| \nabla \dr_R^\eps \|_{H^N} ) \\
	    \leq & C \sqrt{\eps}^3 \mathscr{E}_{N,\eps} (t) \mathscr{D}_{N,\eps} (t) \,.
	  \end{aligned}
	\end{equation}
	Via the inequality \eqref{Du0-d0-HN}, we deduce that
	\begin{equation}\label{IN4-Rd-2}
	  \begin{aligned}
	    & - \l \partial^m ( |\D_{\u_0} \dr_0|^2 \dr_R^\eps ) , \partial^m \D_{\u_0 + \sqrt{\eps} \u_R^\eps } \dr_R^\eps \r \\
	    \leq & C \sum_{m' \leq m} \| \partial^{m-m'} |\D_{\u_0} \dr_0|^2 \|_{L^3} \| \partial^{m'} \dr_R^\eps \|_{L^6} \| \partial^m \D_{\u_0 + \sqrt{\eps} \u_R^\eps } \dr_R^\eps \|_{L^2} \\
	    \leq & C \| \D_{\u_0} \dr_0 \|^2_{H^{N+1}} \| \nabla \dr_R^\eps \|_{H^N} \| \D_{\u_0 + \sqrt{\eps} \u_R^\eps } \dr_R^\eps \|_{H^N} \\
	    \leq & C \sqrt{\eps} ( 1 + \| \nabla \dr_0 \|^6_{H^{N+1}} ) ( \| \nabla \u_0 \|^2_{H^{N+1}} + \| \Delta \dr_0 \|_{H^{N+1}} ) \mathscr{D}_{N,\eps} (t) \\
	    \leq & C \sqrt{\eps} ( 1 + \| \nabla \dr_0 \|^6_{H^{N+2}} ) ( \| \u_0 \|^2_{H^{N+2}} + \| \nabla \dr_0 \|^2_{H^{N+2}} ) \mathscr{D}_{N,\eps} (t) \,.
	  \end{aligned}
	\end{equation}
	It is easily derived that
	\begin{equation}\label{IN4-Rd-3}
	  \begin{aligned}
	    & - 2 \l \partial^m [ \D_{\u_0} \dr_0 \cdot ( \D_{\u_0 + \sqrt{\eps} \u_R^\eps } \dr_R^\eps + \u_R^\eps \cdot \nabla \dr_0 ) \dr_0 ] , \partial^m \D_{\u_0 + \sqrt{\eps} \u_R^\eps } \dr_R^\eps \r \\
	    \leq  & C \| \D_{\u_0} \dr_0 \otimes \dr_0 \|_{L^\infty} \| \partial^m \D_{\u_0 + \sqrt{\eps} \u_R^\eps } \dr_R^\eps \|^2_{L^2} \\
	    & + C \sum_{0 \neq m' \leq m} \| \partial^{m'} ( \D_{\u_0} \dr_0 \otimes \dr_0 ) \|_{L^4} \| \partial^{m-m'} \D_{\u_0 + \sqrt{\eps} \u_R^\eps } \dr_R^\eps \|_{L^4} \| \partial^m \D_{\u_0 + \sqrt{\eps} \u_R^\eps } \dr_R^\eps \|_{L^2} \\
	    & + C \sum_{m' \leq m} \| \partial^{m-m'} ( \D_{\u_0} \dr_0 \otimes \dr_0 \otimes \nabla \dr_0 ) \|_{L^3} \| \partial^{m'} \dr_R^\eps \|_{L^6} \| \partial^m \D_{\u_0 + \sqrt{\eps} \u_R^\eps } \dr_R^\eps \|_{L^2} \\
	    \leq & C \| \D_{\u_0} \dr_0 \|_{H^N} ( 1 + \| \nabla \dr_0 \|_{H^N} ) \| \D_{\u_0 + \sqrt{\eps} \u_R^\eps } \dr_R^\eps \|^2_{H^N} \\
	    & + C \| \D_{\u_0} \dr_0 \|_{H^{N+1}} \| \u_0 \|_{H^{N+1}} ( 1 + \| \nabla \dr_0 \|_{H^{N+1}} ) \| \nabla \u_R^\eps \|_{H^N} \| \D_{\u_0 + \sqrt{\eps} \u_R^\eps } \dr_R^\eps \|_{H^N} \\
	    \leq & C \sqrt{\eps} \| \D_{\u_0} \dr_0 \|_{H^{N+1}} ( 1 + \| \u_0 \|^2_{H^{N+1}} + \| \nabla \dr_0 \|^2_{H^{N+1}} ) \mathscr{D}_{N, \eps} (t) \\
	    \leq & C \sqrt{\eps} ( 1 + \| \u_0 \|^5_{H^{N+2}} + \| \nabla \dr_0 \|^5_{H^{N+2}} ) ( \| \u_0 \|_{H^{N+2}} + \| \nabla \dr_0 \|_{H^{N+2}} ) \mathscr{D}_{N,\eps} (t)
	  \end{aligned}
	\end{equation}
	for all $|m| \leq N$ and $0 < \eps \leq 1$, where the bounds \eqref{Du0-d0-HN} is also utilized.
	
	For the term $ - \lambda_2 \l \partial^m [ ( \A_R^\eps : \dr_R^\eps \otimes \dr_R^\eps ) \dr_0 ] , \partial^m \D_{\u_0 + \sqrt{\eps} \u_R^\eps } \dr_R^\eps \r $, we estimate that
	\begin{equation}\label{IN4-Rd-4}
	  \begin{aligned}
	    & - \lambda_2 \l \partial^m [ ( \A_R^\eps : \dr_R^\eps \otimes \dr_R^\eps ) \dr_0 ] , \partial^m \D_{\u_0 + \sqrt{\eps} \u_R^\eps } \dr_R^\eps \r \\
	    \leq & C \| \partial^m \A_R^\eps \|_{L^2} \| ( \dr_R^\eps \otimes \dr_R^\eps ) \dr_0 \|_{L^\infty} \| \partial^m \D_{\u_0 + \sqrt{\eps} \u_R^\eps } \dr_R^\eps \|_{L^2} \\
	    & + C \sum_{0 \neq m' \leq m} \| \partial^{m-m'} \A_R^\eps \|_{L^2} \\
	    \leq & C \| \nabla \u_R^\eps \|_{H^N} \| \D_{\u_0 + \sqrt{\eps} \u_R^\eps } \dr_R^\eps \|_{H^N} ( \| \dr_R^\eps \|^2_{H^N} + \| \nabla \dr_R^\eps \|^2_{H^N} ) ( 1 + \| \nabla \dr_0 \|_{H^{N+1}} ) \\
	    \leq  & C \sqrt{\eps}^3 ( 1 + \| \nabla \dr_0 \|_{H^{N+1}} ) \mathscr{E}_{N,\eps} (t) \mathscr{D}_{N,\eps} (t) \,.
	  \end{aligned}
	\end{equation}
    Furthermore, by the similar arguments as in \eqref{IN4-Rd-4}, we immediately have
    \begin{equation}\label{IN4-Rd-5}
      \begin{aligned}
        & - \lambda_1 \l \partial^m [ \A_R^\eps : ( \dr_0 \otimes \dr_R^\eps + \dr_R^\eps \otimes \dr_0 ) \dr_R^\eps ] , \partial^m \D_{\u_0 + \sqrt{\eps} \u_R^\eps } \dr_R^\eps \r \\
        & \leq  C \sqrt{\eps}^3 ( 1 + \| \nabla \dr_0 \|_{H^{N+1}} ) \mathscr{E}_{N,\eps} (t) \mathscr{D}_{N,\eps} (t) \,,
      \end{aligned}
    \end{equation}
    and
    \begin{equation}\label{IN4-Rd-6}
      \begin{aligned}
        & - 2 \sqrt{\eps} \lambda_2 \l \partial^m [ ( \A_R^\eps : \dr_R^\eps \otimes \dr_R^\eps ) \dr_R^\eps ] , \partial^m \D_{\u_0 + \sqrt{\eps} \u_R^\eps } \dr_R^\eps \r \\
        \leq & C \sqrt{\eps} \| \nabla \u_R^\eps \|_{H^N} \| \D_{\u_0 + \sqrt{\eps} \u_R^\eps } \dr_R^\eps \|_{H^N} ( \| \dr_R^\eps \|^3_{H^N} + \| \nabla \dr_R^\eps \|^3_{H^N} ) \\
        \leq & C \eps^3 \mathscr{E}^\frac{3}{2}_{N,\eps} (t) \mathscr{D}_{N,\eps} (t) \,.
      \end{aligned}
    \end{equation}
    We now calculate that
    \begin{equation}\label{IN4-Rd-7}
      \begin{aligned}
        & - \lambda_2 \l \partial^m [ ( \A_0 : \dr_R^\eps \otimes \dr_R^\eps ) \dr_R^\eps ] , \partial^m \D_{\u_0 + \sqrt{\eps} \u_R^\eps } \dr_R^\eps \r \\
        \leq & C \sum_{m' \leq m} \| \partial^{m-m'} ( \A_0 : \dr_R^\eps \otimes \dr_R^\eps ) \|_{L^3} \| \partial^{m'} \dr_R^\eps \|_{L^6} \| \partial^m \D_{\u_0 + \sqrt{\eps} \u_R^\eps } \dr_R^\eps \|_{L^2} \\
        \leq & C \| \u_0 \|_{H^{N+2}} ( \| \dr_R^\eps \|^2_{H^N} + \| \nabla \dr_R^\eps \|_{H^N} ) \| \nabla \dr_R^\eps \|_{H^N} \| \D_{\u_0 + \sqrt{\eps} \u_R^\eps } \dr_R^\eps \|_{H^N} \\
        \leq & C \sqrt{\eps}^3 \| \u_0 \|_{H^{N+2}} \mathscr{E}_{N,\eps} (t) \mathscr{D}_{N,\eps} (t) \,.
      \end{aligned}
    \end{equation}
    The following bound holds for all $|m| \leq N$ and $0 < \eps \leq 1$:
    \begin{equation}\label{IN4-Rd-8}
      \begin{aligned}
        & - \sqrt{\eps} \l \partial^m ( |\D_{\u_0 + \sqrt{\eps} \u_R^\eps } \dr_R^\eps + \u_R^\eps \cdot \nabla \dr_0|^2 \dr_0 ) , \partial^m \D_{\u_0 + \sqrt{\eps} \u_R^\eps } \dr_R^\eps \r \\
        \leq & C \sqrt{\eps} \sum_{m' \leq m} \| \partial^{m-m'} | \D_{\u_0 + \sqrt{\eps} \u_R^\eps } \dr_R^\eps + \u_R^\eps \cdot \nabla \dr_0 |^2 \|_{L^2} \| \partial^{m'} \dr_0 \|_{L^\infty} \| \partial^m \D_{\u_0 + \sqrt{\eps} \u_R^\eps } \dr_R^\eps \|_{L^2} \\
        \leq & C \sqrt{\eps} ( \| \D_{\u_0 + \sqrt{\eps} \u_R^\eps } \dr_R^\eps \|^2_{H^N} + \| \u_R^\eps \cdot \nabla \dr_0 \|^2_{H^N} ) ( 1 + \| \nabla \dr_0 \|_{H^{N+1}} ) \| \D_{\u_0 + \sqrt{\eps} \u_R^\eps } \dr_R^\eps \|_{H^N} \\
        \leq & C \sqrt{\eps} ( 1 + \| \nabla \dr_0 \|_{H^{N+1}} ) \| \D_{\u_0 + \sqrt{\eps} \u_R^\eps } \dr_R^\eps \|_{H^N} ( \| \D_{\u_0 + \sqrt{\eps} \u_R^\eps } \dr_R^\eps \|^2_{H^N} + \| \nabla \u_R^\eps \|^2_{H^N} \| \nabla \dr_0 \|^2_{H^{N+1}} ) \\
        \leq & C \sqrt{\eps} ( 1 + \| \u_0 \|^3_{H^{N+1}} + \| \nabla \dr_0 \|^3_{H^{N+1}} ) \mathscr{E}^\frac{1}{2}_{N,\eps} (t) \mathscr{D}_{N,\eps} (t) \,,
      \end{aligned}
    \end{equation}
    where the inequality \eqref{uR-nabla-d0} is also used. We then estimate that
      \begin{align}\label{IN4-Rd-9}
        \no & - 2 \sqrt{\eps} \l \partial^m [ \D_{\u_0} \dr_0 \cdot ( \D_{\u_0 + \sqrt{\eps} \u_R^\eps } \dr_R^\eps + \u_R^\eps \cdot \nabla \dr_0 ) \dr_R^\eps ] , \partial^m \D_{\u_0 + \sqrt{\eps} \u_R^\eps } \dr_R^\eps \r \\
        \no \leq & C \sqrt{\eps} \| \partial^m [ ( \D_{\u_0} \dr_0 \cdot \D_{\u_0 + \sqrt{\eps} \u_R^\eps } \dr_R^\eps ) \dr_R^\eps ] \|_{L^2} \| \partial^m \D_{\u_0 + \sqrt{\eps} \u_R^\eps } \dr_R^\eps \|_{L^2} \\
        \no & + C \sqrt{\eps} \sum_{m' \leq m} \| \partial^{m-m'} [ \D_{\u_0} \dr_0 \cdot ( \u_R^\eps \cdot \nabla \dr_0 ) ] \|_{L^3} \| \partial^{m'} \dr_R^\eps \|_{L^6} \| \partial^m \D_{\u_0 + \sqrt{\eps} \u_R^\eps } \dr_R^\eps \|_{L^2} \\
        \no \leq & C \sqrt{\eps} \| \D_{\u_0} \dr_0 \|_{H^N} \| \dr_R^\eps \|_{H^N} \| \D_{\u_0 + \sqrt{\eps} \u_R^\eps } \dr_R^\eps \|^2_{H^N} \\
        \no & + C \sqrt{\eps} \| \D_{\u_0} \dr_0 \|_{H^{N+1}} \| \nabla \dr_0 \|_{H^{N+1}} ( \| \u_R^\eps \|_{H^N} + \| \nabla \u_R^\eps \|_{H^N} ) \| \nabla \dr_R^\eps \|_{H^N} \| \D_{\u_0 + \sqrt{\eps} \u_R^\eps } \dr_R^\eps \|_{H^N} \\
        \no \leq & C \eps \| \D_{\u_0} \dr_0 \|_{H^N} ( 1 + \| \nabla \dr_0 \|_{H^{N+1}} ) \mathscr{E}^\frac{1}{2}_{N,\eps} (t) \mathscr{D}_{N,\eps} (t) \\
        \no \leq & C \eps ( 1 + \| \nabla \dr_0 \|^4_{H^{N+1}} + \| \u_0 \|^4_{H^{N+1}} ) ( \| \u_0 \|_{H^{N+1}} + \| \nabla \dr_0 \|_{H^{N+1}} ) \mathscr{E}^\frac{1}{2}_{N,\eps} (t) \mathscr{D}_{N,\eps} (t) \\
        \leq & C \eps ( 1 + \| \u_0 \|^5_{H^{N+1}} + \| \nabla \dr_0 \|^5_{H^{N+1}} ) \mathscr{E}^\frac{1}{2}_{N,\eps} (t) \mathscr{D}_{N,\eps} (t)
      \end{align}
    for all $|m| \leq N$ and $0 < \eps \leq 1$, where the last second inequality is implied the bound \eqref{uR-nabla-d0}. Via the same calculations as in the bound \eqref{IN4-Rd-8}, we have
    \begin{equation}\label{IN4-Rd-10}
      \begin{aligned}
        & - \eps \l \partial^m ( | \D_{\u_0 + \sqrt{\eps} \u_R^\eps } \dr_R^\eps + \u_R^\eps \dr_0 |^2 \dr_R^\eps ) , \partial^m \D_{\u_0 + \sqrt{\eps} \u_R^\eps } \dr_R^\eps \r \\
        & \leq C \eps ( 1 + \| \u_0 \|^2_{H^{N+1}} + \| \nabla \dr_0 \|^2_{H^{N+1}} ) \mathscr{E}_{N,\eps} (t) \mathscr{D}_{N,\eps} (t) \,.
      \end{aligned}
    \end{equation}
    It is easy to deduce that
    \begin{equation}\label{IN4-Rd-11}
      \begin{aligned}
        & - \l \partial^m [ \u_R^\eps \cdot \nabla ( \D_{\u_0} \dr_0 ) ] , \partial^m \D_{\u_0 + \sqrt{\eps} \u_R^\eps } \dr_R^\eps \r \\
        \leq & \| \partial^m [ \u_R^\eps \cdot \nabla ( \D_{\u_0} \dr_0 ) ] \|_{L^2} \| \partial^m \D_{\u_0 + \sqrt{\eps} \u_R^\eps } \dr_R^\eps \|_{L^2} \\
        \leq & C \| \u_R^\eps \|_{H^N} \| \nabla ( \D_{\u_0} \dr_0  ) \|_{H^N} \| \D_{\u_0 + \sqrt{\eps} \u_R^\eps } \dr_R^\eps \|_{H^N} \\
        \leq & C ( 1 + \| \nabla \dr_0 \|^3_{H^{N+1}} ) ( \| \nabla \u_0 \|_{H^{N+1}} + \| \Delta \dr_0 \|_{H^{N+1}} ) \| \u_R^\eps \|_{H^N} \| \D_{\u_0 + \sqrt{\eps} \u_R^\eps } \dr_R^\eps \|_{H^N} \\
        \leq & C \sqrt{\eps} ( 1 + \| \nabla \dr_0 \|^3_{H^{N+1}} ) ( \| \nabla \u_0 \|_{H^{N+1}} + \| \Delta \dr_0 \|_{H^{N+1}} ) \mathscr{E}^\frac{1}{2}_{N,\eps} (t) \mathscr{D}^\frac{1}{2}_{N,\eps} (t) \,,
      \end{aligned}
    \end{equation}
    where the last second inequality is derived from the bound \eqref{Du0-d0-HN}. We next estimate that
    \begin{equation}\label{IN4-Rd-12}
      \begin{aligned}
        & - \l \partial^m [ \u_0 \cdot \nabla ( \u_R^\eps \cdot \nabla \dr_0 ) ] , \partial^m \D_{\u_0 + \sqrt{\eps} \u_R^\eps } \dr_R^\eps \r \\
        = & - \l \partial^m ( \u_0 \cdot \nabla \u_R^\eps \cdot \nabla \dr_0 + \u_0 \cdot \u_R^\eps \cdot \nabla \nabla \dr_0 ) , \partial^m \D_{\u_0 + \sqrt{\eps} \u_R^\eps } \dr_R^\eps \r \\
        \leq & C \| \nabla \partial^m \u_R^\eps \|_{L^2} \| \u_0 \cdot \nabla \dr_0 \|_{L^\infty} \| \partial^m \D_{\u_0 + \sqrt{\eps} \u_R^\eps } \dr_R^\eps \|_{L^2} \\
        & + \sum_{0 \neq m' \leq m} \| \nabla \partial^{m-m'} \u_R^\eps \|_{L^4} \| \partial^{m'} ( \u_0 \cdot \nabla \dr_0 ) \|_{L^4} \| \partial^m \D_{\u_0 + \sqrt{\eps} \u_R^\eps } \dr_R^\eps \|_{L^2} \\
        & + C \sum_{m' \leq m} \| \partial^{m-m'} ( \u_0 \cdot \nabla \u_0 \cdot \nabla \nabla \dr_0 ) \|_{L^3} \| \partial^{m'} \u_R^\eps \|_{L^6} \| \partial^m \D_{\u_0 + \sqrt{\eps} \u_R^\eps } \dr_R^\eps \|_{L^2} \\
        \leq & C \| \nabla \u_R^\eps \|_{H^N} \| \D_{\u_0 + \sqrt{\eps} \u_R^\eps } \dr_R^\eps \|_{H^N} \| \u_0 \|_{H^{N+2}} \| \nabla \dr_0 \|_{H^{N+2}} ( 1 + \| \u_0 \|_{H^{N+2}} ) \\
        \leq & C \sqrt{\eps} \| \u_0 \|_{H^{N+2}} \| \nabla \dr_0 \|_{H^{N+2}} ( 1 + \| \u_0 \|_{H^{N+2}} ) \mathscr{D}_{N,\eps} (t) \,.
      \end{aligned}
    \end{equation}
    Furthermore, we have
    \begin{equation}\label{IN4-Rd-13}
      \begin{aligned}
        & - \sqrt{\eps} \l \partial^m [ \u_R^\eps \cdot \nabla ( \u_R^\eps \cdot \nabla \dr_0 ) ] , \partial^m \D_{\u_0 + \sqrt{\eps} \u_R^\eps } \dr_R^\eps \r \\
        \leq & C \sqrt{\eps} \| \u_R^\eps \|^2_{L^6} \| \nabla \partial^m \nabla \dr_0 \|_{L^6} \| \partial^m \D_{\u_0 + \sqrt{\eps} \u_R^\eps } \dr_R^\eps \|_{L^2} \\
        & + C \sqrt{\eps} \sum_{0 \neq m' \leq m} \| \partial^{m'} ( \u_R^\eps \otimes \u_R^\eps ) \|_{L^4} \| \nabla \partial^{m-m'} \nabla \dr_0 \|_{L^4} \| \partial^m \D_{\u_0 + \sqrt{\eps} \u_R^\eps } \dr_R^\eps \|_{L^2} \\
        \leq & C \sqrt{\eps} \| \nabla \dr_0 \|_{H^{N+2}} \| \u_R^\eps \|_{H^N} \| \nabla \u_R^\eps \|_{H^N} \| \D_{\u_0 + \sqrt{\eps} \u_R^\eps } \dr_R^\eps \|_{H^N} \\
        \leq & C \sqrt{\eps}^3 \| \nabla \dr_0 \|_{H^{N+2}} \mathscr{E}^\frac{1}{2}_{N,\eps} (t) \mathscr{D}_{N,\eps} (t) \,.
      \end{aligned}
    \end{equation}
    Recalling the definition of $\mathcal{R}_{\dr}$ in \eqref{Rd} and collecting the previous bounds \eqref{IN4-Rd-1}, \eqref{IN4-Rd-2}, \eqref{IN4-Rd-3}, \eqref{IN4-Rd-4}, \eqref{IN4-Rd-5}, \eqref{IN4-Rd-6}, \eqref{IN4-Rd-7}, \eqref{IN4-Rd-8}, \eqref{IN4-Rd-9}, \eqref{IN4-Rd-10}, \eqref{IN4-Rd-11}, \eqref{IN4-Rd-12} and \eqref{IN4-Rd-13}, we obtain
    \begin{equation}\label{IN4-Rd-14}
      \begin{aligned}
        & \l \partial^m \mathcal{R}_{\dr} , \partial^m \D_{\u_0 + \sqrt{\eps} \u_R^\eps } \dr_R^\eps \r \\
        \leq & C \sqrt{\eps} ( 1 + \| \u_0 \|^5_{H^{N+2}} + \| \nabla \dr_0 \|^5_{H^{N+2}} ) [ \mathscr{E}^\frac{1}{2}_{N,\eps} (t) + \mathscr{E}^2_{N,\eps} (t) ] \mathscr{D}_{N,\eps} (t) \\
        + & C \sqrt{\eps} ( 1 + \| \u_0 \|^7_{H^{N+2}} + \| \nabla \dr_0 \|^7_{H^{N+2}} ) ( \| \u_0 \|_{H^{N+2}} + \| \nabla \dr_0 \|_{H^{N+2}} ) \mathscr{D}_{N,\eps} (t) \\
        + & C \sqrt{\eps} ( 1 + \| \nabla \dr_0 \|^2_{H^{N+2}} ) ( \| \nabla \u_0 \|_{H^{N+2}} + \| \Delta \dr_0 \|_{H^{N+2}} ) \mathscr{E}^\frac{1}{2}_{N,\eps} (t) \mathscr{D}^\frac{1}{2}_{N,\eps} (t)
      \end{aligned}
    \end{equation}
    for all $|m| \leq N$ and $\eps \in ( 0, \eps_0 ] $.

    By the similar estimates on the quantity $ \l \partial^m \mathcal{R}_{\dr} , \partial^m \D_{\u_0 + \sqrt{\eps} \u_R^\eps } \dr_R^\eps \r $ in \eqref{IN4-Rd-14}, we obtain
    \begin{equation}\label{IN4-Rd-15}
      \begin{aligned}
        &  \l \partial^m \mathcal{R}_{\dr} , \partial^m ( \u_R^\eps \cdot \nabla \dr_0 ) \r \\
        \leq & C ( 1 + \| \u_0 \|^5_{H^{N+2}} + \| \nabla \dr_0 \|^5_{H^{N+2}} ) [ \mathscr{E}^\frac{1}{2}_{N,\eps} (t) + \mathscr{E}^2_{N,\eps} (t) ] \mathscr{D}^\frac{1}{2}_{N,\eps} (t) \| \u_R^\eps \cdot \nabla \dr_0 \|_{H^N} \\
        + & C ( 1 + \| \u_0 \|^7_{H^{N+2}} + \| \nabla \dr_0 \|^7_{H^{N+2}} ) ( \| \u_0 \|_{H^{N+2}} + \| \nabla \dr_0 \|_{H^{N+2}} ) \mathscr{D}^\frac{1}{2}_{N,\eps} (t) \| \u_R^\eps \cdot \nabla \dr_0 \|_{H^N} \\
        + & C ( 1 + \| \nabla \dr_0 \|^2_{H^{N+2}} ) ( \| \nabla \u_0 \|_{H^{N+2}} + \| \Delta \dr_0 \|_{H^{N+2}} ) \mathscr{E}^\frac{1}{2}_{N,\eps} (t) \| \u_R^\eps \cdot \nabla \dr_0 \|_{H^N} \\
        \leq & C \sqrt{\eps} ( 1 + \| \u_0 \|^6_{H^{N+2}} + \| \nabla \dr_0 \|^6_{H^{N+2}} ) [ \mathscr{E}^\frac{1}{2}_{N,\eps} (t) + \mathscr{E}^2_{N,\eps} (t) ] \mathscr{D}_{N,\eps} (t) \\
        + & C \sqrt{\eps} ( 1 + \| \u_0 \|^8_{H^{N+2}} + \| \nabla \dr_0 \|^8_{H^{N+2}} ) ( \| \u_0 \|_{H^{N+2}} + \| \nabla \dr_0 \|_{H^{N+2}} ) \mathscr{D}_{N,\eps} (t) \\
        + & C \sqrt{\eps} ( 1 + \| \nabla \dr_0 \|^3_{H^{N+2}} ) ( \| \nabla \u_0 \|_{H^{N+2}} + \| \Delta \dr_0 \|_{H^{N+2}} ) \mathscr{E}^\frac{1}{2}_{N,\eps} (t) \mathscr{D}^\frac{1}{2}_{N,\eps} (t) \,,
      \end{aligned}
    \end{equation}
    where we make use of the inequality \eqref{uR-nabla-d0}, hence for $N \geq 2$
    \begin{equation*}
      \begin{aligned}
        \| \u_R^\eps \cdot \nabla \dr_0 \|_{H^N} \leq C \| \nabla \u_R^\eps \|_{H^N} \| \nabla \dr_0 \|_{H^{N+1}} \leq C \sqrt{\eps} \| \nabla \dr_0 \|_{H^{N+1}} \mathscr{D}^\frac{1}{2}_{N,\eps} (t) \,.
      \end{aligned}
    \end{equation*}

    Notice that the bound
    \begin{equation}
    \begin{aligned}
    \| \partial^m \dr_R^\eps \|_{L^2} \leq \| \partial^m \dr_R^\eps \|_{L^6} |\T^3|^\frac{1}{3} \leq C \| \nabla \dr_R^\eps \|_{H^N} \leq C \sqrt{\eps} \mathscr{D}^\frac{1}{2}_{N,\eps} (t)
    \end{aligned}
    \end{equation}
    holds for all $|m| \leq N$. Then, from the similar arguments in \eqref{IN4-Rd-14} and the previous bound, we deduce that
    \begin{equation}\label{IN4-Rd-16}
    \begin{aligned}
    &  \delta \l \partial^m \mathcal{R}_{\dr} , \partial^m \dr_R^\eps \r \\
    \leq & C ( 1 + \| \u_0 \|^5_{H^{N+2}} + \| \nabla \dr_0 \|^5_{H^{N+2}} ) [ \mathscr{E}^\frac{1}{2}_{N,\eps} (t) + \mathscr{E}^2_{N,\eps} (t) ] \mathscr{D}^\frac{1}{2}_{N,\eps} (t) \| \partial^m \dr_R^\eps \|_{L^2} \\
    + & C ( 1 + \| \u_0 \|^7_{H^{N+2}} + \| \nabla \dr_0 \|^7_{H^{N+2}} ) ( \| \u_0 \|_{H^{N+2}} + \| \nabla \dr_0 \|_{H^{N+2}} ) \mathscr{D}^\frac{1}{2}_{N,\eps} (t) \| \partial^m \dr_R^\eps \|_{L^2} \\
    + & C ( 1 + \| \nabla \dr_0 \|^2_{H^{N+2}} ) ( \| \nabla \u_0 \|_{H^{N+2}} + \| \Delta \dr_0 \|_{H^{N+2}} ) \mathscr{E}^\frac{1}{2}_{N,\eps} (t) \| \partial^m \dr_R^\eps \|_{L^2} \\
    \leq & C \sqrt{\eps} ( 1 + \| \u_0 \|^5_{H^{N+2}} + \| \nabla \dr_0 \|^5_{H^{N+2}} ) [ \mathscr{E}^\frac{1}{2}_{N,\eps} (t) + \mathscr{E}^2_{N,\eps} (t) ] \mathscr{D}_{N,\eps} (t) \\
    + & C \sqrt{\eps} ( 1 + \| \u_0 \|^7_{H^{N+2}} + \| \nabla \dr_0 \|^7_{H^{N+2}} ) ( \| \u_0 \|_{H^{N+2}} + \| \nabla \dr_0 \|_{H^{N+2}} ) \mathscr{D}_{N,\eps} (t) \\
    + & C \sqrt{\eps} ( 1 + \| \nabla \dr_0 \|^2_{H^{N+2}} ) ( \| \nabla \u_0 \|_{H^{N+2}} + \| \Delta \dr_0 \|_{H^{N+2}} ) \mathscr{E}^\frac{1}{2}_{N,\eps} (t) \mathscr{D}^\frac{1}{2}_{N,\eps} (t) \,.
    \end{aligned}
    \end{equation}
    Finally, from the inequalities \eqref{IN4-Rd-14}, \eqref{IN4-Rd-15} and \eqref{IN4-Rd-16}, we derive that
    \begin{equation}\label{IN4-Rd}
      \begin{aligned}
        \mathcal{I}_N^{(4)} (\mathcal{R}_{\dr}) = & \sum_{|m| \leq N} \l \partial^m \mathcal{R}_{\dr} , \partial^m \D_{\u_0 + \sqrt{\eps} \u_R^\eps } \dr_R^\eps - \partial^m ( \u_R^\eps \cdot \nabla \dr_0) + \delta \partial^m \dr_R^\eps \r \\
        \leq & C \sqrt{\eps} ( 1 + \| \u_0 \|^6_{H^{N+2}} + \| \nabla \dr_0 \|^6_{H^{N+2}} ) [ \mathscr{E}^\frac{1}{2}_{N,\eps} (t) + \mathscr{E}^2_{N,\eps} (t) ] \mathscr{D}_{N,\eps} (t) \\
        + & C \sqrt{\eps} ( 1 + \| \u_0 \|^8_{H^{N+2}} + \| \nabla \dr_0 \|^8_{H^{N+2}} ) ( \| \u_0 \|_{H^{N+2}} + \| \nabla \dr_0 \|_{H^{N+2}} ) \mathscr{D}_{N,\eps} (t) \\
        + & C \sqrt{\eps} ( 1 + \| \nabla \dr_0 \|^3_{H^{N+2}} ) ( \| \nabla \u_0 \|_{H^{N+2}} + \| \Delta \dr_0 \|_{H^{N+2}} ) \mathscr{E}^\frac{1}{2}_{N,\eps} (t) \mathscr{D}^\frac{1}{2}_{N,\eps} (t)
      \end{aligned}
    \end{equation}
    holds for all $N \geq 2$ and $ 0 < \eps \leq \eps_0 $.

    As a consequence, substituting the inequalities \eqref{IN4-Cd}, \eqref{IN4-AB}, \eqref{IN4-Sd1}, \eqref{IN4-Sd2} and \eqref{IN4-Rd} into \eqref{IN4-1} reduces to
    \begin{equation}\label{IN4-Bnd}
      \begin{aligned}
        \mathcal{I}_N^{(4)} \leq & C ( 1 + \| \u_0 \|^9_{H^{N+2}} + \| \nabla \dr_0 \|^9_{H^{N+2}} ) ( \| \nabla \u_0 \|_{H^{N+2}} + \| \Delta \dr_0 \|_{H^{N+2}} ) \mathscr{D}^\frac{1}{2}_{N,\eps} (t) \\
        + & C \sqrt{\eps} ( 1 + \| \u_0 \|^6_{H^{N+2}} + \| \nabla \dr_0 \|^6_{H^{N+2}} ) [ \mathscr{E}^\frac{1}{2}_{N,\eps} (t) + \mathscr{E}^2_{N,\eps} (t) ] \mathscr{D}_{N,\eps} (t) \\
        + & C \sqrt{\eps} ( 1 + \| \u_0 \|^8_{H^{N+2}} + \| \nabla \dr_0 \|^8_{H^{N+2}} ) ( \| \u_0 \|_{H^{N+2}} + \| \nabla \dr_0 \|_{H^{N+2}} ) \mathscr{D}_{N,\eps} (t) \\
        + & C \sqrt{\eps} ( 1 + \| \nabla \dr_0 \|^3_{H^{N+2}} ) ( \| \nabla \u_0 \|_{H^{N+2}} + \| \Delta \dr_0 \|_{H^{N+2}} ) \mathscr{E}^\frac{1}{2}_{N,\eps} (t) \mathscr{D}^\frac{1}{2}_{N,\eps} (t)
      \end{aligned}
    \end{equation}
    for all $N \geq 2$ and $ 0 < \eps \leq \eps_0 $.\\

    {\em Step 5. Close the a priori uniform energy estimates.}

    Via plugging the bounds \eqref{IN1-Bnd}, \eqref{IN2-Bnd}, \eqref{IN3-Bnd} and \eqref{IN4-Bnd} into the relation \eqref{u-d-decay-2}, we obtain
    \begin{equation}
      \begin{aligned}
        & \tfrac{1}{2} \tfrac{\d}{\d t} \mathscr{E}_{N,\eps} (t) + \mathscr{D}_{N,\eps} (t) \\
        \leq & C  ( \| \nabla \u_0 \|_{H^{N+2}} + \| \Delta \dr_0 \|_{H^{N+2}} ) \mathscr{D}^\frac{1}{2}_{N,\eps} (t) \\
        + & C ( \| \u_0 \|^9_{H^{N+2}} + \| \nabla \dr_0 \|^9_{H^{N+2}} ) ( \| \nabla \u_0 \|_{H^{N+2}} + \| \Delta \dr_0 \|_{H^{N+2}} ) \mathscr{D}^\frac{1}{2}_{N,\eps} (t) \\
        + & C \sqrt{\eps} ( 1 + \| \u_0 \|^6_{H^{N+2}} + \| \nabla \dr_0 \|^6_{H^{N+2}} ) [ \mathscr{E}^\frac{1}{2}_{N,\eps} (t) + \mathscr{E}^2_{N,\eps} (t) ] \mathscr{D}_{N,\eps} (t) \\
        + & C \sqrt{\eps} ( 1 + \| \u_0 \|^8_{H^{N+2}} + \| \nabla \dr_0 \|^8_{H^{N+2}} ) ( \| \u_0 \|_{H^{N+2}} + \| \nabla \dr_0 \|_{H^{N+2}} ) \mathscr{D}_{N,\eps} (t) \\
        + & C \sqrt{\eps} ( 1 + \| \nabla \dr_0 \|^3_{H^{N+2}} ) ( \| \nabla \u_0 \|_{H^{N+2}} + \| \Delta \dr_0 \|_{H^{N+2}} ) \mathscr{E}^\frac{1}{2}_{N,\eps} (t) \mathscr{D}^\frac{1}{2}_{N,\eps} (t)
      \end{aligned}
    \end{equation}
    for all $0 < \eps \leq \eps_0 \leq 1$. Furthermore, by the relations \eqref{Bnd-Es0-Ds0} and the bound \eqref{u0-d0-Bnd} we know that
    \begin{equation}
      \begin{aligned}
        & \| \nabla \dr_0 \|^2_{H^{N+2}} + \| \u_0 \|^2_{H^{N+2}} \leq c_0^{-1} \mathscr{E}_{S_{\!_N} , 0} (t) \leq C(\beta_{S_{\!_N},0}) \,, \\
        & \| \nabla \u_0 \|^2_{H^{N+2}} + \| \Delta \dr_0 \|^2_{H^{N+2}} \leq c_0^{-1} \mathscr{D}_{S_{\!_N} , 0} (t) \,.
      \end{aligned}
    \end{equation}
    As a result, we deduce that for all $N \geq 2$ and $0 < \eps \leq \eps_0$
    \begin{equation}
      \begin{aligned}
        & \tfrac{1}{2} \tfrac{\d}{\d t} \mathscr{E}_{N,\eps} (t) + \mathscr{D}_{N,\eps} (t) \\
        \leq & C \mathscr{D}^\frac{1}{2}_{S_{\!_N} , 0} (t) \mathscr{D}^\frac{1}{2}_{N,\eps} (t) + C \big( \mathscr{E}^\frac{1}{2}_{S_{\!_N}, 0} (t) + \mathscr{E}^\frac{1}{2}_{N, \eps} (t)  \big) \mathscr{D}^\frac{1}{2}_{S_{\!_N}, 0} (t) \mathscr{D}^\frac{1}{2}_{N,\eps} (t) \\
        + & C \sqrt{\eps} \big( \mathscr{E}^\frac{1}{2}_{S_{\!_N}, 0} (t) + \mathscr{E}^\frac{1}{2}_{N,\eps} (t) + \mathscr{E}^2_{N,\eps} (t) \big) \mathscr{D}_{N,\eps} (t) \,,
      \end{aligned}
    \end{equation}
    which immediately implies by the Young's inequality that
    \begin{equation}
      \begin{aligned}
        \tfrac{\d}{\d t} \mathscr{E}_{N,\eps} (t) + \mathscr{D}_{N,\eps} (t) \leq & C \sqrt{\eps} \big( \mathscr{E}^\frac{1}{2}_{S_{\!_N}, 0} (t) + \mathscr{E}^\frac{1}{2}_{N,\eps} (t) + \mathscr{E}^2_{N,\eps} (t) \big) \mathscr{D}_{N,\eps} (t) \\
        + & C \mathscr{D}_{S_{\!_N} , 0} (t) + C \big( \mathscr{E}^\frac{1}{2}_{S_{\!_N}, 0} (t) + \mathscr{E}^\frac{1}{2}_{N, \eps} (t)  \big) \mathscr{D}^\frac{1}{2}_{S_{\!_N}, 0} (t) \mathscr{D}^\frac{1}{2}_{N,\eps} (t)
      \end{aligned}
    \end{equation}
    for all $0 < \eps \leq \eps_0$. Taking a large constant $\theta_0 \gg 1$ and adding the $\theta_0$ times of the differential inequality \eqref{u0-d0-differential-ineq} to the previous inequality gives us
    \begin{equation}
      \begin{aligned}
        & \tfrac{\d}{\d t} \Big[ \mathscr{E}_{N,\eps} (t) + \theta_0 \mathscr{E}_{S_{\!_N}, 0} (t) \Big] + \mathscr{D}_{N,\eps} (t) + \tfrac{\theta_0}{2} \mathscr{D}_{S_{\!_N}, 0} (t) \\
        \leq & C \big[ \mathscr{E}^2_{N,\eps} (t) + \mathscr{E}^\frac{1}{2}_{N,\eps} (t) + \mathscr{E}^\frac{1}{2}_{S_{\!_N},\eps} (t) \big] \big[ \mathscr{D}_{N,\eps} (t) + \tfrac{\theta_0}{2} \mathscr{D}_{S_{\!_N}, 0} (t) \big]
      \end{aligned}
    \end{equation}
    for all $N \geq 2$ and $0 < \eps \leq \eps_0$. Then the proof of Proposition \ref{Prop-Unif-Bnd} is finished.
\end{proof}

\section{Global well-posedness of the remainder system: proof of Theorem \ref{Thm-main}}\label{Sec:Proof-Main-Thm}.

In this section, we aim at completing the proof of Theorem \ref{Thm-main}. Without loss of generality, based on the a priori energy estimate \eqref{Apriori-Energy-Inq} derived in Proposition \ref{Prop-Unif-Bnd}, we prove the global existence of the remainder system \eqref{Remainder-u-d-2} with the initial data \eqref{IC-Rmd-2}. Then, combining the solution $(\u_0 , \dr_0)$ to the limit equations \eqref{PLQ} with initial conditions \eqref{IC-ParaLQ} given in Proposition \ref{Prop-WP-PLQ}, we know that
$$ ( \u^\eps , \dr^\eps ) = ( \u_0 + \sqrt{\eps} \u_R^\eps , \dr_0 + \sqrt{\eps} \dr_R^\eps ) $$
obeys the first three equations of the system \eqref{HLQ} with the initial data \eqref{IC-HyperLQ}. We remark that the similar arguments will also justify the global existence to the remainder system \eqref{Remainder-u-d}-\eqref{IC-Remainder}, which is the remainder system with respect to the ill-posedness initial data. Thus, we obtain a global classical solution
\begin{equation*}
  \begin{aligned}
    ( \u^\eps , \dr^\eps ) = ( \u_0 + \sqrt{\eps} \u_R^\eps , \dr_0 + \eps \D_I^\eps + \sqrt{\eps} \dr_R^\eps )
  \end{aligned}
\end{equation*}
to the original system \eqref{HLQ}-\eqref{IC-HyperLQ}. We omit the details of this case here.

We now introduce a mollifier over the periodic space variable. Recall that $ \T^3 = \R^3 / \mathbb{L}^3 $, where $\mathbb{L}^3 \subset \R^3$ is some $3$-dimensional lattice. Let $ \varphi \in C^\infty (\R^3)$ be such that $\varphi \geq 0$, $\int_{\R^3} \varphi (x) \d x = 1$, and $\varphi (x) = 0$ for $|x| > 1$. We then define $\varphi^\zeta (x) \in C^\infty ( \T^3 )$ by
\begin{equation*}
  \begin{aligned}
    \varphi^\zeta (x) = \tfrac{1}{\zeta^3} \sum_{l \in \mathbb{L}^3} \varphi \left( \tfrac{x+l}{\zeta} \right)
  \end{aligned}
\end{equation*}
for any $\zeta > 0$. Then we define a mollifier $\J_\zeta$ as
$$ \J_\zeta f (x) = \varphi^\zeta * f (x) = \int_{\T^3} \varphi^\zeta (x-y) f(y) \d y \,. $$

Next we prove the main results of this paper.

\begin{proof}[Proof of Theorem \ref{Thm-main}]

We first construct the following approximate system of the remainder equations \eqref{Remainder-u-d}-\eqref{IC-Remainder}
\begin{equation}\label{Appr-sys}
  \left\{
    \begin{array}{r}
      \partial_t \u_{R,\zeta}^\eps - \tfrac{1}{2} \mu_4 \J_\zeta \Delta \J_\zeta \u_{R,\zeta}^\eps + \nabla p_{R,\zeta}^\eps = \mu_1 \J_\zeta \div \big( ( \J_\zeta \A_{R,\zeta}^\eps : \dr_0 \otimes \dr_0 ) \dr_0 \otimes \dr_0 \big) \\ [2mm]
      \qquad\qquad\qquad\qquad\qquad\qquad\qquad + \J_\zeta \mathcal{K}_{\u, \zeta} + \J_\zeta \div \big( \mathcal{C}_{\u, \zeta} + \mathcal{T}_{\u, \zeta} + \sqrt{\eps} \mathcal{R}_{\u , \zeta} \big) \,, \\ [3mm]
       \div \, \u_{R, \zeta}^\eps = 0 \,, \qquad\qquad\qquad\qquad\qquad\qquad\qquad  \\ [3mm]
      \partial_t ( \D_{\u_0 + \sqrt{\eps} \u_{R,\zeta}^\eps} \dr_{R,\eps}^\eps )_\zeta + \J_\zeta \big[ ( \u_0 + \sqrt{\eps} \J_\zeta \u_{R,\zeta}^\eps ) \cdot \nabla ( \D_{\u_0 + \sqrt{\eps} \u_{R,\zeta}^\eps} \dr_{R,\eps}^\eps )_\zeta \big] \\ [2mm]
      \qquad\qquad\qquad + \tfrac{- \lambda_1}{\eps} ( \D_{\u_0 + \sqrt{\eps} \u_{R,\zeta}^\eps} \dr_{R,\eps}^\eps )_\zeta - \tfrac{1}{\eps} \J_\zeta \Delta \J_\zeta \dr_{R,\zeta}^\eps + \partial_t ( \J_\zeta \u_{R,\eps}^\eps \cdot \nabla \dr_0 ) \\ [2mm]
      \qquad\qquad\qquad = \tfrac{1}{\eps} \J_\zeta \mathcal{C}_{\dr, \zeta} + \tfrac{1}{\eps} \J_\zeta \mathcal{S}^1_{\dr, \zeta} + \tfrac{1}{\sqrt{\eps}} \mathcal{S}^2_{\dr , \zeta} + \J_\zeta \mathcal{R}_{\dr, \zeta} \,, \\ [3mm]
      \partial_t  \dr_{R,\zeta}^\eps = ( \D_{\u_0 + \sqrt{\eps} \u_{R,\zeta}^\eps} \dr_{R,\eps}^\eps )_\zeta - \J_\zeta \big( ( \u_0 + \sqrt{\eps} \J_\zeta \u_{R,\zeta}^\eps ) \cdot \nabla \J_\zeta \dr_{R,\zeta}^\eps \big)
    \end{array}
  \right.
\end{equation}
with the initial conditions
\begin{equation}
  \begin{aligned}
    \big( \u_{R, \zeta}^\eps , \dr_{R, \zeta}^\eps , ( \D_{\u_0 + \sqrt{\eps} \u_{R,\zeta}^\eps} \dr_{R,\eps}^\eps )_\zeta  \big) \big|_{t=0} = ( \J_\zeta \u_R^{\eps, in}, \J_\zeta \dr_R^{\eps, in} , \J_\zeta \widetilde{\dr}_R^{\eps, in}) \,,
  \end{aligned}
\end{equation}
where
$$ \A_{R, \zeta}^\eps = \tfrac{1}{2} ( \nabla \u_{R,\zeta}^\eps + ( \nabla \u_{R,\zeta}^\eps )^\top ) \,, \  \B_{R, \zeta}^\eps = \tfrac{1}{2} ( \nabla \u_{R,\zeta}^\eps - ( \nabla \u_{R,\zeta}^\eps )^\top ) \,,$$
and the symbols $\mathcal{K}_{\u,\zeta}$, $\mathcal{C}_{\u,\zeta}$, $\mathcal{T}_{\u,\zeta}$, $\mathcal{R}_{\u,\zeta}$, $\mathcal{C}_{\dr,\zeta}$, $\mathcal{S}^1_{\dr,\zeta}$, $\mathcal{S}^1_{\dr,\zeta}$ and $\mathcal{R}_{\dr,\zeta}$ are the same form as  $\mathcal{K}_{\u}$, $\mathcal{C}_{\u}$, $\mathcal{T}_{\u}$, $\mathcal{R}_{\u}$, $\mathcal{C}_{\dr}$, $\mathcal{S}^1_{\dr}$, $\mathcal{S}^1_{\dr}$ and $\mathcal{R}_{\dr}$, respectively (just replacing the symbols $\u_R^\eps$, $\dr_R^\eps$ and $\D_{\u_0 + \sqrt{\eps} \u_{R,\zeta}^\eps} \dr_R^\eps$ with the corresponding symbols $\u_{R,\zeta}^\eps$, $\dr_{R,\zeta}^\eps$ and $( \D_{\u_0 + \sqrt{\eps} \u_{R,\zeta}^\eps} \dr_{R,\eps}^\eps )_\zeta$). The previous approximate system can be regarded as an ordinary differential equation in $H^N$ for any $\eps > 0$ by verifying the conditions of Cauchy-Lipschitz theorem. Thus it admits a unique solution $( \u_{R, \zeta}^\eps , \dr_{R, \zeta}^\eps )$ in $ C([0,T_{\eps,\zeta} ) ; H^N) $ with the maximal time interval $ [ 0, T_{\eps,\zeta} )$.

We define the following approximate energy $ \mathscr{E}_{N,\eps}^\zeta (t) $ and the approximate energy dissipative rate $\mathscr{D}_{N,\eps}^\zeta (t)$:
\begin{equation}
  \begin{aligned}
    \mathscr{E}_{N,\eps}^\zeta (t) = & \tfrac{1}{\eps} \| \u_{R,\zeta}^\eps \|^2_{H^N} + ( 1 - \delta ) \| ( \D_{\u_0 + \sqrt{\eps} \u_{R,\zeta}^\eps} \dr_{R,\eps}^\eps )_\zeta \|^2_{H^N} + \tfrac{1}{\eps} \| \nabla \J_\zeta \dr_{R,\zeta}^\eps \|^2_{H^N} \\
    + & ( \tfrac{- \delta \lambda_1}{\eps} - \tfrac{5}{4} \delta ) \| \dr_{R,\zeta}^\eps \|^2_{H^N} + \| \J_\zeta \u_{R,\zeta}^\eps \cdot \nabla \dr_0 + \tfrac{\delta}{2} \dr_{R,\zeta}^\eps \|^2_{H^N} \\
    + & \delta \| ( \D_{\u_0 + \sqrt{\eps} \u_{R,\zeta}^\eps} \dr_{R,\eps}^\eps )_\zeta + \dr_{R,\zeta}^\eps \|^2_{H^N} \\
    + & 2 \sum_{|m| \leq N} \l \partial^m ( \J_\zeta \u_{R,\zeta}^\eps \cdot \nabla \dr_0 ) , \partial^m ( \D_{\u_0 + \sqrt{\eps} \u_{R,\zeta}^\eps} \dr_{R,\eps}^\eps )_\zeta \r \,,
  \end{aligned}
\end{equation}
and
\begin{equation}
  \begin{aligned}
    \mathscr{D}_{N,\eps}^\zeta (t) = & \tfrac{3 \mu_4}{8 \eps} \| \nabla \J_\zeta \u_{R, \zeta}^\eps \|^2_{H^N} + \tfrac{\delta}{2 \eps} \| \nabla \J_\zeta \dr_{R,\zeta}^\eps \|^2_{H^N} - \delta \| ( \D_{\u_0 + \sqrt{\eps} \u_{R,\zeta}^\eps} \dr_{R,\eps}^\eps )_\zeta \|^2_{H^N} \\
    + & \tfrac{\mu_1}{\eps} \sum_{|m| \leq N} \| ( \partial^m \J_\zeta \A_{R,\zeta}^\eps ) : \dr_0 \otimes \dr_0 \|^2_{L^2} \\
    + & \tfrac{1}{\eps} ( \mu_5 + \mu_6 + \tfrac{\lambda_2^2}{\lambda_1} ) \sum_{|m| \leq N} \| ( \partial^m \J_\zeta \A_{R,\zeta}^\eps ) \dr_0 \|^2_{L^2} \\
    + & \tfrac{- \lambda_1}{\eps} \sum_{|m| \leq N} \| \partial^m ( \D_{\u_0 + \sqrt{\eps} \u_{R,\zeta}^\eps} \dr_{R,\eps}^\eps )_\zeta + ( \partial^m \J_\zeta \B_{R,\zeta}^\eps ) \dr_0 + \tfrac{\lambda_2}{\lambda_1} ( \partial^m \J_\zeta \A_{R,\zeta}^\eps ) \dr_0 \|^2_{L^2} \,,
  \end{aligned}
\end{equation}
where the constant $\delta \in ( 0, \tfrac{1}{2} ]$ is the same as that mentioned in $\mathscr{E}_{N,\eps} (t)$ and $\mathscr{D}_{N,\eps} (t)$. As shown in Lemma \ref{Lmm-Energies}, when $N \geq 2$ and $0 < \eps \leq \eps_0$, the approximate energy $ \mathscr{E}_{N,\eps}^\zeta (t) $ and the approximate energy dissipative rate $\mathscr{D}_{N,\eps}^\zeta (t)$ are nonnegative. Moreover,
\begin{equation}
  \begin{aligned}
    \mathscr{E}_{N,\eps}^\zeta (t) \thicksim \tfrac{1}{\eps} \| \u_{R,\zeta}^\eps \|^2_{H^N} + \| ( \D_{\u_0 + \sqrt{\eps} \u_{R,\zeta}^\eps} \dr_{R,\eps}^\eps )_\zeta \|^2_{H^N} + \tfrac{1}{\eps} \| \nabla \J_\zeta \dr_{R,\zeta}^\eps \|^2_{H^N} + \tfrac{1}{\eps} \| \dr_{R,\zeta}^\eps \|^2_{H^N} \,,
  \end{aligned}
\end{equation}
and
\begin{equation}
  \begin{aligned}
    \mathscr{D}_{N,\eps}^\zeta (t) \thicksim & \tfrac{1}{\eps} \| \nabla \J_\zeta \u_{R,\zeta}^\eps \|^2_{H^N} + \tfrac{1}{\eps} \| \nabla \J_\zeta \dr_{R,\zeta}^\eps \|^2_{H^N} + \| ( \D_{\u_0 + \sqrt{\eps} \u_{R,\zeta}^\eps} \dr_{R,\eps}^\eps )_\zeta \|^2_{H^N} \\
    & + \tfrac{1}{\eps} \sum_{|m| \leq N} \| \partial^m ( \D_{\u_0 + \sqrt{\eps} \u_{R,\zeta}^\eps} \dr_{R,\eps}^\eps )_\zeta + ( \partial^m \J_\zeta \B_{R,\zeta}^\eps ) \dr_0 + \tfrac{\lambda_2}{\lambda_1} ( \partial^m \J_\zeta \A_{R,\zeta}^\eps ) \dr_0 \|^2_{L^2} \,.
  \end{aligned}
\end{equation}

Via the similar arguments in Proposition \ref{Prop-Unif-Bnd}, we can derive that for all $0 < \eps \leq \eps_0$, $\zeta > 0$ and $t \in [ 0, T_{\eps, \zeta} )$
\begin{equation}\label{Unif-Diffe-Ineq}
  \begin{aligned}
    \tfrac{\d}{\d t} & \Big[ \mathscr{E}^\zeta_{N,\eps} (t) + \theta_0 \mathscr{E}_{S_{\!_N}, 0} (t) \Big] + \mathscr{D}^\zeta_{N,\eps} (t) + \tfrac{\theta_0}{2} \mathscr{D}_{S_{\!_N}, 0} (t) \\
    \leq & C \big[ ( \mathscr{E}^\zeta_{N,\eps} (t) )^2 + (\mathscr{E}^\zeta_{N,\eps} (t) )^\frac{1}{2} + \mathscr{E}^\frac{1}{2}_{S_{\!_N},0} (t) \big] \big[ \mathscr{D}^\zeta_{N,\eps} (t) + \tfrac{\theta_0}{2} \mathscr{D}_{S_{\!_N}, 0} (t) \big] \,,
  \end{aligned}
\end{equation}
where the constant $C > 0$, $\eps_0 \in ( 0, 1 ]$ and $\theta_0 \gg 1$ are mentioned in Proposition \ref{Prop-Unif-Bnd}.

Next we prove that the maximal time $T_{\eps, \zeta} = + \infty$ under the small size constraint of the initial energy $E^{in}$, defined in \eqref{Initial-Energy}. More precisely, there is a small $\xi_0 > 0$, independent of $\eps$ and $\zeta$, such that if $E^{in} \leq \xi_0$, we have $T_{\eps, \zeta} = + \infty $ and
\begin{equation}
  \begin{aligned}
    \mathscr{E}^\zeta_{N,\eps} (t) + \theta_0 \mathscr{E}_{S_{_{\!N}}, 0} (t) + \tfrac{1}{2} \int_0^t \big[ \mathscr{D}^\zeta_{N,\eps} (\tau) + \tfrac{\theta_0}{2} \mathscr{D}_{S_{_{\!_N}}, 0} (\tau) \big] \d \tau \leq (1 + \theta_0) \xi_0
  \end{aligned}
\end{equation}
holds for all $t \geq 0$, $\zeta > 0$ and $ 0 < \eps \leq \eps_0 $.

Assume that $T_{\eps, \zeta} < + \infty$, We define a time number $T_{\eps, \zeta}^*$ as
\begin{equation}
  \begin{aligned}
    T_{\eps, \zeta}^* = \sup \Big\{  \tau \in [ 0 , T_{\eps, \zeta} ) ; \sup_{t \in [0, \tau]} C \big[ ( \mathscr{E}^\zeta_{N,\eps} (t) )^2 + (\mathscr{E}^\zeta_{N,\eps} (t) )^\frac{1}{2} + \mathscr{E}^\frac{1}{2}_{S_{\!_N},0} (t) \big] \leq \tfrac{1}{2}  \Big\} \in [ 0 , T_{\eps, \zeta} ) \,.
  \end{aligned}
\end{equation}
By Lemma \ref{Lmm-Energies} and the definition of the initial energy $E^{in}$ in \eqref{Initial-Energy}, we have
\begin{equation}
  \begin{aligned}
    & C \big[ ( \mathscr{E}^\zeta_{N,\eps} (0) )^2 + (\mathscr{E}^\zeta_{N,\eps} (0) )^\frac{1}{2} + \mathscr{E}^\frac{1}{2}_{S_{\!_N},0} (0) \big] \\
    \leq & C \big[ ( C_2 E^{in} )^2 + 2 \sqrt{C_2 E^{in}} \big] \\
    \leq & C \big[ ( C_2 \xi_0 )^2 + 2 \sqrt{C_2 \xi_0} \big]
  \end{aligned}
\end{equation}
holds for all $0 < \eps \leq \eps_0$, where the constants $C_2$ and $\eps_0$ are given in Lemma \ref{Lmm-Energies}. As listed in Proposition \ref{Prop-WP-PLQ} of this paper, we first require the number $\xi_0 \leq \beta_{S_{\!_N}, 0}$ such that the initial energy bound $E^{in} \leq \xi_0$ guarantees the global existence results of Wang-Zhang-Zhang \cite{Wang-Zhang-Zhang-2013-ARMA} to the limit system \eqref{PLQ} with the initial data \eqref{IC-ParaLQ}. We thus choose the small positive constant $\xi_0 \in ( 0, \beta_{S_{\!_N}, 0} ] $ such that
$$ C \big[ ( C_2 \xi_0 )^2 + 2 \sqrt{C_2 \xi_0} \big] \leq \tfrac{1}{4} \,.$$
Specifically, one can take any $ \xi_0 \in (0 , \min \{ 1 , \beta_{S_{\!_N}, 0}, \tfrac{1}{C_2} \tfrac{1}{144 C^2 C_2} \} ] \subset ( 0, 1 ] $. As a result, we know that if $E^{in} \leq \xi_0$, then for all $\eps \in ( 0, \eps_0 ]$
\begin{equation}\label{Initial-Bnd}
  \begin{aligned}
    C \big[ ( \mathscr{E}^\zeta_{N,\eps} (0) )^2 + (\mathscr{E}^\zeta_{N,\eps} (0) )^\frac{1}{2} + \mathscr{E}^\frac{1}{2}_{S_{\!_N},0} (0) \big] \leq \tfrac{1}{4} \,.
  \end{aligned}
\end{equation}
By the initial energy bound \eqref{Initial-Bnd} and the continuity of the energy functional $ \mathscr{E}^\zeta_{N,\eps} (t) $ in $ t \in [ 0, T_{\eps, \zeta} ) $, we imply that $ T_{\eps, \zeta}^* > 0 $.

We now claim that there is a small $ \xi_0 \in (0 , \min \{ 1 , \beta_{S_{\!_N}, 0}, \tfrac{1}{C_2} \tfrac{1}{144 C^2 C_2} \} ] $ such that if $E^{in} \leq \xi_0$, then $ T_{\eps, \zeta}^* = T_{\eps, \zeta} $ holds for all $0 < \eps \leq \eps_0$ and $\zeta > 0$. Indeed, if $ T_{\eps, \zeta}^* < T_{\eps, \zeta} $ for all $ \xi_0 \in (0 , \min \{ 1 , \beta_{S_{\!_N}, 0}, \tfrac{1}{C_2} \tfrac{1}{144 C^2 C_2} \} ] $, the inequality \eqref{Unif-Diffe-Ineq} tells us that
\begin{equation}
  \begin{aligned}
    \tfrac{\d}{\d t} & \Big[ \mathscr{E}^\zeta_{N,\eps} (t) + \theta_0 \mathscr{E}_{S_{\!_N}, 0} (t) \Big] + \tfrac{1}{2} \big[ \mathscr{D}^\zeta_{N,\eps} (t) + \tfrac{\theta_0}{2} \mathscr{D}_{S_{\!_N}, 0} (t) \big] \leq 0
  \end{aligned}
\end{equation}
holds for all $t \in [ 0, T_{\eps,\zeta}^* ]$ and for all $\zeta > 0$, $\eps \in ( 0, \eps_0 ]$. Hence, integrating the previous differential inequality over $[0, t]$ reduces to
\begin{equation}
  \begin{aligned}
    & \mathscr{E}_{N,\eps}^\zeta (t) + \theta_0 \mathscr{E}_{S_{_{\!N}}, 0} (t) + \tfrac{1}{2}  \int_0^t \big[ \mathscr{D}^\zeta_{N,\eps} (\tau) + \tfrac{\theta_0}{2} \mathscr{D}_{S_{\!_N}, 0} (\tau) \big] \d \tau \\
    \leq & \mathscr{E}_{N,\eps}^\zeta (0) + \theta_0 \mathscr{E}_{S_{_{\!N}}, 0} (0) \\
    \leq & ( 1 + \theta_0 ) \big[ \mathscr{E}_{N,\eps}^\zeta (0) +  \mathscr{E}_{S_{_{\!N}}, 0} (0) \big] \\
    \leq & ( 1 + \theta_0 ) C_2 \xi_0
  \end{aligned}
\end{equation}
holds for all  $t \in [ 0, T_{\eps,\zeta}^* ]$ and for all $\zeta > 0$, $\eps \in ( 0, \eps_0 ]$, which means that
\begin{equation}
  \begin{aligned}
    \mathscr{E}_{N,\eps}^\zeta (t) + \mathscr{E}_{S_{_{\!N}}, 0} (t)  \leq ( 1 + \theta_0 ) C_2 \xi_0 \,.
  \end{aligned}
\end{equation}
If we choose a fixed $ \xi_0 = \min \{ 1 , \beta_{S_{\!_N}, 0}, \tfrac{1}{C_2 (1 + \theta_0)} \tfrac{1}{144 C^2 C_2 (1 + \theta_0)} \} \in (0 , \min \{ 1 , \beta_{S_{\!_N}, 0}, \tfrac{1}{C_2} \tfrac{1}{144 C^2 C_2} \} ] $, then the previous energy bound implies that under the constraint $E^{in} \leq \xi_0$
\begin{equation}
  \begin{aligned}
    & C \big[ ( \mathscr{E}^\zeta_{N,\eps} (t) )^2 + (\mathscr{E}^\zeta_{N,\eps} (t) )^\frac{1}{2} + \mathscr{E}^\frac{1}{2}_{S_{\!_N},0} (t) \big] \\
    \leq & C \Big[ ( C_2 ( 1 + \theta_0 ) \xi_0 )^2 + 2 \sqrt{C_2 (1 + \theta_0) \xi_0} \Big] \leq \tfrac{1}{4} < \tfrac{1}{2}
  \end{aligned}
\end{equation}
holds for all $t \in [ 0, T_{\eps,\zeta}^* ]$. Thus the continuity of $\mathscr{E}^\zeta_{N,\eps} (t)$ yields that there exists a $t^* > T_{\eps,\zeta}^*$ such that
$$ \sup_{t \in [0, t^*]} C \big[ ( \mathscr{E}^\zeta_{N,\eps} (t) )^2 + (\mathscr{E}^\zeta_{N,\eps} (t) )^\frac{1}{2} + \mathscr{E}^\frac{1}{2}_{S_{\!_N},0} (t) \big] \leq \tfrac{1}{2} \,, $$
which contradicts to the definition of $T_{\eps,\zeta}^*$. Consequently, we have $T_{\eps,\zeta}^* = T_{\eps,\zeta} < + \infty$.

Therefore, it must hold that at time $t = T_{\eps,\zeta}$
\begin{equation}
  \begin{aligned}
    \mathscr{E}^\zeta_{N,\eps} (t) + \mathscr{E}^\frac{1}{2}_{S_{\!_N},0} (t) < + \infty \,.
  \end{aligned}
\end{equation}
We then can extend the solution $( \u_{R,\zeta}^\eps , \dr_{R,\zeta}^\eps )$ of the approximate system \eqref{Appr-sys} to a lager interval $ [ 0, T_{\eps, \zeta} + \kappa ) $ for some $\kappa > 0$. This contradicts to the maximality of $T_{\eps,\zeta}$. As a consequence, there is a small $\xi_0 > 0$, independent of $\eps$ and $\zeta$, such that if $E^{in} \leq \xi_0$ for all $\eps \in ( 0, \eps_0 ]$, we have $T_{\eps, \zeta} = + \infty $ and
\begin{equation}
\begin{aligned}
\sup_{t \geq 0} \big[ \mathscr{E}^\zeta_{N,\eps} (t) + \mathscr{E}_{S_{_{\!N}}, 0} (t) \big] + \tfrac{1}{2} \int_0^\infty \big[ \mathscr{D}^\zeta_{N,\eps} (\tau) + \tfrac{1}{2} \mathscr{D}_{S_{_{\!_N}}, 0} (\tau) \big] \d \tau \leq (1 + \theta_0) \xi_0
\end{aligned}
\end{equation}
holds for all $\zeta > 0$ and $ 0 < \eps \leq \eps_0 $.

Then, by compactness arguments (let $\zeta \rightarrow 0$), we get vector field $ ( \u_R^\eps , \dr_R^\eps ) \in \R^3 \times \R^3 $ satisfying
$$ \u_R^\eps , \D_{\u_0 + \sqrt{\eps} \u_{R,\zeta}^\eps} \dr_R^\eps , \nabla \dr_R^\eps, \dr_R^\eps \in L^\infty ( \R^+; H^N ) \,, \ \nabla \u_R^\eps \in L^2 ( \R^+; H^N ) $$
for all $0 < \eps \leq \eps_0$, which solves the remainder system \eqref{Remainder-u-d}-\eqref{IC-Remainder}. Moreover, $(\u_R^\eps, \dr_R^\eps)$ obeys the energy bound
\begin{equation*}
  \begin{aligned}
    \sup_{t \geq 0} \left( \tfrac{1}{\eps} \| \u_R^\eps \|^2_{H^N} + \| \D_{\u_0 + \sqrt{\eps} \u_{R,\zeta}^\eps} \dr_R^\eps \|^2_{H^N} + \tfrac{1}{\eps} \| \dr_R^\eps \|^2_{H^{N+1}} \right) (t) + \tfrac{1}{\eps} \int_0^\infty \| \nabla \u_R^\eps \|^2_{H^N} (t) \d t \leq C \xi_0
  \end{aligned}
\end{equation*}
for some $C > 0$, which is uniform in $\eps \in ( 0, \eps_0 ]$. It is easy to know that
\begin{equation*}
  \begin{aligned}
    2 (\dr_0 \cdot \dr_R^\eps) (0,x) + \sqrt{\eps} |\dr_R^\eps|^2 (0,x) = 0 \,.
  \end{aligned}
\end{equation*}
As a consequence, Lemma \ref{Lmm-Constraints} (or Remark \ref{Rmk-constraint-2}) implies the constraint \eqref{Constraint-2} holds at any time.Then the proof of Theorem \ref{Thm-main} is finished.
\end{proof}

\appendix

\section{Remainder systems}\label{Appendix}

In this section, we will present the tedious terms of the remainder system \eqref{Remainder-u-d}. After submitting the ansatz \eqref{Ansatz} into the hyperbolic Ericksen-Leslie's liquid crystal model \eqref{HLQ}, one obtain the reminder system \eqref{Remainder-u-d}, namely,
\begin{equation*}
\left\{
\begin{array}{c}
\partial_t \u_R^\eps - \tfrac{1}{2} \mu_4 \Delta \u_R^\eps + \nabla p_R^\eps = \mu_1 \div \big[ (\A_R^\eps : \dr_0 \otimes \dr_0) \dr_0 \otimes \dr_0 \big] \\
\qquad\qquad + \mathcal{K}_{\u} + \div ( \mathcal{C}_{\u} + \mathcal{T}_{\u} + \sqrt{\eps} \mathcal{R}_{\u} ) + \eps \div \mathcal{Q}_{\u}(\D_I) \,,\\[2mm]
\div \u_R^\eps = 0 \,, \\[2mm]
\D^2_{\u_0 + \sqrt{\eps} \u_R^\eps} \dr_R^\eps + \tfrac{- \lambda_1}{\eps} \D_{\u_0 + \sqrt{\eps} \u_R^\eps} \dr_R^\eps - \tfrac{1}{\eps} \Delta \dr_R^\eps + \partial_t ( \u_R^\eps \cdot \nabla \dr_0 + \sqrt{\eps} \u_R^\eps \cdot \nabla \D_I^\eps ) \\
= \tfrac{1}{\eps} \mathcal{C}_{\dr} + \tfrac{1}{\eps} \mathcal{S}^1_{\dr} + \tfrac{1}{\sqrt{\eps}} \mathcal{S}^2_{\dr} + \mathcal{R}_{\dr} + \mathcal{Q}_{\dr}(\D_I)
\end{array}
\right.
\end{equation*}
with the constraint
\begin{equation}
\begin{aligned}
2 \dr_0 \cdot ( \dr_R^\eps + \sqrt{\eps} \D_I^\eps ) + \sqrt{\eps} | \dr_R^\eps + \sqrt{\eps} \D_I^\eps |^2 = 0 \,,
\end{aligned}
\end{equation}
where the tensor $\mathcal{C}_{\u}$ is defined in \eqref{C-u} and the vector field $\mathcal{C}_{\dr}$ is given in \eqref{C-d}, the linear tensor term $\mathcal{T}_{\u}$ is
\begin{equation}\label{Tu}
\begin{aligned}
\mathcal{T}_{\u} = & \mu_1 \big[ ( \A_0 : (\dr_R^\eps \otimes \dr_0 + \dr_0 \otimes \dr_R^\eps )) \dr_0 \otimes \dr_0 + (\A_0 : \dr_0 \otimes \dr_0) ( \dr_0 \otimes \dr_R^\eps + \dr_R^\eps \otimes \dr_0 ) \big] \\
+ & \mu_2 \big[ (\D_{\u_0} \dr_0 + \B_0 \dr_0) \otimes \dr_R^\eps + ( \B_0 \dr_R^\eps + \u_R^\eps \cdot \nabla \dr_0 ) \otimes \dr_0  \big] \\
+ & \mu_3 \big[ \dr_R^\eps \otimes (\D_{\u_0} \dr_0 + \B_0 \dr_0)  + \dr_0  \otimes ( \B_0 \dr_R^\eps + \u_R^\eps \cdot \nabla \dr_0 ) \big] \\
+ & \mu_5 \big[ (\A_0 \dr_0) \otimes \dr_R^\eps + (\A_0 \dr_R^\eps) \otimes \dr_0 \big] + \mu_6 \big[  \dr_R^\eps \otimes (\A_0 \dr_0) + \dr_0 \otimes (\A_0 \dr_R^\eps) \big] \,,
\end{aligned}
\end{equation}
the linear vector field $\mathcal{K}_{\u}$ is
\begin{equation}\label{Lu}
\begin{aligned}
\mathcal{K}_{\u} = & - \u_0 \cdot \nabla \u_R^\eps - \u_R^\eps \cdot \nabla\u_0 - \sqrt{\eps} \u_R^\eps \cdot \nabla \u_R^\eps \\
& - \div ( \nabla \dr_0 \odot \nabla_R^\eps + \nabla \dr_R^\eps \odot \nabla \dr_0 + \sqrt{\eps} \nabla \dr_R^\eps \odot \nabla \dr_R^\eps )\,,
\end{aligned}
\end{equation}
the singular linear term $\mathcal{S}^1_{\dr}$ is of the form
\begin{equation}\label{Sd-1}
\begin{aligned}
\mathcal{S}^1_{\dr} = & 2 ( \nabla \dr_0 \cdot \nabla \dr_R^\eps ) \dr_0 + |\nabla \dr_0|^2 \dr_R^\eps + \lambda_1 ( \u_R^\eps \cdot \nabla \dr_0 + \B_0 \dr_R^\eps ) \\
+ & \lambda_2 \big[ \A_0 \dr_R^\eps - ( \A_R^\eps : \dr_0 \otimes \dr_0 ) \dr_0 - (\A_0 : \dr_0 \otimes \dr_0) \dr_R^\eps - 2 ( \A_0 :  \dr_0 \otimes \dr_R^\eps ) \dr_0 \big] \,,
\end{aligned}
\end{equation}
the singular nonlinear term $\mathcal{S}^2_{\dr}$ is defined as
\begin{equation}\label{Sd-2}
\begin{aligned}
\mathcal{S}^2_{\dr} = & |\nabla \dr_R^\eps|^2 \dr_0 + 2 ( \nabla \dr_0 \cdot \nabla \dr_R^\eps ) \dr_R^\eps - \D^2_{\u_0} \dr_0 - |\D_{\u_0} \dr_0|^2 \dr_0 + \lambda_1 \B_R^\eps \dr_R^\eps \\
+ & \lambda_2 \big[ \A_R^\eps \dr_R^\eps - \A_0 : ( \dr_0 \otimes \dr_R^\eps + \dr_R^\eps \otimes \dr_0 ) \dr_R^\eps - ( \A_R^\eps : \dr_0 \otimes \dr_0 ) \dr_R^\eps \big] \\
- & \lambda_2 \big[ \A_R^\eps : ( \dr_R^\eps \otimes \dr_0 + \dr_0 \otimes \dr_R^\eps ) \dr_0 + ( \A_0 : \dr_R^\eps \otimes \dr_R^\eps ) \dr_0 \big] \,,
\end{aligned}
\end{equation}
the nonsingular nonlinear term $\mathcal{R}_{\dr}$ is
\begin{equation}\label{Rd}
\begin{aligned}
\mathcal{R}_{\dr} = & ( |\nabla \dr_R^\eps|^2 - |\D_{\u_0} \dr_0|^2 ) \dr_R^\eps - 2 \big[ \D_{\u_0} \dr_0 \cdot ( \D_{\u_0 + \sqrt{\eps} \u_R^\eps} \dr_R^\eps + \u_R^\eps \cdot \nabla \dr_0 ) \big] \dr_0 \\
- & \lambda_2 \big[ (\A_R^\eps : \dr_R^\eps \otimes \dr_R^\eps ) \dr_0 + \A_R^\eps : ( \dr_R^\eps \otimes \dr_0 + \dr_0 \otimes \dr_R^\eps ) \dr_R^\eps + ( \A_0 : \dr_R^\eps \otimes \dr_R^\eps ) \dr_R^\eps \big] \\
- & \sqrt{\eps} \big[ | \D_{\u_0 + \sqrt{\eps} \u_R^\eps} \dr_R^\eps + \u_R^\eps \cdot \nabla \dr_0 |^2 \dr_0 + 2 \lambda_2 ( \A_R^\eps : \dr_R^\eps \otimes \dr_R^\eps ) \dr_R^\eps \\
& \qquad + 2 \D_{\u_0} \dr_0 \cdot ( \D_{\u_0 + \sqrt{\eps} \u_R^\eps} \dr_R^\eps + \u_R^\eps \cdot \nabla \dr_0 ) \dr_R^\eps \big] \\
- & \eps | \D_{\u_0 + \sqrt{\eps} \u_R^\eps} \dr_R^\eps + \u_R^\eps \cdot \nabla \dr_0 |^2 \dr_R^\eps \\
- & \u_R^\eps \cdot \nabla (\D_{\u_0} \dr_0 ) - \u_0 \cdot \nabla ( \u_R^\eps \cdot \nabla \dr_0 ) - \sqrt{\eps} \u_R^\eps \cdot \nabla ( \u_R^\eps \cdot \nabla \dr_0 ) \,,
\end{aligned}
\end{equation}
and the nonsingular nonlinear tensor $\mathcal{R}_{\u}$ is
\begin{equation}\label{Ru}
\begin{aligned}
\mathcal{R}_{\u} = \mathcal{M}_1 + \sqrt{\eps} \mathcal{M}_2 + \sqrt{\eps}^2 \mathcal{M}_3 + \sqrt{\eps}^3 \mathcal{M}_4 \,.
\end{aligned}
\end{equation}
Here the term $\mathcal{M}_1$ is
\begin{equation}\label{M1}
\begin{aligned}
\mathcal{M}_1 = & \mu_1 \big[ ( \A_0 : \dr_R^\eps \otimes \dr_R^\eps + 2 \A_R^\eps : \dr_0 \otimes \dr_R^\eps ) \dr_0 \otimes \dr_0 + (\A_0 : \dr_0 \otimes \dr_0 ) \dr_R^\eps \otimes \dr_R^\eps  \\
& \quad + ( 2 \A_0 : \dr_0 \otimes \dr_R^\eps   +  \A_R^\eps : \dr_0 \otimes \dr_0 ) (  \dr_R^\eps \otimes \dr_0 + \dr_0 \otimes \dr_R^\eps ) \big] \\
+ & \mu_2 \big[ ( \B_R^\eps \dr_R^\eps ) \otimes \dr_0 + ( \D_{\u_0 + \sqrt{\eps} \u_R^\eps} \dr_R^\eps + \u_R^\eps \cdot \nabla \dr_0 + \B_0 \dr_R^\eps + \B_R^\eps \dr_0 ) \otimes \dr_R^\eps \big] \\
+ & \mu_3 \big[ \dr_0 \otimes ( \B_R^\eps \dr_R^\eps )  + \dr_R^\eps \otimes ( \D_{\u_0 + \sqrt{\eps} \u_R^\eps} \dr_R^\eps + \u_R^\eps \cdot \nabla \dr_0 + \B_0 \dr_R^\eps + \B_R^\eps \dr_0 ) \big] \\
+ & \mu_5 \big[ ( \A_R^\eps \dr_R^\eps ) \otimes \dr_0 + ( \A_R^\eps \dr_0 + \A_0 \dr_R^\eps ) \otimes \dr_R^\eps \big] \\
+ & \mu_6 \big[ \dr_0  \otimes  ( \A_R^\eps \dr_R^\eps ) + \dr_R^\eps \otimes ( \A_R^\eps \dr_0 + \A_0 \dr_R^\eps ) \big] \,,
\end{aligned}
\end{equation}
the term $\mathcal{M}_2$ is
\begin{equation}\label{M2}
\begin{aligned}
\mathcal{M}_2 = & \mu_1 \big[ ( \A_R^\eps : \dr_0 \otimes \dr_0 + 2 \A_0 : \dr_0 \otimes \dr_R^\eps ) \dr_R^\eps \otimes \dr_R^\eps + ( \A_R^\eps : \dr_R^\eps \otimes \dr_R^\eps ) \dr_0 \otimes \dr_0 \\
& \quad + ( 2 \A_R^\eps : \dr_0 \otimes \dr_R^\eps + \A_0 : \dr_R^\eps \otimes \dr_R^\eps ) ( \dr_0 \otimes \dr_R^\eps + \dr_R^\eps \otimes \dr_0 ) \big] \\
+ & \mu_2 ( \B_R^\eps \dr_R^\eps ) \otimes \dr_R^\eps + \mu_3 \dr_R^\eps \otimes ( \B_R^\eps \dr_R^\eps ) + \mu_5 ( \A_R^\eps \dr_R^\eps ) \otimes \dr_R^\eps + \mu_6 \dr_R^\eps \otimes ( \A_R^\eps \dr_R^\eps ) \,,
\end{aligned}
\end{equation}
the term $\mathcal{M}_3$ is
\begin{equation}\label{M3}
\begin{aligned}
\mathcal{M}_3 = & \mu_1 \big[ ( 2 \A_R^\eps : \dr_0 \otimes \dr_R^\eps ) \dr_R^\eps \otimes \dr_R^\eps + (\A_0 : \dr_R^\eps \otimes \dr_R^\eps) \dr_R^\eps \otimes \dr_R^\eps \\
& + ( \A_R^\eps : \dr_R^\eps \otimes \dr_R^\eps ) ( \dr_R^\eps \otimes \dr_0 + \dr_0 \otimes \dr_R^\eps ) \big] \,,
\end{aligned}
\end{equation}
the term $\mathcal{M}_4$ is
\begin{equation}\label{M4}
\mathcal{M}_4 = \mu_1 ( \A_R^\eps : \dr_R^\eps \otimes \dr_R^\eps ) \dr_R^\eps \otimes \dr_R^\eps \,.
\end{equation}
Moreover, the tensor term $\mathcal{Q}_{\u} (\D_I)$ involving initial layer in the $\u_R^\eps$-equation of \eqref{Remainder-u-d} reads
\begin{equation}\label{Q-u}
\begin{aligned}
\mathcal{Q}_{\u} (\D_I) = & \tfrac{1}{\sqrt{\eps}} \mathcal{Q}_{\u}^1 + \mathcal{Q}_{\u}^2 + \sqrt{\eps} \mathcal{Q}_{\u}^3 + \sqrt{\eps}^2 \mathcal{Q}_{\u}^4 + \sqrt{\eps}^3 \mathcal{Q}_{\u}^5 + \sqrt{\eps}^4 \mathcal{Q}_{\u}^6 + \sqrt{\eps}^5 \mathcal{Q}_{\u}^7 + \sqrt{\eps}^6 \mathcal{Q}_{\u}^8 \,,
\end{aligned}
\end{equation}
where the term $\mathcal{Q}_{\u}^1$ is
\begin{equation}
\begin{aligned}
\mathcal{Q}_{\u}^1 = & - \nabla \dr_0 \odot \nabla \D_I^\eps + \nabla \D_I^\eps \odot \nabla \dr_0 \\
+ & \mu_1 \big[ 2 ( \A_0 : \dr_0 \otimes \D_I^\eps ) \dr_0 \otimes \dr_0 + ( \A_0 : \dr_0 \otimes \dr_0 ) ( \D_I^\eps \otimes \dr_0 + \dr_0 \otimes \D_I^\eps ) \big] \\
+ & \mu_2 \big[ ( \D_{\u_0} \dr_0 + \B_0 \dr_0 ) \otimes \D_I^\eps + ( \D_{\u_0} \D_I^\eps + \B_0 \D_I^\eps ) \otimes \dr_0 \big] \\
+ & \mu_3 \big[ \D_I^\eps \otimes  ( \D_{\u_0} \dr_0 + \B_0 \dr_0 ) + \dr_0 \otimes ( \D_{\u_0} \D_I^\eps + \B_0 \D_I^\eps ) \big] \\
+ & \mu_5 \big[ ( \A_0 \dr_0 ) \otimes \D_I^\eps + ( \A_0 \D_I^\eps ) \otimes \dr_0 \big] + \mu_6 \big[ \D_I^\eps \otimes ( \A_0 \dr_0 )  + \dr_0 \otimes ( \A_0 \D_I^\eps ) \big] \,,
\end{aligned}
\end{equation}
the term $\mathcal{Q}_{\u}^2$ is
\begin{equation}
\begin{aligned}
\mathcal{Q}_{\u}^2 = & - \nabla \D_I^\eps \odot \nabla \dr_R^\eps - \nabla \dr_R^\eps \odot \nabla \D_I^\eps + 2 \mu_1 (\A_0 : \dr_0 \otimes \D_I^\eps) ( \dr_R^\eps \otimes \dr_0 + \dr_0 \otimes \dr_R^\eps ) \\
+ & 2 \mu_1 ( \A_R^\eps : \D_I^\eps \otimes \dr_0 + \A_0 : \D_I^\eps \otimes \dr_R^\eps ) \dr_0 \otimes \dr_0   \\
+ & \mu_1 ( \A_R^\eps : \dr_0 \otimes \dr_0 + 2 \A_0 : \dr_0 \otimes \dr_R^\eps ) ( \D_I^\eps \otimes \dr_0 + \dr_0 \otimes \D_I^\eps ) \\
+ & \mu_2 ( \u_R^\eps \cdot \nabla \dr_0 + \D_{\u_0 + \sqrt{\eps} \u_R^\eps } \dr_R^\eps + \B_R^\eps \dr_0 + \B_0 \dr_R^\eps ) \otimes \D_I^\eps \\
+ & \mu_2 \big[ ( \u_R^\eps \cdot \nabla \D_I^\eps + \B_R^\eps \D_I^\eps ) \otimes \dr_0 + ( \D_{\u_0} \D_I^\eps + \B_0 \D_I^\eps ) \otimes \dr_R^\eps \big] \\
+ & \mu_3 \D_I^\eps \otimes ( \u_R^\eps \cdot \nabla \dr_0 + \D_{\u_0 + \sqrt{\eps} \u_R^\eps } \dr_R^\eps + \B_R^\eps \dr_0 + \B_0 \dr_R^\eps ) \\
+ & \mu_3 \big[ \dr_0 \otimes ( \u_R^\eps \cdot \nabla \D_I^\eps + \B_R^\eps \D_I^\eps ) + \dr_R^\eps \otimes ( \D_{\u_0} \D_I^\eps + \B_0 \D_I^\eps ) \big] \\
+ & \mu_5 \big[ ( \A_0 \dr_R^\eps + \A_R^\eps \dr_0 ) \otimes \D_I^\eps + (\A_0 \D_I^\eps) \otimes \dr_R^\eps + (\A_R^\eps \D_I^\eps) \otimes \dr_0 \big] \\
+ & \mu_6 \big[ \D_I^\eps \otimes ( \A_0 \dr_R^\eps + \A_R^\eps \dr_0 ) + \dr_R^\eps \otimes (\A_0 \D_I^\eps) + \dr_0 \otimes (\A_R^\eps \D_I^\eps) \big] \,,
\end{aligned}
\end{equation}
the term $\mathcal{Q}_{\u}^3$ is
\begin{equation}
\begin{aligned}
\mathcal{Q}_{\u}^3 = & - \nabla \D_I^\eps \odot \nabla \D_I^\eps + \mu_1 \big[ ( \A_0 : \D_I^\eps \otimes \D_I^\eps ) \dr_0 \otimes \dr_0 + ( \A_0 : \dr_0 \otimes \dr_0 ) \D_I^\eps \otimes \D_I^\eps \big] \\
+ & \mu_1 \big[ 2 ( \A_0 : \D_I^\eps \otimes \dr_0 ) ( \D_I^\eps \otimes \dr_0 + \dr_0 \otimes \D_I^\eps ) + 2 ( \A_R^\eps : \D_I^\eps \otimes \dr_R^\eps ) \dr_0 \otimes \dr_0 \big] \\
+ & \mu_1 \big[ 2 ( \A_0 : \D_I^\eps \otimes \dr_0 ) \dr_R^\eps \otimes \dr_R^\eps + 2 ( \A_R^\eps : \D_I^\eps \otimes \dr_0 ) ( \dr_0 \otimes \dr_R^\eps + \dr_R^\eps \otimes \dr_0 ) \big] \\
+ & \mu_1 ( 2 \A_0 : \D_I^\eps \otimes \dr_R^\eps + \A_R^\eps : \D_I^\eps \otimes \dr_R^\eps  ) ( \dr_0 \otimes \dr_R^\eps + \dr_R^\eps \otimes \dr_0 ) \\
+ & \mu_1 ( \A_0 : \dr_R^\eps \otimes \dr_R^\eps + 2 \A_R^\eps : \dr_R^\eps \otimes \dr_0 ) ( \dr_0 \otimes \D_I^\eps + \D_I^\eps \otimes \dr_0 ) \\
+ & \mu_1 ( \A_R^\eps : \dr_0 \otimes \dr_0 + 2 \A_0 : \dr_0 \otimes \dr_R^\eps ) ( \dr_R^\eps \otimes \D_I^\eps + \D_I^\eps \otimes \dr_R^\eps ) \\
+ & \mu_2 \big[ ( \D_{\u_0} \D_I^\eps + \B_0 \D_I^\eps ) \otimes \D_I^\eps + ( \B_R^\eps \dr_R^\eps ) \otimes \D_I^\eps + ( \u_R^\eps \cdot \nabla \D_I^\eps + \B_R^\eps \D_I^\eps ) \otimes \dr_R^\eps \big] \\
+ & \mu_3 \big[ \otimes \D_I^\eps ( \D_{\u_0} \D_I^\eps + \B_0 \D_I^\eps ) + \D_I^\eps \otimes ( \B_R^\eps \dr_R^\eps ) + \dr_R^\eps \otimes ( \u_R^\eps \cdot \nabla \D_I^\eps + \B_R^\eps \D_I^\eps ) \big] \\
+ & \mu_5 \big[ (\A_0 \D_I^\eps) \otimes \D_I^\eps + ( \A_R^\eps \dr_R^\eps ) \otimes \D_I^\eps + ( \A_R^\eps \D_I^\eps ) \otimes \dr_R^\eps \big] \\
+ & \mu_6 \big[ \D_I^\eps \otimes (\A_0 \D_I^\eps) + \D_I^\eps \otimes ( \A_R^\eps \dr_R^\eps ) + \dr_R^\eps \otimes ( \A_R^\eps \D_I^\eps ) \otimes \dr_R^\eps \big] \,,
\end{aligned}
\end{equation}
the term $\mathcal{Q}_{\u}^4$ is
\begin{equation}
\begin{aligned}
\mathcal{Q}_{\u}^4 = & \mu_1 \big[ ( \A_R^\eps : \D_I^\eps \otimes \D_I^\eps ) \dr_0 \otimes \dr_0 + ( \A_R^\eps : \dr_0 \otimes \dr_0 + 2 \A_0 : \dr_0 \otimes \dr_R^\eps ) \D_I^\eps \otimes \D_I^\eps \big] \\
+ & \mu_1 \big[ ( \A_0 : \D_I^\eps \otimes \D_I^\eps ) ( \dr_R^\eps \otimes \dr_0 + \dr_0 \otimes \dr_R^\eps ) + 2 ( \A_0 : \D_I^\eps \otimes \dr_0 ) ( \dr_R^\eps \otimes \D_I^\eps + \D_I^\eps \otimes \dr_R^\eps ) \big] \\
+ & 2 \mu_1 ( \A_R^\eps : \D_I^\eps \otimes \dr_0 + \A_0 : \D_I^\eps \otimes \dr_R^\eps ) ( \D_I^\eps \otimes \dr_0 + \dr_0 \otimes \D_I^\eps ) \\
+ & 2 \mu_1 ( \A_0 : \D_I^\eps \otimes \dr_R^\eps + \A_R^\eps : \D_I^\eps \otimes \dr_0  ) \dr_R^\eps \otimes \dr_R^\eps \\
+ & \mu_1 ( \A_0 : \dr_R^\eps \otimes \dr_R^\eps + 2 \A_R^\eps : \dr_R^\eps \otimes \dr_0  ) ( \D_I^\eps \otimes \dr_R^\eps + \dr_R^\eps \otimes \D_I^\eps ) \\
+ & \mu_1 \big[ 2 ( \A_R^\eps : \D_I^\eps \otimes \dr_R^\eps ) ( \dr_R^\eps \otimes \dr_0 + \dr_0 \otimes \dr_R^\eps ) + ( \A_R^\eps : \dr_R^\eps \otimes \dr_R^\eps ) ( \D_I^\eps \otimes \dr_0 + \dr_0 \otimes \D_I^\eps ) \big] \\
+ & \mu_2 ( \u_R^\eps \cdot \nabla \D_I^\eps + \B_R^\eps \D_I^\eps ) \otimes \D_I^\eps + \mu_3 \D_I^\eps \otimes ( \u_R^\eps \cdot \nabla \D_I^\eps + \B_R^\eps \D_I^\eps ) \\
+ & \mu_5 (\A_R^\eps \D_I^\eps) \otimes \D_I^\eps + \mu_6 \D_I^\eps \otimes (\A_R^\eps \D_I^\eps) \,,
\end{aligned}
\end{equation}
the term $\mathcal{Q}_{\u}^5$ is
\begin{equation}
\begin{aligned}
\mathcal{Q}_{\u}^5 = & \mu_1 \big[ 2 ( \A_0 : \D_I^\eps \otimes \dr_0 ) \D_I^\eps \otimes \D_I^\eps + ( \A_0 : \D_I^\eps \otimes \D_I^\eps ) ( \D_I^\eps \otimes \dr_0 + \dr_0 \otimes \D_I^\eps ) \big] \\
+ & \mu_1 \big[ ( \A_0 : \D_I^\eps \otimes \D_I^\eps ) \dr_R^\eps \otimes \dr_R^\eps + 2 ( \A_R^\eps : \D_I^\eps \otimes \dr_R^\eps ) ( \dr_0 \otimes \D_I^\eps + \D_I^\eps \otimes \dr_0 ) \big] \\
+ & \mu_1 ( 2 \A_R^\eps : \dr_R^\eps \otimes \dr_0 + \A_0 : \dr_R^\eps \otimes \dr_R^\eps ) \D_I^\eps \otimes \D_I^\eps \\
+ & 2 \mu_1 ( \A_R^\eps : \D_I^\eps \otimes \dr_0 + \A_0 : \D_I^\eps \otimes \dr_R^\eps ) ( \D_I^\eps \otimes \dr_R^\eps + \dr_R^\eps \otimes \D_I^\eps ) \\
+ & \mu_1 \big[ 2 ( \A_R^\eps : \D_I^\eps \otimes \dr_R^\eps ) \dr_R^\eps \otimes \dr_R^\eps + ( \A_R^\eps : \dr_R^\eps \otimes \dr_R^\eps ) ( \D_I^\eps \otimes \dr_R^\eps + \dr_R^\eps \otimes \D_I^\eps ) \big] \,,
\end{aligned}
\end{equation}
the term $\mathcal{Q}_{\u}^6$ is
\begin{equation}
\begin{aligned}
\mathcal{Q}_{\u}^6 = & \mu_1 \big[ 2 ( \A_R^\eps : \D_I^\eps \otimes \dr_0 ) \D_I^\eps \otimes \D_I^\eps + ( \A_R^\eps : \D_I^\eps \otimes \D_I^\eps ) ( \dr_0 \otimes \D_I^\eps + \D_I^\eps \otimes \dr_0 ) \big] \\
+ & \mu_1 \big[ 2 ( \A_0 : \D_I^\eps \otimes \dr_R^\eps ) \D_I^\eps \otimes \D_I^\eps + ( \A_0 : \D_I^\eps \otimes \D_I^\eps ) ( \dr_0 \otimes \D_I^\eps + \D_I^\eps \otimes \dr_0 ) \big] \\
+ & \mu_1 \big[ ( \A_R^\eps : \D_I^\eps \otimes \D_I^\eps ) \dr_R^\eps \otimes \dr_R^\eps + ( \A_R^\eps : \dr_R^\eps \otimes \dr_R^\eps ) \D_I^\eps \otimes \D_I^\eps \big] \\
+ & 2 \mu_1 ( \A_R^\eps : \D_I^\eps \otimes \dr_R^\eps  ) ( \dr_R^\eps \otimes \D_I^\eps + \D_I^\eps \otimes \dr_R^\eps ) \,,
\end{aligned}
\end{equation}
the term $\mathcal{Q}_{\u}^7$ is
\begin{equation}
\begin{aligned}
\mathcal{Q}_{\u}^7 = & \mu_1 (  \A_0 : \D_I^\eps \otimes \D_I^\eps + 2 \A_R^\eps : \D_I^\eps \otimes \dr_R^\eps ) \D_I^\eps \otimes \D_I^\eps \\
+ & \mu_1 ( \A_R^\eps : \D_I^\eps \otimes \D_I^\eps ) ( \dr_R^\eps \otimes \D_I^\eps + \D_I^\eps \otimes \dr_R^\eps )
\end{aligned}
\end{equation}
and the term $\mathcal{Q}_{\u}^8$ is
\begin{equation}
\begin{aligned}
\mathcal{Q}_{\u}^8 = \mu_1 ( \A_R^\eps : \D_I^\eps \otimes \D_I^\eps ) \D_I^\eps \otimes \D_I^\eps \,.
\end{aligned}
\end{equation}
Finally, the vector field term $\mathcal{Q}_{\dr} (\D_I)$ involving the initial layer structure in the $\dr_R^\eps$-equation of \eqref{Remainder-u-d} is defined as
\begin{equation}\label{Q-d}
\begin{aligned}
\mathcal{Q}_{\dr}(\D_I) = \tfrac{1}{\sqrt{\eps}} \mathcal{Q}_{\dr}^1 + \mathcal{Q}_{\dr}^2 + \sqrt{\eps} \mathcal{Q}_{\dr}^3 + \sqrt{\eps}^2 \mathcal{Q}_{\dr}^4 + \sqrt{\eps}^3 \mathcal{Q}_{\dr}^5 + \sqrt{\eps}^4 \mathcal{Q}_{\dr}^6 + \sqrt{\eps}^5 \mathcal{Q}_{\dr}^7 + \sqrt{\eps}^6 \mathcal{Q}_{\dr}^8 \,,
\end{aligned}
\end{equation}
where the term $\mathcal{Q}_{\dr}^1$ is
\begin{equation}
\begin{aligned}
\mathcal{Q}_{\dr}^1 = & \lambda_1 ( \u_0 \cdot \nabla \D_I^\eps + \B_0 \D_I^\eps ) + \lambda_2 \A_0 \D_I^\eps + \bm{\gamma}_0 \D_I^\eps + 2 ( \nabla \dr_0 \cdot \nabla \D_I^\eps - \lambda_2 \A_0 : \D_I^\eps \otimes \dr_0 ) \dr_0 \,,
\end{aligned}
\end{equation}
the term $\mathcal{Q}_{\dr}^2$ is
\begin{equation}
\begin{aligned}
\mathcal{Q}_{\dr}^2 = & \lambda_1 ( \u_R^\eps \cdot \D_I^\eps + \B_R^\eps \D_I^\eps ) + \lambda_2 \A_R^\eps \D_I^\eps + ( 2 \nabla \dr_0 \cdot \nabla \D_I^\eps - 2 \lambda_1 \A_0 : \D_I^\eps \otimes \dr_0 ) \dr_R^\eps \\
& + ( 2 \nabla \D_I^\eps \cdot \nabla \dr_R^\eps - \lambda_2 \A_0 : \D_I^\eps \otimes \dr_R^\eps - 2 \lambda_2 \A_R^\eps : \D_I^\eps \otimes \dr_0 ) \dr_0 \\
& + ( 2 \nabla \dr_0 \cdot \nabla \dr_R^\eps - 2 \lambda_2 \A_0 : \dr_0 \otimes \dr_R^\eps - \lambda_2 \A_R^\eps : \dr_0 \otimes \dr_0 ) \D_I^\eps \,,
\end{aligned}
\end{equation}
the term $\mathcal{Q}_{\dr}^3$ is
\begin{equation}
\begin{aligned}
\mathcal{Q}_{\dr}^3 = & - 2 \u_0 \cdot \nabla \partial_t \D_I^\eps - \partial_t \u_0 \cdot \nabla \D_I^\eps - \u_0 \cdot \nabla ( \u_0 \cdot \nabla \D_I^\eps ) \\
+ & ( |\nabla \D_I^\eps|^2 - \lambda_2 \A_0 : \D_I^\eps \otimes \D_I^\eps ) \dr_0 + 2 ( \nabla \dr_0 \cdot \nabla \D_I^\eps - \lambda_2 \A_0 : \D_I^\eps \otimes \dr_0 ) \D_I^\eps \\
+ & 2 ( \nabla \dr_R^\eps \cdot \nabla \D_I^\eps - \lambda_2 \A_0 : \D_I^\eps \otimes \dr_R^\eps - \lambda_2 \A_R^\eps \D_I^\eps \otimes \dr_0 ) \dr_R^\eps - 2 ( \D_{\u_0} \dr_0 \cdot \D_{\u_0} \D_I^\eps ) \dr_0 \\
- & \u_R^\eps \cdot \nabla ( \u_R^\eps \cdot \nabla \D_I^\eps ) - 2 ( \D_{\u_0} \dr_0 \cdot \D_{\u_0} \D_I^\eps + \lambda_2 \A_R^\eps : \dr_R^\eps \otimes \D_I^\eps ) \dr_0 - |\D_{\u_0} \dr_0|^2 \D_I^\eps \,,
\end{aligned}
\end{equation}
the term $\mathcal{Q}_{\dr}^4$ is
\begin{equation}
\begin{aligned}
\mathcal{Q}_{\dr}^4 = & ( |\nabla \D_I^\eps|^2 - \lambda_2 \A_0 : \D_I^\eps \otimes \D_I^\eps ) \dr_R^\eps - \lambda_2 ( \A_R^\eps : \D_I^\eps \otimes \D_I^\eps ) \dr_0 \\
+ & 2 ( \nabla \dr_R^\eps \cdot \nabla \D_I^\eps - \lambda_2 \A_0 : \D_I^\eps \otimes \dr_R^\eps - \lambda_2 \A_R^\eps : \D_I^\eps \otimes \dr_0 ) \D_I^\eps \\
- & \u_R^\eps \cdot \nabla \D_{\u_0} \D_I^\eps - 2 ( \D_{\u_0} \dr_0 \cdot \D_{\u_0} \D_I^\eps + \lambda_2 \A_R^\eps : \D_I^\eps \otimes \dr_R^\eps ) \dr_R^\eps \\
- & 2 ( \D_{\u_0} \dr_0 \cdot ( \u_R^\eps \cdot \nabla \D_I^\eps ) + ( \u_R^\eps \cdot \nabla \dr_0 + \D_{\u_0 + \sqrt{\eps} \u_R^\eps } \dr_R^\eps ) \cdot \D_{\u_0} \D_I^\eps ) \dr_0 \\
- & \big[ \lambda_2 \A_R^\eps : \dr_R^\eps \otimes \dr_R^\eps + ( \u_R^\eps \cdot \nabla \dr_0 + \D_{\u_0 + \sqrt{\eps} \u_R^\eps} \dr_R^\eps ) \cdot \D_{\u_0} \dr_0 \big] \D_I^\eps \\
- & 2 \big[ \D_{\u_0} \dr_0 \cdot ( \u_R^\eps \cdot \nabla \D_I^\eps ) + \D_{\u_0} \D_I^\eps \cdot ( \u_R^\eps \cdot \nabla \dr_0 + \D_{\u_0 + \sqrt{\eps} \u_R^\eps } \dr_R^\eps ) \big] \dr_0 \\
- & 2 \D_{\u_0} \dr_0 \cdot ( \u_R^\eps \cdot \nabla \dr_0 + \D_{\u_0 + \sqrt{\eps} \u_R^\eps } \dr_R^\eps ) \D_I^\eps - 2 ( \D_{\u_0} \dr_0 \cdot \D_{\u_0}  \D_I^\eps ) \dr_R^\eps \,,
\end{aligned}
\end{equation}
the term $\mathcal{Q}_{\dr}^5$ is
\begin{equation}
\begin{aligned}
\mathcal{Q}_{\dr}^5 = & ( |\nabla \D_I^\eps|^2 - \lambda_2 \A_0 : \D_I^\eps \otimes \D_I^\eps ) \D_I^\eps - |\D_{\u_0} \D_I^\eps|^2 \dr_0 - \lambda_2 ( \A_R^\eps : \D_I^\eps \otimes \D_I^\eps ) \dr_R^\eps \\
- & 2 ( \D_{\u_0} \dr_0 \cdot \D_{\u_0} \D_I^\eps ) \dr_0 + 2 ( \u_R^\eps \cdot \nabla \dr_0 + \D_{\u_0 + \sqrt{\eps} \u_R^\eps } \dr_R^\eps ) \cdot ( \u_R^\eps \cdot \nabla \D_I^\eps ) \dr_0 \\
- & 2 \big[ \D_{\u_0} \dr_0 \cdot ( \u_R^\eps \cdot \nabla \D_I^\eps ) + ( \u_R^\eps \cdot \nabla \dr_0 + \D_{\u_0 + \sqrt{\eps} \u_R^\eps } \dr_R^\eps ) \cdot \D_{\u_0} \D_I^\eps \big] \dr_R^\eps \\
- & | \u_R^\eps \cdot \nabla \dr_0 + \D_{\u_0 + \sqrt{\eps} \u_R^\eps } \dr_R^\eps  |^2 \D_I^\eps \,,
\end{aligned}
\end{equation}
the term $\mathcal{Q}_{\dr}^6$ is
\begin{equation}
\begin{aligned}
\mathcal{Q}_{\dr}^6 = & - \lambda_2 ( \A_R^\eps : \D_I^\eps \otimes \D_I^\eps ) \D_I^\eps - |\D_{\u_0} \D_I^\eps|^2 \dr_R^\eps - 2 \D_{\u_0} \D_I^\eps \cdot ( \u_R^\eps \cdot \nabla \D_I^\eps ) \dr_0 \\
- & 2 \D_{\u_0} \dr_0 \cdot ( \u_R^\eps \cdot \nabla \D_I^\eps ) \D_I^\eps - 2 \D_{\u_0} \D_I^\eps \cdot ( \u_R^\eps \cdot \nabla \dr_0 + \D_{\u_0 + \sqrt{\eps} \u_R^\eps } \dr_R^\eps ) \D_I^\eps - |\D_{\u_0} \D_I^\eps|^2 \dr_R^\eps \\
- & 2 \big[ \D_{\u_0} \dr_0 \cdot ( \u_R^\eps \cdot \nabla \D_I^\eps ) + \D_{\u_0} \D_I^\eps \cdot ( \u_R^\eps \cdot \nabla \dr_0 + \D_{\u_0 + \sqrt{\eps} \u_R^\eps } \dr_R^\eps ) \big] \D_I^\eps \\
- &  2 \D_{\u_0} \D_I^\eps \cdot ( \u_R^\eps \cdot \nabla \D_I^\eps ) \dr_0 - 2 ( \u_R^\eps \cdot \nabla \dr_0 ) \cdot ( \u_R^\eps \cdot \nabla \D_I^\eps ) \dr_R^\eps \\
- & 2 ( \u_R^\eps \cdot \nabla \dr_0 + \D_{\u_0 + \sqrt{\eps} \u_R^\eps } \dr_R^\eps ) \cdot ( \u_R^\eps \cdot \nabla \D_I^\eps ) \dr_R^\eps - 2 \D_{\u_0 + \sqrt{\eps} \u_R^\eps } \dr_R^\eps \cdot ( \u_R^\eps \cdot \nabla \D_I^\eps ) \dr_R^\eps \,,
\end{aligned}
\end{equation}
the term $\mathcal{Q}_{\dr}^7$ is
\begin{equation}
\begin{aligned}
\mathcal{Q}_{\dr}^7 = & - | \D_{\u_0} \D_I^\eps |^2 \D_I^\eps - |\u_R^\eps \cdot \nabla \D_I^\eps|^2 \dr_0 - 2 \D_{\u_0} \D_I^\eps \cdot ( \u_R^\eps \cdot \nabla \D_I^\eps ) \dr_R^\eps \\
& - 2 ( \u_R^\eps \cdot \nabla \D_I^\eps ) \cdot ( \u_R^\eps \cdot \nabla \dr_0 + \D_{\u_0 + \sqrt{\eps} \u_R^\eps } \dr_R^\eps ) \D_I^\eps
\end{aligned}
\end{equation}
and the term $\mathcal{Q}_{\dr}^8$ is
\begin{equation}
\begin{aligned}
\mathcal{Q}_{\dr}^8 = & - 2 \D_{\u_0} \D_I^\eps \cdot ( \u_R^\eps \cdot \nabla \D_I^\eps ) \D_I^\eps - |\u_R^\eps \cdot \nabla \D_I^\eps|^2 ( \dr_R^\eps + \sqrt{\eps} \D_I^\eps ) \,.
\end{aligned}
\end{equation}


\section*{Acknowledgment}

This work is supported by the grants from the National Natural Science Foundation of China under contract No. 11471181 and No. 11731008.

\bigskip

\bibliography{reference}

\end{document}